\documentclass[12pt]{article}

\usepackage{amsmath,amsthm,amssymb,mathtools}
\usepackage{a4wide,color,tikz-cd,hyperref,enumitem}
\usepackage[T1]{fontenc}   
\hypersetup{colorlinks=true,linkcolor=blue,citecolor=blue,urlcolor=blue,linktocpage}
\setlist[enumerate,1]{label={(\roman*)}}
\DeclareFontFamily{U}{mathc}{}
\DeclareFontShape{U}{mathc}{m}{it}%
{<->s*[1] mathc10}{}
\DeclareMathAlphabet{\mathscr}{U}{mathc}{m}{it}

\newtheorem{thm}{Theorem}[subsection]
\newtheorem{prp}[thm]{Proposition}
\newtheorem{lem}[thm]{Lemma}
\newtheorem{cor}[thm]{Corollary}
\newtheorem{con}[thm]{Conjecture}

\newtheorem{thm-intro}{Theorem}
\newtheorem{cor-intro}[thm-intro]{Corollary}

\theoremstyle{definition}
\newtheorem{dfn}[thm]{Definition}
\newtheorem{exm}[thm]{Example}
\newtheorem{rmk}[thm]{Remark}
 
\theoremstyle{remark}
\newtheorem{prb}{Problem}

\newcommand{\mx}[1]{\begin{pmatrix}#1\end{pmatrix}}
\newcommand{\vb}{\,|\,}

\newcommand{\Z}{\mathbb{Z}}

\newcommand{\R}{\mathbb{R}}
\newcommand{\C}{\mathbb{C}}
\renewcommand{\P}{\mathbb{P}}
\newcommand{\RP}{\mathbb{RP}}
\newcommand{\cA}{\mathcal{A}}
\newcommand{\cB}{\mathcal{B}}

\newcommand{\cC}{\mathcal{C}}
\newcommand{\cD}{\mathcal{D}}
\newcommand{\cF}{\mathcal{F}}

\newcommand{\cL}{\mathcal{L}}
\newcommand{\cS}{\mathcal{S}}
\newcommand{\cW}{\mathcal{W}}
\newcommand{\cP}{\mathcal{P}}

\newcommand{\cT}{\mathcal{T}}
\newcommand{\cM}{\mathcal{M}}
\newcommand{\fX}{\mathfrak{X}}
\newcommand{\fc}{\mathfrak{c}}
\newcommand{\fr}{\mathfrak{r}}
\newcommand{\fb}{\mathfrak{b}}
\newcommand{\fm}{\mathfrak{m}}
\newcommand{\fd}{\mathfrak{d}}
\newcommand{\fs}{\mathfrak{s}}

\newcommand{\inter}{\textup{int}}

\newcommand{\id}{\textup{id}}
\newcommand{\supp}{\textup{supp}}
\newcommand{\sing}{\textup{ss}}
\newcommand{\colim}{\textup{colim}}
\newcommand{\holim}{\textup{holim}}

\newcommand{\HW}{\textup{HW}}
\newcommand{\CF}{\textup{CF}}
\newcommand{\CW}{\textup{CW}}
\newcommand{\starr}{\textup{star}}
\newcommand{\tb}{\textup{tb}}

\newcommand{\dd}{^{\diamond}}
\renewcommand{\th}{^{\text{th}}}
\newcommand{\st}{^{\text{st}}}
\newcommand{\nd}{^{\text{nd}}}
\newcommand{\Zero}{\textup{Zero}}
\newcommand{\Crit}{\textup{Crit}}
\newcommand{\dgCat}{\textup{dgCat}}
\newcommand{\op}{\textup{op}}
\newcommand{\dg}{\textup{dg}}
\newcommand{\cox}{\boldsymbol{c}}
\newcommand{\mon}{\boldsymbol{m}}
\renewcommand{\k}{\boldsymbol{k}}
\newcommand{\sgn}{\textup{sgn}}

\newcommand{\modk}{\textup{Mod}_{\k}}
\newcommand{\loc}{\textup{Loc}}
\newcommand{\sh}{\textup{Sh}}

\newcommand{\Sh}{\mathscr{Sh}}
\newcommand{\Modk}{\mathscr{Mod}_{\k}}
\newcommand{\Loc}{\mathscr{Loc}}
\newcommand{\mSh}{\mathscr{\mu Sh}}
\newcommand{\Hom}{\mathscr{Hom}}

\newcommand{\Perf}{\mathscr{Perf}}
\newcommand{\Perfk}{\mathscr{Perf}_{\k}}
\newcommand{\Fun}{\mathscr{Fun}}
\newcommand{\KS}{\mathscr{KS}}
\newcommand{\CE}{\mathscr{CE}}

\title{Microlocal Sheaves on Pinwheels}
\author{Do\u{g}ancan Karaba\c{s}}
\date{\today}

\begin{document}

\maketitle

\begin{abstract}
	In this thesis, we study the wrapped Fukaya category of the rational homology ball $B_{p,q}$ and the traditional/wrapped microlocal sheaves on its skeleton $L_{p,q}$, called pinwheel. We explicitly calculate both for $q=1$, and show they match in wrapped case.
\end{abstract}

\tableofcontents

\section{Introduction}

\subsection{Summary}

Given coprime integers $p$ and $q$ such that $p\geq 2$ and $0<q<p$, the \textit{rational homology ball $B_{p,q}$} is defined as a certain quotient of $A_{p-1}$-Milnor fibre. Explicitly, if we present $A_{p-1}$-Milnor fibre as
\[A_{p-1}:=\{(x,y,z)\vb z^p+2xy=1\}\subset\C^3\]
then $B_{p,q}:=A_{p-1}/\Gamma_{p,q}$ where
\begin{align*}
	\Gamma_{p,q}\colon\Z/p\times A_{p-1} &\rightarrow A_{p-1}\\
	(\xi,(x,y,z))&\mapsto(\xi x,\xi^{-1}y,\xi^q z)
\end{align*}
where $\Z/p$ is presented as $\{e^{2\pi i k/p}\vb k\in\Z\}$. There is a Lefschetz fibration of $A_{p-1}$ given by
\begin{align*}
	\pi\colon A_{p-1} &\rightarrow \C\\
	(x,y,z)&\mapsto z
\end{align*}
which has $p$ critical points. The action of $\Gamma_{p,q}$ can be visualised using the Lefschetz fibration: It rotates the base $\C_z$ around the origin by the angle $2\pi q/p$, and the fibres by $2\pi/p$. The Lefschetz fibration of $A_{p-1}$ is shown in Figure \ref{fig:intro-lef-fib}, where the crosses are the critical values of $\pi$. See Section \ref{sec:rational-homology-ball} for more detail.

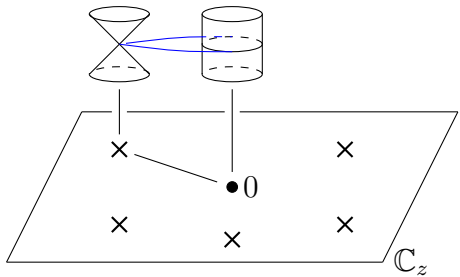
\begin{figure}[h]
	\centering
	
	\begin{tikzpicture}
	\draw (-3,-1) -- (2,-1) node[right]{$\C_z$};;
	\draw (-2,1) -- (-1.6,1);
	\draw (-1.4,1) -- (-0.1,1);
	\draw (0.1,1) -- (3,1);
	\draw (-3,-1) -- (-2,1);
	\draw (2,-1) -- (3,1);
	
	\newcommand{\cross}[3][thick]{
		\draw[#1] (#2-0.1,#3-0.1) -- (#2+0.1,#3+0.1);
		\draw[#1] (#2-0.1,#3+0.1) -- (#2+0.1,#3-0.1)
	}
	
	\draw[fill] (0,0) circle (2pt) node[right]{$0$};
	\cross{1.5}{0.5};
	\cross{1.5}{-0.5};
	\cross{0}{-0.7};
	\cross{-1.5}{-0.5};
	\cross{-1.5}{0.5};
	
	\draw[domain=-1.3:-0.2,variable=\x] plot ({\x},{(0.5/-1.5)*\x});
	\draw (0,0.2) -- (0,1.3);
	\draw (-1.5,0.7) -- (-1.5,1.3);
	
	\draw (0.4,1.5) arc (0:-180:0.4 and 0.1);
	\draw[dashed] (0.4,1.5) arc (0:180:0.4 and 0.1);
	\draw (0.4,1.9) arc (0:-180:0.4 and 0.1);
	\draw[dashed] (0.4,1.9) arc (0:180:0.4 and 0.1);
	\draw (0,2.3) ellipse (0.4 and 0.1);
	\draw (-0.4,1.5) -- (-0.4,2.3);
	\draw (0.4,1.5) -- (0.4,2.3);
	
	\draw (-1.1,1.5) arc (0:-180:0.4 and 0.1);
	\draw[dashed] (-1.1,1.5) arc (0:180:0.4 and 0.1);
	\draw (-1.9,1.5) -- (-1.1,2.3);
	\draw (-1.1,1.5) -- (-1.9,2.3);
	\draw (-1.5,2.3) ellipse (0.4 and 0.1);
	
	\draw[blue] (-1.5,1.9) .. controls (-0.7,1.8) .. (0,1.8);
	
	\begin{scope}
	\clip (-0.4,1.5) rectangle (-1.4,2.5);
	\draw[blue] (-1.5,1.9) .. controls (-0.7,2) .. (0,2);
	\end{scope}
	
	\begin{scope}
	\clip (-0.4,1.5) rectangle (0.4,2.5);
	\draw[blue,dashed] (-1.5,1.9) .. controls (-0.7,2) .. (0,2);
	\end{scope}
	
	\end{tikzpicture}
	
	\caption{Lefschetz fibration $\pi\colon A_{p-1}\to\C_z$}
	\label{fig:intro-lef-fib}
\end{figure}

$B_{p,q}$ is a $4$-dimensional Weinstein manifold, moreover in \cite{lekili}, it is shown that there is no closed exact embedded Lagrangian in $B_{p,q}$. This makes the Fukaya categories of $B_{p,q}$ using the definition in \cite{seidel} uninteresting if we only allow closed exact embedded Lagrangians inside. Instead, we must consider immersed Lagrangians also. Note that $2c_1(B_{p,q})\neq 0\in H^2(B_{p,q};\Z)\simeq\Z/p$ if $p\geq 3$. Hence we consider Fukaya categories as $\Z/2$-graded, and get the following theorem:

\begin{thm-intro}[Theorem \ref{thm:wrapped-rational}]\label{thm:intro-wrapped}
	Let $\k$ be a field of characteristic zero. For $p\geq 3$, the $\k$-linear wrapped Fukaya category $\cW(B_{p,1})$ is generated by a Lagrangian cocore whose endomorphism algebra is given by the semifree differential graded algebra (dga) $\cA_{p,1}$ generated by the degree $1$ elements $x_i$ for $p\geq i\geq 1$, and $y_{ij}$ for $p\geq i,j\geq 1$, such that
	\begin{align*}
		dx_i&=-\delta_{i,p}+\sum_{j=1}^{i-1} x_{i-j}\circ x_j\\
		dy_{ij} &= \delta_{ij} + \sum_{k=1}^{i-1} x_{i-k}\circ y_{kj}+\sum_{k=j+1}^{p}y_{ik}\circ  x_{k-j}
	\end{align*}
	where $\delta_{ij}$ is Kronecker delta. Hence, there is an $A_{\infty}$-quasi-equivalence of pretriangulated $A_{\infty}$-categories over $\k$
	\[\cW(B_{p,1})\simeq\Perf(\cA_{p,1})\]
	where $\Perf(\cA_{p,1})$ is the triangulated envelope of $\cA_{p,1}$.
\end{thm-intro}

We prove this as follows: In \cite{CDRGG} it is shown that Lagrangian cocores, cocores of critical (i.e. $n$-dimensional) Weinstein handles of a $2n$-dimensional Weinstein manifold, generate the wrapped Fukaya category. Moreover, for any field $\k$ of characteristic zero, \cite{bee} and \cite[Theorem 2]{ekholm-lekili} showed that the endomorphism algebra of Lagrangian cocores is given by the Chekanov-Eliashberg dga $\CE^*(\Lambda)$ of the Legendrian attaching sphere $\Lambda$ of the critical handles. For $4$-dimensional Weinstein manifolds, $\CE^*(\Lambda)$ can be explained combinatorially using the Legendrian surgery diagram by \cite{subcritical}.

By inspecting the Lefschetz fibration of $A_{p-1}$, we see that it consists of one $0$-handle, one $1$-handle, and $p$ $2$-handles whose cores are Lefschetz thimbles, and attached as shown in Figure \ref{fig:intro-leg-milnor}.

\begin{figure}[h]
	\centering
	
	\begin{tikzpicture}
	\draw (-5,0) circle (1);
	\draw (-4,0) arc (0:-180:1 and 0.15);
	\draw[dashed] (-4,0) arc (0:180:1 and 0.15);
	\draw (5,0) circle (1);
	\draw (6,0) arc (0:-180:1 and 0.15);
	\draw[dashed] (6,0) arc (0:180:1 and 0.15);
	
	\node[left] at ({-5+cos(asin(0.9))},0.9){\tiny 1};
	\node[left] at ({-5+cos(asin(0.7))},0.7){\tiny 2};
	\node[left] at ({-5+cos(asin(0.9))},-0.9){\tiny p};
	
	\node[right] at ({5-cos(asin(0.9))},0.9){\tiny 1};
	\node[right] at ({5-cos(asin(0.7))},0.7){\tiny 2};
	\node[right] at ({5-cos(asin(0.9))},-0.9){\tiny p};
	
	\draw ({-5+cos(asin(0.9))},0.9) -- ({5-cos(asin(0.9))},0.9);
	\draw ({-5+cos(asin(0.7))},0.7) -- ({5-cos(asin(0.7))},0.7);
	\draw ({-5+cos(asin(0.9))},-0.9) -- ({5-cos(asin(0.9))},-0.9);
	
	\draw[dashed] (0,0.6) -- (0,-0.8) node[left,midway]{\small $p$};
	\end{tikzpicture}
	
	\caption{Legendrian surgery diagram of $A_{p-1}$}
	\label{fig:intro-leg-milnor}
\end{figure}
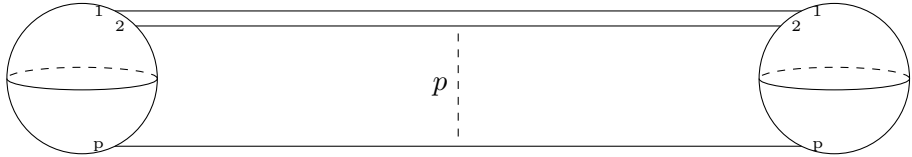

After taking quotient by $\Gamma_{p,1}$, by \cite{lekili} we see that $B_{p,1}$ has one $0$-handle, one $1$-handle, and one $2$-handle which is attached as in Figure \ref{fig:intro-leg-bp1}.

\begin{figure}[h]
	\centering
	
	\begin{tikzpicture}    
	\draw (-5,0) circle (1);
	\draw (-4,0) arc (0:-180:1 and 0.15);
	\draw[dashed] (-4,0) arc (0:180:1 and 0.15);
	\draw (5,0) circle (1);
	\draw (6,0) arc (0:-180:1 and 0.15);
	\draw[dashed] (6,0) arc (0:180:1 and 0.15);
	
	\node[left] at ({-5+cos(asin(0.9))},0.9){\tiny 1};
	\node[left] at ({-5+cos(asin(0.7))},-0.7){\tiny p-1};
	\node[left] at ({-5+cos(asin(0.9))},-0.9){\tiny p};
	
	\node[right] at ({5-cos(asin(0.9))},0.9){\tiny 1};
	\node[right] at ({5-cos(asin(0.7))},0.7){\tiny 2};
	\node[right] at ({5-cos(asin(0.9))},-0.9){\tiny p};
	
	\draw ({-5+cos(asin(0.9))},0.9) -- ({5-cos(asin(0.7))},0.7);
	\draw ({-5+cos(asin(0.7))},-0.7) -- ({5-cos(asin(0.9))},-0.9);
	\draw ({-5+cos(asin(0.9))},-0.9) -- (-2,-1) -- (-3,1.1) -- ({5-cos(asin(0.9))},0.9);
	
	\draw[dashed] (-3.5,0.75) -- (-3.5,-0.65) node[right,midway]{\small $p$};
	\draw[dashed] (3,0.6) -- (3,-0.8);  
	\end{tikzpicture}
	
	\caption{Legendrian surgery diagram of $B_{p,1}$}
	\label{fig:intro-leg-bp1}
\end{figure}
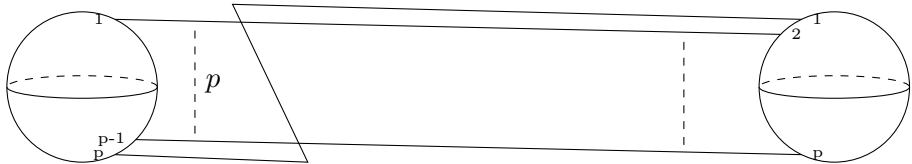

If we denote the attaching circle in Figure \ref{fig:intro-leg-bp1} by $\Lambda_{p,1}$, we get
\[\cW(B_{p,1})\simeq\Perf(\CE^*(\Lambda_{p,1}))\ .\]
$\CE^*(\Lambda_{p,1})$ is a semifree dga which can be calculated by first resolving the Legendrian surgery diagram to get the Lagrangian projection diagram, then its generators are coming from the crossings and the 1-handle as shown in Figure \ref{fig:intro-gen-bp1}, and their differential is determined by counting appropriate disks bounded by the generators in the diagram, see Proposition \ref{prp:ce-bp1} and \cite{subcritical}. Finally, we simplify the dga algebraically to get Theorem \ref{thm:intro-wrapped}, see Section \ref{sec:wrapped-rational} for these calculations.

\begin{figure}[h]
	\centering
	
	\begin{tikzpicture}[scale=1.27]
	\draw (-5,0) circle (1) node[blue]{\small $c_{ij},c'_{ij}$};
	\draw (-4,0) arc (0:-180:1 and 0.15);
	\draw[dashed] (-4,0) arc (0:180:1 and 0.15);
	\draw (5,0) circle (1) node[blue]{\small $c_{ij},c'_{ij}$};
	\draw (6,0) arc (0:-180:1 and 0.15);
	\draw[dashed] (6,0) arc (0:180:1 and 0.15);
	
	\node[left] at ({-5+cos(asin(0.9))},0.9){\tiny $1$};
	\node[left] at ({-5+cos(asin(0.7))},0.7){\tiny $2$};
	\node[left] at ({-5+cos(asin(0.5))},-0.5){\tiny $p-2$};
	\node[left] at ({-5+cos(asin(0.7))},-0.7){\tiny $p-1$};
	\node[left] at ({-5+cos(asin(0.9))},-0.9){\tiny $p$};
	
	\node[right] at ({5-cos(asin(0.9))},0.9){\tiny $p$};
	\node[right] at ({5-cos(asin(0.7))},0.7){\tiny $p-1$};
	\node[right] at ({5-cos(asin(0.5))},-0.5){\tiny $3$};
	\node[right] at ({5-cos(asin(0.7))},-0.7){\tiny $2$};
	\node[right] at ({5-cos(asin(0.9))},-0.9){\tiny $1$};
	
	\draw ({-5+cos(asin(0.9))},-0.9) -- (-3.1,-0.9);
	\draw plot[smooth] coordinates{(-2.9,-0.9) (-2.5,-1) (-2.5,-1.4) (-2.8,-1.4) (-3,-0.9) (-3,1.1) (1.5,1.1) (2.5,-1) (3,-0.9) ({5-cos(asin(0.9))},-0.9)};
	
	\begin{scope}
	\clip (-6,-1.3) rectangle (-3.25,1.3) (-3.1,-1.3) rectangle (2.35,1.3) (2.45,-1.3) rectangle (6,1.3);
	\draw plot[smooth] coordinates{({-5+cos(asin(0.9))},0.9) (1,0.9) (2,-1) (2.5,-0.7) ({5-cos(asin(0.7))},-0.7)};
	\end{scope}
	
	\begin{scope}
	\clip (-6,-1.3) rectangle (-3.35,1.3) (-3.15,-1.3) rectangle (1.8,1.3) 	(1.95,-1.3) rectangle (2.25,1.3) (2.4,-1.3) rectangle (6,1.3);
	\draw plot[smooth] coordinates{({-5+cos(asin(0.7))},0.7) (0.5,0.7) (1.5,-1) (2,-0.5) ({5-cos(asin(0.5))},-0.5)};
	\end{scope}
	
	\begin{scope}
	\clip (-6,-1.3) rectangle (-3.2,1.3) (-3,-1.3) rectangle (-0.2,1.3) (0.4,-1.3) rectangle (0.7,1.3) (0.82,-1.3) rectangle (1.2,1.3) (1.35,-1.3) rectangle (1.8,1.3) (1.95,-1.3) rectangle (6,1.3);
	\draw plot[smooth] coordinates{({-5+cos(asin(0.5))},-0.5) (-1,-0.5) (-0.5,-1) (1,0.7) ({5-cos(asin(0.7))},0.7)};
	\end{scope}
	
	\begin{scope}
	\clip (-6,-1.3) rectangle (-3.15,1.3) (-2.95,-1.3) rectangle (-0.8,1.3) (-0.67,-1.3) rectangle (-0.5,1.3) (0.1,-1.3) rectangle (0.3,1.3) (0.45,-1.3) rectangle (0.7,1.3) (0.95,-1.3) rectangle (1.55,1.3) (1.75,-1.3) rectangle (6,1.3);
	\draw plot[smooth] coordinates{({-5+cos(asin(0.7))},-0.7) (-1.5,-0.7) (-1,-1) (0.7,0.9) ({5-cos(asin(0.9))},0.9)};
	\end{scope}
	
	\draw[dashed] (-3.8,0.6) -- (-3.8,-0.4) node[right,midway]{\small $p$};
	\draw[dashed] (-1.7,0.6) -- (-1.7,-0.4);
	\draw[dashed] (-0.2,-0.3) -- (0.1,0.2);
	\draw[dashed] (-0.2,-0.9) -- (1.2,-0.9);
	\draw[dashed] (1.2,0.5) -- (1.6,-0.4);
	\draw[dashed] (1.7,0.6) -- (2.1,-0.3);
	\draw[dashed] (3.2,0.6) -- (3.2,-0.4);
	
	\fill[blue] (-3.18,0.95) circle (1pt) node[above left]{\tiny $a_1$};
	\fill[blue] (-3.25,0.76) circle (1pt) node[above left,yshift=-0.1cm]{\tiny $a_2$};
	\fill[blue] (-3.1,-0.48) circle (1pt) node[above left,yshift=-0.1cm]{\tiny $a_{p-2}$};
	\fill[blue] (-3.05,-0.68) circle (1pt) node[above left,yshift=-0.1cm]{\tiny $a_{p-1}$};
	\fill[blue] (-3,-0.9) circle (1pt) node[above left,yshift=-0.1cm]{\tiny $a_p$};
	
	\fill[blue] (1.67,1) circle (1pt) node[above]{\tiny $b_{p1}$};
	\fill[blue] (0.82,0.94) circle (1pt) node[above]{\tiny $b_{p2}$};
	\fill[blue] (0.38,0.72) circle (1pt) node[above]{\tiny $b_{p3}$};
	\fill[blue] (-0.73,-0.67) circle (1pt) node[above]{\tiny $b_{p(p-1)}$};
	\fill[blue] (1.88,0.8) circle (1pt) node[above right,yshift=-0.14cm]{\tiny $b_{(p-1)1}$};
	\fill[blue] (1.28,0.77) circle (1pt) node[above,yshift=-0.05cm]{\tiny $b_{(p-1)2}$};
	\fill[blue] (0.76,0.57) circle (1pt) node[above,yshift=-0.05cm]{\tiny $b_{(p-1)3}$};
	\fill[blue] (2.34,-0.48) circle (1pt) node[above right,yshift=-0.08cm]{\tiny $b_{31}$};
	\fill[blue] (1.86,-0.55) circle (1pt) node[above right]{\tiny $b_{32}$};
	\fill[blue] (2.4,-0.74) circle (1pt) node[above right,yshift=-0.06cm]{\tiny $b_{21}$};
	
	\end{tikzpicture}
	
	\caption{Generators of $\CE^*(\Lambda_{p,1})$ on the Lagrangian projection of $\Lambda_{p,1}$}
	\label{fig:intro-gen-bp1}
\end{figure}
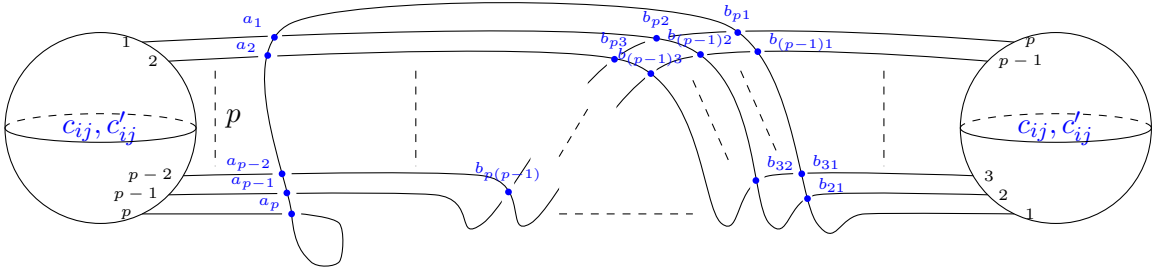

Next, we study the microlocal sheaves on the skeleton of $B_{p,1}$. The skeleton of $A_{p-1}$ is obtained by the union of the cores of its Weinstein handles which is given by $V_p\times S^1$, where $V_p$ is $p$-valent vertex (equally spaced $p$ spikes attached to a point), and then attaching disk to each of the $p$ connected component of its boundary. To get the skeleton of $B_{p,1}$, we divide it by the action of $\Gamma_{p,1}$. Then its skeleton is described as follows: First we take $V_p\times[0,1]$, then identify $V_p\times\{0\}$ with $V_p\times\{1\}$ after rotating the former by the angle $2\pi/p$ clockwise. We denote the resulting space by $S_{p,1}$. Its boundary is a circle, and we attach a disk to it to get the skeleton of $B_{p,1}$, which we denote by $L_{p,1}$, and call it a \textit{pinwheel} (a terminology introduced in \cite{evans}). A neighbourhood of its core circle is shown in Figure \ref{fig:intro-core-circle}.

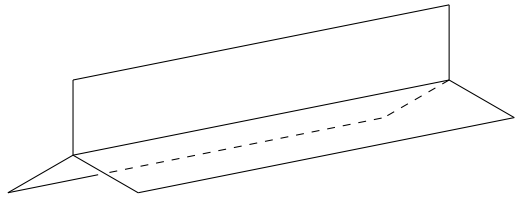
\begin{figure}[h]
	\centering
	
	\begin{tikzpicture}
	\draw (0,0) -- (0,1);
	\draw (0,0) -- ({-sqrt(3)/2},-0.5);
	\draw (0,0) -- ({sqrt(3)/2},-0.5);
	\draw (0,0) -- (5,1);
	\draw (0,1) -- (5,2);
	\draw ({sqrt(3)/2},-0.5) -- ({5+sqrt(3)/2},0.5);
	\draw ({-sqrt(3)/2},-0.5) -- ({5*0.24-sqrt(3)/2},-0.5+0.24);
	\draw[dashed] ({5*0.27-sqrt(3)/2},-0.5+0.27) -- ({5-sqrt(3)/2},0.5);
	\draw (5,1) -- (5,2);
	\draw (5,1) -- ({5+sqrt(3)/2},0.5);
	\draw[dashed] (5,1) -- ({5-sqrt(3)/2},0.5);
	
	\end{tikzpicture}
	
	\caption{A neighbourhood of an arc in the core circle of $L_{3,1}$}
	\label{fig:intro-core-circle}
\end{figure}

For a commutative ring $\k$, traditional (resp. wrapped) microlocal sheaves $\mSh(L)$ (resp. $\mSh^w(L)$) on a skeleton $L$, studied in \cite{combinatorics} and \cite{wrapped} (see Section \ref{sec:microlocal} for a review), form a $\k$-linear dg category, and it can be calculated by first finding microlocal sheaves on the small pieces of the skeleton and then gluing them together by taking homotopy limit (resp. homotopy colimit) in the category of dg categories $\dgCat$, which is studied in Section \ref{sec:dgcat}. It is conjectured (Conjecture \ref{con:microlocal}, \ref{con:microlocal-wrapped}) that the compact (resp. wrapped) Fukaya category of a Weinstein manifold and the traditional (resp. wrapped) microlocal sheaves on its skeleton are quasi-equivalent. First conjecture is proved for cotangent bundles by Nadler and Zaslow in \cite{nadler-zaslow,microbranes}, where traditional microlocal sheaves are defined for a (singular) conic Lagrangian $L$ inside a cotangent bundle $T^*M$ as the dg category $\Sh_L(M)$ of the complexes of sheaves of $\k$-modules on the base $M$ whose singular support lies inside $L$ and the cohomology of its stalks is of finite rank. Second conjecture is proved for cotangent bundles by Ganatra, Pardon, and Shende in \cite{gps3}, where wrapped microlocal sheaves are defined as the full dg subcategory $\Sh^w_L(M)$ of compact objects in $\Sh\dd_L(M)$, which is the dg category defined similar to $\Sh_L(M)$ where we drop the ``finite-rank stalk'' condition. See Chapter \ref{chp:sheaves} for the precise definitions, and Section \ref{sec:fukaya-weinstein} for the review of these results.

Locally, a neighbourhood of the skeleton of a Weinstein manifold can be seen as a cotangent bundle by Darboux theorem, so for a general Weinstein manifold we can define microlocal sheaves locally by carrying its skeleton to a cotangent bundle and doing the calculations there. We prove the following theorem for $L_{p,1}$:

\begin{thm-intro}[Theorem \ref{thm:msh-pinwheel}]
	Let $\k$ be a commutative ring. For $p\geq 3$, microlocal sheaves on the pinwheel $L_{p,1}$ are given by
	\begin{align*}
		\mSh(L_{p,1})&\simeq\Perfk(\cA_{p,1})\\
		\mSh^w(L_{p,1})&\simeq\Perf(\cA_{p,1})
	\end{align*}
	where $\cA_{p,1}$ is the semifree dga in Theorem \ref{thm:intro-wrapped}.
\end{thm-intro}

$\Perfk(\cA_{p,1})$ stands for the dg category of the $A_{\infty}$-functors from $\cA_{p,1}^\op$ to the (co)chain complexes of $\k$-modules with cohomology of finite rank. We get the immediate corollary which confirms Conjecture \ref{con:microlocal-wrapped} for the Weinstein manifolds $B_{p,1}$:

\begin{cor-intro}[Corollary \ref{cor:bp1-lp1}]
	Let $\k$ be a field of characteristic zero. For $p\geq 3$, we have the $A_{\infty}$-quasi-equivalence of pretriangulated $A_{\infty}$-categories over $\k$
	\[\mSh^w(L_{p,1})\simeq\cW(B_{p,1})\ .\]
\end{cor-intro}

To prove the theorem, our strategy is to calculate large microlocal sheaves $\mSh\dd(L_{p,1})$ instead, which is defined similar to $\mSh(L_{p,1})$ except that $\mSh(L_{p,1})$ necessarily contains the complexes of sheaves of $\k$-modules such that the cohomology of its stalks is of finite rank, whereas $\mSh\dd(L_{p,1})$ does not have a boundedness condition on the cohomology, so it is a much larger category. Once we know $\mSh\dd(L_{p,1})$, we can also find $\mSh^w(L_{p,1})$. By definition in \cite{wrapped}, it is the full dg subcategory of compact objects inside $\mSh\dd(L_{p,1})$.

First, we calculate the (large) microlocal sheaves on its local piece $V_p\times (0,1)$. By the stabilisation property (Proposition \ref{prp:stabilisation}), we have $\mSh\dd(V_p\times (0,1))\simeq \mSh\dd(V_p)$. Microlocal sheaves on the $p$-valent vertex are obtained as follows: We symplectically embed $V_p$ with its neighbourhood to $T^*\R$. However, for $p>4$ it is impossible to make $V_p$ a conic Lagrangian in $T^*\R$. Therefore, we lift $V_p$ to a Legendrian $V_p$ in $T^{\infty}\R^2$ by a process explained in Section \ref{sec:microlocal}, which preserves the exact symplectic geometry of $V_p$. Then we take the cone $\R_{>0}V_p$ of $V_p\subset T^{\infty}\R^2$ and add the zero section $0_{\R^2}$, we call $V_p'\subset T^*\R^2$ the resulting singular Lagrangian. Then we have
\[\mSh\dd(V_p)\simeq\mSh\dd(V_p')/\Loc\dd(\R^2)\]
where we divide by local systems $\Loc\dd(\R^2)$ on $\R^2$ because we added the zero section $\R^2$. We can calculate $\mSh\dd(V_p')$ since $V_p'$ is a conic Lagrangian in the cotangent bundle $T^*\R^2$, it is the complexes of sheaves of $\k$-modules on $\R^2$ whose singular support lies inside $V_p'$. See Chapter \ref{chp:sheaves} for the combinatorial calculation of this category, and see the proof of Proposition \ref{prp:msh-vertex}. $\mSh\dd(V'_p)$ is the $A_{\infty}$-modules over the $p$-cyclic category, i.e. the dg category with $p$ objects and a closed degree zero morphism between each pair of objects in cyclic manner. For $p=4$, an object of $\mSh\dd(V_p')$, which is a complex of sheaves of $\k$-modules on $\R^2$, is represented in Figure \ref{fig:intro-vertex-front} shown in red, where $A_i\in\Modk$ for all $i$, and $\Modk$ is the dg category of (co)chain complexes of $\k$-modules. Also, the curve in Figure \ref{fig:intro-vertex-front} is the front projection of $V'_p$ in $\R^2$ and blue directions are the fibre directions of $V'_p$. Higher morphisms are hidden in the figure.

\begin{figure}[h]
	\centering
	
	\begin{tikzpicture}
		\newcommand{\pta}[1]{(2-(#1+10)/90)}
		\newcommand{\ptb}[1]{(4-(#1+5)/30)}
		
		\newcommand{\edge}[2]{
			\draw[domain=0:{4-((#1)/30)},variable=\x] plot ({\x},{(tan(#1)/2)*(\x^2)}) node[right] {edge #2};
			\draw[->,blue] ({\pta{#1}},{(tan(#1)/2)*(\pta{#1}^2)}) -- ({\pta{#1}+(0.2*(1/sqrt(1+(tan(#1)*\pta{#1})^2)))*tan(#1)*\pta{#1}},{(tan(#1)/2)*(\pta{#1}^2)-(0.2*(1/sqrt(1+(tan(#1)*\pta{#1})^2)))});
			\draw[->,blue] ({\ptb{#1}},{(tan(#1)/2)*(\ptb{#1}^2)}) -- ({\ptb{#1}+(0.2*(1/sqrt(1+(tan(#1)*\ptb{#1})^2)))*tan(#1)*\ptb{#1}},{(tan(#1)/2)*(\ptb{#1}^2)-(0.2*(1/sqrt(1+(tan(#1)*\ptb{#1})^2)))})
		}
		
		\node[left] at (0,0) {$0$};
		\draw[->,blue] (0,0) -- (0,-0.2);
		\edge{0}{1};
		\edge{10}{2};
		\edge{30}{3};
		\edge{60}{4};
		
		\node[red](A1) at (3.2,0.4) {$A_1$};
		\node[red](A2) at (2.7,1.3) {$A_2$};
		\node[red](A3) at (2.1,2.2) {$A_3$};
		\node[red](A4) at (-2,0.2) {$A_4$};
		
		\draw[red,->] (A1) to["$a_1$" {xshift=0.2cm, yshift=0.1cm}] (A2);
		\draw[red,->] (A2) to["$a_2$" {xshift=0.2cm, yshift=0.1cm}] (A3);
		\draw[red,->] (A3) to [bend right,"$a_3$" {xshift=0.2cm, yshift=0.1cm}] (A4);
		\draw[red,->] (A4) to [bend right,"$a_0$" {xshift=0.2cm, yshift=0cm}] (A1);
	\end{tikzpicture}
	
	\caption{An object of $\mSh\dd(V_4')$}
	\label{fig:intro-vertex-front}
\end{figure}

Dividing $\mSh\dd(V_p')$ by $\Loc(\R^2)$ can be explained by setting the stalk of the sheaves to zero at some fixed point. In Figure \ref{fig:intro-vertex-front}, we can select a point which makes $A_4\simeq 0$, which gives us the object
\[A:=A_1\xrightarrow{a_1} A_2\xrightarrow{a_2} A_3\ .\]
In general, we get
\[\mSh\dd(V_p)\simeq\mSh\dd(V_p')/\Loc\dd(\R^2)\simeq\Modk(A_{p-1})\]
where $A_{p-1}$ stands for the $A_{p-1}$-quiver (where the morphisms are closed degree zero), and $\Modk(A_{p-1})$ stands for the $A_{\infty}$-modules over $A_{p-1}$-quiver, which is studied extensively in Section \ref{sec:quiver}.

From now on, we fix $p=4$, since the calculations have the similar ingredients when $p$ is higher. We have the restriction maps $j_i\colon\mSh\dd(V_4)\to\Modk$ to the $i\th$ edge of $V_4$ given by the cones of the arrows crossing the corresponding edge in Figure \ref{fig:intro-vertex-front}. Note that we have $A_4\simeq 0$ there. This gives
\[j_1(A)=A_1, j_2(A)=C(a_1), j_3(A)=C(a_2), j_4(A)=A_3[1]\]
where $C(a_i)$ is the cone of the morphism $a_i$.

Next, we glue the pieces using the tools developed in Section \ref{sec:gluing}. To get $S_{4,1}$, we glue the ends of $V_4\times [0,1]$ by identify $V_4\times\{0\}$ with $V_4\times\{1\}$ after rotating the former by the angle $2\pi/p$ clockwise. By the \hyperlink{circle}{Circle Lemma}, this gives a dg category whose objects are the homotopy equivalences in $\mSh\dd(V_4)$ sending an object $A$ to its rotation $r(A)$ clockwise. This rotation can be easily seen if we work with $\mSh\dd(V_4')$ as shown in Figure \ref{fig:intro-vertex-rotated}. However, when we restrict to $\mSh\dd(V_4)$, it is not immediately clear what this rotation functor should be. To understand it, we focus on the edges: The rotation functor $r$ on $\mSh\dd(V_4)$ shifts the restriction to edges clockwise, i.e.
\[j_1\circ r(A)=C(a_1), j_2\circ r(A)=C(a_2), j_3\circ r(A)=A_3[1], j_4\circ r(A)=A_1\ .\]
One can show that there is no functor satisfying this if the categories are $\Z$-graded. Hence, we start to work with $\Z/2$-graded categories on this stage onwards. Then there exist a unique (up to natural equivalence) functor satisfying the above relations, which is the Coxeter functor $\cox_{4,1}$, given by
\[\cox_{4,1}(A):=C(a_1)\xrightarrow{\mx{\id & 0\\0 & a_2}}C(a_2\circ a_1)\xrightarrow{\mx{\id & 0}} A_1[1]\]
shown by Proposition \ref{prp:rotation-coxeter}.

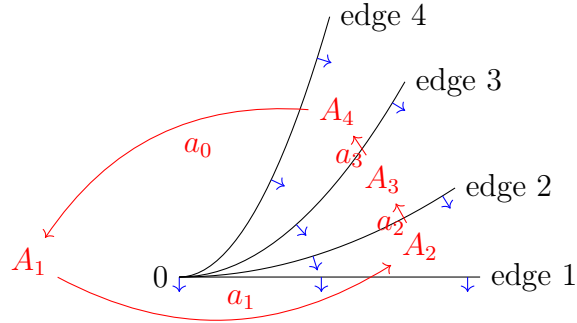
\begin{figure}[h]
	\centering
	
	\begin{tikzpicture}
		\newcommand{\pta}[1]{(2-(#1+10)/90)}
		\newcommand{\ptb}[1]{(4-(#1+5)/30)}
		
		\newcommand{\edge}[2]{
			\draw[domain=0:{4-((#1)/30)},variable=\x] plot ({\x},{(tan(#1)/2)*(\x^2)}) node[right] {edge #2};
			\draw[->,blue] ({\pta{#1}},{(tan(#1)/2)*(\pta{#1}^2)}) -- ({\pta{#1}+(0.2*(1/sqrt(1+(tan(#1)*\pta{#1})^2)))*tan(#1)*\pta{#1}},{(tan(#1)/2)*(\pta{#1}^2)-(0.2*(1/sqrt(1+(tan(#1)*\pta{#1})^2)))});
			\draw[->,blue] ({\ptb{#1}},{(tan(#1)/2)*(\ptb{#1}^2)}) -- ({\ptb{#1}+(0.2*(1/sqrt(1+(tan(#1)*\ptb{#1})^2)))*tan(#1)*\ptb{#1}},{(tan(#1)/2)*(\ptb{#1}^2)-(0.2*(1/sqrt(1+(tan(#1)*\ptb{#1})^2)))})
		}
		
		\node[left] at (0,0) {$0$};
		\draw[->,blue] (0,0) -- (0,-0.2);
		\edge{0}{1};
		\edge{10}{2};
		\edge{30}{3};
		\edge{60}{4};
		
		\node[red](A1) at (3.2,0.4) {$A_2$};
		\node[red](A2) at (2.7,1.3) {$A_3$};
		\node[red](A3) at (2.1,2.2) {$A_4$};
		\node[red](A4) at (-2,0.2) {$A_1$};
		
		\draw[red,->] (A1) to["$a_2$" {xshift=0.2cm, yshift=0.1cm}] (A2);
		\draw[red,->] (A2) to["$a_3$" {xshift=0.2cm, yshift=0.1cm}] (A3);
		\draw[red,->] (A3) to [bend right,"$a_0$" {xshift=0.2cm, yshift=0.1cm}] (A4);
		\draw[red,->] (A4) to [bend right,"$a_1$" {xshift=0.2cm, yshift=0cm}] (A1);
	\end{tikzpicture}
	
	\caption{The rotation of the object in Figure \ref{fig:intro-vertex-front} in $\mSh\dd(V_4')$}
	\label{fig:intro-vertex-rotated}
\end{figure}

Then the objects of $\mSh\dd(S_{4,1})$ are of the form
\[\begin{tikzcd}
	A\dar["f"]\\ \cox_{4,1}(A)
\end{tikzcd}
:=
\begin{tikzcd}[ampersand replacement=\&]
	A_1\dar["f_1"]\rar["a_1"]\drar["h_1"] \& A_2\dar["f_2"]\rar["a_2"]\drar["h_2"] \& A_3\dar["f_3"]\\
	C(a_1)\rar["\mx{\id & 0\\0 & a_2}"'] \& C(a_2\circ a_1)\rar["\mx{\id & 0}"'] \& A_1[1]
\end{tikzcd}\]
where $f$ is a homotopy equivalence, i.e. $f_1,f_2,f_3$ are homotopy equivalences, and the squares commute up to homotopy inside, see Proposition \ref{prp:quiver-algebra}. We will refer this object simply as $f$.

Note that we have the restriction functor $\mon\colon\mSh\dd(S_{4,1})\to\Loc\dd(S^1)$ induced by the inclusion $\partial S_{4,1}\simeq S^1\hookrightarrow S_{4,1}$. The objects of $\Loc\dd(S^1)$ are of the form $(X,m)$ where $X\in\Modk$ and $m\colon X\to X$ is a homotopy equivalence, called monodromy. If we write $\mon(f)=(\mon(f)_1,\mon(f)_0)$, the monodromy $\mon(f)_0\colon \mon(f)_1\to \mon(f)_1$ can be obtained as follows: We fix a $4$-valent vertex $V_4$ in $S_{4,1}$. We start from the first edge of $V_4$, and trace the boundary of $S_{4,1}$ until we come back to the first edge of $V_4$, which will complete the circle. Schematically, we can express this as
\[\text{edge 1}\to\text{edge 2}\to\text{edge 3}\to\text{edge 4}\to\text{edge 1}\]
where $f$ associates the object $j_i(A)$ to $i\th$ edge, and morphisms are coming from $f$. Explicitly, $\mon(f)_0$ is given by
\begin{align*}
	A_1\xrightarrow{f_1} C(a_1)\xrightarrow{\mx{f_1 & 0\\h_1 & f_2}}C(C(a_1)&\xrightarrow{\mx{\id & 0\\0 & a_2}} C(a_2\circ a_1))\simeq C(a_2)\xrightarrow{\mx{f_2 & 0 \\h_2 & f_3}}\\
	&\xrightarrow{\mx{f_2 & 0 \\h_2 & f_3}}C(C(a_2\circ a_1)\xrightarrow{\mx{\id & 0}}A_1[1])\simeq A_3[1]\xrightarrow{f_3}A_1
\end{align*}
where the equivalences are natural equivalences explained in Section \ref{sec:quiver}, and given in Proposition \ref{prp:monodromy} explicitly.

Lastly, we attach a disk to the boundary of $S_{4,1}$ to get $L_{4,1}$. By \hyperlink{disk}{Disk Lemma}, the effect of the attaching a disk is to kill the monodromy, i.e. adding a degree $1$ element $\gamma$ such that $d\gamma=\mon(f)_0-\id$. Before proceeding, we observe that there is a way to simplify the objects of $\mSh\dd(S_{4,1})$ by getting rid of the vertical and diagonal maps, starting from the right (by applying Lemma \ref{lem:simplified-pinwheel-1} and \ref{lem:simplified-pinwheel-2} successively), and get the objects of the form
\[\begin{tikzcd}[ampersand replacement=\&]
A\dar["f", ":=" xshift=1cm] \& [30pt] A_1\dar["f_1=\mx{f_1^1\\f_1^2\\f_1^3}"']\rar["\mx{f_1^1\\f_1^2}"]\drar["0"] \& [20pt]C(f_1^1)\dar["\id"]\rar["\mx{\id & 0}"]\drar["0"] \& A_1[1]\dar["\id"]\\ [20pt]
\cox_{4,1}(A) \& C\mx{f_1^1\\f_1^2}\rar["\mx{\id & 0 & 0\\0 & \id & 0}"'] \& C(f_1^1)\rar["\mx{\id & 0}"'] \& A_1[1]
\end{tikzcd}\]
where $f_1$ is a homotopy equivalence. We can equivalently characterise this condition by asking $df_1=0$ which gives
\[df_1^i=\sum_{j=1}^{i-1}f_1^{i-j}\circ f_1^j\]
where $f_1^i\colon A_1\to A_1$ is a degree $1$ morphism for $i=1,2,3$, and requiring $C(f_1)\simeq 0$. The monodromy $\mon(f)$ is given by
\[\mon(f)_0=\id\circ\mx{0 & \id & f_1^1}\circ\id\circ\mx{0 & \id & 0 & f_1^1 & 0\\ 0 & 0 & \id & f_1^2 & 0\\0 & 0 & 0 & 0 & \id}\circ\mx{f_1^1 & 0 & 0\\f_1^2 & 0 & 0\\f_1^3 & 0 & 0\\ 0 & \id & 0 \\ 0& 0 &\id}\circ\mx{f_1^1\\f_1^2\\f_1^3}=\sum_{j=1}^{4-1} f_1^{i-j}\circ f_1^j\]
and if we set $f_1^4:=\gamma$ we get
\[df_1^i=-\delta_{i,4}+\sum_{j=1}^{i-1} f_1^{i-j}\circ f_1^j\]
for $i=1,\dots,4$.

Finally, we can describe the condition $C(f_1)\simeq 0$ by introducing a degree $1$ morphism $\varepsilon\colon C(f_1)\to C(f_1)$ satisfying $d\varepsilon=\id$ following \cite{drinfeld}, which in turn introduces degree $1$ morphisms $g^{ij}\colon A_1\to A_1$ for $1\leq i,j\leq 4$ where
\[dg^{ij}=\sum_{k=1}^{i-1} f_1^{i-k}\circ g^{kj}+\sum_{k=j+1}^4 g^{ik}\circ f_1^{k-j}\]
see the proof of Theorem \ref{thm:msh-pinwheel}. This shows that the objects of $\mSh\dd(L_{4,1})$ are of the form $(A_1,(f_1^i)_{i=1}^4,(g^{ij})_{1\leq i,j\leq 4})$ where $A_1\in\Modk$ and $f_1^i,g^{ij}$ are degree $1$ morphisms $A_1\to A_1$ satisfying above differential relations. This shows that
\[\mSh\dd(L_{4,1})\simeq \Modk(\cA_{4,1})\]
and if we restrict to the objects with cohomology of finite rank, we get
\[\mSh(L_{4,1})\simeq \Perfk(\cA_{4,1})\]
and if we take the compact objects, we get
\[\mSh^w(L_{4,1})\simeq \Perf(\cA_{4,1}) \ .\]
Note that here we focused on objects, but their morphisms match also: We can glue the morphisms, and simplify them to show this similarly.

\subsection{Background}

Let $\k$ be a commutative ring, and fix $R=\Z$ (or $\Z/N$ when we specify so). We mostly work with \textit{pretriangulated (small, $\k$-linear) differential graded (dg) categories}, whose definition and properties can be found in \cite{dgcat} and \cite{toen}. In particular, for any two objects $A,B\in\cC$ of a dg category $\cC$, their morphism space $\Hom_{\cC}(A,B)$ is an $R$-graded $\k$-module, and equipped with a degree $1$ morphism $d$ satisying $d\circ d=0$, called \textit{differential}. We write $\Hom_{\cC}^n(A,B)$ for the $\k$-module of degree $n$ morphisms from $A$ to $B$, and we drop the subscript ``$\cC$'' if it is clear from the context. The degree of a morphism $f$ is also denoted by $|f|$.

The identity $\id\in\Hom_{\cC}(A,A)$ is degree zero and $d(\id)=0$. Also, the graded Leibniz rule holds for the composition: If $f\in\Hom^n_{\cC}(A,B)$ and $g\in\Hom^m_{\cC}(B,C)$, then
\[d(g\circ f)=dg\circ f + (-1)^m g\circ df\ .\]
Note that we read the compositions from right to left.

We write $H^n\cC$ for the $\k$-linear category with the same objects as $\cC$ and whose morphism space is $H^n\Hom_{\cC}(A,B)$. We define $H^*\cC$ similarly by replacing $n$ with $*$. Note its morphism space is graded.

A \textit{dg functor} $F$ between the dg categories $\cC,\cD$ is a $\k$-linear functor preserving grading and satisfying $dF(f)=F(df)$ for any morphism $f$. We write $\Fun_{\dg}(\cC,\cD)$ for the set of dg functors between $\cC$ and $\cD$, and moreover it is a dg category whose morphisms are natural transformations between dg functors differential is obtained by differentiating each component of a natural transformation.

A dg functor is a (dg) \textit{quasi-equivalence} (resp. \textit{quasi-isomorphism}) if $H^0F\colon H^0\cC\to H^0\cD$ is an equivalence (resp. isomorphism), and $H^*F\colon H^*\cC\to H^*\cD$ is full and faithful. A chain map between chain complexes inducing an isomorphism between their homology is also called \textit{quasi-isomorphism}, hence the last condition can be restated as: $F$ induces a quasi-isomorphism between the morphism spaces.

We also work with \textit{pretriangulated $A_{\infty}$-categories}, which appear naturally in the context of symplectic geometry, specifically in Floer thoery. See \cite{seidel} for the definitions and properties regarding them. In particular, for any two objects $A,B\in\cC$ of an $A_{\infty}$-category $\cC$, their morphism space $\Hom_{\cC}(A,B)$ is an $R$-graded $\k$-module equipped with \textit{$A_{\infty}$-composition maps} $\mu^n$ for $n\geq 1$ satisfying $A_{\infty}$-relations in \cite{seidel}. Note that a dg category can be trivially seen as an $A_{\infty}$-category by setting $\mu^1$ equal to the differential, $\mu^2$ to the composition, and $\mu^n=0$ for $n\geq 3$ after appropriately arranging the signs.

If $\cC$ and $\cD$ are $A_{\infty}$-categories, $\Fun(\cC,\cD)$ denotes the set of \textit{$A_{\infty}$-functors} between $\cC$ and $\cD$, moreover it is an $A_{\infty}$-category whose morphisms are \textit{$A_{\infty}$-natural transformations}. If $\cD$ is a dg category, then $\Fun(\cC,\cD)$ is also a dg category.

When we work with dg categories, we localise the dg functors at quasi-isomorphisms. Whereas for $A_{\infty}$-categories, we do not need to localise the morphisms, because quasi-isomorphisms are already invertible up to homotopy. In particular, instead of working with $\Fun_{\dg}(\cC,\cD)$ and localise it, we can work with $\Fun(\cC,\cD)$ directly. In the view of this, we call $\cC$ and $\cD$ \textit{quasi-equivalent}, if there is a roof of (dg) quasi-equivalences between $\cC$ or $\cD$, or equivalently, an \textit{$A_{\infty}$-quasi-equivalence} between $\cC$ and $\cD$.

We define $\Modk$ as the dg category of (co)chain complexes of $\k$-modules, and $\Perfk$ as the dg category of (co)chain complexes of $\k$-modules with cohomology of finite rank. Then the dg category $\Fun(\cC^\op,\Modk)$ is called \textit{$A_{\infty}$-modules over $\cC$}, denoted by $\Modk(\cC)$. Similarly, the dg category $\Fun(\cC^\op,\Perfk)$ is called \textit{perfect $A_{\infty}$-modules over $\cC$}, denoted by $\Perfk(\cC)$.

If a dg category ($A_{\infty}$-category) $\cC$ has only one object $A$, we call it \textit{differential graded algebra (dga)} (\textit{$A_{\infty}$-algebra}). Hence we work with $\Hom_{\cC}(A,A)$ instead of $\cC$ in that case.

We say a pretriangulated $A_{\infty}$-category $\cC$ is \textit{generated} by its full $A_{\infty}$-subcategory $\cD$ if the \textit{triangulated envelope} $\Perf(\cD)$ of $\cD$ is quasi-equivalent to $\cC$, i.e. if enlarging $\cD$ by taking all cones and shifts in $\cD$ in arbitrarily many times results in $\cC$. If $\cD$ is a dg or an $A_{\infty}$-algebra, we also write $\cC\simeq\Perf(\Hom_{\cD}(A,A))$ where $A$ is the unique object of $\cD$.

On the geometry side, we work with manifolds which are smooth ($C^{\infty}$) without boundary unless otherwise specified. Lagrangian and Legendrian submanifolds are assumed to be smooth unless they are told to be singular. We denote the cotangent bundle of the manifold $M$ by $T^*M$, and its zero section by $0_M$, which we simply write as $M$ when there is no possibility of confusion. An element of $T^*M$ is presented as $(x,p)$ where $x\in M$ and $p$ is a cotangent vector at $x$. For a submanifold $L\subset M$, we write $T^*L$ for the cotangent bundle of $L$ inside $T^*M$, and $N^*L$ for the normal bundle of $L$.

We define the \textit{cosphere bundle} $T^{\infty}M$ as the quotient of $T^*M\setminus 0_M$ by the action of $\R_{>0}$. An element of $T^{\infty}M$ is presented as $(x,[p])$ where $x\in M$, $p\neq 0$ is a cotangent vector at $x$, and $[p]$ is the equivalence class of $p$ such that $[p]=[cp]$ if $c>0$. We write $T^{\infty,-}(M)$ for the open subset of $T^{\infty}(M)$ whose elements of the form $(x_1,\ldots,x_n,[p_1,\ldots,p_{n-1},-1])$ if $M$ is $n$-dimensional.

\subsection{Acknowledgements}

I would like to thank my PhD supervisor Yank{\i} Lekili for suggesting me to compute the microlocal sheaves on the pinwheels, and leading me to prove Theorem \ref{thm:intro-wrapped}. His guidance and immense knowledge helped me in writing many parts of this thesis.

This document is the author's PhD thesis written at King's College London. This research was partially funded by a Royal Society grant RG130456.

\section{Constructible Sheaves and Singular Support}\label{chp:sheaves}

In this chapter, we will review some materials from sheaf theory, in particular constructible sheaves, and describe the connection between sheaf theory and symplectic geometry via the singular support of a sheaf. We will also define wrapped constructible sheaves which will be useful when we define wrapped microlocal sheaves in Section \ref{sec:microlocal}. The main reference for this section is \cite{kashiwara-schapira}. We also refer \cite{wrapped} and \cite{shende-treumann-zaslow} frequently for some useful information.

\subsection{Constructible Sheaves}

In this section, we will define constructible sheaves and study their combinatorial nature. Let $M$ be a topological space throughout the section.

\begin{dfn}
    A \textit{presheaf} $\cF$ on a topological space $M$ with values in the category $\cC$ is a functor $\cF\colon\text{Top}(M)^{\op}\rightarrow\cC$ where ``op'' means the opposite category, and $\text{Top}(M)$ is the category with objects as open subsets of $M$ and morphisms as inclusions of the open subsets. If furthermore, for every open $U\subset M$ and its open covering $\{U_i\}$ we have
    \[\cF(U)\simeq\lim\left(\prod_i\cF(U_i)\rightrightarrows\prod_{j,k}\cF(U_j\cap U_k)\right)\]
    then $\cF$ is called \textit{sheaf}. $\cF(U)$ is called \textit{sections on $U$} and it is also denoted by $\Gamma(U;\cF)$ where $\Gamma(U;\boldsymbol{\cdot})$ is the functor from the category of sheaves to the category $\cC$. Given $V\subset U$, we call the associated morphism $\cF(U)\to\cF(V)$ \textit{restriction morphism}.
    
    We define $\sh(M)$ as the abelian category of sheaves on $M$ with values in the abelian category $\modk$ of $\k$-modules, which is a functor category. Our main interest is the derived category $D(M)$ of $\sh(M)$. Note that $\sh(M)$ can be regarded as the full subcategory of $D(M)$ consisting of objects concentrated in degree zero. To deal with $D(M)$, we will work with its dg enhancement $\Sh(M)$, which is defined as the dg derived category of complexes of sheaves on $M$ with values in $\modk$ (or equivalently, sheaves on $M$ with values in $\Modk$), where the morphisms are obtained by localising the usual complexes of maps between complexes at quasi-isomorphisms (or equivalently, by taking dg quotient by the acyclic objects in the sense of \cite{drinfeld}). The complexes are $R$-graded, where $R=\Z$ or $R=\Z/N$ for some $N$. $\Sh(M)$ is a pretriangulated dg category and by taking its (ungraded) cohomology category, we have
    \[H^0(\Sh(M))\simeq D(M)\]
    as triangulated categories. See \cite[Section 2.1-2]{microbranes} for the detailed discussion.
    
    For the operations on sheaves and useful relations, refer to \cite{kashiwara-schapira}, \cite{viterbo}, and \cite[Appendix]{combinatorics}.
\end{dfn}

\begin{rmk}
    We mostly care about derived functors, in particular derived sections. Given the open sets $U,V\subset M$, we can glue the derived sections using
    \[R\Gamma(U\cup V;\cF)\simeq C(R\Gamma(U;\cF)\oplus R\Gamma(V;\cF)\to R\Gamma(U\cap V;\cF))[-1]\]
    where $R\Gamma(U;\boldsymbol{\cdot})$ is the right derived functor for $\Gamma(U;\boldsymbol{\cdot})$, $C(\boldsymbol{\cdot})$ is the cone functor, the isomorphism is quasi-isomorphism. 
\end{rmk}

\begin{dfn}
    For $A\in\modk$, we define the \textit{constant sheaf} $A_M\in\sh(M)$ as
    \[A_M(U)=\{f\colon U\to A\vb f\text{ is locally constant}\}\]
    and restrictions are restrictions of functions.
\end{dfn}
    
\begin{exm}
    The constant sheaf $\k_M$ on $M$ is given by $\Gamma(U;\k_M)=H^0(U;\k)$ with the obvious restriction maps. The derived sections are given by $R\Gamma(U;\k_M)=C^*(U;\k)$, the singular chain complex associated to $U$.
\end{exm}
    
\begin{dfn}\label{dfn:locally-constant}
    The \textit{stalk} of the sheaf $\cF$ at the point $x\in M$ is defined as
    \[\cF_x:=\lim_{\underset{U\ni x}{\longrightarrow}}\cF(U)\]
    where the limit is the direct limit. The \textit{support} of $\cF$ is defined as the closure of the set $\{x\in M\vb \cF_x\not\simeq 0\}$. $\cF$ is called \textit{locally constant} or \textit{local system} if for each point $x$, there is a neighbourhood $U$ of $x$ such that $\cF|_U$ is a constant sheaf, where $\cF|_U$ is defined as $j^*(\cF)$ for the inclusion $j\colon U\hookrightarrow M$.

    We define $\loc(M)$ as the full subcategory of $\sh(M)$ consisting of locally constant sheaves. It is well known that if $M$ is connected, we have
    \[\loc(M)\simeq \modk(\pi_1(M))\]
    where $\modk(\pi_1(M))$ stands for the category of functors from $\pi_1(M)$ (as a category) to $\modk$.
    
    We define $\Loc(M)$ as the full dg subcategory of $\Sh(M)$ consisting of sheaves $\cF\in\Sh(M)$ such that $H^*\cF$ is a locally constant sheaf with perfect stalks (i.e. stalks with cohomology of finite rank). If we drop the perfect stalk condition, we get the dg category $\Loc\dd(M)$. Note that $\loc(M)$ can be regarded as the full subcategory of $H^0(\Loc\dd(M))$ consisting of objects concentrated in degree zero. As shown in \cite[Example 1.1]{wrapped}, if $M$ is connected, we have
    \begin{align*}
        \Loc\dd(M)&\simeq\Modk(C_{-*}(\Omega M))\\
        \Loc(M)&\simeq\Perfk(C_{-*}(\Omega M))
    \end{align*}
    where $\Omega M$ is the based loop space of $M$.
\end{dfn}

\begin{prp}\label{prp:locally-constant}
    If $M$ is locally contractible, the following are equivalent:
    \begin{itemize}
        \item $\cF\in\Loc\dd(M)$,
        
        \item Whenever $U$ and $V$ are contractible, $V\subset U$, the restriction $R\Gamma(U;\cF)\to R\Gamma(V;\cF)$ is a quasi-isomorphism.
    \end{itemize}
\end{prp}

From now on, assume that $M$ is a smooth manifold.
    
\begin{dfn}[\cite{mather}]
    Let $S$ be a subset of $M$. A \textit{Whitney stratification} of $S$ is $\cS=\{\cS_i\}$ where $\cS_i$ are pairwise disjoint connected smooth submanifolds of $M$ which lie in $S$, called \textit{stratum}, such that $M=\bigsqcup_i\cS_i$ and $\cS$ satisfies the following conditions:
    \begin{enumerate}
        \item Locally finiteness: Every point of $M$ has a neighbourhood which intersects finitely many strata,
            
        \item Condition of the frontier: For each stratum $\cS_i$ of $\cS$ its frontier $(\bar\cS_i - \cS_i)\cap S$ is a union of strata,
            
        \item Whitney condition A: Given any $(\cS_i,\cS_j)$, $y\in\cS_j$, and given any sequence $\{x_n\}$ of points in $\cS_i$ such that $x_n\to y$ and $T_{x_n}\cS_i$ converges to some subspace $V\subset T_y M$, we have $T_y\cS_j \subset V$,
            
        \item Whitney condition B: Given any $(\cS_i,\cS_j)$, $y\in\cS_j$, consider everything in a coordinate chart around $y$ which is identified with $\R^n$. Identify $T_p^*\R^n$ with $\R^n$ for any $p\in\R^n$ in the standard way. Let $\{x_n\}$ be a sequence of points in $\cS_i$, converging to $y$ and $\{y_n\}$ a sequence of points in $\cS_j$, also converging to $y$. Suppose $T_{x_n}\cS_i$ converges to some subspace $V \subset \R^n$ and that $x_n\neq y_n$ for all $n$ and the secants $\widehat{x_n y_n}$ (the line in $\R^n$ which is parallel to the line joining $x_n$ and $y_n$ and passes through the origin) converge (in projective space $\P^{n-1}$) to some line $L \subset \R^n$. Then $L\subset V$.
    \end{enumerate}
	If $\cS$ satisfies only the first two conditions, we call $\cS$ a \textit{stratification}, and $M$ a \textit{stratified space}. If $M$ is a subset of a symplectic manifold and consists of isotropic strata, we call $M$ \textit{isotropic subset}. If moreover, all of its connected components has a Lagrangian stratum, we call $M$ \textit{singular Lagrangian}.
\end{dfn}

\begin{prp}[\cite{mather}]
    Whitney condition B implies Whitney condition A.
\end{prp}
    
\begin{dfn}
    Let $\cF$ be a sheaf on $S\subset M$. Given a Whitney stratification $\cS$ of $S$, $\cF$ is called \textit{$\cS$-constructible} if $(H^*\cF)|_{\cS_i}$ is locally constant and has perfect stalks for each $i$. $\cF$ is called \textit{constructible} if there exists a Whitney stratification $\cS$ of $S$ such that $\cF$ is $\cS$-constructible. If we drop the perfect stalk condition, then we get \textit{large $\cS$-constructible} and \textit{large constructible} sheaf, respectively.
    
    We define $\Sh_c(M)$, $\Sh_c\dd(M)$, $\Sh_{\cS}(M)$, $\Sh\dd_{\cS}(M)$ as the full dg subcategory of $\Sh(M)$ consisting of constructible, large constructible, $\cS$-constructible, large $\cS$-constructible sheaves, respectively.
\end{dfn}

By definition, constructible sheaves have a combinatorial nature. We will exploit this by realising $\cS$-constructible sheaves by $A_{\infty}$-representations of quivers. For that, we need the following definitions:

\begin{dfn}
    Given the Whitney stratification $\cS=\{\cS_i\}$ of $S\subset M$, the \textit{star} of the stratum $\cS_i$ is defined as
    \[\starr(\cS_i)=\bigsqcup_{\cS_i\subset\bar\cS_j}\cS_j\ .\]
    We also consider $\cS$ as the $A_{\infty}$-category with objects as strata $\cS_i$ and there is a degree 0 morphism $\cS_i\rightarrow\cS_j$ whenever $\cS_i\subset\starr(\cS_j)$ with $\mu^n=0$ for all $n\neq 2$ and $\mu^2$ is the obvious composition (we take the linearisation of the category to make it $\k$-linear). We call $\cS$ \textit{regular cell complex} if all strata and stars are contractible.
\end{dfn}

\begin{lem}
    If $\cS=\{\cS_i\}$ is a Whitney stratification of $S\subset M$, then $\starr(\cS_i)$ is open in $S$ for any $i$.
\end{lem}

\begin{proof}
    Pick $x\in\starr(\cS_i)$. Then $x\in\cS_j$ for some $\cS_j\subset\starr(\cS_i)$. Take a neighbourhood $U$ of $x$ such that it intersects finitely many strata and any smaller neighbourhood of $x$ inside $U$ still intersects the same strata. Such neighbourhood exists, since otherwise one can choose a smaller neighbourhood inside $U$ which intersects less strata. This process will end since the number of strata intersecting with the neighbourhood is finite.
    
    Pick a stratum $\cS_k$ such that $\cS_k\cap U\neq\emptyset$. Any smaller neighbourhood of $x$ inside $U$ intersects with $\cS_k$, hence we have $x\in\bar\cS_k$. Since $\cS_k$ is a union of strata, we must have $\cS_j\subset\bar\cS_k$. We also have $\cS_i\subset\bar\cS_j$, therefore $\cS_i\subset\bar\cS_k$ and $\cS_k\subset\starr(\cS_i)$. This shows $\starr(\cS_i)$ is open in $S$.
\end{proof}

By seeing $\Sh\dd_{\cS}(M)$ as $A_{\infty}$-category, we have the following proposition:
    
\begin{prp}[\cite{shende-treumann-zaslow}, \cite{gps3}]\label{prp:combinatorics-constructible}
    Let $\cS$ be a Whitney stratification of $M$. We have the $A_{\infty}$- functor $\Gamma_{\cS}\colon\Sh\dd_{\cS}(M)\rightarrow\Modk(\cS)$ given by \[\Gamma_{\cS}(\cF)=[\cS_i\mapsto R\Gamma(\starr(\cS_i);\cF)]\ .\]
    If $\cS$ is a regular cell complex, then $\Gamma_{\cS}$ is a quasi-equivalence.
    
    Same statement holds if we replace $\Sh\dd_{\cS}(M)$ with $\Sh_{\cS}(M)$ and $\Modk(\cS)$ with $\Perfk(\cS)$.
\end{prp}
    
\begin{proof}
    See \cite[Section 2.3-4]{microbranes}.
\end{proof}

\begin{rmk}
    If $\cS$ is a regular cell complex, then the restriction map from $R\Gamma(\starr(\cS_i);\cF)$ to the stalk of $\cF$ at any point of $\cS_i$ is a quasi-isomorphism.
\end{rmk}
    
\begin{exm}
    Consider the stratification $\cS$ of the circle $S^1$ consisting of the circle itself. Then $\Sh\dd_{\cS}(S^1)$ consists of locally constant sheaves on $S^1$, whereas $\Modk(\cS)$ consists of constant sheaves on $S^1$. This means $\Modk(\cS)$ does not record the monodromy. The reason is the stratum ``circle'' is not contractible.
\end{exm}

\begin{exm}
    Consider the stratification $\cS$ of $S^1$ consisting of a point and an arc. Then again locally constant sheaves appear in $\Sh\dd_{\cS}(S^1)$, however not in $\Modk(\cS)$. The reason is the star of a point is circle which is not contractible.
\end{exm}
    
\begin{exm}\label{exm:circle-local-system}
    Consider the stratification $\cS$ of $S^1$ consisting of two points and two arcs. $\cS$ is a regular cell complex, hence by Proposition \ref{prp:combinatorics-constructible} we have
    \[\Sh\dd_{\cS}(S^1)\simeq\Modk(\cS)\]
    where the objects of the $\Modk(\cS)$ are given by the following diagram:
    \[\begin{tikzcd}[column sep=10pt, row sep=10pt]
          & A\ar[ld]\ar[rd] &   \\
        B &   & C \\
          & D\ar[lu]\ar[ru] &
    \end{tikzcd}\]
    where $A,B,C,D\in\Modk$ and the arrows represent morphisms between them. As for locally constant sheaves, $\Loc\dd(S^1)$ is a full dg subcategory of $\Sh\dd_{\cS}(S^1)$, hence of $\Modk(\cS)$ by setting all stalks as quasi-isomorphic, i.e. $A\simeq B\simeq C\simeq D$, and all morphisms as quasi-isomorphisms by Proposition \ref{prp:locally-constant}. One can show that such objects are isomorphic to
    \[\begin{tikzcd}[column sep=10pt, row sep=10pt]
          & A\ar[ld,"\id"',"\sim" {rotate=45, xshift=-0.2cm}]\ar[rd,"m","\sim"' {rotate=-45, xshift=0.2cm}] &   \\
        A &   & A \\
          & A\ar[lu,"\id","\sim"' {rotate=-45, xshift=-0.2cm}]\ar[ru,"\id"',"\sim" {rotate=45, xshift=0.2cm}] &
    \end{tikzcd}\]
    where $m$ is the composition of arrows (after inverting them whenever needed) in the previous diagram, which we call the monodromy. Hence
    \[\Loc\dd(S^1)=\{(A,m)\vb A\in\Modk, m\colon A\to A\text{ is quasi-isomorphism}\}\]
    and also
    \[\Loc(S^1)=\{(A,m)\vb A\in\Perfk, m\colon A\to A\text{ is quasi-isomorphism}\}\ .\]
\end{exm}
    
\begin{exm}
    Consider the stratification $\cS$ of the disk (with boundary) consisting of a point and an arc in the boundary, and the interior of the disk. Then $\cS$ is a regular cell complex, but Proposition \ref{prp:combinatorics-constructible} does not hold since disk is not a manifold (it is a manifold with boundary). Indeed, $\Modk(\cS)$ does not capture all the sheaves which are locally constant when restricted to the boundary and vanish in the interior of the disk.
\end{exm}

\subsection{Singular Support}

For a smooth manifold $M$, we will define singular support for sheaves in $\Sh(M)$ following \cite{kashiwara-schapira}. One can also restrict the attention to $\Sh_c\dd(M)$ and define singular support using stratified Morse theory following \cite{macpherson}, as in \cite{shende-treumann-zaslow}.

For a closed set $Z\subset M$, we define the functor
\[\Gamma_Z\colon\Sh(M)\to\Sh(M)\]
such that for an open set $U$, we have
\[\Gamma_Z(\cF)(U)=\ker(\cF(U)\to\cF(U\setminus Z))\ .\]
Note that we have the exact triangle
\[R\Gamma_Z(\cF)(U)\to R\Gamma(U;\cF)\to R\Gamma(U\setminus Z;\cF)\xrightarrow{[1]}\]
which gives
\[R\Gamma_Z(\cF)(U)\simeq C(R\Gamma(U;\cF)\to R\Gamma(U\setminus Z;\cF))[-1]\ .\]

\begin{dfn}[\cite{kashiwara-schapira}, \cite{viterbo}]
    For $\cF\in\Sh(M)$, the \textit{microstalk} $\cF_{(x,p),f}$ of $\cF$ at $(x,p)\in T^*M$ relative to the smooth function $f\colon M\to\R$ with $f(x)=0$ and $df_x=p$ is defined as
    \[\cF_{(x,p),f}:=(R\Gamma_{\{y\in M \vb f(y)\geq 0\}}(F))_x\ .\] 
    The \textit{singular support (microsupport)} $\sing(\cF)$ of $\cF$ is the subset of $T^*M$ given by the closure of the points $(x,p)\in T^*M$ with $\cF_{(x,p),f}\not\simeq 0$ for some smooth function $f\colon M\to\R$ with $f(x)=0$ and $df_x=p$. We call $(x,p)\in T^*M$ \textit{characteristic} for $\cF$ if $(x,p)\in\sing(\cF)$. Also, we call a subset $L\subset T^*M$ \textit{conic} if $(x,kp)\in L$ whenever $(x,p)\in L$ for any $k\in\R_{>0}$.
\end{dfn}

\begin{rmk}
    We can calculate the microstalk $\cF_{(x,p),f}$ by
    \[\cF_{(x,p),f}=\lim_{\underset{U\ni x}{\longrightarrow}} C(R\Gamma(U;\cF)\to R\Gamma(\{y\in U \vb f(y)< 0\};\cF))[-1]\ .\]
    Basically, $\sing(\cF)$ is the collection of points and directions, along which the sheaf $\cF$ changes.
\end{rmk}

\begin{prp}[\cite{kashiwara-schapira}, \cite{kuwagaki}]\label{prp:test-function}
    Given $(x,p)\in T^*M$ and $f\colon M\to\R$ with $f(x)=0$ and $df_x=p$, let $L\subset T^*M$ be a conic Lagrangian such that $(x,p)$ is a smooth point on $L$ and $\sing(\cF)\subset L$ in a neighbourhood of $(x,p)$. If $L$ and the graph $\Gamma_{df}$ of $df$ intersect transversely at $(x,p)$, then $\cF_{(x,p),f}$ does not depend on $f$ up to shifts. For such $f$, if $\cF_{(x,p),f}\simeq 0$, then $(x,p)\notin\sing(F)$.
\end{prp}

\begin{dfn}
    We call the function $f$ in Proposition \ref{prp:test-function} \textit{proper test function} for $\cF$ at $(x,p)$. Using this, we define the \textit{microstalk} of $\cF$ at $(x,p)$ (up to shifts)
    \[\cF_{(x,p)}:=\cF_{(x,p),f}\]
    where $f$ is a proper test function for $\cF$ at $(x,p)$.
\end{dfn}

\begin{rmk}
    With this new definition, we can simply say that the singular support $\sing(\cF)$ of $\cF$ is the closure of the points $(x,p)\in T^*M$ where the microstalk of $\cF$ does not vanish.
\end{rmk}
    
\begin{prp}[\cite{viterbo}]\label{prp:sing-properties}
    $\sing(\cF)$ has the following properties:
        
    \begin{enumerate}
        \item $\sing(\cF)$ is closed and conic,
            
        \item $\sing(\cF)$ is determined locally, i.e. $\sing(\cF)\cap T^*U=\sing(\cF|_U)$ for open $U\subset M$,
            
        \item $\sing(\cF)\cap 0_M=\supp(\cF)$,
        
        \item $\sing(\cF)\subset\bigcup_i\sing(H^i\cF)$,
            
        \item For an exact triangle $\cF_1\rightarrow\cF_2\rightarrow\cF_3\xrightarrow{[1]}$ we have
        \[\sing(\cF_j)\triangle\sing(\cF_k)\subset\sing(\cF_i)\subset\sing(\cF_j)\cup\sing(\cF_k)\]
        for $\{i,j,k\}=\{1,2,3\}$.
            
    \end{enumerate}
\end{prp}
    
\begin{exm}
    For the constant sheaf $\k_M$, we have $R\Gamma(U;\k_M)=C^*(U;\k)$ and $\sing(\k_M)=0_M$. Since singular support is locally determined, if $\cF$ is a locally constant sheaf on $M$ which is not the zero sheaf, then $\sing(\cF)=0_M$ also.
\end{exm}
    
\begin{exm}
    Let $Z\subset M$ be the closure of an open set in $M$. If $i\colon Z\hookrightarrow M$ is the inclusion, then we can define the sheaf $\k_Z$ on $M$ as $i_*(\k_Z)$ where the second $\k_Z$ is the constant sheaf on $Z$. We have $\Gamma(U;\k_Z)=H^0(U\cap Z;\k)$ and $R\Gamma(U;\k_Z)=C^*(U\cap Z;\k)$. If $Z$ has smooth boundary $\partial Z$, we get
    \[\sing(\k_Z)=0_Z\cup\{(x,-\lambda\nu_Z(x))\in T^*M\vb x\in\partial Z,\lambda>0\}\]
    where $\nu_Z(x)$ is the exterior normal covector of $Z$ at $x$.
\end{exm}
    
\begin{exm}
    If $j\colon Y\hookrightarrow M$ is the inclusion of the open set $Y\subset M$, then we can define the sheaf $\k_Y$ on $M$ as $j_!(\k_Y)$ where the second $\k_Y$ is the constant sheaf on $Y$. We have $\Gamma(U;\k_Y)=\k^{c(U\cap\bar Y,\partial Y)}$ where $c(U\cap\bar Y,\partial Y)$ is the number of the connected components of $U\cap\bar Y$ which do not intersect with $\partial Y$. We have the exact sequence
    \[0\rightarrow \k_Y\rightarrow \k_M\rightarrow \k_Z\rightarrow 0\]
    where $Z=M\setminus Y$ and using Proposition \ref{prp:sing-properties}(v), if $Y$ has smooth boundary, we get
    \[\sing(\k_Y)=0_{\bar Y}\cup\{(x,\lambda\nu_Y(x))\in T^*M\vb x\in\partial Y,\lambda>0\}\ .\]
\end{exm}

We have the following proposition which makes Proposition \ref{prp:test-function} useful for constructible sheaves:

\begin{prp}[\cite{shende-treumann-zaslow}]\label{prp:sing-conormal}
    Let $\cS$ be a Whitney stratification of $M$ and $\cF\in\Sh(M)$. Then $\cF\in\Sh\dd_{\cS}(M)$ if and only if $\sing(\cF)\subset N^*\cS$,
    where $N^*\cS$ is defined as $\bigsqcup_{\cS_i\in\cS}N^*\cS_i$.
\end{prp}

\begin{rmk}
    If $L=N^*S$ for some submanifold $S$ of $M$ and $f\colon M\to\R$ is smooth, then $\Gamma_{df}$ and $L$ are transverse at $(x,p)\in T^*M$ if and only if $x$ is nondegenerate critical point of $f|_S$. Hence, assuming $(x,p)$ is a smooth point of $N^*\cS$ and $(x,p)\in N^*\cS_i$, $f$ is a proper test function for $\cF\in\Sh\dd_{\cS}(M)$ at $(x,p)$ if and only if $x$ is nondegenerate critical point of $f|_{\cS_i}$
\end{rmk}

\begin{exm}
    Let us consider an example in more detail: Let $Z=\{y\leq x^3\}\subset\R^2$, $i\colon Z\hookrightarrow\R^2$ be the inclusion and $\cF=i_*(\k_Z)$. It can be seen that $\cF$ is $\cS$-constructible where strata in $\cS$ are $\partial Z=\{y=x^3\}$ and the two disjoint complements of the boundary. By Proposition \ref{prp:sing-conormal} $\sing(F)\subset N^*\cS$. We want to know if $(0,-dy)\in N^*\partial Z$ is in $\sing(\cF)$ or not.

    If we try the test function $f(x,y)=-y$, we get $\cF_{(0,-dy),f}\simeq 0$. However, the problem is $\Gamma_{df}$ and $N^*\partial Z$ are not transverse, since $0$ is a degenerate critical point of $f|_{\partial Z}(x)=x^3$. This means $f$ is not a proper test function.

    If we try the test function $f(x,y)=-x^2-y$, then $\Gamma_{df}$ and $N^*\partial Z$ are transverse, since $0$ is a nondegenerate critical point of $f|_{\partial Z}(x)=-x^2-x^3$. Hence $f$ is a proper test function. We have
    \[F_{(0,-dy)}=F_{(0,-dy),f}\simeq C(\k\xrightarrow{\mx{\id\\ \id}}\k^2)\not\simeq 0\]
    hence $(0,-dy)\in\sing(\cF)$.
\end{exm}

\begin{exm}
    If $S$ is a closed submanifold of $M$, we have the stratification $\cS$ of $M$ consisting of $S$ and the connected components of the complement which are open submanifolds. Consider the sheaf $\k_S$ on $M$ defined as $i_*(\k_S)$ where the second $\k_S$ is the constant sheaf on $S$ and $i\colon S\hookrightarrow M$ is the inclusion. It is $\cS$-constructible, hence by Proposition \ref{prp:sing-conormal} we have $\sing(\k_S)\subset N^*\cS=N^*S\cup 0_M$. In fact, it is easy to see that we have $\sing(\k_S)=N^*S$.
\end{exm}

The singular support has a deep geometrical meaning, given by the following theorem by Kashiwara and Schapira:
    
\begin{thm}[\cite{kashiwara-schapira}, \cite{kashiwara1983}]\label{thm:kashiwara-schapira}
    If $\cF\in\Sh(M)$, then $\sing(\cF)$ is a coisotropic subset of $T^*M$. Moreover, $\cF\in\Sh\dd_c(M)$ if and only if $\sing(\cF)$ is a (singular) Lagrangian.
\end{thm}

Now, we are ready to define our main object of interest:

\begin{dfn}\label{dfn:sheaf-category-sing}
    For a conic subset $L\subset T^*M$, we define $\Sh_L(M)$ as the full dg subcategory of $\Sh(M)$ consisting of sheaves with perfect stalks whose singular support lies in $L$. If we drop the perfect stalk condition, we get the dg category $\Sh\dd_{L}(M)$.
    
    We also define $\Sh_L(M)_0$ as the dg quotient $\Sh_L(M)/\Loc(M)$, and $\Sh_L(M)\dd_0$, as the dg quotient $\Sh\dd_L(M)/\Loc\dd(M)$. Equivalently, as in \cite{STWZ}, we could define them as the full dg subcategory of $\Sh_L(M)$, respectively $\Sh_L\dd(M)$, consisting of sheaves with zero stalk at any specified point. Note that the equivalence depends on the point.
\end{dfn}

\begin{rmk}
    With this new definition, Proposition \ref{prp:sing-conormal} can be rephrased as
    \begin{align*}
        \Sh\dd_{\cS}(M)&=\Sh\dd_{N^*\cS}(M)\\
        \Sh_{\cS}(M)&=\Sh_{N^*\cS}(M)\ .
    \end{align*}
\end{rmk}
    
\begin{cor}\label{cor:sing-constructible}
    If $L$ is conic (singular) Lagrangian in $T^*M$, $\Sh\dd_{L}(M)$ is a full dg subcategory of $\Sh\dd_c(M)$. If moreover $L\subset N^*\cS$ for a Whitney stratification $\cS$ of $M$, then $\Sh\dd_L(M)$ is a full dg subcategory of $\Sh\dd_{\cS}(M)$. Similar statements hold for $\Sh_L(M)$.
\end{cor}

\begin{proof}
    Let $\cF\in\Sh_L(M)$. Then we have $\sing(\cF)\subset L$, and $\sing(\cF)$ is coisotropic, which implies $\sing(\cF)$ is a Lagrangian. Hence by Theorem \ref{thm:kashiwara-schapira} $\cF\in\Sh\dd_c(M)$. If we assume $L\subset N^*\cS$, then we have
    \[\Sh\dd_L(M)\subset\Sh\dd_{N^*\cS}(M)=\Sh\dd_{\cS}(M)\]
    by Proposition \ref{prp:sing-conormal}.
\end{proof}

This corollary allows us to describe $\Sh\dd_L(M)$ combinatorially: We first stratify $M$ with the Whitney stratification $\cS$ such that $\cS$ is regular cell complex and $L\subset N^*\cS$. By Corollory \ref{cor:sing-constructible}, we have $\Sh\dd_L(M)\subset\Sh\dd_{\cS}(M)$ as the full dg subcategory. Also, for $\Sh\dd_{\cS}(M)$ we have the combinatorial description $\Sh\dd_{\cS}(M)\simeq\Modk(\cS)$ by Proposition \ref{prp:combinatorics-constructible}. Hence we have
\[\Sh\dd_L(M)\subset\Modk(\cS)\]
as the full dg subcategory, and $\Sh\dd_L(M)$ can be described by imposing the microsupport condition $\sing(\cF)\subset L$ on the objects $\cF\in\Modk(\cS)$. Similarly, we have
\[\Sh_L(M)\subset\Perfk(\cS)\]
as the full dg subcategory. Moreover, $\Sh\dd_L(M)_0$ and $\Sh_L(M)_0$ can be obtained by setting the stalk at any specified point as $0$.

\begin{rmk}
    All previous examples regarding the calculation of singular support can be approached by this combinatorial way.
\end{rmk}

\begin{exm}
    Let $M=\R$ and $L=0_\R\cup\{(0,\lambda dx)\vb \lambda>0\}\subset T^*\R$ which is a conic Lagrangian, shown in Figure \ref{fig:conic-lag-in-TR}:
    \begin{figure}[h]
        \centering
        
        \begin{tikzpicture}
            \draw (0,0) -- (0,2) node[right]{$L$};
            \draw[thick,red] (-2,0) -- (0,0) node[midway,below]{$\cS_1$};
            \draw[thick,purple] (0,0) -- (2,0) node[midway,below]{$\cS_3$} node[right,black]{$\R$};
            \fill[blue] (0,0) circle (2pt) node[below]{$\cS_2$};
            \node at (-2,2) {$T^*\R$};
        \end{tikzpicture}
        
        \caption{The conic Lagrangian $L$ in $T^*\R$}
        \label{fig:conic-lag-in-TR}
    \end{figure}
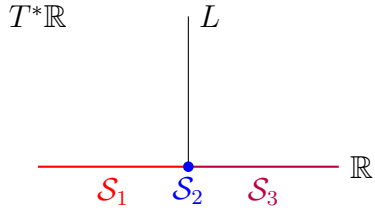
    
    To calculate $\Sh_L(M)$, we choose the Whitney stratification $\cS=\{\cS_1,\cS_2,\cS_3\}$ of $\R$ where
    \[
        \cS_1=\{0\},
        \cS_2=\{x\in\R\vb x<0\},
        \cS_3=\{x\in\R\vb x>0\}
    \]
    shown by colours in Figure \ref{fig:conic-lag-in-TR}. Note that $\cS$ is regular cell complex and $L\subset N^*\cS$. Then we have
    \[\Sh\dd_L(M)\subset\Modk(\cS)\]
    as the full dg subcategory. $\Modk(\cS)$ has the objects
    \[\begin{tikzcd}[column sep=10pt, row sep=10pt]
          & A\ar[ld,"f"']\ar[rd,"g"] &   \\
        B &   & C
    \end{tikzcd}\]
    where $A,B,C\in\Modk$ and the microsupport condition forces $g$ to be quasi-isomorphism. So, one can present the objects of $\Sh\dd_L(M)$ as
    \[\begin{tikzcd}
        A\ar[r,"f"] & B
    \end{tikzcd}\]
    which are representations of $A_2$-quiver
    \[\bullet\longrightarrow\bullet\]
    hence we get
    \[\Sh\dd_L(M)=\Modk(A_2)\]
    and also
    \[\Sh_L(M)=\Perfk(A_2)\ .\]
    By setting the stalk at $0$ $A\simeq 0$, we also have
    \begin{align*}
        \Sh\dd_L(M)_0&=\Modk\\
        \Sh_L(M)_0&=\Perfk\ .
    \end{align*}
\end{exm}

$\Sh\dd_{L}(M)$ has stabilisation property (see \cite{kashiwara-schapira}, and \cite{h-principle}), which is useful in calculations by reducing the dimension:

\begin{prp}\label{prp:stabilisation}
	Let $L\subset T^*M$ be a conic set, then $L\times \R\subset T^*M\times T^*\R\simeq T^*(M\times\R)$ is also conic and we have
	\[\Sh\dd_{L\times\R}(M\times\R)\simeq\Sh\dd_L(M)\ .\]
	Same equivalence holds if we remove the diamonds.
\end{prp}

Next, we will present a very important theorem which clarifies the dependence of $\Sh\dd_L(M)$ on the Lagrangian $L$:
    
\begin{dfn}\label{dfn:associated-conic}
   	Given a subset $\Lambda\subset T^{\infty}M$, we define the cone of $\Lambda$ in $T^*M$ as
    \[\R_{>0}\Lambda:=\{(x,p)\in T^*M\vb (x,[p])\in\Lambda\}\ .\]
    The \textit{associated conic subset} for $\Lambda$ is defined as
    \[L_{\Lambda}=0_M\cup\R_{>0}\Lambda\subset T^*M\]
    and we set
    \begin{align*}
        \Sh\dd_{\Lambda}(M)&:=\Sh\dd_{L_{\Lambda}}(M)\\
        \Sh_{\Lambda}(M)&:=\Sh_{L_{\Lambda}}(M)\ .
    \end{align*}
\end{dfn}
    
\begin{thm}[\cite{guillermou-kashiwara-schapira}, \cite{shende-treumann-zaslow}]
	Let $\Lambda\subset T^{\infty}M$ be a closed (singular) Legendrian and $h_t$ be a Hamiltonian isotopy of $T^{\infty}M$ (with $h_0=\id$). If $h_t$ has compact horizontal support (i.e. there exists an open $U\subset M$ with compact closure such that $h_t(x,[p])=(x,[p])$ for $x\notin U$), then $\Sh\dd_{\Lambda}(M)$ and $\Sh\dd_{h_t(\Lambda)}(M)$ are quasi-equivalent for any $t$.
\end{thm}

\begin{rmk}
    This shows that $\Sh\dd_{\Lambda}(M)_0$ and $\Sh\dd_{h_t(\Lambda)}(M)_0$ are also quasi-equivalent since they are obtained by taking dg quotient by $\Loc\dd(M)$. Also, same equivalences hold if we remove the diamonds.
\end{rmk}

\subsection{Wrapped Constructible Sheaves}

The material in this section is introduced in \cite{wrapped}. The term ``wrapped'' is justified by the result in \cite{gps3} (Theorem \ref{thm:gps}) relating wrapped constructible sheaves and (partially) wrapped Fukaya categories.

Let $M$ be a smooth manifold, $\cS$ be a Whitney stratification on $M$, $L$ be a conic subset of $T^*M$, and $\Lambda$ be a subset in $T^{\infty}M$.

\begin{dfn}
	We define the dg categories $\Sh_c^w(M)$, $\Sh_{\cS}^w(M)$, $\Sh_{L}^w(M)$ as the full dg subcategories of compact objects in the dg categories $\Sh_c\dd(M)$, $\Sh_{\cS}\dd(M)$, $\Sh_{L}\dd$(M), respectively. We refer their objects as \textit{wrapped constructible sheaves}. We also write $\Sh_{\Lambda}^w(M)$ for $\Sh_{L_{\Lambda}}^w(M)$.
\end{dfn}

One has the following useful proposition regarding compact objects:

\begin{prp}[{\cite[Proposition 6.3]{rickard-morita}, \cite[Proposition 6.4]{bokstedt-homotopy}}]\label{prp:mod-perf}
	Given an $A_{\infty}$-category $\cC$, the full dg subcategory of compact objects in $\Modk(\cC)$ is generated by $\cC$.
\end{prp}

By Proposition \ref{prp:combinatorics-constructible}, this implies that
\[\Sh_{\cS}^w(M)\simeq\Perf(\cS)\]
if $\cS$ is a regular cell complex, as remarked in \cite{gps3}.

\begin{exm}\label{exm:wrapped-local-loop}
	We can define $\Loc^w(M)$ as  the full dg subcategory of compact objects in $\Loc\dd(M)$. Then if $M$ is connected, we have
	\[\Loc^w(M)\simeq\Perf(C_{-*}(\Omega M))\]
	since $\Loc\dd(M)\simeq\Modk(C_{-*}(\Omega M))$
\end{exm}

\section{Microlocal Sheaves and Fukaya Category}

For a given Weinstein pair $(W,\Lambda)$, we will define the infinitesimally wrapped and (partially) wrapped Fukaya categories. Then, we will try to give a definition for traditional and wrapped microlocal sheaves associated to $(W,\Lambda)$, which are conjecturally expected to correspond to the respective Fukaya categories. This definition works for the application in this thesis, namely the calculation of microlocal sheaves on pinwheels in Section \ref{sec:msh-pinwheel}.

\subsection{Weinstein Manifolds}\label{sec:weinstein}

In this section, we will review some material regarding to exact symplectic, Liouville, and Weinstein manifolds. Definitions and results in this chapter are mostly taken from \cite{weinstein} and \cite{revisited}.

\begin{dfn}
    An \textit{exact symplectic manifold} $(W,\theta)$ is a symplectic manifold $(W,\omega)$ with a 1-form $\theta$ on $W$, called \textit{Liouville form}, such that $d\theta=\omega$. An \textit{exact symplectomorphism} $F\colon (W,\theta)\to(W',\theta')$ between exact symplectic manifolds is a diffeomorphism such that $F^*\theta'-\theta$ is exact.
    
    A Lagrangian $L$ in $(W,\theta)$ is called an \textit{exact} if there exist a smooth function $f\colon L\to\R$, called \textit{potential}, such that $\theta|_L=df$. Such $L$ is called \textit{compactly exact} if furthermore $f$ is compactly supported. A singular Lagrangian is \textit{(compactly) exact} if there exist a (compactly supported) continuous function $f\colon L\to\R$ such that $f|_{L_i}$ is smooth and $\theta|_{L_i}=df|_{L_i}$ for any of its stratum $L_i$. 
\end{dfn}

\begin{prp}\label{prp:exact-to-exact}
	An exact symplectomorphism maps a (singular) exact Lagrangian to a (singular) exact Lagrangian.
\end{prp}

\begin{proof}
	Let $F\colon (W,\theta)\to(W',\theta')$ be an exact symplectomorphism. Then there exists a smooth $f\colon W\to\R$ such that $df=F^*\theta'-\theta$. Let $L'$ be an exact singular Lagrangian with the strata $L'_i$. Then $L:=F^{-1}(L')$ is a singular Lagrangian with the strata $L_i:=F^{-1}(L'_i)$. Exactness of $L'$ implies that there exists a continuous function $g\colon L'\to\R$ such that $\theta'|_{L'_i}=dg|_{L'_i}$. Define the continuous function $h\colon L\to\R$ by $(g\circ F-f)|_L$. We have
	\[\theta|_{L_i}=(F^*\theta')|_{L_i}-df|_{L_i}=d(g\circ F)|_{L_i}-df|_{L_i}=dh|_{L_i}\]
	hence $L$ is exact.
\end{proof}

\begin{dfn}
    The \textit{Liouville vector field} $X$ for the Weinstein manifold $(W,\theta)$ is defined by the equality $\iota_X\omega=\theta$, where $\iota_X\omega$ is the interior product defined as $\iota_X\omega(Y):=\omega(X,Y)$. We denote its flow by $X^t$ and call it \textit{Liouville flow}. Note that it is  symplectically expanding, i.e. $(X^t)^*\omega = e^t\omega$ (and moreover, $(X^t)^*\theta = e^t\theta$), hence $-X$ is contracting. A subset $S\subset W$ is called \textit{conic} if it is invariant under Liouville flow.
\end{dfn}

\begin{lem}\label{lem:liouville-form-field}
    Let $L$ be a Lagrangian in an exact symplectic manifold and $x$ be a point in $L$. The Liouville form $\theta$ vanishes on $T_xL$ if and only if the Liouville vector field $X$ at $x$ lies inside $T_xL$.
\end{lem}

\begin{proof}
    If $\theta$ vanishes on $T_xL$, we get $\omega(X,Y)=0$ for all $Y\in T_xL$. Then $\omega$ vanish on
    \[\{Y+kX\vb Y\in T_xL,k\in\R\}\ .\]
    If we assume $X\notin T_xL$, this subspace cannot be a isotropic since it has the dimension $\dim L+1$ and $L$ is Lagrangian. 
    
    Conversely, let $i\colon L\hookrightarrow W$ be the inclusion and assume $X\in T_xL$. For any $Y\in T_xL$ we have
    \[\theta|_L(Y)=(i^*\theta)(Y)=\theta(di(Y))=\omega(X,di(Y))=0\]
    since $di(Y)\in T_xL$ and $\omega|_L=0$.
\end{proof}

\begin{prp}\label{prp:conic-exact}
    If $L$ is conic Lagrangian in an exact symplectic manifold $(W,\theta)$, then $\theta|_L=0$. In particular, conic Lagrangians are compactly exact.
\end{prp}

\begin{proof}
    If $L$ is conic, then it is invariant under Liouville flow. So at each point $x$ of $L$, the Liouville vector field lies inside $T_xL$. Then the result follows from Lemma \ref{lem:liouville-form-field}.
\end{proof}

We have also the partial converse of this proposition:

\begin{prp}
    If $L$ is closed Lagrangian in an exact symplectic manifold $(W,\theta)$ with $\theta|_L=0$, then $L$ is conic.
\end{prp}

\begin{proof}
    By Lemma \ref{lem:liouville-form-field}, at each point $x$ of $L$ the Liouville vector field lies inside $T_xL$. Closedness guarantees that $L$ does not escape under its flow.
\end{proof}

\begin{exm}\label{exm:cotangent-standard}
    The cotangent bundle $(T^*M,\omega=\sum_idx_i\wedge dp_i)$ of an $n$-dimensional smooth manifold $M$ is an exact symplectic manifold with the standard Liouville form $\theta=-\sum_i p_i dx_i$. Its Liouville vector field is $X=\sum_i p_i \partial_{p_i}$, where $\partial_{p_i}$ is the partial derivative $\partial/\partial p_i$, see Figure \ref{fig:standard-Liouville} for a particular example. Its flow is given by $X^t=(x_i,p_i e^t)$ which is complete. The skeleton is given by $\fX_{T^*M}=0_M$. Note that it is exact Lagrangian, moreover it is conic.
\end{exm}

\begin{figure}[h]
    \centering
    
    \begin{tikzpicture}
        \newcommand{\arrow}[3][blue]{\draw[#1] (#2,#3*0.5) -- (#2-0.1,#3*0.4);\draw[#1] (#2,#3*0.5) -- (#2+0.1,#3*0.4);\draw[#1] (#2,#3*0.2) -- (#2,#3*0.5)}
    
        \draw (-2,0) -- (2,0) node[right]{$\R$};
        \node at (-2,1){$T^*\R$};
        
        \arrow{0}{1};
        \arrow{0.7}{1};
        \arrow{1.4}{1};
        \arrow{-0.7}{1};
        \arrow{-1.4}{1};
        \arrow{0}{-1};
        \arrow{0.7}{-1};
        \arrow{1.4}{-1};
        \arrow{-0.7}{-1};
        \arrow{-1.4}{-1};
    \end{tikzpicture}
    
    \caption{The Liouville vector field associated to $\theta=-pdx$ in $T^*\R$}
    \label{fig:standard-Liouville}
\end{figure}
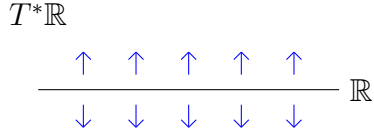

\begin{dfn}
    A \textit{Liouville manifold} $(W,\theta)$ is an exact symplectic manifold whose Liouville vector field $X$ is complete and there is an exhaustion of $W$ by compact domains $\{W_k\}_{k=1}^{\infty}$ with smooth boundaries such that $X$ is outwardly transverse to $\partial W_k$ for each $k$.
    
    The \textit{skeleton} of the Liouville manifold $(W,\theta)$ is defined by
    \[\fX_W=\bigcup_{k=1}^{\infty}\bigcap_{t>0}X^{-t}(W_k)\]
    where $X^{-t}$ is the negative Liouville flow of the Liouville vector field $X$. Note that this does not depend on the exhaustion $\{W_k\}$. It is easy to see that $\fX_W$ is homotopy equivalent to $W$. 
    
    $(W,\theta)$ is called \textit{finite-type} if $\fX_W$ is compact. In that case, we can fix a single compact domain $W^c$ containing $\fX_W$ such that the Liouville vector field is outward pointing on $\partial W^c$. Then we can regard $W$ as
    \[W^c\cup_{\partial W^c}(\partial W^c\times[0,\infty))\]
    where we call $W^c$ \textit{Liouville domain}, $\partial W^c$ \textit{ideal contact boundary} of $W$, $\partial W^c\times[0,\infty)$ \textit{cylindrical end}, . The Liouville form becomes $e^r\alpha$ on the cylindrical end, where $r$ is the radial coordinate and $\alpha=\theta|_{\partial W^c}$, and Liouville vector field becomes $\partial_r$. $\partial W^c$ is naturally a contact manifold with the contact form $\alpha$. Note that in the case of finite-type Liouville manifolds, the skeleton is just
    \[\fX_W=\bigcap_{t>0}X^{-t}(W^c)\]
    which is just tracing back the Liouville flow.
\end{dfn}

\begin{rmk}\label{rmk:conic-set}
    If $(W,\theta)$ is a finite-type Liouville manifold with the Liouville domain $W^c$,  then we have the exact symplectomorphism
    \[W\setminus\fX_W\overset{\sim}{\longrightarrow}\partial W^c\times\R\]
    where $\partial W^c\times\R$ has the Liouville form $e^r\alpha$, $r$ is the radial coordinate, and $\alpha$ is the contact form of $\partial W^c$. With this formulation, a conic set $L$ can be written as
    \[L=L'\sqcup(\Lambda\times\R)\]
    where $L'=L\cap\fX_W$ and $\Lambda=L\cap\partial W^c$.
\end{rmk}
    
\begin{dfn}
	For a subset $\Lambda\subset\partial W^c$, we define the cone of $\Lambda$ in $W$ as
	\[\R_{>0}\Lambda:=\bigcup_{t\in\R}X^t(\Lambda)\]
	for the Liouville flow $X^t$. The \textit{associated conic subset} for $\Lambda$ is defined as
	\[L_{\Lambda}:=\fX_W\sqcup\R_{>0}\Lambda\subset W\]
	as in Definition \ref{dfn:associated-conic} in cotangent bundle case. Note that in Remark \ref{rmk:conic-set}, $\Lambda\times\R$ is the cone of $\Lambda$, hence
	\[L=L'\sqcup\R_{>0}\Lambda\subset L_{\Lambda}\ .\]
\end{dfn}

Compactly exact Lagrangians also have a similar presentation:

\begin{prp}\label{prp:cylindrical-end}
    Let $L$ be an exact closed Lagrangian in a finite-type Liouville manifold $(W,\theta)$. Then $L$ is compactly exact if and only if $L$ is a Lagrangian with cylindrical end, i.e.
    \[L=L^c\cup_{\partial L^c}(\partial L^c\times[0,\infty))\]
    where $L^c\subset W^c$ for a Liouville domain $W^c$ and $\partial L^c\subset\partial W^c$ is the Legendrian boundary of $L^c$ in $W^c$.
\end{prp}

\begin{proof}
    Since $L$ is compactly exact and $W$ is finite-type, there is a Liouville domain $W^c$ such that $\fX_W\subset W^c$ and $\theta|_L=0$ outside $W^c$. By Lemma \ref{lem:liouville-form-field}, the Liouville vector field is tangent to $L$ outside $W^c$ and since $L$ is closed, $L$ is invariant under Liouville flow there. Hence outside $W^c$, $L$ can be modelled as
    \[\partial L^c\times [0,\infty)\subset\partial W^c\times [0,\infty)\]
    where $L^c=L\cap W^c$ and $\partial L^c\subset\partial W^c$.
    
    Conversely, if $L$ is a Lagrangian with cylindrical end, then $\theta$ vanish on $\partial L^c\times (0,\infty)$ by Lemma \ref{lem:liouville-form-field} since the Liouville vector field lies in its tangent bundle by definition. Hence $L$ is compactly exact.
\end{proof}

\begin{exm}
    $T^*M$ as in Example \ref{exm:cotangent-standard} is not a Liouville manifold if $M$ is noncompact, since here is no compact exhaustion $\{W_k\}$ such that the Liouville vector field is outwardly transverse to $\partial W_k$. Nevertheless, we can still see it as a (noncompact) domain $W^c$ with cylindrical end attached, where
    \[W^c=\{(x_i,p_i)\in T^*M\vb |(p_1,\ldots,p_n)|\leq 1\}\]
    and $\partial W^c$ can be canonically identified with $T^{\infty}M$.
\end{exm}

\begin{exm}
    If $M$ is compact, $T^*M$ as in Example \ref{exm:cotangent-standard} is a Liouville manifold.
\end{exm}

\begin{exm}\label{exm:cotangent-nonstandard}
    Consider $(T^*\R^n,\omega=\sum_idx_i\wedge dp_i)$ with a different (nonstandard) Liouville form $\theta=\frac{1}{2}\sum_i(x_idp_i-p_idx_i)$. Its Liouville vector field is $X=\frac{1}{2}\sum_i(x_i\partial_{x_i}+p_i\partial_{p_i})$, see Figure \ref{fig:nonstandard-Liouville} for a particular example. Its flow is $X^t=(x_ie^{\frac{1}{2}t},p_ie^{\frac{1}{2}t})$ which is complete. Its skeleton is given by $\fX_{T^*\R^n}=\{0\}$, a single point. With this Liouville form, $T^*\R^n$ is Liouville since there is a compact exhaustion by closed balls around $0$. It can be seen as a compact domain $W^c$ with cylindrical end attached, where
    \[W^c=\{(x_i,p_i)\in T^*\R^n\vb |(x_1,\ldots,x_n,p_1,\ldots,p_n)|\leq 1\}\ .\]
    Note that $\partial W^c$ is not contactomorphic to $T^{\infty}\R^n$ with this Liouville form.
\end{exm}

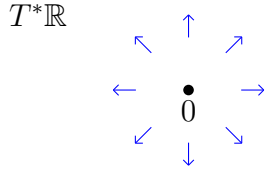
\begin{figure}[h]
    \centering
    
    \begin{tikzpicture}
        \newcommand{\arrow}[3][blue]{
        \draw[#1] ({(#2+0.3)*cos(#3)+0.1*cos(#3+135)},{(#2+0.3)*sin(#3)+0.1*sin(#3+135)}) -- ({(#2+0.3)*cos(#3)},{(#2+0.3)*sin(#3)});
        \draw[#1] ({(#2+0.3)*cos(#3)+0.1*cos(#3+225)},{(#2+0.3)*sin(#3)+0.1*sin(#3+225)}) -- ({(#2+0.3)*cos(#3)},{(#2+0.3)*sin(#3)});
        \draw[#1] ({#2*cos(#3)},{#2*sin(#3)}) -- ({(#2+0.3)*cos(#3)},{(#2+0.3)*sin(#3)})
        }
    
        \fill (0,0) circle (2pt) node[below]{$0$};
        \node at (-2,1){$T^*\R$};
        \node at (2,1){};
        
        \arrow{0.7}{0};
        \arrow{0.7}{45};
        \arrow{0.7}{90};
        \arrow{0.7}{135};
        \arrow{0.7}{180};
        \arrow{0.7}{225};
        \arrow{0.7}{270};
        \arrow{0.7}{315};
    \end{tikzpicture}
    
    \caption{The Liouville vector field associated to $\theta=\frac{1}{2}(xdp-pdx)$}
    \label{fig:nonstandard-Liouville}
\end{figure}

In general, the skeleton of a Liouville manifold may not be isotropic, see \cite{McDuff1991}, but we have a better behaved class of exact symplectic manifolds, which will be our central object to study:

\begin{dfn}
    A \textit{Weinstein manifold} $(W,\theta,\phi)$ is an exact symplectic manifold $(W,\theta)$ together with an exhausting (i.e. proper and bounded below) Morse function $\phi\colon W\to\R$ such that the Liouville vector field $X$ is complete and gradient-like for $\phi$, i.e.
    \[d\phi(X)\geq\delta(|X|^2+|d\phi|^2)\]
    for some $\delta>0$, where $|X|$ is the norm with respect to some Riemannian metric on $W$ and $|d\phi|$ is the dual norm. Here, $\phi$ is called \textit{Lyapunov function} for $X$. We call the pair $(\theta,\phi)$ a \textit{Weinstein structure} on $W$.
    
    One can relax the conditions on the Lyapunov function, e.g. instead of Morse functions, we could work with Morse-Bott ones. We will work with Lyapunov functions which are Morse.
\end{dfn}

\begin{rmk}
    In \cite[Section 9.3-4]{weinstein}, it is shown that the above inequality implies the following:
    \begin{enumerate}
    	\item The zero locus of $X$, denoted by $\Zero(X)$, and the critical locus of $\phi$, denoted by $\Crit(\phi)$, coincide. Moreover, a zero of $X$ is nondegenerate if and only if it is a nondegenerate critical point of $\phi$,
    	
    	\item A nondegenerate zero of $X$ is hyperbolic, i.e. all eigenvalues of the linearisation of $X$ at the zeroes have nonzero real part,
    	
    	\item All the limit points of the Liouville flow are in $\Zero(X)$.
    \end{enumerate}
    
  	The first two points imply that our Lyapunov function has only hyperbolic zeroes. Then we can define the \textit{stable manifold} at $p\in \Zero(X)$ as
    \[W_p^-:=\{x\in W\vb \lim_{t\to\infty} X^t(x)=p\}\]
    and the \textit{unstable manifold} at $p$ as
    \[W_p^+:=\{x\in W\vb \lim_{t\to -\infty} X^t(x)=p\}\]
    where $\dim W_p^-$ is equal to the index of $\phi$ at $p$, and $\dim W_p^+=\dim W-\dim W_p^-$. By the third point, we can express the skeleton of $W$ as the union of the stable manifolds, i.e.
    \[\fX_W=\{x\in W\vb \lim_{t\to\infty}X^t(x)\in\text{Zero}(X)\}\]
    or equivalently,
    \[\fX_W=\{x\in W\vb \lim_{t\to\infty}X^t(x)\in\text{Crit}(\phi)\}\ .\]
\end{rmk}

\begin{lem}[\cite{weinstein}]\label{lem:stable}
	For any $p\in\Zero(X)$, the stable manifold $W_p^-$ is isotropic (i.e. $\omega$-isotropic), moreover it is $\theta$-isotropic. The unstable manifold $W_p^+$ is coisotropic.
\end{lem}

\begin{prp}\label{prp:isotropic-skeleton}
    Given a Weinstein manifold $W$, there is a stratification of its skeleton $\fX_W$ by isotropic submanifolds of $W$. Moreover, $\fX_W$ is Whitney stratifiable if the Liouville flow is Morse-Smale (i.e. stable and unstable manifolds intersect transversely at every critical point) and near critical points the Liouville vector field is gradient with respect to an Euclidean metric.
\end{prp}

\begin{proof}
	First statement follows from Lemma \ref{lem:stable}, last one is proven in \cite{laudenbach}.
\end{proof}

\begin{rmk}
	Any Weinstein manifold is Liouville, by defining the compact exhaustion $\{W_k\}$ by $W_k=\{\phi\leq c_k\}$ where $\{c_k\}$ is a sequence of regular values of $\phi$ such that $\lim_{k\to\infty}c_k=\infty$. However, not every Liouville manifold is diffeomorphic to a Weinstein manifold, see \cite{McDuff1991}.
	
	It is easy to see that a Weinstein manifold is finite-type if and only if $\phi$ has finitely many critical points. If the Weinstein manifold is finite, it has the Liouville domain $W_k=\{\phi\leq c_k\}$ for some $k$, which is called \textit{Weinstein domain}.
\end{rmk}

\begin{exm}\label{exm:cotangent-weinstein}
    If $M$ is compact, after perturbing the Liouville vector field (see \cite{weinstein}), $T^*M$ in Example \ref{exm:cotangent-standard} is Weinstein with the Lyapunov function $\phi\colon T^*M\to\R$ defined as
    \[\phi(x_i,p_i)=f(x_1,\ldots,x_n)+\frac{1}{2}\sum_{i=1}^n p_i^2\]
    where $f\colon M\to\R$ is a sufficiently small Morse function. If we had allowed Morse-Bott functions, we could have chosen
     \[\phi(x_i,p_i)=\frac{1}{2}\sum_{i=1}^n p_i^2\]
    and no perturbation would be needed. Observe that the skeleton $\fX_{T^*M}=0_M$ is isotropic, moreover it is Lagrangian.
\end{exm}

\begin{exm}\label{exm:cotangent-weinstein-nonstandard}
    $T^*\R^n$ in Example \ref{exm:cotangent-nonstandard} is Weinstein with the Lyapunov function $\phi\colon T^*\R^n\to\R$ defined as
    \[\phi(x_i,p_i)=\frac{1}{4}\sum_{i=1}^n x_i^2+p_i^2\ .\]
    Observe that the skeleton $\fX_{T^*\R^n}=\{0\}$ is isotropic, but it is not Lagrangian.
\end{exm}

There is a nice way to deal with Weinstein manifolds:

\begin{dfn}\label{dfn:weinstein-handle}
	The Lyapunov function $\phi$, being exhausting Morse, gives a handle decomposition of the Weinstein domain $W^c$ with finitely many handles, called \textit{Weinstein handle decomposition}. The regular sublevel set $W^a:=\{\phi\leq a\}$ is a Weinstein domain for every regular value $a$ of $\phi$, with the contact boundary $\partial W^a=\{\phi=a\}$. The core of a handle attached to $W^a$, given by the stable manifold $W_p^-\cap\{\phi\geq a\}$ at a critical point $p$ of $\phi$, is isotropic, and its attaching sphere $W_p^-\cap\{\phi=a\}$ is isotropic in $\partial W^a$. The cocore of a handle attached to $W^a$, given by the unstable manifold $W_p^+\cap\{\phi\leq b\}$ for some small enough $b> \phi(p)$, is coisotropic. Conversely, using Weinstein handles we can get a Weinstein domain by attaching them along isotropic attaching spheres, so that the Weinstein structure extends to the resulting Weinstein domain. Note that the skeleton can be thought as the union of the cores of the handles, although we should take the whole stable manifold instead of just the core.
	
	Let $W$ be $2n$-dimensional. Since stable manifolds are isotropic, $\phi$ has critical points with index at most $n$, hence we have only $k$-handles for $k\leq n$. A $k$-handle with $k<n$ is called \textit{subcritical handle}, and an $n$-handle is called \textit{critical handle}. If a Weinstein domain has only subcritical handles, then it is called \textit{subcritical}. The symplectic topology of subcritical Weinstein domains is trivial by \cite{weinstein}, therefore only the critical handles contribute to the symplectic structure. The normal bundle of the attaching sphere of a critical handle automatically trivialises since the core of the critical handle is Lagrangian and by Weinstein neighbourhood theorem its neighbourhood gives the handle. Hence we do not need to specify normal framing for the attaching spheres of critical handles. This implies that a Weinstein domain can be described by a Legendrian link in the boundary of a subcritical Weinstein domain, i.e. by \textit{Legendrian surgery diagram}. We refer to the cocore of a critical handle as \textit{Lagrangian cocore}.
\end{dfn}

\begin{exm}\label{exm:cotangent-cocore}
	If $M$ is compact, the Weinstein manifold $T^*M$ in Example \ref{exm:cotangent-weinstein} has the Weinstein handle decomposition by the thickening of the handles of $M$ coming from the Morse function $f$. The Lagrangian cocore is the the cotangent fibre $T^*_pM$, where $p\in M$ is the critical point of $f$ with the index $n$.
\end{exm}

The class of Weinstein manifolds is not large enough for our purposes. For example, when $M$ is not compact the cotangent bundle $T^*M$ is not Weinstein. This leads us the following series of definitions, following \cite{revisited}:

\begin{dfn}\label{dfn:weinstein-pair}
	Let $N$ be a contact manifold with the contact form $\alpha$. A codimension $1$ submanifold $\Sigma$ of $N$ with boundary is called \textit{Weinstein hypersurface}, if $\Sigma$ is equipped with a Weinstein structure $(\alpha|_{\Sigma},\phi)$ such that its Liouville vector field is outwardly transverse to $\partial\Sigma$. Note that the skeleton $\fX_{\Sigma}$ of $\Sigma$ is stratified by isotropic submanifolds of $N$.
	
	A \textit{contact surrounding} $U(\Sigma)\subset N$ of a Weinstein hypersurface $\Sigma$ is obtained as follows: Take a bigger Weinstein hypersurface $\tilde\Sigma\supset\Sigma$ equipped with a Lyapunov function $\tilde\phi\leq\varepsilon^2$ such that $\tilde\phi$ has no critical point in $\tilde\Sigma\setminus\Sigma$, $\tilde\phi|_{\Sigma}=\phi$, and $\tilde\phi|_{\partial\tilde\Sigma}=\varepsilon^2$ for some $\varepsilon>0$. Consider a neighbourhood $\tilde U$ of $\tilde\Sigma$ diffeomorphic to $\tilde\Sigma\times(-\varepsilon,\varepsilon)$ such that $\alpha|_{\tilde U}=\pi^*(\alpha|_{\tilde\Sigma})+du$ where $u$ is the coordinate for $(-\varepsilon,\varepsilon)$ and $\pi\colon \tilde U\to\tilde\Sigma$ is the projection. Then we set
	\[U(\Sigma):=\{(t,u)\in \tilde U\simeq\tilde \Sigma\times(-\varepsilon,\varepsilon)\vb \tilde\phi(t)+u^2\leq\varepsilon^2\}\ .\]
	
	For a finite-type Weinstein manifold $W$, and a Weinstein hypersurface $\Sigma\subset\partial W^c$, the pair $(W,\Sigma)$ (or sometimes $(W,\fX_{\Sigma})$) is called a \textit{Weinstein pair}. Its skeleton is defined as
	\[\fX_{W,\Sigma}:=L_{\fX_{\Sigma}}\ .\]
	Note that not every closed singular $\Lambda\subset W^c$ is a skeleton of some $\Sigma$, hence $(W,\Lambda)$ is not a Weinstein pair for such $\Lambda$.
\end{dfn}

\begin{prp}[\cite{revisited}]\label{prp:weinstein-adjusted}
	Let $(W^c,\theta,\phi)$ be a finite-type Weinstein domain, and $\Sigma$ be a Weinstein hypersurface in $\partial W^c$. By seeing $W^c$ as a manifold with boundary, there is a Weinstein structure $(\theta',\phi')$ on $W^c$ with the Liouville vector field $X'$ satisfying
	\begin{enumerate}
		\item $d\theta'=d\theta$, and $\theta'=\theta$ outside a neighbourhood of $\Sigma$,
		
		\item $X'$ is tangent to $\partial W^c$ on $U(\Sigma)$ and transverse to $\partial W^c$ elsewhere,
		
		\item $X'|_{U(\Sigma)}=X_{\tilde\Sigma}+u\partial_u$,
		
		\item $\phi'|_{U(\Sigma)}=\phi_{\tilde\Sigma}+u^2$, and all critical values of $\phi'$ are smaller than $\phi'|_{\partial U(\Sigma)}=\varepsilon^2$,
		
		\item $\fX_{(W^c,\theta',\phi')}=\fX_{(W^c,\theta,\phi),\Sigma}$.
	\end{enumerate}
	where $U(\Sigma)$ and the related notations are as defined in Definition \ref{dfn:weinstein-pair}, $X_{\tilde\Sigma}$ and $\phi_{\tilde\Sigma}$ are the Liouville vector field and Lyapunov function on $\tilde\Sigma$, respectively.
\end{prp}

\begin{dfn}\label{dfn:weinstein-sector}
	Given a Weinstein pair $(W,\Sigma)$, the new Weinstein structure $(\theta',\phi')$ on $W^c$ obtained from Proposition \ref{prp:weinstein-adjusted} is said to be \textit{adjusted} to the Weinstein pair $(W,\Sigma)$. We know $\phi'$ has no critical values greater than or equal to $\varepsilon^2$, so it makes sense to define
	\[W'^c:=\{x\in W^c\vb \phi'(x)\leq\varepsilon^2\}\]
	which is a manifold with corner $\partial U(\Sigma)$. We define the \textit{Weinstein sector} $(W',\theta',\phi')$ corresponding to the Weinstein pair $(W,\Sigma)$ as the resulting space after attaching a cylindrical end to $W'^c$ along $\partial W'^c=\{\phi'=\varepsilon^2\}$. Note that $W'$ is a manifold with boundary if $\Sigma$ is nonempty, hence not a Weinstein manifold.
	
	Similarly, if we start with a Weinstein sector $W'$, we can get a Weinstein pair $(W,\Sigma)$ corresponding to $W'$. So there is a correspondence between Weinstein pairs and Weinstein sectors, and we can use them interchangeably.
\end{dfn}

\begin{rmk}
	A Weinstein sector is a special case of a Liouville sector, which is introduced in \cite{gps1}. Weinstein sectors are also defined in \cite{CDRGG}. An important fact to highlight is that there is no closed Reeb orbit on $U(\Sigma)$, hence one can show that the moduli space of holomorphic curves in a Weinstein sector is compact.
\end{rmk}

\begin{exm}
	The cotangent bundle of a manifold with boundary is a Weinstein sector. In particular, $T^*\R^n$ with the standard Weinstein structure from Example \ref{exm:cotangent-weinstein} can be thought as a Weinstein sector as follows: Equip the closed $2n$-disk
	\[D^{2n}=\left\{(x_i,p_i)_{i=1}^n\in\R^{2n}\,\left|\, \sum_{i=1}^n x_i^2+p_i^2\leq 1\right.\right\}\]
	with a Weinstein structure coming from Example \ref{exm:cotangent-weinstein-nonstandard}. Then $D^{2n}$ is a Weinstein domain. Define the Legendrian
	\[\Lambda:=\left\{(x_i,p_i)_{i=1}^n\in\R^{2n}\,\left|\, \sum_{i=1}^n x_i^2= 1,p_i=0\right.\right\}\subset\partial D^{2n}\]
	and let $\Sigma$ be a Weinstein hypersurface in $\partial D^{2n}$ such that $\fX_{\Sigma}=\Lambda$. Then by Proposition \ref{prp:weinstein-adjusted}, there is an adjusted Weinstein structure on $D^{2n}$ such that we can define
	\[W'^c:=\left\{(x_i,p_i)_{i=1}^n\in D^{2n}\,\left|\, \sum_{i=1}^n p_i^2\leq\varepsilon^2\right.\right\}\]
	for some $\varepsilon<1$ with $\fX_{W'^c}=\{(x_i,p_i)_{i=1}^n\in D^{2n}\vb p_i=0\}$. After attaching a cylindrical end to $W'^c$ along $\partial W'^c$ and possibly perturbing the Weinstein structure, we get $T^*D^n$ with the standard Weinstein structure, which makes it a Weinstein sector corresponding to the Weinstein pair $(D^{2n},\Lambda)$. $T^*\R^n$ can be considered as $T^*D^n$ after compactifying $\R^n$ with its boundary at infinity, namely $(n-1)$-sphere $S^{n-1}$.
\end{exm}

\subsection{Fukaya Category of Weinstein Manifolds}\label{sec:fukaya-weinstein}

We will introduce the compact, infinitesimally wrapped, wrapped, and partially wrapped Fukaya categories of immersed Lagrangians of Weinstein manifolds as defined in \cite{akaho-joyce}, \cite{nadler-zaslow}, \cite{abouzaid-wrapped}, and \cite{sylvan}, respectively, omitting many details which can be found in \cite{seidel}, \cite{fooo1}, and \cite{fooo2}. One can find the sketch of some arguments in \cite{HMS} and \cite{auroux}. If we have a Weinstein sector instead of a Weinstein manifold, then we work with the Weinstein pair $(W,\Lambda)$ it is adjusted to.

From this section onwards, any Weinstein manifold $(W,\theta,\phi)$ is assumed to be finite-type, hence it can equivalently be characterised by a Weinstein domain with cylindrical end. An immersed Lagrangian $L$ in $W$ is given by the smooth immersion $i\colon L\to W$ with $i^*d\theta=0$, which is assumed to be proper, and an embedding outside finitely many points. All the self-intersections of $L$ are transverse double self-intersections, or more generally, clean intersections, see \cite{alston-bao}. Lagrangians in a Fukaya category are equipped with a brane structure, i.e. a local system, a grading, and a relative pin structure.
	
\begin{dfn}\label{dfn:compact-Fukaya}
	The \textit{compact Fukaya category} $\cF(W)$ of the Weinstein manifold $W$ is the triangulated envelope of the $A_{\infty}$-category of compact exact immersed Lagrangians $L$ in $W$. The morphisms between embedded Lagrangians $L_0$ and $L_1$ are given by the $R$-graded $\k$-module
	\[\CF^*(L_0,L_1):=\k\langle\phi^1_H(L_0)\cap L_1\rangle\]
	for $R=\Z/N$ or $\Z$, with the $A_{\infty}$-operations counting pseudoholomorphic polygons with Lagrangian boundary conditions, where $\phi^1_H$ is the time-1 map of a suitable Hamiltonian $H$ on $W$ and $\phi^1_H(L_0)$ intersects with $L_1$ transversally. We denote its cohomology by $HF^*(L_0,L_1)$. See \cite{seidel} for details.
	
	The problem occurs when Lagrangians are immersed, because the differential does not necessarily square to zero, which is caused by holomorphic disks (``teardrops'') that immersed Lagrangians may bound. One way to resolve this issue is to consider only the Lagrangians which do not bound any holomorphic disks. Such Lagrangians are called \textit{unobstructed}. But this is a very restricted class. Instead, we consider Lagrangians with bounding cochains, and proceed as in \cite{akaho-joyce} to define the Floer cohomology.
\end{dfn}

\begin{rmk}\label{rmk:grading-structure}
	Given a Weinstein manifold $W$, $\cF(W)$ can be made $R$-graded if and only if $2c_1(W)=0\in H^2(W;R)$ for $R=\Z/N$ or $\Z$. In particular, $\cF(W)$ can always be made $\Z/2$-graded. Grading structure on Lagrangians depends on an initial choice on $W$, called a \textit{grading structure} on $W$, which is a trivialization $\eta_W^2$ of the bicanonical bundle $\kappa_W^2$, where $\kappa_W=(\bigwedge^{\dim W} T^{\text{hol}}W)^{-1}$ where $T^{\text{hol}}W$ is holomorphic tangent bundle of $W$. $\cF(W)$ depends only on the homotopy class of $\eta_W^2$, which gives $H^1(W;R)$ many choices if $\cF(W)$ is $R$-graded. In the case of $\Z/2$-graded $\cF(W)$, there is a canonical choice for $\eta_W^2$ such that the grading of an intersection point is even if the Lagrangians intersect positively, and odd otherwise. We will use this canonical choice when $\cF(W)$ is $\Z/2$-graded. See \cite{seidel} and \cite{graded} for more details.
\end{rmk}

If we allow noncompact Lagrangians, we need to deal with their intersection at infinity. We consider $W$ as a Weinstein domain $W^c$ with cylindrical end. We have two different ways to follow:

In the first approach, introduced by Nadler and Zaslow in \cite{nadler-zaslow}, we fix a singular closed Legendrian $\Lambda$ in $\partial W^c$. Then we consider exact immersed Lagrangians $L$ with cylindrical end in $W$ satisfying $L\cap\partial W^c\subset\Lambda$. Note that instead of asking for cylindrical ends, we could equivalently require our Lagrangians to be compactly exact by Proposition \ref{prp:cylindrical-end}.

As for morphisms, we choose Hamiltonians whose time-1 map separates the Lagrangians in the cylindrical end of $W$. For that, we make use of the following lemma which is a consequence of Curve Selection Lemma (see \cite[Lemma 3.1]{milnor-singular}):

\begin{lem}[{\cite{HMS}, \cite[Lemma 5.2.5]{nadler-zaslow}}]\label{lem:infinitesimally-wrapped}
	Let $\phi^t_R$ be the Reeb flow on a contact manifold $N$. For any two Legendrians $\Lambda_0$ and $\Lambda_1$ in $N$, there exists $\varepsilon>0$ such that $\phi^t_R(\Lambda_0)\cap \Lambda_1 =\emptyset$ for all $t\in(0,\varepsilon)$.
\end{lem}

We choose a Hamiltonian $H$ which at infinity is of the form $H = h(e^r)$, where $r$ is the radial coordinate in the cylindrical end, and $h$ is a function with $h'\in(0,\varepsilon)$. Then the Hamiltonian vector field $X_H$ at infinity is $h'(e^r)$ times the Reeb vector field $R$ on $\partial W^c$, so its time-1 map perturbs $L_0$ infinitesimally at infinity. Hence by Lemma \ref{lem:infinitesimally-wrapped} it separates $L_0$ and $L_1$ in the cylindrical end, so that they only intersect in the compact domain $W^c$. The rest is as in Definition \ref{dfn:compact-Fukaya}.

\begin{dfn}
	We call the pretriangulated $A_{\infty}$-category with Lagrangians and morphisms as above \textit{infinitesimally wrapped Fukaya category} $\cF(W,\Lambda)$. See \cite{nadler-zaslow}, \cite{microbranes}, and \cite{HMS} for more details. Note that if $\Lambda=\emptyset$, then $\cF(W,\Lambda)=\cF(W)$.
\end{dfn}

\begin{rmk}
	In \cite{nadler-zaslow}, Nadler and Zaslow defined the $A_{\infty}$-category $\cF(T^*M,\Lambda)$ for a smooth manifold $M$ (not necessarily compact) for embedded Lagrangians. When $M$ is compact, $T^*M$ is a Weinstein manifold and this matches with our definition, because the triangulated envelope gives the same $A_{\infty}$-category. They proved the following important theorem relating constructible sheaves to the Fukaya category, which is be one of our main guides:
\end{rmk}

\begin{thm}[\cite{nadler-zaslow}, \cite{microbranes}]\label{thm:nadler-zaslow}
	There is an $A_{\infty}$-quasi-equivalence (microlocalisation)
	\[\mu_M\colon\Sh_{\Lambda}(M)\xrightarrow{\sim}\cF(T^*M,\Lambda)\]
	where $\mu_M$ sends the standard sheaf $i_*\cL_S$ associated to a local system $\cL_S$ on a submanifold $i\colon S\hookrightarrow M$ to the perturbation of its singular support $\sing(i_*\cL_S)\subset L_{\Lambda}$, which is the standard Lagrangian
	\[L_{S,m,\cL_S} := N^*S + \Gamma_{-d \log (m|_S)}\]
	for some $m\colon S\to\R_{\geq 0}$ whose zero set is precisely $\partial S$ and it is equipped with the canonical brane structure where the local system is coming from $\cL_S$. The standard sheaves, whose singular support inside $L_{\Lambda}$, generate $\Sh_{\Lambda}(M)$, hence $\mu_M$ is automatically defined for any arbitrarily $\cF\in\Sh_{\Lambda}(M)$.
\end{thm}

In the second approach, we consider exact immersed Lagrangians $L$ with cylindrical end in $W$. To define the morphisms between give two Lagrangians $L_0$ and $L_1$, we perturb $L_0$ in the direction of Reeb vector field, but this time instead of perturbing $L_0$ infinitesimally away from $L_1$, we wrap $L_0$ around $L_1$ with increasing speed as $r$ increases, where $r$ is the radial coordinate of the cylindrical end, by choosing a Hamiltonian $H$ such that $H= h(e^r)$ at infinity with $\lim_{r\to\infty} h'(e^r)=\infty$. Then we define the morphisms as in Definition \ref{dfn:compact-Fukaya}. Equivalently, they can be characterised by the perturbed intersection points of $L_0$ and $L_1$ in $W^c$, and the Reeb chords of arbitrary length from $L_0\cap \partial W^c$ to $L_1\cap \partial W^c$. In that case, the $A_{\infty}$-operations count punctured pseudoholomorphic polygons in $W^c$ whose boundaries lie inside Lagrangians and which converge asymptotically to the Reeb chords at the punctures. See \cite{ekholm-lekili} for details.

\begin{dfn}
	We call the pretriangulated $A_{\infty}$-category with Lagrangians and morphisms as above \textit{wrapped Fukaya category} $\cW(W)$. The morphism complex between Lagrangians $L_0$ and $L_1$ is denoted by $\CW^*(L_0,L_1)$ and its cohomology by $\HW^*(L_0,L_1)$. See \cite{abouzaid-wrapped}, \cite{HMS}, and \cite{CDRGG} for more details, and also for the equivalent definition by ``direct limit''.
\end{dfn}

One has the following generation result for wrapped Fukaya categories, which will be one of our main tool for calculations:

\begin{thm}[\cite{CDRGG}]\label{thm:cdrgg}
	If $W$ is a (finite-type) Weinstein manifold, then its wrapped Fukaya category $\cW(W)$ is generated by the Lagrangian cocores.
\end{thm}

\begin{rmk}
	If $W=T^*M$ with the standard Weinstein structure and compact $M$, the result can be translated to ``a cotangent fibre of $M$ generates $\cW(T^*M)$'', as proved by Abouzaid in \cite{abouzaid-wrapped-generation}, see Example \ref{exm:cotangent-cocore}. Also he proved that:
\end{rmk}

\begin{thm}[\cite{abouzaid-loops}]\label{thm:abouzaid-wrapped-loop}
	The endomorphism algebra of a cotangent fibre of $M$ at a point $p\in M$ is given by
	\[\CW^*(T_p^*M,T_p^*M)\simeq C_{-*}(\Omega_p M)\]
	hence if $M$ is connected, there is an $A_{\infty}$-quasi-equivalence
	\[\cW(T^*M)\simeq\Perf(C_{-*}(\Omega M))\ .\]
\end{thm}

For a general Weinstein manifold $W$ presented by a Legendrian surgery diagram, we have a very important surgery formula when $\k$ is a field of characteristic zero, which describes the endomorphism algebra of Lagrangian cocores:

\begin{thm}[{\cite[Theorem 5.8]{bee}, \cite[Theorem 2]{ekholm-lekili}}]\label{thm:bee}
	Assume $W$ has $k$ many critical handles whose attaching spheres are the Legendrians $\Lambda_i$ on the boundary of the subcritical part of $W$, and whose Lagrangian cocores are $L_i$ for $i=1,\ldots,k$. Define the Legendrian $\Lambda:=\bigoplus_{i=1}^k\Lambda_i$ and the Lagrangian $L:=\bigoplus_{i=1}^k L_i$. Then we have
	\[\CW^*(L,L)\simeq\CE^*(\Lambda)\]
	where $\CE^*(\Lambda)$ is the \textit{Chekanov-Eliashberg dga} of $\Lambda$ (called the Legendrian homology algebra $\text{LHA}(\Lambda)$ in \cite{bee}).
\end{thm}

Hence, Theorem \ref{thm:cdrgg} and \ref{thm:bee} together imply that for a Weinstein manifold $W$ with a Legendrian surgery diagram where the attaching spheres of the critical handles are the Legendrians $\{\Lambda_i\}_{i=1}^k$ on the boundary of the subcritical part of $W$, we have
\[\cW(W)\simeq\Perf(\CE^*(\Lambda))\]
where $\Lambda:=\bigoplus_{i=1}^k\Lambda_i$. In our calculations, we have a combinatorial description of $\CE^*(\Lambda)$, see Section \ref{sec:wrapped-rational}.

For a given singular closed Legendrian $\Lambda$ in $\partial W^c$, Sylvan in \cite{sylvan} defined:

\begin{dfn}
	The \textit{partially wrapped Fukaya category} $\cW(W,\Lambda)$ is $A_{\infty}$-subcategory of $\cW(W)$ consisting of objects which avoid $\Lambda$ and the only Reeb chords it contains as morphisms are the ones which do not intersect with $\Lambda$. Hence we call $\Lambda$ \textit{stop}. Note that if $\Lambda=\emptyset$, then $\cW(W,\Lambda)=\cW(W)$.
\end{dfn}

For $W=T^*M$, Ganatra, Pardon, and Shende proved the ``wrapped'' version of Theorem \ref{thm:nadler-zaslow}, which is the main motivation for our Corollary \ref{cor:bp1-lp1}:

\begin{thm}[\cite{gps3}]\label{thm:gps}
	There is an $A_{\infty}$-quasi-equivalence
	\[\Sh^w_{\Lambda}(M)\simeq\cW(T^*M,\Lambda)\ .\]
\end{thm}

Note that this implies Theorem \ref{thm:abouzaid-wrapped-loop} partially by setting $\Lambda=\emptyset$, which gives
\[\cW(T^*M)\simeq\Loc^w(M)\simeq \Perf(C_{-*}(\Omega M))\]
where the last equivalence is shown in Example \ref{exm:wrapped-local-loop}. Also, it is possible to get the rest of the theorem as explained by \cite[Corollary 6.1]{gps3}.

\begin{rmk}
	As remarked in \cite{auslander}, given $\Lambda\in\partial W^c$ we have functors
	\[\cF(W) \to \cF(W,\Lambda) \to \cW(W,\Lambda) \to \cW(W)\]
	where the first two functors are full and faithful embeddings, and the last functor is the composition
	\[\cW(W,\Lambda) \to\cW(W,\Lambda)/\cD \to \cW(W)\]
	where $\cD$ is the full subcategory of $\cW(W,\Lambda)$ generated by the objects supported near $\Lambda$, the first functor is the localisation functor, and the second one is induced by the inclusion $\cW(W,\Lambda) \to \cW(W)$, which is an $A_{\infty}$-quasi-equivalence by \cite[Theorem 1.16]{gps2}.
\end{rmk}

\subsection{Kashiwara-Schapira Stack and Microlocal Sheaves}\label{sec:microlocal}

Let $M$ be a smooth manifold, and $\Lambda$ be a closed singular Legendrian in $T^{\infty}M$. In the previous section, Nadler-Zaslow theorem (Theorem \ref{thm:nadler-zaslow}) gave the equivalence
\[\cF(T^*M,\Lambda)\simeq\Sh_{\Lambda}(M)\]
and Ganatra-Pardon-Shende theorem (Theorem \ref{thm:gps}) gave the equivalence
\[\cW(T^*M,\Lambda)\simeq \Sh^w_{\Lambda}(M)\ .\]
By Proposition \ref{prp:combinatorics-constructible}, Corollary \ref{cor:sing-constructible}, and Proposition \ref{prp:mod-perf}, we saw that $\Sh_{\Lambda}(M)$ and $\Sh^w_{\Lambda}(M)$ can be calculated combinatorially. These are great results for the computability of $\cF(T^*M,\Lambda)$ and $\cW(T^*M,\Lambda)$.

A natural question to ask is whether these relations generalise to Weinstein manifolds. For that, by noting that $M$ can be seen as the skeleton of $T^*M$, we can write
\begin{align*}
	\cF(T^*M,\Lambda)&\simeq\Sh_\Lambda(\fX_{T^*M})\\
	\cW(T^*M,\Lambda)&\simeq\Sh^w_\Lambda(\fX_{T^*M})
\end{align*}
which suggest us to replace $T^*M$ with a general Weinstein manifold $W$ equipped with a closed singular Legendrian $\Lambda\subset \partial W^c$ to generalise the result.

One can quickly notice that the category $\Sh_{\Lambda}(\fX_W)$ does not make sense in general since $\fX_W$ does not need to be a smooth manifold, instead it is stratified by isotropic submanifolds of $W$ by Proposition \ref{prp:isotropic-skeleton}. However, we can proceed as follows: The Weinstein pair $(W,\Lambda)$ corresponds to a Weinstein sector $W'$ whose skeleton is $\fX_{W'}=L_{\Lambda}$. We can perturb the Weinstein structure of $W'$ in such a way that for each open set $V\subset L_{\Lambda}$, we have a conic open set $U\subset W'$ such that $U\cap L_{\Lambda}=V$. Considered with its boundary, $U$ is a Weinstein sector with $\fX_{U}=V$. In particular, if $V$ is a smooth Lagrangian, then $U$ can be seen as the cotangent bundle $T^*V$, for which we have the equivalences
\begin{align*}
	\cF(U)&\simeq\Loc(V)\\
	\cW(U)&\simeq\Loc^w(V)\ .
\end{align*}
If $V$ is singular, we do not have such simple description, but we have the following conjecture for infinitesimal Fukaya categories:

\begin{con}\label{con:microlocal}
	Let $W$ be a Weinstein manifold and $\Lambda\subset\partial W^c$ be a closed singular Legendrian. There is a sheaf of (pretriangulated) dg categories $\mSh_{W,\Lambda}$ on the conic isotropic subset $L_{\Lambda}$ satisfying
	\[\mSh_{W,\Lambda}(V)\simeq\cF(U',\Lambda')\]
	where $V$ is an open set of $L_{\Lambda}$, $U$ is a conic open subset of $W'$ with $U\cap L_{\Lambda}=V$, $W'$ is the Weinstein sector corresponding to $(W,\Lambda)$, and $(U',\Lambda')$ is the Weinstein pair corresponding to $U$. In particular, its global sections are
	\[\mSh_{W,\Lambda}(L_{\Lambda})\simeq \cF(W,\Lambda)\ .\]
\end{con}

\begin{rmk}
	If $V$ is a smooth Lagrangian, by the above discussion we have
	\[\mSh_{W,\Lambda}(V)\simeq \Loc(V)\ .\]
	If $V$ is not Lagrangian, we have
	\[\mSh_{W,\Lambda}(V)\simeq 0\]
	since $\fX_U=V$ which implies $U$ is subcritical and $\cF(U',\Lambda')\simeq 0$.
\end{rmk}

Note that we consider $L_{\Lambda}$ not just as a topological space, instead a singular Lagrangian in its Weinstein surrounding. Hence, we do not consider topologically same skeletons $L_{\Lambda}$ and $L_{\Lambda'}$ as the same, if their Weinstein surroundings cannot be deformed to each other. See Remark \ref{rmk:pants-torus} for the comparison of a pair of pants with a punctured torus, whose skeleta are the same as topological spaces.

\begin{dfn}
	The sheaf of dg categories $\mSh_{W,\Lambda}$ above is called \textit{(traditional) microlocal stack}, and its global section $\mSh_{W,\Lambda}(L_{\Lambda})$ is called \textit{(traditional) microlocal sheaves}. We shortly write $\mSh$ for $\mSh_{W,\Lambda}$ when the Weinstein pair $(W,\Lambda)$ is clear from the context. If $\Lambda$ is not specified, i.e. $\Lambda=\emptyset$, we write $\mSh_W$ instead of $\mSh_{W,\Lambda}$, and the global section is $\mSh_W(\fX_W)$ in that case.
\end{dfn}

\begin{rmk}
	The sheaf $\mSh$ is with values in the category of dg categories $\dgCat$ whose objects we consider up to quasi-equivalence. This means that the correct notion of limit is homotopy limit, which we need to use when gluing: For every open $V\subset L_{\Lambda}$ and its open covering $\{V_i\}$ we have
	\[\mSh(V)\simeq\holim\left(\prod_i\mSh(V_i)\rightrightarrows\prod_{j,k}\mSh(V_j\cap V_k)\right)\ .\]
	The model structure and homotopy limit on $\dgCat$ will be studied in Section \ref{sec:dgcat}. Also, for Conjecture \ref{con:microlocal} to hold, dg categories are expected to be $R$-graded if $2c_1(W)=0\in H^2(W;R)$ for $R=\Z/N$ of $\Z$, see Remark \ref{rmk:grading-structure}.
\end{rmk}

So, if we can define such a sheaf, then we can glue the dg categories of sheaves on the local pieces of $L_{\Lambda}$ to get the the infinitesimal Fukaya category $\cF(W,\Lambda)$. For $W=T^*M$, it can be defined using the following sheaf defined by Kashiwara and Schapira:

\begin{dfn}[\cite{kashiwara-schapira}, \cite{combinatorics}]\label{dfn:ks-stack}
	The \textit{Kashiwara-Schapira stack (KS stack)} is a sheaf of pretriangulated dg categories on $T^*M$ with conic topology (i.e. topology with conic open sets) defined as follows: For each conic open set $U\subset T^*M$, we define $KS^{\text{pre}}(U)$ as the dg quotient $\Sh(M)/\Sh_{T^*M\setminus U}(M)$. This gives a presheaf of dg categories on $T^*M$. For a given conic set $L\subset T^*M$, $\KS^{\text{pre}}_L$ is a presheaf such that $\KS^{\text{pre}}_L(U)$ is full dg subcategory of $\KS^{\text{pre}}(U)$ consisting of sheaves whose singular support lies in $L$. The sheafification of the presheaf $\KS^{\text{pre}}_L$ is $\KS_L$, called \textit{KS stack} on $T^*M$ supported on $L$. The restriction map $\KS_L(U)\to\KS_L(V)$ sends a sheaf to another with the same microstalks.
\end{dfn}

The following proposition describes $\KS_L$ explicitly:

\begin{prp}[\cite{kashiwara-schapira}, \cite{combinatorics}]\label{prp:ks-explicit}
	Let $U\subset T^*M$ be a conic open set, and $\pi\colon T^*M\to M$ be the projection. We have
	\begin{enumerate}
		\item If $U=T^*\pi(U)$, then $\KS_L(U)\simeq\Sh_{L\cap U}(\pi(U))$,
		
		\item If $U\cap M=\emptyset$ and $U$ is sufficiently small, then
		\[\KS_L(U)\simeq\Sh_{\pi(U)\cup(L\cap U)}(\pi(U))/\Loc(\pi(U))\simeq\Sh_{\pi(U)\cup(L\cap U)}(\pi(U))_0\ .\]		
	\end{enumerate}
\end{prp}

By definition, $\KS_L(U)$ vanishes when $U\cap L$ is empty. Therefore $\KS_L$ is supported on $L$. Hence it makes sense to regard $\KS_L$ as a sheaf on $L$ and write
\[\KS_L(V):=\KS_L(U)\]
where $V$ is an open subset of $L_{\Lambda}$ and $\R_{>0}V=L_{\Lambda}\cap U$ for some conic open $U\subset T^*M$.

\begin{dfn}
	For a given closed singular Legendrian $\Lambda\subset T^{\infty}M$, we define $\mSh_{T^*M,\Lambda}$ on $L_{\Lambda}\subset T^*M$ as $\KS_{L_{\Lambda}}$.
\end{dfn}

\begin{rmk}
	Note that $\mSh_{T^*M,\Lambda}$ defined above for cotangent bundles satisfies the conditions in Conjecture \ref{con:microlocal}. In particular,
	\[\mSh_{T^*M,\Lambda}(L_{\Lambda})=\KS_{L_{\Lambda}}(T^*M)\simeq\Sh_{L_{\Lambda}}(M)\simeq\cF(T^*M,\Lambda)\ .\]
\end{rmk}

Next, we will try to give a construction of the sheaf $\mSh_{W,\Lambda}$ for a general Weinstein pair $(W,\Lambda)$ of dimension $2n$, following mostly \cite{combinatorics} and \cite{wrapped}. This will be an informal discussion. Equivalently, we will define $\mSh_{W,\Lambda}|_V$ for any $V$ in an open cover $\{V_i\}$ of $L_{\Lambda}$:
\begin{enumerate}
	\item If $V$ is not Lagrangian, then $\mSh_{W,\Lambda}|_V$ is the zero sheaf.\
	
	\item If $V$ is smooth Lagrangian, then $\mSh_{W,\Lambda}|_V:=\mSh_{T^*V}$.
	
	\item If $V$ is singular Lagrangian and sufficiently small, then by Darboux theorem we take an appropriate open subset $U\subset W$ such that $V=U\cap L_{\Lambda}$ with the symplectomorphism $\varphi\colon U\to T^*\R^n$. If $\varphi(V)$ is conic in $T^*\R^n$ and contains $\R^n$, we define $\mSh_{W,\Lambda}|_V(Y):=\mSh_{T^*\R^n,\Lambda'}(\varphi(Y))$ for any open $Y\subset V$, where we define $\Lambda'\subset T^{\infty}\R^n$ by $L_{\Lambda'}=\varphi(V)$.
\end{enumerate}

Note that $\varphi$ above is in fact an exact symplectomorphism since $H^1(T^*\R^n;\R)\simeq 0$. So it makes sense to deal with the pair $(T^*\R^n,\varphi(V))$ instead of $(U,V)$ since we mostly care about exact symplectic structure.

There are some problems defining $\mSh_{W,\Lambda}$ as above, coming form the third point:

\begin{prb}
	$\varphi(V)$ may not be conic in $T^*\R^n$.
\end{prb}

In that case we may go to one dimension higher and see it as a conic Lagrangian there as follows: $V$ is conic, hence exact by Proposition \ref{prp:conic-exact}. $\varphi$ is an exact symplectomorphism, therefore by Proposition \ref{prp:exact-to-exact} $\varphi(V)$ is an exact singular Lagrangian. So there exists continuous $f\colon \varphi(V)\to\R$ such that $\theta=df$ on the strata of $\varphi(V)$. Then we can lift $\varphi(V)$ to the singular Legendrian inside the contactisation of $T^*\R^n$ using $-f$, explicitly
\[\Gamma_{\varphi(V),-f}=\{(x,p;t)\vb (x,p)\in\varphi(V),t=-f(x,p)\}\subset T^*\R^n\times \R_t\]
where $T^*\R^n\times \R_t$ has the contact form $\alpha=\theta+dt$ and $\theta$ is the standard Liouville form of $T^*\R^n$. Moreover, $T^*\R^n\times\R_t$ is contactomorphic to $T^{\infty,-}(\R^n\times\R_t)$ via the map \[(x,p;t)\mapsto (x,t;[p,-1])\]
and $\Gamma_{\varphi(V),-f}$ is mapped to
\[\Lambda_V:=\{(x,t;[p,-1])\vb (x,p)\in\varphi(V),t=-f(x,p)\}\subset T^{\infty,-}(\R^n\times\R_t)\]
where the contact form on $T^{\infty,-}(\R^n\times\R_t)$ is $\alpha$. Since $L_{\Lambda_V}\subset T^*\R^{n+1}$ is conic, now we can define
\[\mSh_{W,\Lambda}|_V(Y):=\mSh_{T^*\R^{n+1},\Lambda_V}(\R_{>0}\Lambda_Y)\]
for any open $Y\subset V$. Still we need to answer the question: If $\varphi(V)$ is already conic containing $\R^n$, and $\Lambda'$ is defined by $L_{\Lambda'}=\varphi(V)$, are $\mSh_{T^*\R^n,\Lambda'}(\varphi(Y))$ and $\mSh_{T^*\R^{n+1},\Lambda_V}(\R_{>0}\Lambda_Y)$ quasi-isomorphic?

Note that there is a noncharacteristic deformation of $\Lambda_V$ to a Legendrian $\Lambda_{\text{arb}}$ with \textit{arboreal singularities}, defined by Nadler in \cite{arboreal}, such that:

\begin{thm}[\cite{nonchar}]\label{thm:arboreal}
	The sections $\mSh_{T^*\R^{n+1},\Lambda_V}(\R_{>0}\Lambda_V)$ and $\mSh_{T^*\R^{n+1},\Lambda_{\text{arb}}}(\R_{>0}\Lambda_{\text{arb}})$ are quasi-isomorphic.
\end{thm}
The nice thing about arboreal singularities is that $\mSh_{T^*\R^{n+1},\Lambda_{\text{arb}}}$ around an arboreal singularity is explicitly given by $A_{\infty}$-modules over an associated tree, shown in \cite{arboreal}. Also, one can do such deformations before lifting, so that we do not need to lift the skeleton at all.

For $i\neq j$, $\mSh_{W,\Lambda}|_{V_i}$ and $\mSh_{W,\Lambda}|_{V_j}$ are defined in different domains, hence we have the following problem:

\begin{prb}
	How to identify $\mSh_{W,\Lambda}|_{V_i}$ and $\mSh_{W,\Lambda}|_{V_j}$ when they are restricted to $V_i\cap V_j$?
\end{prb}

This is a crucial problem, since the result of gluing depends on these identifications. Our claim is that for $R=\Z/N$ or $\Z$, after taking the open cover $\{V_i\}$ sufficiently fine, there are $H^1(W;R)\simeq H^1(L_{\Lambda};R)$ many different such identifications for $R$-graded $\mSh_{W,\Lambda}$, and they correspond to different grading structures on the Weinstein manifold $W$. In our calculations in later chapters, we will use a particular identification for $\Z/2$-graded $\mSh_{W,\Lambda}$ which corresponds to the standard $\Z/2$-grading structure of the Weinstein manifold.

\begin{prb}
	How does $\mSh_{W,\Lambda}$ depend on the pair $(W,\Lambda)$?
\end{prb}

Let $(W,\Lambda)$ and $(W',\Lambda')$ be two Weinstein pairs such that there is an exact symplectomorphism $\varphi\colon W\to W'$ sending $L_{\Lambda}$ to $L_{\Lambda'}$. Then by the construction of $\mSh$, we have
\[\mSh_{W,\Lambda}\simeq\mSh_{W',\Lambda'}\ .\]
However, this is not enough for us in most of the cases. An example, where we consider Weinstein sectors corresponding to Weinstein pairs, is as follows: Let $W=\R^2=\{(x,y)\}$ be a Weinstein sector with the skeleton
\[T_1=\{y=0\}\cup\{x=0,y\geq 0\}\]
and $W'=\R^2$ be a Weinstein sector with the skeleton
\[T_2=\{x=y,y\leq 0\}\cup\{x=-y,y\leq 0\}\cup\{x=0,y\geq 0\}\ .\]
It is not hard to check that $\mSh_W(T_1)\simeq\mSh_{W'}(T_2)$, see the proof of Proposition \ref{prp:msh-vertex}. However, there is no exact symplectomorphism $\varphi\colon\R^2\to\R^2$ sending $T_1$ to $T_2$, since $\{y=0\}$ in $T_1$, as a smooth submanifold, is required to be mapped to a smooth submanifold of $\R^2$, which cannot be contained in $T_2$. However, there is a deformation of the Weinstein structure of $W$ to $W'$ through Weinstein structures on $\R^2$.

In general, let $(W,\Lambda)$ and $(W',\Lambda')$ be Weinstein pairs such that $\tilde W$ and $\tilde W'$ are the corresponding Weinstein sectors, respectively. We expect that if there is a deformation of Weinstein structures of $\tilde W$ and $\tilde W'$ to each other through Weinstein structures, which we call \textit{Weinstein homotopy}, then
\[\mSh_{W,\Lambda}(L_{\Lambda})\simeq\mSh_{W',\Lambda'}(L_{\Lambda'})\ .\]
See \cite{weinstein} and \cite{revisited} for more information about Weinstein homotopies.

Another problem which we will not address is:

\begin{prb}
	Does $\mSh$ depend on the choices of $\{V_i\}$, $U$, $\varphi$, $f$?
\end{prb}

In the end, we define $\mSh_{W,\Lambda}$ for Weinstein manifolds as:

\begin{dfn}\label{dfn:microlocal}
	For any small enough open $V\subset L_{\Lambda}$, we define
	\[\mSh_{W,\Lambda}|_V(Y):=\mSh_{T^*\R^{n+1},\Lambda_V}(\R_{>0}\Lambda_Y)\]
	for any open $Y\subset V$.
	Observe that we have
	\[\mSh_{W,\Lambda}|_V(V)=\KS_{L_{\Lambda_V}}(\R_{>0}\Lambda_V)\simeq\Sh_{\Lambda_V}(\R^{n+1})_0\]
	by Proposition \ref{prp:ks-explicit}. Hence the global section of $\mSh_{W,\Lambda}$ is given by
	\[\mSh_{W,\Lambda}(L_{\Lambda})\simeq\holim\left(\prod_i\Sh_{\Lambda_{V_i}}(\R^{n+1})_0\rightrightarrows\prod_{j,k}\Sh_{\Lambda_{V_j\cap V_k}}(\R^{n+1})_0\right)\ .\]
\end{dfn}

\begin{rmk}
	Note that with this definition, if $V$ is not Lagrangian, then $\Lambda_V$ is not Legendrian, hence by Theorem \ref{thm:kashiwara-schapira}
	\[\mSh_{W,\Lambda}(V)\simeq \Sh_{\Lambda_V}(\R^{n+1})_0\simeq 0\ .\]
	If $V$ is a smooth Lagrangian, then it is not hard to see that
	\[\mSh_{W,\Lambda}(V)\simeq\Loc(V)\ .\]
	However, we do not know whether $\mSh_{W,\Lambda}$ satisfies Conjecture \ref{con:microlocal} yet.
\end{rmk}

There is a similar conjecture for (partially) wrapped Fukaya categories:

\begin{con}\label{con:microlocal-wrapped}
	Let $W$ be a Weinstein manifold and $\Lambda\subset\partial W^c$ be a closed singular Legendrian. There is a cosheaf of (pretriangulated) dg categories $\mSh^w_{W,\Lambda}$ on the conic isotropic subset $L_{\Lambda}$ satisfying
	\[\mSh^w_{W,\Lambda}(V)\simeq\cW(U',\Lambda')\]
	where $V$ is an open set of $L_{\Lambda}$, $U$ is a conic open subset of $W'$ with $U\cap L_{\Lambda}=V$, $W'$ is the Weinstein sector corresponding to $(W,\Lambda)$, and $(U',\Lambda')$ is the Weinstein pair corresponding to $U$. In particular, its global sections are
	\[\mSh^w_{W,\Lambda}(L_{\Lambda})\simeq \cW(W,\Lambda)\ .\]
\end{con}

Our result (Corollary \ref{cor:bp1-lp1}) confirms this conjecture for $(W,\Lambda)=(B_{p,1},\emptyset)$.

\begin{rmk}
	If $V$ is a smooth Lagrangian, by the above discussion we have
	\[\mSh^w_{W,\Lambda}(V)\simeq \Loc^w(V)\ .\]
	If $V$ is not Lagrangian, we have
	\[\mSh^w_{W,\Lambda}(V)\simeq 0\]
	since $\fX_U=V$ which implies $U$ is subcritical and $\cW(U',\Lambda')\simeq 0$.
\end{rmk}

\begin{dfn}
	The cosheaf of dg categories $\mSh^w_{W,\Lambda}$ above is called \textit{wrapped microlocal stack}, and its global section $\mSh^w_{W,\Lambda}(L_{\Lambda})$ is called \textit{wrapped microlocal sheaves}. We shortly write $\mSh^w$ for $\mSh^w_{W,\Lambda}$ when $(W,\Lambda)$ is clear from the context. If $\Lambda$ is not specified, i.e. $\Lambda=\emptyset$, we write $\mSh^w_W$ instead of $\mSh^w_{W,\Lambda}$, and the global section is $\mSh^w_W(\fX_W)$ in that case.
\end{dfn}

We will try to construct this cosheaf $\mSh^w_{W,\Lambda}$ following \cite{wrapped} and the previous steps towards defining $\mSh_{W,\Lambda}$:

\begin{dfn}
	For a given conic set $L\subset T^*M$, we define \textit{large KS stack} $\KS_L\dd$ as in Definition \ref{dfn:ks-stack} by replacing every $\Sh$ by $\Sh\dd$. It is a sheaf of dg categories like $\KS_L$, and Proposition \ref{prp:ks-explicit} holds for $\KS\dd_L$ after replacing $\Sh$ by $\Sh\dd$ and $\Loc$ by $\Loc\dd$.
	
	For any open set $V\subset L$, we define $\KS^w_L(V)$ as the full dg subcategory of compact objects in $\KS\dd_L(V)$. It is shown in \cite{wrapped} that $\KS^w_L$ is a cosheaf of dg categories, and we call it the \textit{wrapped KS stack}. Proposition \ref{prp:ks-explicit} holds for $\KS^w_L$ after replacing $\Sh$ by $\Sh^w$ and $\Loc$ by $\Loc^w$ by definition.
	
	Finally for Weinstein manifolds, we define the \textit{large microlocal stack} $\mSh\dd_{W,\Lambda}$ on $L_{\Lambda}\subset W$ by the same construction of $\mSh_{W,\Lambda}$ after replacing $\KS$ by $\KS\dd$. Then we define the $\mSh^w_{W,\Lambda}$ on $L_{\Lambda}$ by setting $\mSh^w_{W,\Lambda}(V)$ as the full dg subcategory of compact objects in $\mSh\dd_{W,\Lambda}(V)$ for any open set $V\subset L_{\Lambda}$. By definition, the construction for $\mSh_{W,\Lambda}$ holds similarly for $\mSh^w_{W,\Lambda}$ after replacing $\KS$ by $\KS^w$.
\end{dfn}

We restate Proposition \ref{prp:mod-perf} with these new definitions:

\begin{prp}\label{prp:mod-perf-microlocal}
	Assume $\mSh_{W,\Lambda}\dd(V)\simeq\Modk(\cC)$ for some open subset $V\subset L_{\Lambda}$ and an $A_{\infty}$-category $\cC$. Then we have
	\begin{align*}
	\mSh_{W,\Lambda}(V)&\simeq\Perfk(\cC)\\
	\mSh^w_{W,\Lambda}(V)&\simeq\Perf(\cC)\ .
	\end{align*}
\end{prp}

So presenting $\mSh\dd_{W,\Lambda}(V)$ as $\Modk(\cC)$ is enough to calculate $\mSh_{W,\Lambda}(V)$ and $\mSh^w_{W,\Lambda}(V)$. We will make use of this observation in our calculations.

\begin{rmk}
	There is an another way to follow to define the traditional/wrapped microlocal stack, outlined in \cite{h-principle}. Roughly, we embed a thickening of the whole Weinstein manifold $W$ into $T^{\infty}\R^k$ for some large $k$ and define the microlocal stack by the KS stack, as we did before.
\end{rmk}

\section{Algebraic Tools}\label{chp:algebraic-tools}

In the previous section, we defined the microlocal stack $\mSh_{W,\Lambda}$ which is a sheaf with values in the category of dg categories $\dgCat$. In Section \ref{sec:dgcat}, we will study a model structure on $\dgCat$ and describe the homotopy limit. Using this description, we will prove two important lemmas in Section \ref{sec:gluing}, which we will make use of when gluing microlocal sheaves on pinwheels in Section \ref{sec:msh-pinwheel}. In Section \ref{sec:quiver}, we will study the $A_{\infty}$-modules over $A_n$-quiver, and realise them as the microlocal sheaves on the $(n+1)$-valent vertex. Together with the dg functors defined on them, they will be used to calculate the microlocal sheaves on pinwheels locally, their restriction maps, and associated monodromy map in Section \ref{sec:msh-pinwheel}.
    
\subsection{Model Structure and Homotopy Limit on DG categories}\label{sec:dgcat}

We start with the definitions which we will refer throughout this section:

\begin{dfn}
	Let $\cC$ be a dg category. Two morphisms $f,g\colon A\rightarrow B$ in $\cC$ are called \textit{homotopic} if there exists $\xi$ such that $d\xi=f-g$. We denote this by $f\sim g$ or $f\overset{\xi}{\sim}g$. A closed degree $0$ morphism $f\colon A\rightarrow B$ in $\cC$ is called a \textit{homotopy equivalence} if it is an isomorphism in the homotopy category $H^0(\cC)$, i.e. there exists a closed degree $0$ morphism $g$ such that $g\circ f\sim\id$ and $f\circ g\sim\id$. In this case, we call $A$ and $B$ \textit{homotopy equivalent}, and $g$ an \textit{inverse of $f$ in homotopy}.
\end{dfn}

We have the following basic properties for homotopic maps:

\begin{prp}\label{prp:homotopic-prop}
	If $x\sim y$ and $F$ is a dg functor, then
	\begin{enumerate}
		\item $(a\circ x\circ b+c)\sim(a\circ y\circ b+c)$ if $da=db=0$,
		
		\item $F(x)\sim F(y)$,
		
		\item If $f$ is a homotopy equivalence with the inverse $g$ in homotopy, then $F(f)$ is a homotopy equivalence with the inverse $F(g)$ in homotopy.
	\end{enumerate}
\end{prp}

When showing a morphism is a homotopy equivalence, in some circumstances we do not need to show that the inverse is closed:

\begin{prp}\label{prp:closed-inverse}
	Let $f\colon A\rightarrow B$ be a closed degree $0$ morphism, and let $g\colon B\rightarrow A$ be such that $g\circ f=\id$ and $f\circ g\sim\id$. Then $g$ is also a closed degree $0$ morphism, and consequently $f$ is a homotopy equivalence with the inverse $g$.
\end{prp}

\begin{proof}
	Obviously, $g$ is degree $0$. Since $g\circ f=\id$, we have
	\[0=d(g\circ f)=dg\circ f + g\circ df=dg\circ f\]
	since $df=0$. Also, $f\circ g\sim\id$ gives $d(f\circ g)=0$. Observe that we have $g=g\circ f\circ g$ and
	\[dg=d(g\circ (f\circ g))=(dg\circ f)\circ g + g\circ d(f\circ g)=0\ .\]
\end{proof}

Homotopy equivalence is similarly defined in $A_{\infty}$ setting by replacing $d$ with $\mu^1$. Then we have the following interpretation of \cite[Lemma 1.6]{seidel}:

\begin{lem}\label{lem:seidel-homotopy-equivalence}
	Let $F,G\in\Fun(\cC,\cD)$ be $A_{\infty}$-functors between $A_{\infty}$-categories $\cC$ and $\cD$. A natural transformation $T\in\Hom(F,G)$ is a homotopy equivalence if and only if $\mu^1(T)=0$ and $T^0(A)\in\Hom(F(A),G(A))$ is a homotopy equivalence for every $A\in\cC$.
\end{lem}

This lemma, and its following implication will be useful in Section \ref{sec:quiver}: 

\begin{lem}\label{lem:adjusting}
	Let $F$ be a dg functor between dg categories $\cC$ and $\cD$. Assume for each $A_i\in\cC$, there is a homotopy equivalence $m_i\colon F(A_i)\to B_i$ for some $B_i\in\cD$. Then we can construct a dg functor $G$ between $\cC$ and $\cD$ such that $G(A_i)=B_i$ for $A_i\in\cC$, and for a morphism $a\colon A_i\to A_j$, $G(a)=m_j\circ F(a)\circ m_i'$ where $m_i'$ is the inverse of $m_i$ in homotopy. Moreover, we have the natural equivalence $F\simeq G$, i.e. $F$ and $G$ are homotopy equivalent in $\Fun(\cC,\cD)$ via the $A_{\infty}$-natural transformation $T\colon F\to G$ defined by
	\begin{align*}
		T^0(A_i)&=m_i &&\text{for }A_i\in\cC\\
		T^1(a)&=(-1)^{|a|} m_j\circ F(a)\circ \xi_i &&\text{for }a\in\Hom_{\cC}(A_i,A_j)\\
		T^k&=0 &&\text{for }k>1
	\end{align*}
	where $d\xi_i=\id-m_i'\circ m_i$.
\end{lem}

\begin{proof}
	One can easily check that $\mu^1(T)=0$, and since $T^0(A_i)$ is a homotopy equivalence for every $A_i\in\cC$, by Lemma \ref{lem:seidel-homotopy-equivalence} $T$ is a homotopy equivalence.
\end{proof}

Next, we study the category of dg categories, $\dgCat$, via its model structure:
    
\begin{thm}\cite{tabuada-model}
	The category dgCat has a cofibrantly generated model structure whose weak equivalences are quasi-equivalences, and whose fibrations are full dg functors $F\colon\cC\rightarrow \cD$ such that for each isomorphism $n\colon F(A)\rightarrow B$ in $H^0(\cD)$ there exists an isomorphism $m$ in $H^0(\cC)$ with $F(m) = n$. Every object of $\dgCat$ is fibrant within this model structure.
\end{thm}

\begin{rmk}\label{rmk:model-z2}
	If we work with $\Z/2$-graded version of $\dgCat$, i.e. if the dg categories are $\Z/2$-graded, then it has the same model structure as above, shown in \cite{dyckerhoff-kapranov}.
\end{rmk}

\begin{rmk}\label{rmk:model-pretriangulated}
	If $\cC$ and $\cD$ are pretriangulated dg categories, a dg functor $F\colon\cC\rightarrow\cD$ gives the functor $H^0 F\colon H^0(\cC)\rightarrow H^0(\cD)$ which is exact. Hence we do not need to require exactness of dg functors between pretriangulated categories. This shows that the category of pretriangulated dg category is a full subcategory of $\dgCat$, so it can be given the same model structure as above.
\end{rmk}

\begin{dfn}
	In a model category $\cM$, a \textit{path object} for $A\in\cM$ is an object $A^I\in\cM$ equipped with a weak equivalence $A\to A^I$ and a fibration $A^I\to A \times A$ whose composition gives the diagonal map $A \to A \times A$.
\end{dfn}

By the axioms of the model category, every object in a model category has a path object. In \cite{canonaco}, path objects are explicitly constructed in the model category $\dgCat$:

\begin{dfn}
	The \textit{path space} $\cP(\cC)$ of a dg category $\cC$ is defined as follows:
	\[\cP(\cC):=\{(M_1,m,M_2)\vb M_1,M_2\in\cC, m\colon M_1\rightarrow M_2\text{ is a homotopy equivalence}\}\ .\]
	The degree $k$ morphisms $\Hom_{\cP(\cC)}^k((M_1,m,M_2),(N_1,n,N_2))$, or $\Hom_{\cP(\cC)}^k(m,n)$ shortly, are given by
	\[\Hom^k_{\cC}(M_1,N_1)\oplus\Hom^{k-1}_{\cC}(M_1,N_2)\oplus\Hom^k_{\cC}(M_2,N_2)\ .\]
	A morphism $(\mu_1,\mu_0,\mu_2)\in\Hom_{\cP(\cC)}^k((M_1,m,M_2),(N_1,n,N_2))$ can also presented by the diagram
	\[\begin{tikzcd}
		M_1\rar["m"]\dar["\mu_1"]\drar["\mu_0"] & M_2\dar["\mu_2"]\\
		N_1\rar["n"] & N_2
	\end{tikzcd}\ .\]
	Its differential is
	\[d(\mu_1,\mu_0,\mu_2)=(d\mu_1,d\mu_0+(-1)^k(n\circ\mu_1-\mu_2\circ m),d\mu_2)\ ,\]
	the composition rule is
	\[(\nu_1,\nu_0,\nu_2)\circ(\mu_1,\mu_0,\mu_2)=(\nu_1\circ\mu_1,\nu_2\circ\mu_0+(-1)^k\nu_0\circ\mu_1,\nu_2\circ\mu_2)\ ,\]
	and the identity is $(\id,0,\id)$. $\cP(\cC)$ is equipped with two dg functors
	\[\pi_1,\pi_2\colon \cP(\cC)\to\cC\]
	such that $\pi_i(M_1,m,M_2)=M_i$ and $\pi_i(\mu_1,\mu_0,\mu_2)=\mu_i$ for $i=1,2$.
\end{dfn}

\begin{rmk}
	$\cP(\cC)$ is a full dg category of $\Modk(A_2)$, where $A_2$ is the $A_2$-quiver which is an ($A_{\infty}$-)category with two objects and a morphism between them. We study it in Section \ref{sec:quiver}.
\end{rmk}

\begin{lem}\label{lem:path-space}
	Every path space $\cP(\cC)$ is a path object for $\cC$ in $\dgCat$.
\end{lem}

\begin{proof}
	Observe that the composition \[\cC\overset{i}{\hookrightarrow}\cP(\cC)\xrightarrow{(\pi_1,\pi_2)}\cC\times\cC\]
	is the diagonal map, where $i$ is the obvious embedding of $\cC$ into $\cP(\cC)$, which is a quasi-equivalence, and $(\pi_1,\pi_2)$ is a fibration.
\end{proof}

Using path spaces, we can represent homotopy limits by ordinary limits:

\begin{prp}\label{prp:holim-lim}
	Let $\cC_0,\cC_1,\cC_2$ be dg categories and $p_i\colon\cC_i\to\cC_0$ be a dg functor for $i=1,2$. Then we have the following quasi-equivalence:
	\[
	\holim\left(
	\begin{tikzcd}
		\cC_1\ar[rd,"p_1"'] & & \cC_2\ar[ld,"p_2"]\\
		& \cC_0 &
	\end{tikzcd}
	\right)\simeq\lim\left(
	\begin{tikzcd}
		\cC_1\ar[rd,"p_1"'] & & \cP(\cC_0)\ar[ld,"\pi_1"]\ar[rd,"\pi_2"'] & & \cC_2\ar[ld,"p_2"]\\
		& \cC_0 & & \cC_0 &
	\end{tikzcd}
	\right)
	\]
\end{prp}

\begin{proof}
	It is easy to see that
	\[\holim(\cC_1\xrightarrow{p_1}\cC_0\xleftarrow{p_2}\cC_2)\simeq\holim(\cC_1\xrightarrow{p_1}\cC_0\xleftarrow{\id}\cC_0\xrightarrow{\id}\cC_0\xleftarrow{p_2}\cC_2)\ .\]
	Then we can replace $\cC_0\xrightarrow{(\id,\id)}\cC_0\times\cC_0$ with $\cP(\cC_0)\xrightarrow{(\pi_1,\pi_2)}\cC_0\times\cC_0$ (fibrant replacement) by Lemma \ref{lem:path-space}:
	\[\holim(\cC_1\xrightarrow{p_1}\cC_0\xleftarrow{\id}\cC_0\xrightarrow{\id}\cC_0\xleftarrow{p_2}\cC_2)\simeq \holim(\cC_1\xrightarrow{p_1}\cC_0\xleftarrow{\pi_1}\cP(\cC_0)\xrightarrow{\pi_2}\cC_0\xleftarrow{p_2}\cC_2)\]
	Finally, in the last expression all objects are fibrant, and $\pi_1$ and $\pi_2$ are fibrations, hence homotopy limit becomes ordinary limit (see \cite[Proposition A.2.4.4]{lurie-htt}).
\end{proof}

Above proposition shows that for $\cC=\holim(\cC_1\xrightarrow{p_1}\cC_0\xleftarrow{p_2}\cC_2)$ we have \[\cC\simeq\{M=(M_1\vb m\vb M_2) \vb M_1\in\cC_1,M_2\in\cC_2,(p_1(M_1),m,p_2(M_2))\in\cP(\cC_0)\}\ .\]
The degree $k$ morphisms are
\[\Hom_{\cC}^k(M,N)=\{\mu=(\mu_1\vb\mu_0\vb\mu_2)\vb \mu_i\in\Hom^k(M_i,N_i),\mu_0\in\Hom^{k-1}(p_1(M_1),p_2(N_2))\}\]
for $i=1,2$, with the differential
\[d(\mu_1\vb\mu_0\vb\mu_2)=(d\mu_1\vb d\mu_0+(-1)^k(n\circ p_1(\mu_1)-p_2(\mu_2)\circ m)\vb d\mu_2)\ ,\]
the composition
\[(\nu_1\vb\nu_0\vb\nu_2)\circ(\mu_1\vb\mu_0\vb\mu_2)=(\nu_1\circ\mu_1\vb p_2(\nu_2)\circ\mu_0+(-1)^k\nu_0\circ p_1(\mu_1)\vb\nu_2\circ\mu_2)\ ,\]
and the identity $(\id\vb 0\vb \id)$. Write $p_i(M_i)=M^i$ and $p_i(\mu_i)=\mu^i$ for $i=1,2$ from now on.

\begin{rmk}
	As shown in Remark \ref{rmk:model-pretriangulated}, the model structure of $\dgCat$ can be restricted to the category of pretriangulated dg categories. In particular, this implies that taking homotopy limit of pretriangulated dg categories gives a pretriangulated dg category.
\end{rmk}

Rest of the section is devoted to understanding homotopic morphisms and homotopy equivalences in homotopy limits. Fix a homotopy limit $\cC=\holim(\cC_1\xrightarrow{p_1}\cC_0\xleftarrow{p_2}\cC_2)$, $M,N\in\cC$, and $\mu,\nu\in\Hom^k(M,N)$.

First, it is easy to observe the following lemma which we will refer often when we prove the lemmas in Section \ref{sec:gluing}:

\hypertarget{homotopy}{}

\begin{lem}[Homotopy Lemma]\label{lem:homotopy-lemma}
	$\mu\sim\nu$ if and only if there exist $\xi_1$ and $\xi_2$ such that $d\xi_i=\mu_i-\nu_i$ for $i=1,2$ and $(-1)^k(\xi^2\circ m-n\circ\xi^1)\sim\mu_0-\nu_0$.
\end{lem}

Next, we will try to describe homotopy equivalences:

\begin{lem}\label{lem:closed-inverse}
	If $\mu_1$ and $\mu_2$ are homotopy equivalences with the inverses $\nu_1$ and $\nu_2$, respectively, and $d\mu=0$, then there exists $\nu=(\nu_1\vb\nu_0\vb\nu_2)$ such that $d\nu=0$.
\end{lem}

\begin{proof}
	We have $\nu_i\circ\mu_i\sim\id$ and $\mu_i\circ\nu_i\sim\id$ for $i=1,2$, by applying $p_i$, we get $\nu^i\circ\mu^i\sim\id$ and $\mu^i\circ\nu^i\sim\id$ for $i=1,2$. Since $d\mu=0$, we get $(d\mu)_0=d\mu_0+n\circ\mu^1-\mu^2\circ m=0$. This implies $n\circ\mu^1\sim\mu^2\circ m$. Using Proposition \ref{prp:homotopic-prop} and above relations, we get $\nu^2\circ n\sim m\circ\nu^1$. Hence there exists $\nu_0$ such that $d\nu_0+m\circ\nu^1-\nu^2\circ n=0$. We have also $d\nu_i=0$ for $i=1,2$. After defining $\nu:=(\nu_1\vb\nu_0\vb\nu_2)$, this implies $d\nu=0$.
\end{proof}

Before stating the next lemma, we define
\[\epsilon_{\mu,\nu,\xi_1,\xi_2}:=(-1)^k(\mu_0-\nu_0)+(n\circ\xi^1-\xi^2\circ m)\ .\]
Note that if $\mu\sim\nu$ with $d\xi_i=\mu_i-\nu_i$ for $i=1,2$, we have $d\epsilon_{\mu,\nu,\xi_1,\xi_2}=0$ by \hyperlink{homotopy}{Homotopy Lemma}.

\hypertarget{equivalence}{}

\begin{lem}[Equivalence Lemma]
	$\mu$ is a homotopy equivalence if and only if $\mu_1$ and $\mu_2$ are homotopy equivalences and $n\circ\mu^1\sim\mu^2\circ m$.
\end{lem}

\begin{proof}
	We will construct a left inverse and a right inverse of $\mu$ in homotopy. Let $\nu_i$ be inverses of $\mu_i$ in homotopy for $i=1,2$. We have $d\mu=0$ by assumptions. Then by Lemma \ref{lem:closed-inverse}, there exists $\nu=(\nu_1\vb\nu_0\vb\nu_2)$ such that $d\nu=0$. We will modify $\nu$ to get a left inverse of $\mu$.
	
	Let $d\xi_i=\nu_i\circ\mu_i-\id$ for $i=1,2$ and $\epsilon:=\epsilon_{\nu\circ\mu,\id,\xi_1,\xi_2}$. Define $\nu':=(\nu_1\vb \nu_0-\epsilon\circ\nu^1\vb \nu_2)$. Note that $d\nu'=0$ , since $d\nu=0$ and $d\epsilon=0$. Then $\epsilon_{\nu'\circ\mu,\id,\xi}=\epsilon-\epsilon\circ\nu^1\circ\mu^1\sim\epsilon-\epsilon=0$. Hence by \hyperlink{homotopy}{Homotopy Lemma}, $\nu'\circ\mu\sim\id$, i.e. $\nu'$ is a left inverse of $\mu$ in homotopy.
	
	Similarly, $\nu$ can be modified as $\nu''$ such that $\nu''$ becomes a right inverse of $\mu$ in homotopy. Note that $\nu'\circ\mu\sim\id$ gives $\nu'\circ\mu\circ\nu''\sim\nu''$, and hence $\nu'\sim\nu''$. Then $\mu\circ\nu''\sim\id$ implies $\mu\circ\nu'\sim\id$. Therefore $\nu'$ is an inverse of $\mu$ in homotopy and $\mu$ is a homotopy equivalence.
\end{proof}

\hyperlink{equivalence}{Equivalence Lemma} will be our main tool along with \hyperlink{homotopy}{Homotopy Lemma} for proving lemmas in Section \ref{sec:gluing}.

\subsection{Gluing Lemmas}\label{sec:gluing}
    
We will present two important lemmas for our calculations regarding gluing in Section \ref{sec:msh-pinwheel}. Note that by Remark \ref{rmk:model-z2}, these lemmas hold for both $\Z$- and $\Z/2$-graded dg categories.

\hypertarget{circle}{}

\begin{lem}[Circle Lemma]\label{lem:circle}
	Let $\cC$ be a dg category, and $\cD$ be the homotopy limit of the diagram
	\[\begin{tikzcd}[column sep=10pt, row sep=10pt]
		& \cC & \\
		\cC\urar["\id"]\drar["p_1"'] & & \cC\ular["\id"']\dlar["p_2"] \\
		& \cC &
	\end{tikzcd}\]
	i.e. $\cD=\holim(\cC\xrightarrow{(\id,p_1)}\cC\times\cC\xleftarrow{(\id,p_2)}\cC)$. Let $\cD'$ be the dg category defined as
	\begin{gather*}
    	\cD'=\{(A,m)\vb A\in\cC, m\colon p_1(A)\rightarrow p_2(A)\textup{ is a homotopy equivalence}\}\\
        \Hom_{\cD'}^k((A,m),(B,n))=\Hom_{\cC}^k(A,B)\oplus\Hom_{\cC}^{k-1}(p_1(A),p_2(B))\ .
    \end{gather*}
    A morphism $(f,h)\in\Hom_{\cD'}^k((A,m),(B,n))$ can be also presented by the diagrams
    \[\begin{tikzcd}
    	A\dar["{f\hspace{2em}\textup{and}}"] & [20pt] p_1(A)\rar["m"]\dar["p_1(f)"]\drar["h"] & p_2(A)\dar["p_2(f)"]\\
    	B & p_1(B)\rar["n"] & p_2(B)
    \end{tikzcd}\ .\]
    Its differential is
    \[d(f,h)=(df,dh+(-1)^k(n\circ p_1(f)-p_2(f)\circ m))\ ,\]
    the composition rule is
    \[(f',h')\circ(f,h)=(f'\circ f,p_2(f')\circ h+(-1)^k h'\circ p_1(f))\ ,\]
    and the identity is $(\id,0)$. Then $\cD$ is quasi-equivalent to $\cD'$. Moreover, $(f,h)$ in $\cD'$ is a homotopy equivalence if and only if $f$ is a homotopy equivalence in $\cC$ and we have that $dh=p_2(f)\circ m-n\circ p_1(f)$.
\end{lem}
        
\begin{proof}
	Define the dg functor $F\colon \cD'\rightarrow\cD$ such that $F(A,m)=(A\vb(\id,m)\vb A)$ for $(A,m)\in\cD'$ and $F(f,h)=(f\vb (0,h)\vb f)$ for $(f,h)\in\Hom_{\cD'}((A,m),(B,n))$.
        
	To show essential surjectivity of $H^0 F$, pick $(A\vb(a,b)\vb B)\in\cD$. Let $a'$ be an inverse of $a$ in homotopy. We claim that it is homotopy equivalent to
    \[F(A,p_2(a')\circ b)=(A\vb(\id,p_2(a')\circ b)\vb A)\ .\]
    Indeed, we have the homotopy equivalences $\id:A\rightarrow A$ and $a:A\rightarrow B$, and
    \[(a,b)\circ(\id,\id)\sim(a,p_2(a))\circ(\id,p_2(a')\circ b)\ .\] By \hyperlink{equivalence}{Equivalence Lemma}, our claim is true.
        
    Next, we want to show that $H^*F$ is full and faithful. To show that $H^*F$ is injective on morphisms, assume $F(f,h)\sim F(f',h')$ and $d(f,h)=d(f',h')=0$. Then by \hyperlink{homotopy}{Homotopy Lemma}
    \[(-1)^{|f|}((\xi_2,p_2(\xi_2))\circ(\id,m)-(\id,n)\circ(\xi_1,p_1(\xi_1)))\sim(0,h-h')\]
    for some $\xi_i$ with $d\xi_i=f-f'$ for $i=1,2$. In particular we have $\xi_1\sim\xi_2$, and this implies $p_2(\xi_1)\sim p_2(\xi_2)$. Using this, we get
    \[h-h'\sim(-1)^{|\xi^1|}(n\circ p_1(\xi_1)-p_2(\xi_2)\circ m)\sim(-1)^{|\xi^1|}(n\circ p_1(\xi_1)-p_2(\xi_1)\circ m)\ .\]
    Hence there exists $\gamma$ such that $d(\xi^1,\gamma)=(f,h)-(f',h')$ and $(f,h)\sim(f',h')$. This proves $H^*F$ is injective on morphisms.
        
    Finally, we will show that $H^*F$ is surjective on morphisms. Pick \[\mu=(f\vb(\alpha,\beta)\vb g)\in\Hom^k(F(A,m),F(B,n))\]
    with $d\mu=0$. We claim that it is homotopic to $\nu=F(f,\beta-p_2(\alpha)\circ m)$. Indeed, choosing $\xi_1=0$ and $\xi_2=(-1)^k\alpha$, we have $(-1)^k((-1)^k\alpha,(-1)^kp_2(\alpha)\circ m)=(\alpha,p_2(\alpha)\circ m)$. Also $d\xi_i=\mu_i-\nu_i$, hence by \hyperlink{homotopy}{Homotopy Lemma} $\mu\sim\nu$. Also $d(f,\beta-p_2(\alpha)\circ m)=0$, therefore $H^*F$ is surjective on morphisms. Consequently, $F$ is a quasi-equivalence and $\cD$ and $\cD'$ are quasi-equivalent.
        
    To characterise the homotopy equivalences in $\cD'$, note that $(f,h)$ is a homotopy equivalence if and only if $F(f,h)$ is a homotopy equivalence. Since $F(f,h)=(f\vb(0,h)\vb f)$ is in $\cD$, by \hyperlink{equivalence}{Equivalence Lemma}, $F(f,h)$ is a homotopy equivalence if and only if $f$ is a homotopy equivalence and $dh=p_2(f)\circ m-n\circ p_1(f)$.
\end{proof}
    
Here is an easy application of the lemma:
        
\begin{exm}
	We can calculate $\Loc\dd(S^1)$ locally for the circle $S^1=[0,1]/\{0,1\}$ using this lemma. First, note that the large microlocal stack $\mSh\dd=\mSh\dd_{T^*S^1}$ is defined on the skeleton $S^1$ and its global section is $\mSh\dd(S^1)\simeq\Loc\dd(S^1)$. To calculate it via gluing, partition $S^1$ into pieces:
    \[S^1\simeq\colim((0,1)\xleftarrow{(i,i)}(0,1/2)\sqcup(1/2,1)\xrightarrow{(i,i)}(0,1))\]
    where $i$ are inclusions. Applying $\mSh\dd$, we get
    \[\mSh\dd(S^1)\simeq\holim(\mSh\dd((0,1))\xleftarrow{(\id,\id)}\mSh\dd((0,1/2))\times\mSh\dd((1/2,1))\xrightarrow{(\id,\id)}\mSh\dd((0,1)))\ .\]
    Since for any open interval $U$ we have $\mSh\dd(U)\simeq\Loc\dd(U)\simeq\Modk$, we can write
   	\[\Loc\dd(S^1)\simeq\holim(\Modk\xleftarrow{(\id,\id)}\Modk\times\Modk\xrightarrow{(\id,\id)}\Modk)\ .\]
   	Then by setting $p_1=\id$ and $p_2=\id$ in \hyperlink{circle}{Circle Lemma}, we get:
    \begin{gather*}
    	\Loc\dd(S^1)\simeq\{(A,m)\vb A\in\Modk, m\colon A\rightarrow A\text{ is a homotopy equivalence}\}\\
        \Hom_{\Loc\dd(S^1)}^k((A,m),(B,n))=\Hom_{\Modk}^k(A,B)\oplus\Hom_{\Modk}^{k-1}(A,B)\ .
   	\end{gather*}
   	A morphism $(f,h)\in\Hom_{\Loc\dd(S^1)}^k((A,m),(B,n))$ can be also presented by the diagram
   	\[\begin{tikzcd}
		A\rar["m"]\dar["f"]\drar["h"] & A\dar["f"]\\
	   	B\rar["n"] & B
   	\end{tikzcd}\ .\]
    Its differential is
   	\[d(f,h)=(df,dh+(-1)^k(n\circ f-f\circ m))\ ,\]
    the composition rule is
    \[(f',h')\circ(f,h)=(f'\circ f,f'\circ h+(-1)^k h'\circ f)\ ,\]
    and the identity is $(\id,0)$. Moreover, $(f,h)$ in $\Loc\dd(S^1)$ is a homotopy equivalence if and only if $f$ is a homotopy equivalence and $dh=f\circ m-n\circ f$. One can calculate $\Loc(S^1)$ in the same way by replacing $\Modk$ with $\Perfk$. Compare with Example \ref{exm:circle-local-system}.   
\end{exm}

From now on, if $M=(A,m)\in\Loc\dd(S^1)$, we write $M_1:=A$ and $M_0:=m$. Also, if $\mu=(f,h)$ is a morphism in $\Loc\dd(S^1)$, we write $\mu_1:=f$ and $\mu_0:=h$. Our second lemma is as follows:

\hypertarget{disk}{}
        
\begin{lem}[Disk Lemma]\label{lem:disk-gluing}
	Let $Y$ be the skeleton of a Weinstein manifold $W$, and $\mSh\dd=\mSh\dd_W$ be the large microlocal stack on $Y$. Assume we can write
    \[Y\simeq\colim(X\overset{i}{\hookleftarrow}S^1\times B^1\hookrightarrow{B^2})\]
    where $B^n$ is an open $n$-ball in $\R^n$, $X$ is an open subset of $Y$ with the boundary $\partial X=S^1$ whose neighbourhood is $S^1\times B^1$, and the maps are standard inclusions into the neighbourhood of boundaries. If $\mon=\mSh\dd(i)$, then we have
    \begin{gather*}
    	\mSh\dd(Y)\simeq\{(\Gamma,\gamma)\vb \Gamma\in\mSh\dd(X),\mon(\Gamma)_0\overset{\gamma}{\sim}\id\}\\
        \Hom_{\mSh\dd(Y)}^k((\Gamma,\gamma),(\Gamma',\gamma'))=\Hom_{\mSh\dd(X)}^k(\Gamma,\Gamma')\oplus\Hom_{\Modk}^{k-2}(\mon(\Gamma)_1,\mon(\Gamma')_1)\ .
    \end{gather*}
    A morphism $(f,h)\in\Hom_{\mSh\dd(Y)}^k((\Gamma,\gamma),(\Gamma',\gamma'))$ can be also presented by the diagrams
    \[\begin{tikzcd}
    	& [20pt]
	    & [20pt]
	    \mon(\Gamma)_1\ar[rrr,"\mon(\Gamma)_0"]\ar[drr,"\gamma"]\ar[dl,"\id"]\ar[dd,"\mon(f)_1" yshift=-0.5cm]\ar[dddrr,"h"]\ar[ddrrr,dashed,gray,"\mon(f)_0" {xshift=1cm, yshift=-0.6cm}] & [-20pt]
	    & [-20pt]
	    & [20pt]
	    \mon(\Gamma)_1\ar[dl,"\id"]\ar[dd,,"\mon(f)_1"]
	    \\
	    \Gamma\ar[dd,"f\hspace{2em}\textup{and}"] &
		\mon(\Gamma)_1\ar[rrr,crossing over,"\id" xshift=-0.2cm]\ar[dd,"\mon(f)_1"] &
		&
		&
		\mon(\Gamma)_1\ar[dd,,"\mon(f)_1" yshift=0.7cm] &
		\\
		&
		&
		\mon(\Gamma')_1\ar[-,r,shorten >= -0.4cm]\ar[drr,"\gamma'"]\ar[dl,"\id"] &
		{}\ar[-,r,shorten <= 0.5cm] &
		{}\ar[r,"\mon(\Gamma')_0"] &
		\mon(\Gamma')_1\ar[dl,"\id"]
		\\
		\Gamma' &
		\mon(\Gamma')_1\ar[rrr,"\id" xshift=-0.2cm] &
		&
		&
		\mon(\Gamma')_1 &
    \end{tikzcd}\ .\]
    Its differential is
    \[d(f,h)=(df,dh+\gamma'\circ \mon(f)_1+(-1)^k(\mon(f)_0-\mon(f)_1\circ\gamma))\ ,\]
    the composition rule is
    \[(f',h')\circ(f,h)=(f'\circ f,\mon(f')_1\circ h+(-1)^k h'\circ \mon(f)_1)\ ,\]
    and the identity is $(\id,0)$. Moreover, $(f,h)$ in $\mSh\dd(Y)$ is a homotopy equivalence if and only if $f$ is a homotopy equivalence and $dh=\mon(f)_1\circ\gamma-\gamma'\circ \mon(f)_1-\mon(f)_0$. Everything similarly holds for $\mSh(Y)$ after replacing $\mSh\dd(X)$ by $\mSh(X)$, $\Loc\dd(S^1)$ by $\Loc(S^1)$, and $\Modk$ by $\Perfk$.
\end{lem}
    
\begin{dfn}
	We call $\mon\colon\mSh\dd(X)\to\Loc\dd(S^1)$ the \textit{monodromy functor}, and $\mon(\Gamma)_0$ the \textit{monodromy} of $\Gamma$.
\end{dfn}

\begin{proof}[Proof of \hyperlink{disk}{Disk Lemma}]
	Applying $\mSh\dd$ to the colimit diagram, we get	
	\[\mSh\dd(Y)\simeq\holim(\mSh\dd(X)\xrightarrow{\mon}\mSh\dd(S^1\times B^1)\leftarrow\mSh\dd(B^2))\ .\]
	We know $\mSh\dd(B^2)\simeq\Loc\dd(B^2)\simeq\Modk$, and by Proposition \ref{prp:stabilisation} we have
	\[\mSh\dd(S^1\times B^1)\simeq\mSh\dd(S^1)\simeq\Loc\dd(S^1)\ .\]
	Then the diagram becomes
 	\[\mSh\dd(Y)\simeq\holim(\mSh\dd(X)\xrightarrow{\mon}\Loc\dd(S^1)\xleftarrow{r}\Modk)\]
   	where $r(A)=(A,\id)$ on objects and $r(f)=(f,0)$ on morphisms. Define
    \begin{gather*}
    	\cC=\{(\Gamma,\gamma)\vb \Gamma\in\mSh\dd(X),\mon(\Gamma)_0\overset{\gamma}{\sim}\id\}\\
        \Hom_{\cC}^k((\Gamma,\gamma),(\Gamma',\gamma'))=\Hom_{\mSh\dd(X)}^k(\Gamma,\Gamma')\oplus\Hom_{\Modk}^{k-2}(\mon(\Gamma)_1,\mon(\Gamma')_1)
   	\end{gather*}
    with the differential
    \[d(f,h)=(df,dh+\gamma'\circ \mon(f)_1+(-1)^k(\mon(f)_0-\mon(f)_1\circ\gamma))\]
   	the composition
    \[(f',h')\circ(f,h)=(f'\circ f,\mon(f')_1\circ h+h'\circ \mon(f)_1)\]
    and the identity $(\id,0)$. We will show $\mSh\dd(Y)$ and $\cC$ are quasi-equivalent.
    
    Define the dg functor $F\colon\cC\rightarrow\mSh\dd(Y)$ such that $F(\Gamma,\gamma)=(\Gamma\vb(\id,\gamma)\vb \mon(\Gamma)_1)$ for $(\Gamma,\gamma)\in\cC$ and $F(f,h)=(f\vb (0,h)\vb \mon(f)_1)$ for $(f,h)\in\Hom_{\cC}((\Gamma,\gamma),(\Gamma',\gamma'))$.
        
   	To show essential surjectivity of $H^0 F$, pick $(\Gamma\vb(\alpha,\gamma)\vb A)\in\mSh\dd(Y)$. Let $\alpha'$ be an inverse of $\alpha$ in homotopy, with $d\hat\alpha=\id-\alpha'\circ\alpha$ and $d\check\alpha=\id-\alpha\circ\alpha'$. We claim that $(\Gamma\vb(\alpha,\gamma)\vb A)$ is homotopy equivalent to \[F(\Gamma,\alpha'\circ\gamma+\hat\alpha\circ(\mon(\Gamma)_0-\id))=(\Gamma\vb(\id,\alpha'\circ\gamma+\hat\alpha\circ(\mon(\Gamma)_0-\id))\vb \mon(\Gamma)_1)\ .\]
    Indeed, we have the homotopy equivalences $\id\colon\Gamma\rightarrow \Gamma$ and $\alpha\colon \mon(\Gamma)_1\rightarrow A$, and using $d\gamma=\alpha\circ(\mon(\Gamma)_0-\id)$, we get
    \[(\alpha,0)\circ(\id,\alpha'\circ\gamma+\hat\alpha\circ(\mon(\Gamma)_0-\id))\overset{(\alpha\circ\hat\alpha-\check\alpha\circ\alpha,-\check\alpha\circ\gamma)}{\sim}(\alpha,\gamma)\circ(\id,0)\ .\]
    By \hyperlink{equivalence}{Equivalence Lemma}, our claim is true.
        
    Next, we want to show that $H^*F$ is full and faithful. To show that $H^*F$ is injective on morphisms, pick
    \[(f,h),(f',h')\in H^k\Hom_{\cC}((\Gamma,\gamma),(\Gamma',\gamma'))\]
    and assume $F(f,h)\sim F(f',h')$. Then there exists $(\xi\vb(\alpha,\beta)\vb \eta)$ such that
    \[d(\xi\vb(\alpha,\beta)\vb \eta)=(f-f'\vb(0,h-h')\vb(\mon(f-f')_1)\ .\]
    In particular, we get
    \begin{align*}
     	d\xi&=f-f'\\
        d(\alpha,\beta)&-(-1)^k((\id,\gamma')\circ(\mon(\xi)_1,\mon(\xi)_0)-(\eta,0)\circ(\id,\gamma))=(0,h-h')
    \end{align*}
    which gives
    \begin{align*}
    	d\alpha&=(-1)^k(\mon(\xi)_1-\eta)\\
        d\beta&=(h-h')-\gamma'\circ \mon(\xi)_1+(-1)^k(\mon(\xi)_0-\eta\circ\gamma)+(-1)^k\alpha\circ(\mon(\Gamma)_0-\id)\ .
    \end{align*}
    Then we get
    \[d(\xi,\beta-\alpha\circ\gamma)=(f,h)-(f',h')\]
    since $d\gamma=\mon(\Gamma)_0-\id$. Hence we get $(f,h)\sim(f',h')$. This proves $H^*F$ is injective on morphisms.
        
    Finally, we will show that $H^*F$ is surjective on morphisms. Pick
    \[(f\vb(l,h)\vb g)\in\Hom_{\mSh\dd(Y)}^k(F(\Gamma,\gamma),F(\Gamma',\gamma'))\]
    and since it is closed, we get in particular
    \begin{align*}
     	df&=0\\
        dl&=(-1)^k(g-\mon(f)_1)\\
        dh&=-\gamma'\circ \mon(f)_1-(-1)^k(\mon(f)_0-g\circ\gamma)-(-1)^k l\circ(\mon(\Gamma)_0-\id)\ .
    \end{align*}
    Pick $(f,h-l\circ\gamma)\in\Hom_{\cC}^k((\Gamma,\gamma),(\Gamma',\gamma'))$. It is closed, and we claim that
    \[F(f,h-l\circ\gamma)=(f\vb(0,h-l\circ\gamma)\vb \mon(f)_1)\]
   	is homotopic to $(f\vb(l,h)\vb g)$. Indeed, we have
    \[d(0\vb(0,0)\vb (-1)^k l)=(f\vb(l,h)\vb g)-(f\vb(0,h-l\circ\gamma)\vb \mon(f)_1)\]
    hence $H^*F$ is surjective on morphisms. Consequently, $F$ is a quasi-equivalence and $\mSh\dd(Y)$ and $\cC$ are quasi-equivalent.
        
    To characterise the homotopy equivalences in $\cC$, note that $(f,h)$ is a homotopy equivalence if and only if $F(f,h)$ is a homotopy equivalence, and by \hyperlink{equivalence}{Equivalence Lemma} $F(f,h)=(f\vb(0,h)\vb \mon(f)_1)$ is a homotopy equivalence if and only if $f$ is a homotopy equivalence and $dh=\mon(f)_1\circ\gamma-\gamma'\circ \mon(f)_1-\mon(f)_0$.
\end{proof}
    
In short, the effect of attaching disk on microlocal sheaves is adding a degree $-1$ element which kills the monodromy on the boundary.
    
\begin{exm}
	We can calculate $\Loc\dd(S^2)$ using \hyperlink{disk}{Disk Lemma}. Let $\mSh\dd=\mSh\dd_{T^*S^2}$ be the large microlocal stack on the skeleton $S^2$. Its global section is $\mSh\dd(S^2)\simeq\Loc\dd(S^2)$. We can write
    \[S^2\simeq\colim(B^2\overset{p}{\hookleftarrow}S^1\times B^1\overset{i}{\hookrightarrow}{B^2})\]
    where $p$ and $i$ are inclusions. Note that $\mon=\mSh\dd(p)$ is such that $\mon(A)=(A,\id)$ on objects and $\mon(f)=(f,0)$ on morphisms. By the lemma, we get
    \begin{gather*}
    	\Loc\dd(S^2)=\{(A,\gamma)\vb A\in\Modk,\gamma\in\Hom_{\Modk}^{-1}(A,A),d\gamma=0\}\\
      	\Hom_{\Loc\dd(S^2)}^k((A,\gamma),(A',\gamma'))=\Hom_{\Modk}^k(A,A')\oplus\Hom_{\Modk}^{k-2}(A,A')
  	\end{gather*}
  	A morphism $(f,h)\in\Hom_{\Loc\dd(S^2)}^k((A,\gamma),(A',\gamma'))$ can be also presented by the diagram
  	\[\begin{tikzcd}
  		A\rar["\gamma"]\dar["f"]\drar["h"] & A\dar["f"]\\
  		A'\rar["\gamma'"] & A'
  	\end{tikzcd}\ .\]
    Its differential is
    \[d(f,h)=(df,dh+\gamma'\circ f-(-1)^k f\circ\gamma)\ ,\]
    the composition rule is
    \[(f',h')\circ(f,h)=(f'\circ f,f'\circ h+(-1)^k h'\circ f)\ ,\]
    and the identity is $(\id,0)$. Moreover, $(f,h)$ in $\Loc\dd(S^2)$ is a homotopy equivalence if and only if $f$ is a homotopy equivalence and $dh=f\circ\gamma-\gamma'\circ f$. Similarly, we can calculate $\Loc(S^2)$ after replacing $\Modk$ by $\Perfk$.
    
    Note that in Definition \ref{dfn:locally-constant} we remarked that
    \[\Loc\dd(S^2)\simeq\Modk(C_{-*}(\Omega S^2))\]
    and the dga $C_{-*}(\Omega S^2)$ is given by
    \[C_{-*}(\Omega S^2)\simeq \k[x]\]
    with $|x|=-1$ and $dx=0$. This matches with our computation.
\end{exm}

\subsection{$A_{\infty}$-Modules over $A_n$-Quiver}\label{sec:quiver}

In this section, we will study the $A_{\infty}$-modules over $A_n$-quiver, $\Modk(A_n)$, and realise them as the microlocal sheaves on $(n+1)$-valent vertex. Using this geometric realisation, we will define some dg functors on $\Modk(A_n)$ corresponding to the operations on the vertex. In particular, we will define the Coxeter functor on $\Modk(A_n)$ which corresponds to the rotation of the vertex. The results in this section will be extensively used in Section \ref{sec:msh-pinwheel}.

We start by fixing some notations. Given $A\in\Modk$, we denote its components by superscripts: We write $A=(A^i,d_A^i)_{i\in\Z}$ meaning the (co)chain complex
\[\ldots\xrightarrow{d_A^{-2}} A^{-1}\xrightarrow{d_A^{-1}}A^0\xrightarrow{d_A^0}A^1\xrightarrow{d_A^1}A^2\xrightarrow{d_A^2}\ldots\ .\]
Also, we write $f=(f^i)_{i\in\Z}\in\Hom^k_{\Modk}(A,B)$ where $f^i\colon A^i\to B^{i+k}$ denotes the component of $f$ with the domain $A^ i$. Its differential $df=((df)^i)_{i\in\Z}$ can be described by
\[(df)^i=d_B^i\circ f^i-(-1)^k f^{i+1}\circ d_A^i\ .\]
We write $d$ for $d_A$ and $d_B$ in short when it does not arise any confusion.

For $A,B\in\Modk$, we see the elements of $A\oplus B\in\Modk$ as the column matrix $\mx{a \\ b}$, where $a\in A$ and $b\in B$. Then we see the morphism $f\colon A\oplus B\to A'\oplus B'$ as the matrix $\mx{f_{11} & f_{12}\\ f_{21} & f_{22}}$. For typographical reasons, we shortly write
\[\fm(f_{11},f_{12},f_{21},f_{22}):=\mx{f_{11} & f_{12}\\ f_{21} & f_{22}}\ .\] Moreover we write
\[\fb(f_{11},f_{21},f_{22}):=\mx{f_{11} & 0\\ f_{21} & f_{22}}\]
and
\[\fd(f_{11},f_{22}):=\mx{f_{11} & 0\\ 0 & f_{22}}\ .\]
If $f\colon A\oplus B\to A'$, it is represented by
\[\fr(f_1,f_2):=\mx{f_1 & f_2}\]
and if $f\colon A \to A'\oplus B'$, it is represented by
\[\fc(f_1,f_2):=\mx{f_1 \\ f_2}\ .\]
We also write $I_n$ for the $n\times n$ identity matrix, and $0_{n,m}$ for the $n\times m$ zero matrix. However, we simply write $\id$ for $I_n$ and $0$ for $0_{n,m}$ when their dimensions are clear from the context.

\begin{dfn}
	For $n\geq 1$, \textit{(the path algebra of) $A_n$-quiver} is a $k$-linear category with $n$ objects and with the morphisms
	\[\bullet_1\leftarrow\bullet_2\leftarrow\ldots\leftarrow\bullet_n\]
	where the composition is the concatenation of the arrows. We will see it as $A_{\infty}$-category where $\mu^i=0$ for $i\neq 2$ and $\mu^2$ is the composition. We denote this category by $A_n$ and call $\Modk(A_n)$ \textit{$A_{\infty}$-modules over $A_n$-quiver}. $\Modk(A_n)$ is a pretriangulated dg category (see \cite{seidel}).
\end{dfn}

In particular, we can represent the objects of $\Modk(A_2)$ as
\[A_1\xrightarrow{a_1}A_2\]
where $A_1,A_2\in\Modk$ and $a_1\in \Hom^0(A_1,A_2)$ with $da_1=0$. We can also represent it by $A=(A_1,a_1,A_2)$. An element in $\Hom^k_{\Modk(A_2)}(A,B)$ is given by
\[\begin{tikzcd}
	A_1\ar[r,"a_1"]\ar[d,"f_1"]\ar[rd,"h_1"] & A_2\ar[d,"f_2"]\\
	B_1\ar[r,"b_1"] & B_2
\end{tikzcd}\]
where $f_i\in\Hom^k(A_i,B_i)$ for $i=1,2$ and $h_1\in\Hom^{k-1}(A_1,B_2)$, which can be represented by $f=(f_1,h_1,f_2)$. It has the differential
\[d(f_1,h_0,f_2)=(df_1,dh_1+(-1)^k(b_1\circ f_1-f_2\circ a_1),df_2)\ ,\]
the composition
\[(f'_1,h'_1,f'_2)\circ(f_1,h_1,f_2)=(f'_1\circ f_1,f'_2\circ h_1+(-1)^k h'_1\circ f_1,f'_2\circ f_2 )\ ,\]
and the identity is $(\id,0,\id)$. By Lemma \ref{lem:seidel-homotopy-equivalence}, $(f_1,h_2,f_2)\in\Hom^0_{\Modk(A_2)}(A,B)$ is a homotopy equivalence if and only if $f_1$ and $f_2$ are homotopy equivalences and $dh_1=f_{2}\circ a_1-b_1\circ f_1$.

Next, we will define two important functors on $\Modk$ and $\Modk(A_2)$:

\begin{dfn}
   	For a given $n\in\Z$, we define the dg functor $[n]\colon\Modk\to\Modk$, called \textit{shift functor}, as $A[n]^i:=A^{i+n}$ and $d_{A[n]}^i:=(-1)^n d_A^{i+n}$ on an object $A$, and $f[n]^i=f^{i+n}$ on a morphism $f$. It can be also defined on $\Modk(A_n)$ in an obvious way. Clearly, shift functors are auto-quasi-equivalences of $\Modk(A_n)$.
\end{dfn}

\begin{rmk}
	For a degree $k$ morphism $f\colon A\to B$ in $\Modk$, by seeing $d_A\colon A\to A$ and $d_B\colon B\to B$ as degree $1$ morphisms, we can write the differential of $f$ as
	\[df=d_B\circ f -(-1)^k f[1]\circ d_A\ .\]
	Moreover, we write simply $f[n]=f$ when it is clear from the context. Then we can write compactly as
	\[df=d\circ f -(-1)^k f\circ d\ .\]
	Also, we can write $d_{A[n]}=(-1)^n d_A$. Then clearly, $f$ can be regarded as degree $k+n-m$ morphism $f[n]\colon A[n]\to B[m]$ with the differential $(-1)^m df$.
\end{rmk}

\begin{dfn}
	We define the dg functor $C\colon\Modk(A_2)\to\Modk$, called \textit{cone functor}, as
	\[C(A_1\xrightarrow{a_1} A_2)=C(a_1):=A_1[1]\oplus A_2\]
	with the differential
	\[d^i_{C(a_1)}:=\mx{d_{A_1[1]}^i & 0 \\ a_1[1]^i & d_{A_2}^i}=\mx{-d_{A_1}^{i+1} & 0 \\ a_1^{i+1} & d_{A_2}^i}\]
	for an object $A_1\xrightarrow{a_1}A_2$, and
	\[C(f_1,h_1,f_2):=\mx{f_1 & 0 \\ h_1 & f_2}\]
	for a morphism $(f_1,h_1,f_2)$. Shortly, we write
	\[d_{C(a_1)}=\mx{-d & 0 \\ a_1 & d}=\fb(-d,a_1,d)\]
	and
	\[C(f_1,h_1,f_2)=\fb(f_1,h_1,f_2)\ .\]
	We call
	\[A_1\xrightarrow{a_1}A_2\xrightarrow{\fc(0,\id)}C(a_1)\]
	an \textit{exact triangle}.
\end{dfn}

\begin{rmk}
	By Proposition \ref{prp:homotopic-prop}, the cone functor preserves homotopy equivalences since it is a dg functor: If a morphism $(f_1,h_1,f_2)$ in $\Modk(A_2)$ is a homotopy equivalence with an inverse $(f_1',h_1',f_2')$, then $C(f_1,h_1,f_2)$ is a homotopy equivalence in $\Modk$ with the inverse $C(f_1',h_1',f_2')$. As a quick corollary, we get:
\end{rmk}

\begin{prp}\label{prp:cone-zero}
	If $(A_1\xrightarrow{a_1}A_2)\in\Modk(A_2)$, and $a_1$ is a homotopy equivalence, then $C(a_1)\simeq 0$.
\end{prp}

\begin{proof}
	Consider the morphism
	\[\begin{tikzcd}
		A_1\rar["a_1"]\dar["a_1"]\drar["0"] & A_2\dar["\id"]\\
		A_2\rar["\id"] & A_2
	\end{tikzcd}\]
	between $(A_1\xrightarrow{a_1}A_2)$ and $(A_2\xrightarrow{\id}A_2)$. It is a homotopy equivalence since $a_1$ is a homotopy equivalence and the square commutes (up to homotopy inside). Hence its cone $C(a_1,0,\id)$ is a homotopy equivalence between $C(a_1)$ and $C(A_2\xrightarrow{\id}A_2)\simeq 0$.
\end{proof}

Let $P(n,i,j)$ denote the permutation matrix which is obtained by interchanging $i^{\text{th}}$ and $j^{\text{th}}$ rows of the $n\times n$ identity matrix. Then we have the following useful lemma:

\hypertarget{nine}{}

\begin{lem}[Nine Lemma]
	Let $(A_1\xrightarrow{a_1}A_2),(B_1\xrightarrow{b_1}B_2)\in \Modk(A_2)$ and $(f_1,h_1,f_2)$ be a closed degree zero morphism between them. Then $(A_1\xrightarrow{f_1}B_1),(A_2\xrightarrow{f_2}B_2)\in \Modk(A_2)$ and $(a_1,h_1,b_1)$ is a closed degree zero morphism between them. Also, we have
	\[C(C(a_1)\xrightarrow{C(f_1,h_1,f_2)}C(b_1))\simeq C(C(f_1)\xrightarrow{C(a_1,h_1,b_1)}C(f_2))\]
	by the isomorphism $P(4,2,3)$ with the inverse same as itself. Schematically, we can state this lemma as follows: Given the commutative square (up to homotopy indicated in the square)
	\[\begin{tikzcd}
	A_1\arrow[r,"a_1"]\arrow[d,"f_1"]\arrow[rd,"h_1"] & A_2\arrow[d,"f_2"]\\
	B_1\arrow[r,"b_1"] & B_2
	\end{tikzcd}\]
	we get the diagram
	\[\begin{tikzcd}
	A_1\arrow[r,"a_1"]\arrow[d,"f_1"]\arrow[rd,"h_1"] & A_2\arrow[r,"{\fc(0,\id)}"]\arrow[d,"f_2"] & C(a_1)\arrow[d,"{C(f_1,h_1,f_2)}"]\\
	B_1\arrow[r,"b_1"]\arrow[d,"{\fc(0,\id)}"] & B_2\arrow[r,"{\fc(0,\id)}"]\arrow[d,"{\fc(0,\id)}"] & C(b_1)\arrow[dashrightarrow,d]\\
	C(f_1)\arrow[r,"{C(a_1,h_1,b_1)}"] & C(f_2)\arrow[dashrightarrow,r] & X
	\end{tikzcd}\]
	where every row and column is an exact triangle.
\end{lem}

\begin{proof}
	Obviously, $(A_1\xrightarrow{f_1}B_1),(A_2\xrightarrow{f_2}B_2)\in \Modk(A_2)$ and $(a_1,h_1,b_1)$ is a closed degree zero morphism between them. Note that we have
	\[C(C(a_1)\xrightarrow{C(f_1,h_1,f_2)}C(b_1))=C(\fb(f_1,h_2,f_2))\]
	and
	\[C(C(f_1)\xrightarrow{C(a_1,h_1,b_1)}C(f_2))=C(\fb(a_1,h_1,b_1))\ .\]
	Hence we get
	\[C(\fb(f_1,h_1,f_2))=A_1\oplus A_2[1]\oplus B_1[1]\oplus B_2\]
	with
	\[d_{C(\fb(f_1,h_1,f_2))}=
	\mx{
		d & 0 & 0 & 0 \\
		a_1 & -d & 0 & 0 \\
		f_1 & 0 & -d & 0 \\
		h_1 & f_2 & b_1 & d
	}
	\]
	and
	\[C(\fb(a_1,h_1,b_1))=A_1\oplus B_1[1]\oplus A_2[1]\oplus B_2\]
	with
	\[d_{C(\fb(a_1,h_1,b_1))}=
	\mx{
		d & 0 & 0 & 0 \\
		f_1 & -d & 0 & 0 \\
		a_1 & 0 & -d & 0 \\
		h_1 & b_1 & f_2 & d
	}\ .
	\]
	To show $P(4,2,3)\colon C(\fb(f_1,h_1,f_2))\to C(\fb(a_1,h_1,b_1))$ is a homotopy equivalence with the inverse ${P(4,2,3)\colon C(\fb(a_1,h_1,b_1))\to C(\fb(f_1,h_1,f_2))}$, note that $P(4,2,3)\circ P(4,2,3)=\id$, hence by the Proposition \ref{prp:closed-inverse}, we only need to show
	\[P(4,2,3)\colon C(\fb(f_1,h_1,f_2))\to C(\fb(a_1,h_1,b_1))\]
	is degree zero and closed. $P(4,2,3)$ is degree zero obviously, and
	\begin{align*}
	d(P(4,2,&3))=d_{C(\fb(a_1,h_1,b_1))}\circ P(4,2,3)-P(4,2,3)\circ d_{C(\fb(f_1,h_1,f_2))}\\
	&=
	\mx{
		d & 0 & 0 & 0 \\
		f_1 & -d & 0 & 0 \\
		a_1 & 0 & -d & 0 \\
		h_1 & b_1 & f_2 & d
	}
	\mx{
		\id & 0 & 0 & 0 \\
		0 & 0 & \id & 0 \\
		0 & \id & 0 & 0 \\
		0 & 0 & 0 & \id
	}
	-
	\mx{
		\id & 0 & 0 & 0 \\
		0 & 0 & \id & 0 \\
		0 & \id & 0 & 0 \\
		0 & 0 & 0 & \id
	}
	\mx{
		d & 0 & 0 & 0 \\
		a_1 & -d & 0 & 0 \\
		f_1 & 0 & -d & 0 \\
		h_1 & f_2 & b_1 & d
	}\\
	&=0
	\end{align*}
	hence $P(4,2,3)$ is closed, and it is a homotopy equivalence.
\end{proof}

\begin{lem}\label{lem:zero}
	Given $(A_1\xrightarrow{a_1}A_2)\in \Modk(A_2)$, and $A_2\simeq 0$ via degree $-1$ morphism $\alpha\colon A_2\to A_2$ such that $d\alpha=\id$, we have
	\[C(A_1\xrightarrow{a_1}A_2)\simeq A_1[1]\]
	where the equivalence is given by the morphism
	\[\xi_1=\fr(\id,0)\]
	with the inverse
	\[\xi_1'=\fc(\id,-\alpha\circ a_1)\]
	and given $(B_1\xrightarrow{b_1}B_2)\in \Modk(A_2)$, and $B_1\simeq 0$ via degree $-1$ morphism $\beta\colon B_1\to B_1$ such that $d\beta=\id$, we have
	\[C(B_1\xrightarrow{b_1}B_2)\simeq B_2\]
	where the equivalence is given by the morphism
	\[\xi_2=\fr(b_1\circ\beta,\id)\]
	with the inverse
	\[\xi_2'=\fc(0,\id)\ .\]
\end{lem}

\begin{proof}
	Consider the degree zero morphisms $\Xi_1$ and $\Xi_1'$:
	\[\begin{tikzcd}
		(A_1\xrightarrow{a_1}A_2)\arrow[d,"\Xi_1"] & & A_1\arrow[r,"a_1"]\arrow[d,"\id"]\arrow[rd,"0"] & A_2\arrow[d,"0"] \\
		(A_1\xrightarrow{0}0)\arrow[d,"\Xi_1'"] & = & A_1\arrow[r,"0"]\arrow[d,"\id"]\arrow[rd,"-\alpha\circ a_1" xshift=-0.1cm] & 0\arrow[d,"0"] \\
		(A_1\xrightarrow{a_1}A_2) & & A_1\arrow[r,"a_1"] & A_2
	\end{tikzcd}\]
	Obviously, $\Xi_1\circ\Xi_1'=\id$, and
	\[\Xi_1'\circ\Xi_1-\id=(\id,-\alpha\circ a_1,0)\circ(\id,0,0)-\id=(0,-\alpha\circ a_1,-\id)=d(0,0,-\alpha)\]
	and also $d(\Xi_1)=0$. Hence by Proposition \ref{prp:closed-inverse}, $\Xi_1$ is a homotopy equivalence with the inverse $\Xi_1'$. This implies $\xi_1=C(\Xi_1)$ is a homotopy equivalence with the inverse $\xi_1'=C(\Xi_1')$.
	
	Next, consider the degree zero morphisms $\Xi_2$ and $\Xi_2'$:
	\[
	\begin{tikzcd}
	(B_1\xrightarrow{b_1}B_2)\arrow[d,"\Xi_2"] & & B_1\arrow[r,"b_1"]\arrow[d,"0"]\arrow[rd,"b_1\circ\beta"] & B_2\arrow[d,"\id"] \\
	(0\xrightarrow{0}B_2)\arrow[d,"\Xi_2'"] & = & 0\arrow[r,"0"]\arrow[d,"0"]\arrow[rd,"0"] & B_2\arrow[d,"\id"] \\
	(B_1\xrightarrow{b_1}B_2) & & B_1\arrow[r,"b_1"] & B_2
	\end{tikzcd}
	\]
	Obviously, $\Xi_2\circ\Xi_2'=\id$, and
	\[\Xi_2'\circ\Xi_2-\id=(0,0,\id)\circ(0,b_1\circ\beta,\id)-\id=(-\id,b_1\circ\beta,0)=d(-\beta,0,0)\]
	and also $d(\Xi_2')=0$. Hence by Proposition \ref{prp:closed-inverse}, $\Xi_2$ is a homotopy equivalence with the inverse $\Xi_2'$. This implies $\xi_2=C(\Xi_2)$ is a homotopy equivalence with the inverse $\xi_2'=C(\Xi_2')$.
\end{proof}

\begin{lem}\label{lem:zero-cone}
	Given $(A_1\xrightarrow{\fc(x_1,x_2)}C(A_2\xrightarrow{\id}A_2))\in \Modk(A_2)$, we have
	\[C(A_1\xrightarrow{\fc(x_1,x_2)}C(A_2\xrightarrow{\id}A_2))\simeq A_1[1]\]
	where the equivalence is given by the morphism
	\[\xi_1=\fr(\id,0,0)\]
	with the inverse
	\[\xi_1'=\fc(\id,-x_2,0)\]
	and given $(C(A_1\xrightarrow{\id} A_1)\xrightarrow{\fr(y_1,y_2)}A_2)\in \Modk(A_2)$, we have
	\[C(C(A_1\xrightarrow{\id} A_1)\xrightarrow{\fr(y_1,y_2)}A_2)\simeq A_2\]
	where the equivalence is given by the morphism
	\[\xi_2=\fr(0,y_1,\id)\]
	with the inverse
	\[\xi_2'=\fc(0,0,\id)\ .\]
\end{lem}

\begin{proof}
	We have $C(A_2\xrightarrow{\id}A_2)\simeq 0$ via degree $-1$ morphism
	\[\fm(0,\id,0,0)\colon C(A_2\xrightarrow{\id}A_2)\to C(A_2\xrightarrow{\id}A_2)\]
	such that
	\[d(\fm(0,\id,0,0))=\fb(-d,\id,d)\circ \fm(0,\id,0,0)+\fm(0,\id,0,0)\circ\fb(-d,\id,d)=\id\ .\]
	Similarly, $C(A_1\xrightarrow{\id}A_1)\simeq 0$ via degree $-1$ morphism
	\[\fm(0,\id,0,0)\colon C(A_1\xrightarrow{\id}A_1)\to C(A_1\xrightarrow{\id}A_1)\]
	such that $d(\fm(0,\id,0,0))=\id$. Hence by Lemma \ref{lem:zero} we get the result.
\end{proof}

Applying these lemmas, we get the equivalences below. Their explicit description are crucial for Propostion \ref{prp:coxeter-power}, \ref{prp:monodromy}, and \ref{prp:monodromy-morphism}.

\begin{lem}\label{lem:cone-exact}
	Given $(A_1\xrightarrow{a_1}A_2)\in \Modk(A_2)$, we have
	\[C(A_2\xrightarrow{\fc(0,\id)}C(a_1))\simeq A_1[1]\]
	where the equivalence is given by the morphism
	\[\eta_1=\fr(0,\id,0)\]
	with the inverse
	\[\eta_1'=\fc(-a_1,\id,0)\]
	and
	\[C(C(a_1)\xrightarrow{\fr(\id,0)}A_1[1])\simeq A_2[1]\]
	where the equivalence is given by the morphism
	\[\eta_2=\fr(0,\id,a_1)\]
	with the inverse
	\[\eta_2'=\fc(0,\id,0)\ .\]
	Hence we have the sequence
	\[\ldots\rightarrow A_2[-1]\xrightarrow{\fc(0,\id)} C(a_1)[-1]\xrightarrow{\fr(\id,0)}A_1\xrightarrow{a_1}A_2\xrightarrow{\fc(0,\id)}C(a_1)\xrightarrow{\fr(\id,0)}A_1[1]\rightarrow\ldots\]
	where every successive three elements in this sequence makes an exact triangle.
\end{lem}

\begin{proof}
	Apply \hyperlink{nine}{Nine Lemma} on the following commutative square:
	\[
	\begin{tikzcd}
	0\arrow[r]\arrow[d]\arrow[rd] & A_2\arrow[d,"\id"] \\
	A_1\arrow[r,"a_1"] & A_2
	\end{tikzcd}
	\]
	By Lemma \ref{lem:zero-cone} we get
	\[\eta_1\colon C(A_2\xrightarrow{\fc(0,\id)}C(a_1))\xrightarrow[\sim]{P(3,1,2)}C(A_1\xrightarrow{\fc(0,a_1)}C(A_2\xrightarrow{\id}A_2))\xrightarrow[\sim]{\fr(\id,0,0)} A_1[1]\]
	hence $\eta_1=\fr(\id,0,0)\circ P(3,1,2)=\fr(0,\id,0)$, and
	\[\eta_1'\colon A_1[1]\xrightarrow[\sim]{\fc(\id,-a_1,0)}C(A_1\xrightarrow{\fc(0,a_1)}C(A_2\xrightarrow{\id}A_2))\xrightarrow[\sim]{P(3,1,2)}C(A_2\xrightarrow{\fc(0,\id)}C(a_1))\]
	hence $\eta_1'=P(3,1,2)\circ \fc(\id,-a_1,0)=\fc(-a_1,\id,0)$.
	
	Next, apply \hyperlink{nine}{Nine Lemma} on the following commutative square:
	\[
	\begin{tikzcd}
	A_1\arrow[r,"a_1"]\arrow[d,"\id"]\arrow[rd] & A_2\arrow[d]\\
	A_1\arrow[r] & 0
	\end{tikzcd}
	\]
	By Lemma \ref{lem:zero-cone} we get
	\[\eta_2\colon C(C(a_1)\xrightarrow{\fr(\id,0)}A_1[1])\xrightarrow[\sim]{P(3,2,3)}C(C(A_1\xrightarrow{\id}A_1)\xrightarrow{\fr(a_1,0)}A_2[1])\xrightarrow[\sim]{\fr(0,a_1,\id)}A_2[1]\]
	hence $\eta_2=\fr(0,a_1,\id)\circ P(3,2,3)=\fr(0,\id,a_1)$, and
	\[\eta_2'\colon A_2[1]\xrightarrow[\sim]{\fc(0,0,\id)}C(C(A_1\xrightarrow{\id}A_1)\xrightarrow{\fr(a_1,0)}A_2[1])\xrightarrow[\sim]{P(3,2,3)}C(C(a_1)\xrightarrow{\fr(\id,0)}A_1[1])\]
	hence $\eta_2'=P(3,2,3)\circ \fc(0,0,\id)=\fc(0,\id,0)$.
\end{proof}

\begin{lem}\label{lem:cone-triple}
	Given $(A_1\xrightarrow{a_1}A_2\xrightarrow{a_2}A_3)\in \Modk(A_3)$, we have
	\[C(C(a_1)\xrightarrow{\fd(\id,a_2)}C(a_2\circ a_1))\simeq C(a_2)\]
	where the equivalence is given by the morphism
	\[\eta_3=\mx{0 & \id & a_1 & 0 \\ 0 & 0 & 0 & \id}\]
	with the inverse
	\[\eta_3'=\mx{0 & \id & 0 & 0 \\ 0 & 0 & 0 & \id}\]
	and
	\[C(C(a_2\circ a_1)\xrightarrow{\fd(a_1,\id)}C(a_2))\simeq C(a_1)[1]\]
	where the equivalence is given by the morphism
	\[\eta_4=\mx{\id & 0 & 0 & 0 \\ 0 & 0 & \id & 0}\]
	with the inverse
	\[\eta_4'=\mx{\id & 0 & 0 & 0 \\ 0 & -a_2 & \id & 0}\ .\]
\end{lem}

\begin{proof}
	Apply \hyperlink{nine}{Nine Lemma} on the following commutative square:
	\[
	\begin{tikzcd}
		A_1\arrow[r,"a_1"]\arrow[d,"\id"]\arrow[rd,"0"] & A_2\arrow[d,"a_2"]\\
		A_1\arrow[r,"a_2\circ a_1"] & A_3
	\end{tikzcd}
	\]
	By Lemma \ref{lem:zero-cone} we get
	\[\begin{tikzcd}
		\eta_3\colon C(C(a_1)\xrightarrow{\fd(\id,a_2)} C(a_2\circ a_1))\rar["\sim"',"{P(4,2,3)}"] & C(C(A_1\xrightarrow{\id}A_1)\xrightarrow{\fd(a_1,a_2\circ a_1)}C(a_2))\rar["\sim"',"{\fr(0,\fc(a_1,0),I_2)}"] &[15pt] C(a_2)
	\end{tikzcd}\]
	hence
	\[\eta_3=\fr(0,\fc(a_1,0),I_2)\circ P(4,2,3)=\mx{0 & \id & a_1 & 0 \\ 0 & 0 & 0 & \id}\]
	and
	\[\begin{tikzcd}
		\eta_3'\colon C(a_2)\rar["\sim"',"{\fc(0,0,I_2)}"] & C(C(A_1\xrightarrow{\id}A_1)\xrightarrow{\fd(a_1,a_2\circ a_1)}C(a_2))\rar["\sim"',"{P(4,2,3)}"] & C(C(a_1)\xrightarrow{\fd(\id,a_2)}C(a_2\circ a_1))
	\end{tikzcd}\]
	hence
	\[\eta_3'=P(4,2,3)\circ \fc(0,0,I_2)=\mx{0 & \id & 0 & 0 \\ 0 & 0 & 0 & \id}\ .\]
		
	Next, apply \hyperlink{nine}{Nine Lemma} on the following commutative square:
	\[
	\begin{tikzcd}
		A_1\arrow[r,"a_2\circ a_1"]\arrow[d,"a_1"]\arrow[rd,"0"] & A_3\arrow[d,"\id"]\\
		A_2\arrow[r,"a_2"] & A_3
	\end{tikzcd}
	\]
	By Lemma \ref{lem:zero-cone} we get
	\[\begin{tikzcd}
		\eta_4\colon C(C(a_2\circ a_1)\xrightarrow{\fd(a_1,\id)}C(a_2))\rar["\sim"',"{P(4,2,3)}"] & C(C(a_1)\xrightarrow{\fd(a_2\circ a_1,a_2)}C(A_3\xrightarrow{\id}A_3))\rar["\sim"',"{\fr(I_2,0,0)}"] & C(a_1)[1]
	\end{tikzcd}\]
	hence
	\[\eta_4=\fr(I_2,0,0)\circ P(4,2,3)=\mx{\id & 0 & 0 & 0 \\ 0 & 0 & \id & 0}\]
	and
	\[\begin{tikzcd}
		\eta_4'\colon C(a_1)[1]\rar["\sim"',"{\fc(I_2,-\fr(0,a_2),0)}"] & [25pt] C(C(a_1)\rar["{\fd(a_2\circ a_1,a_2)}"] & [5pt] C(A_3\xrightarrow{\id}A_3))\rar["\sim"',"{P(4,2,3)}"] &[-5pt] C(C(a_2\circ a_1)\rar["{\fd(a_1,\id)}"] &[-5pt] C(a_2))
	\end{tikzcd}\]
	hence
	\[\eta_4'=P(4,2,3)\circ \fc(I_2,-\fr(0,a_2),0)=\mx{\id & 0 & 0 & 0 \\ 0 & -a_2 & \id & 0}\ .\]
\end{proof}

Our next task is to realise $\Modk(A_n)$ geometrically:

\begin{dfn}\label{dfn:vertex}
 	For $n\geq 3$, define \textit{$n$-valent vertex} as
 	\[V_n := \{re^{i\alpha_k}\in\C \vb  r \in [0,1), k=1,\ldots,n\}\]
 	where $\alpha_k\in[0,2\pi)$ are distinct angles with $\alpha_{k_1}>\alpha_{k_2}$ if $k_1>k_2$. Equip
 	\[B^2=\{z\in\C\vb |z|<1\}\]
 	with the Weinstein structure pulled back from $\R^2$ with the Weinstein structure from Example \ref{exm:cotangent-weinstein-nonstandard} via the diffeomorphism
 	\begin{align*}
 		B^2&\to\R^2\simeq\C\\
 		z&\mapsto \frac{|z|}{1-|z|}z
 	\end{align*}
 	and define
 	\[\Lambda:=\{e^{i\alpha_k}\in\C \vb k=1,\ldots,n\}\subset\partial B^2\ .\]
 	Then $V_n$ can be seen as the skeleton of the Weinstein pair $(B^2,\Lambda)$. We label the edges so that $k^{\text{th}}$ edge is given by
	\[e_k:=\{re^{i\alpha_k}\in\C \vb  r \in (0,1)\}\]
	for $k=1,\ldots,n$.
\end{dfn}

Let $\mSh\dd=\mSh\dd_{B^2,\Lambda}$ be the large microlocal stack on $\fX_{B^2,\Lambda}=V_n$. We have the following proposition which gives a geometrical meaning to $A_{\infty}$-modules over $A_n$-quivers:
    
\begin{prp}\label{prp:msh-vertex}
	For $n\geq 3$, we have $\mSh\dd(V_n)\simeq \Modk(A_{n-1})$ and the inclusion of $k^{\text{th}}$ edge $i_k\colon e_k\rightarrow V_n$ induces the restriction given by the dg functor $j_k:=\mSh\dd(i_k)$ as
    \begin{align*}
    	j_k\colon \Modk(A_{n-1})&\rightarrow\Modk\\
    	A=(A_1\xrightarrow{a_1}\ldots\xrightarrow{a_{n-2}}A_{n-1})&\mapsto{\begin{cases}
    		A_1 &\text{ if } k=1\\
    		C(a_{k-1}) &\text{ if } 2\geq k\geq n-1\\
    		A_{n-1}[1] &\text{ if } k=n\ .
    	\end{cases}}
    \end{align*}
    Same statements hold for $\mSh(V_n)$ by replacing each $\Modk$ with $\Perfk$.
\end{prp}
    
\begin{proof}
	Assume all edges are with the angles $\alpha_k\in[0,\pi/2)$. The calculation does not depend on the values of $\alpha_k$ up to quasi-equivalence, see Remark \ref{rmk:vertex-reflection}. Let $W'=B^2$ be the Weinstein sector corresponding to $(B^2,\Lambda)$. Define the exact symplectomorphism
	\begin{align*}
		\varphi\colon W'=B^2&\to T^*\R\simeq\C\\
		z&\mapsto \frac{|z|}{1-|z|}z\ .
	\end{align*}
	Obviously, $\varphi$ is just scaling. The transferred skeleton
	\[\varphi(V_n)=\{(0,0)\}\sqcup\bigsqcup_{k=1}^n \varphi(e_k)\subset T^*\R\]
    is exact singular Lagrangian with the obvious strata, via the function $f\colon\varphi(V_n)\to\R$ defined as follows: For any $k$, $\varphi(e_k)$ is a smooth Lagrangian given by the embedding
    \begin{align*}
    	i_k\colon(0,\infty)&\hookrightarrow T^*\R\\
    	r&\mapsto(ra_k,rb_k)
    \end{align*}
    where $a_k:=\cos(\alpha_k)$ and $b_k:=\sin(\alpha_k)$. Then the Liouville form $-pdx$ on $T^*\R=\{(x,p)\}$ restricts on $\varphi(e_k)$ as
    \[i_k^*(-pdx)=-a_k b_k r dr\ .\]
    We define $f_k\colon\varphi(e_k)\to\R$ as
   	\[f_k(r)=-\frac{a_kb_k}{2}r^2\]
   	for $r\in(0,\infty)$, which we can write as
   	\[f_k(x,p)=-\frac{1}{2}c_k x^2\]
	for $(x,p)\in\varphi(e_k)$ and $c_k:=\tan(\alpha_k)$. Clearly $df_k=i_k^*(-pdx)$. Finally, we define the continuous function $f\colon \varphi(V_n)\to\R$ by
	\[f(x,p)=\begin{cases}
		f_k(x,p) &\text{if }(x,p)\in\varphi(e_k)\\
		0 &\text{if }(x,p)=(0,0)
	\end{cases}\]
	and it is a primitive for the Liouville form on $T^*\R$.
        
    By the procedure described in Section \ref{sec:microlocal}, using $f$ we can lift $\varphi(V_n)\subset T^*\R$ to the Legendrian $\Lambda_{V_n}=\{0,0;[0,-1]\}\sqcup\bigsqcup_{k=1}^n \Lambda_k\subset T^{\infty}\R^2$, where
    \[\Lambda_k:=\{(x,-f_k(x,p);[p,-1])\vb (x,p)\in \varphi(e_k)\}=\left\{\left.\left(x,\frac{1}{2}c_kx^2;\left[c_k x,-1\right]\right)\,\right |\, x\in(0,\infty)\right\}\ .\]
    So, the front projection of $\Lambda_{V_n}$ in $\R^2$ is the union of $n$ half parabolas in the first quadrant, meeting at the origin, and the fibre direction at each point is perpendicularly downwards. See Figure \ref{fig:vertex-front} for the case $n=4$.
        
	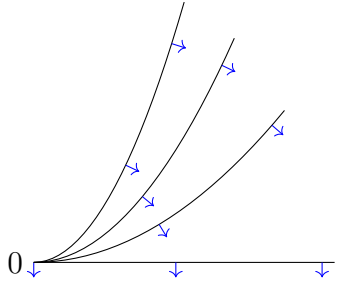
\begin{figure}[h]
    	\centering
            
     	\begin{tikzpicture}
     		\newcommand{\pta}[1]{(2-(#1+10)/90)}
     		\newcommand{\ptb}[1]{(4-(#1+5)/30)}
     	
     		\newcommand{\edge}[1]{
     			\draw[domain=0:{4-((#1)/30)},variable=\x] plot ({\x},{(tan(#1)/2)*(\x^2)});
     			\draw[->,blue] ({\pta{#1}},{(tan(#1)/2)*(\pta{#1}^2)}) -- ({\pta{#1}+(0.2*(1/sqrt(1+(tan(#1)*\pta{#1})^2)))*tan(#1)*\pta{#1}},{(tan(#1)/2)*(\pta{#1}^2)-(0.2*(1/sqrt(1+(tan(#1)*\pta{#1})^2)))});
     			\draw[->,blue] ({\ptb{#1}},{(tan(#1)/2)*(\ptb{#1}^2)}) -- ({\ptb{#1}+(0.2*(1/sqrt(1+(tan(#1)*\ptb{#1})^2)))*tan(#1)*\ptb{#1}},{(tan(#1)/2)*(\ptb{#1}^2)-(0.2*(1/sqrt(1+(tan(#1)*\ptb{#1})^2)))})
     		}
     		
     		\node[left] at (0,0) {$0$};
     		\draw[->,blue] (0,0) -- (0,-0.2);
     		\edge{0};
     		\edge{20};
     		\edge{40};
     		\edge{60};		
		\end{tikzpicture}
            
        \caption{Front projection of $\Lambda_{V_4}$ in $\R^2$ with the fibre directions}
        \label{fig:vertex-front}
  	\end{figure}
  
	By Definition \ref{dfn:microlocal} and Proposition \ref{prp:ks-explicit}, we have
	\[\mSh\dd(V_n)=\mSh\dd_{T^*\R^2,\Lambda_{V_n}}(\R_{>0}\Lambda_{V_n})\simeq\Sh\dd_{\Lambda_{V_n}}(\R^2)_0\ .\]
	Note that we $\Sh\dd_{\Lambda_{V_n}}(\R^2)_0\simeq\Sh\dd_{\Lambda_{V_n}}(\R^2)/\Loc\dd(\R^2)$, and by Corollary \ref{cor:sing-constructible}, $\Sh\dd_{\Lambda_{V_n}}(\R^2)$ can be described as the full dg subcategory of $\Sh\dd_{\cS}(\R^2)$ where $\cS$ is the Whitney stratification of $\R^2$ consisting of the origin, the edges, and the regions outside the edges. We can see $\cS$ as an $A_{\infty}$-category and since $\cS$ is a regular cell complex, by Proposition \ref{prp:combinatorics-constructible} we have
	\[\Sh\dd_{\cS}(\R^2)\simeq\Modk(\cS)\]
	which associates an element of $\Modk$ to each stratum, and a morphism from a stratum to another one which contains the former in its closure. See the Figure \ref{fig:vertex-sheaves} for $n=4$, where $V\in\Modk$ is associated to the origin, $E_k\in\Modk$ to the $(k+1)^{\text{st}}$ edge, and $A_k\in\Modk$ to the region below $(k+1)^{\text{st}}$ edge for $k=0,\ldots,n-1$. Note that the maps $V\to A_k$ and higher maps are not shown in the figure.
	
	\begin{figure}[h]
		\centering
		
		\begin{tikzpicture}
			\node(V) at (0,0) {$V$};
			\node(A0) at (-2,-1) {$A_0$};
			\node(E0) at (2,0) {$E_0$};
			\node(A1) at (3.8,0.5) {$A_1$};
			\node(E1) at (1.8,0.5) {$E_1$};
			\node(A2) at (3.4,1.5) {$A_2$};
			\node(E2) at (1.4,1) {$E_2$};
			\node(A3) at (3,2.5) {$A_3$};
			\node(E3) at (1,1.5) {$E_3$};
			
			\draw[->] (V) -- node[above,yshift=-0.17cm]{$\sim$} (E0);
			\draw[->] (V) -- (E1);
			\draw[->] (V) -- (E2);
			\draw[->] (V) -- (E3);
			\draw[->] (E0) -- node[below,rotate=10]{$\sim$} (A0);
			\draw[->] (E0) -- (A1);
			\draw[->] (E1) -- node[above,yshift=-0.1cm]{$\sim$} (A1);
			\draw[->] (E1) -- (A2);
			\draw[->] (E2) -- node[above,rotate=10,yshift=-0.1cm]{$\sim$} (A2);
			\draw[->] (E2) -- (A3);
			\draw[->] (E3) -- node[above,rotate=20,yshift=-0.1cm]{$\sim$} (A3) ;
			\draw[->] (E3) -- (A0);	
		\end{tikzpicture}
		
		\caption{An object in $\Sh\dd_{\Lambda_{V_4}}(\R^2)$}
		\label{fig:vertex-sheaves}
	\end{figure}
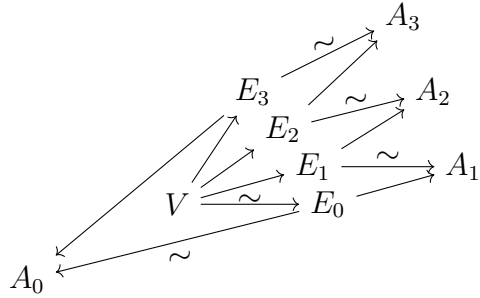
	
	Then $\Sh\dd_{\Lambda_{V_n}}(\R^2)$ can be found by imposing the microlocal conditions in $\Modk(\cS)$ coming from fibre directions: It is easy to see that they require each morphism from an edge to the region below to be quasi-isomorphism, i.e. $E_k\to A_k$ is a quasi-isomorphism for any $k$. At the vertex, it is a bit more complicated. For any $\cF\in\Sh\dd_{\Lambda_{V_n}}(\R^2)$, $\sing(\cF)$ can only have the downward fibre direction at the vertex. This means that for any open set $U\subset\R^2$ containing the origin, the restriction
	\[R\Gamma(U;\cF)\to R\Gamma(\{(x,p)\in U \vb f(x,p)< 0\};\cF)\]
	is a quasi-isomorphism for any smooth function $f\colon \R^2\to\R$ with $f(0,0)=0$ and $df_{(0,0)}\neq -c dp$ for any $c>0$. One can check that it is enough to check this for the function $f(x,p)=-x$ and $U=\R^2$. Then we have the quasi-isomorphism
	\[R\Gamma(\R^2;\cF)\to R\Gamma(\{(x,p)\in \R^2 \vb x>0\};\cF)\]
	where $R\Gamma(\R^2;\cF)\simeq V$ and
	\[R\Gamma(\{(x,p)\in \R^2 \vb x>0\};\cF)\simeq C(E_0\oplus E_1\oplus\ldots\oplus E_{n-1}\to A_1\oplus \ldots\oplus A_{n-1})[-1]\simeq E_0\]
	since $E_k\to A_k$ are quasi-isomorphisms. Hence we get that $V\to E_0$ is a quasi-isomorphism. See the Figure \ref{fig:vertex-sheaves} for an object of $\Sh\dd_{\Lambda_{V_n}}(\R^2)$ for $n=4$.
	
	Finally, it is easy to see that we can replace every quasi-isomorphism by the identity, so we can set $E_k=A_k$ for any $k$ and $V=A_0$. The dg category $\Sh\dd_{\Lambda_{V_n}}(\R^2)_0\subset\Sh\dd_{\Lambda_{V_n}}(\R^2)$ is obtained by setting the stalk to zero at a specified point as remarked in Definition \ref{dfn:sheaf-category-sing}. We choose the origin, hence $V=A_0$ becomes zero. This also implies that all homotopy maps can be set to zero, and we can get rid of the maps $V\to A_k$ by the commutativity of the diagram up to homotopy. This shows that the objects of $\mSh\dd(V_n)\simeq\Sh\dd_{\Lambda_{V_n}}(\R^2)_0$ are given by
	\[A_1\xrightarrow{a_1}A_2\xrightarrow{a_2}\ldots\xrightarrow{a_{n-2}}A_{n-1}\]
	and since it is full dg subcategory of $\Modk(\cS)$, we get $\mSh\dd(V_n)\simeq\Modk(A_{n-1})$.
	
	To find the restriction functor $j_k:=\mSh\dd(i_k)\colon \Modk(A_{n-1})\rightarrow\Modk$ for any $k=1,\ldots,n$, remember that $\cF$ and $j_k(\cF)$ have the same microstalks. Clearly, the microstalk at a point $((x_1,x_2),(p_1,p_2))\in\R_{>0}e_k\subset T^*\R^2$ is given by
	\[\cF_{(x_1,x_2),(p_1,p_2)}\simeq C(E_{k-1}\to A_k)\simeq C(A_{k-1}\xrightarrow{a_{k-1}} A_k)\]
	after setting $A_n:=A_0\simeq 0$, $a_0:=0$, and $a_{n-1}:=0$. Since $j_k(\cF)\in\Modk$, this microstalk is enough to determine it, hence
	\[j_k(\cF)\simeq C(a_{k-1})\ .\]	
\end{proof}

\begin{rmk}\label{rmk:vertex-reflection}
	If $\alpha_k$'s are arbitrarily chosen such that $\alpha_k\neq\pi/2\text{ or }3\pi/2$ for any $k$ (if not, rotate $B^2$), then a similar proof applies, but we get the objects of the form
	\[A_1\to\ldots \to A_{i-1}\to A_i\leftarrow A_{i+1}\leftarrow\ldots\leftarrow A_{n-1}\]
	where $i=n-1$ if all $\alpha_k\in[0,\pi/2)\cup(3\pi/2,2\pi)$, otherwise $i$ is the number of edges with $\alpha_k\in[0,\pi/2)\cup(3\pi/2,2\pi)$. However, all these representations are quasi-equivalent using the reflection functors, see \cite{arboreal}.
\end{rmk}

\begin{rmk}\label{rmk:pants-torus}
	Above remark shows that $\mSh\dd(V_n)$ is only depends on the topology of $V_n$, whereas the restriction maps from $\mSh\dd(V_n)$ to the edges are affected by the cyclic orientation of the edges of $V_n$. This causes some topologically same skeleta to have different microlocal sheaves as follows: Consider a pair of pants and a punctured torus. Topologically they have the same skeleta, as shown in Figure \ref{fig:skeleton-figure-eight}. However, for the pair of pants, we take a four-valent vertex, connect first and second edge, and then third and fourth edge. Whereas for the punctured torus, we take a four-valent vertex, connect first and third edge, and then second and fourth edge. Since the restriction maps from the vertex depend on the order of the edges, the two skeleta will have different gluing diagrams. One can easily confirm that they have different microlocal sheaves. In general, for 1-dimensional skeleta, the microlocal sheaves depend only on the ribbon graph structure of the skeleta.
\end{rmk}

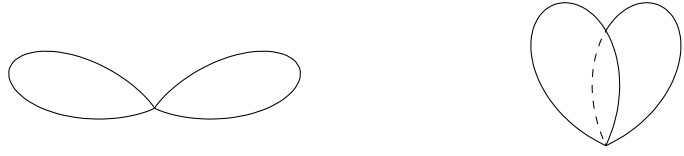
\begin{figure}[h]
	\centering
	
	\begin{tikzpicture}
		\newcommand{\petal}[5][]{
			\draw[domain=0:90,smooth,xshift=#3cm,yshift=#5cm,rotate=#4,#1] plot ({\x}:{2*sin(2*\x)^#2})
		}
		
		\petal{2}{-3}{-30}{0.5};
		\petal{2}{-3}{120}{0.5};
		
		\petal{1}{3}{65}{0};
		
		\begin{scope}
			\clip (3,0) rectangle (4.2,2);
			\petal{1}{3}{25}{0};
		\end{scope}
		
		\begin{scope}
			\clip (3,0) rectangle (2,2);
			\petal[dashed]{1}{3}{25}{0};
		\end{scope}
	
	\end{tikzpicture}
	
	\caption{Skeleta of a pair of pants (left) and a punctured torus (right)}
	\label{fig:skeleton-figure-eight}
\end{figure}

We have a different representation of $\Modk(A_n)$ where the morphisms are simplified, which we will use throughout rest of the thesis:

\begin{prp}\label{prp:quiver-algebra}
	Let $\cC_{n-1}$ be a dg category whose objects are
	\[A_1\xrightarrow{a_1}A_2\xrightarrow{a_2}\ldots\xrightarrow{a_{n-2}}A_{n-1}\]
	where $A_i\in\Modk$, and $a_i\in\Hom_{\Modk}^0(A_i,A_{i+1})$ with $da_i=0$. We also represent an object by $A=(A_1,a_1,A_2,a_2,\ldots,a_{n-2},A_{n-1})=(A_i,a_i)$. A morphism in $\Hom_{\cC_{n-1}}^k(A,B)$ is given by
	\[
	\begin{tikzcd}
	A_1\arrow[r,"a_1"]\arrow[d,"f_1"]\arrow[rd,"h_1"] & A_2\arrow[r,"a_2"]\arrow[d,"f_2"]\arrow[rd,"h_2"] & A_3\arrow[r,"a_3"]\arrow[d,"f_3"] & \ldots\arrow[r,"a_{n-3}"] & A_{n-2}\arrow[d,"f_{n-2}"]\arrow[r,"a_{n-2}"]\arrow[rd,"h_{n-2}"] & A_{n-1}\arrow[d,"f_{n-1}"]\\
	B_1 \arrow[r,"b_1"] & B_2 \arrow[r,"b_2"] & B_3\arrow[r,"b_3"] & \ldots\arrow[r,"b_{n-3}"] & B_{n-2}\arrow[r,"b_{n-2}"] & B_{n-1}
	\end{tikzcd}
	\]
	where $f_i\in\Hom_{\Modk}^k(A_i,B_i)$ and $h_i\in\Hom_{\Modk}^{k-1}(A_i,B_{i+1})$. We also represent a morphism by $f=(f_1,h_1,f_2,\ldots,h_{n-2},f_{n-1})=(f_i,h_i)$. It has the differential
	\[d(f_i,h_i)=(df_i,dh_i+(-1)^k(b_i\circ f_i-f_{i+1}\circ a_i))\ ,\]
	the composition
	\[(f'_i,h'_i)\circ(f_i,h_i)=(f'_i\circ f_i,f'_{i+1}\circ h_i+(-1)^k h'_i\circ f_i)\ ,\]
	and the identity is $(\id,0)$. Then $\cC_{n-1}$ is quasi-equivalent to $\Modk(A_{n-1})$. Moreover, $(f_i,h_i)\in\Hom^0_{\cC_{n-1}}(A,B)$ is a homotopy equivalence if and only if $f_i$ is a homotopy equivalence and $dh_i=f_{i+1}\circ a_i-b_i\circ f_i$ for all $i$.
\end{prp}

\begin{proof}
	We can prove this proposition by explicitly constructing an $A_{\infty}$-quasi-equivalence between $\cC_{n-1}$ and $\Modk(A_{n-1})$. However, we will follow a different approach: By Proposition \ref{prp:msh-vertex} we know $\mSh\dd(V_n)\simeq\Modk(A_{n-1})$. There is another way to calculate $\mSh\dd(V_n)$ using arboreal singularities as remarked in Section \ref{sec:microlocal}. We can expand $\Lambda_{V_n}$ to get $\Lambda_{V'_n}$ with arboreal singularities as shown \cite{nonchar}, then by Theorem \ref{thm:arboreal} we have
	\[\mSh\dd(V_n)\simeq\mSh\dd_{T^*\R^2,\Lambda_{V_n}}(\R_{>0}\Lambda_{V_n})\simeq\mSh\dd_{T^*\R^2,\Lambda_{V'_n}}(\R_{>0}\Lambda_{V'_n})\simeq\mSh\dd(V'_n)\]
	where $V'_n$ is a skeleton associated to $\Lambda_{V'_n}$. See Figure \ref{fig:vertex-arboreal} for this noncharacteristic deformation in the case of $n=4$. Note that although it is not shown in the figure, the edges attached to the arc are deformed near the singular points so that the fibre direction of the arc and the edge at any singular point match.
	
	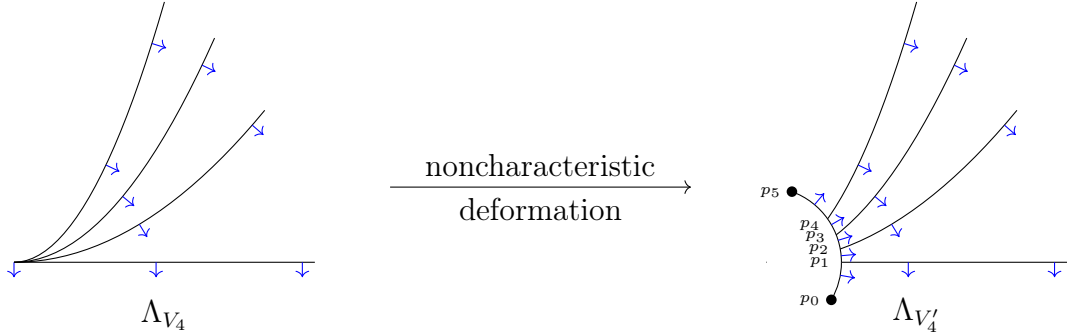
\begin{figure}[h]
		\centering
		
		\begin{tikzpicture}
			\newcommand{\pta}[1]{(2-(#1+10)/90)}
			\newcommand{\ptb}[1]{(4-(#1+5)/30)}
			
			\newcommand{\edge}[2]{
				\draw[domain=0:{4-((#1)/30)},variable=\x] plot ({\x+#2},{(tan(#1)/2)*(\x^2)});
				\draw[->,blue] ({\pta{#1}+#2},{(tan(#1)/2)*(\pta{#1}^2)}) -- ({\pta{#1}+#2+(0.2*(1/sqrt(1+(tan(#1)*\pta{#1})^2)))*tan(#1)*\pta{#1}},{(tan(#1)/2)*(\pta{#1}^2)-(0.2*(1/sqrt(1+(tan(#1)*\pta{#1})^2)))});
				\draw[->,blue] ({\ptb{#1}+#2},{(tan(#1)/2)*(\ptb{#1}^2)}) -- ({\ptb{#1}+#2+(0.2*(1/sqrt(1+(tan(#1)*\ptb{#1})^2)))*tan(#1)*\ptb{#1}},{(tan(#1)/2)*(\ptb{#1}^2)-(0.2*(1/sqrt(1+(tan(#1)*\ptb{#1})^2)))})
			}
		
			\newcommand{\arrow}[2][0]{
				\draw[->,blue] ({#1+cos(#2)},{sin(#2)}) -- ({#1+(1+0.2)*cos(#2)},{(1+0.2)*sin(#2)})
			}
			
			\draw[->,blue] (-5,0) -- (-5,-0.2);
			\edge{0}{-5};
			\edge{20}{-5};
			\edge{40}{-5};
			\edge{60}{-5};

			\begin{scope}[even odd rule]
				\clip (5,0) circle (1) (5,-0.5) rectangle (9,4);
				\edge{0}{5};
				\edge{20}{5};
				\edge{40}{5};
				\edge{60}{5};
			\end{scope}	
			
			\draw (6,0) arc (0:70:1cm);
			\draw (6,0) arc (0:-30:1cm);
			\fill ({5+cos(70)},{sin(70)}) circle (2pt);
			\fill ({5+cos(-30)},{sin(-30)}) circle (2pt);
			\arrow[5]{50};
			\arrow[5]{30};
			\arrow[5]{17};
			\arrow[5]{5};
			\arrow[5]{-10};
			
			\node[left] at ({5+cos(-30)},{sin(-30)}){\tiny $p_0$};
			\node[left] at ({5+cos(0)},{sin(0)}){\tiny $p_1$};
			\node[left] at ({5+cos(10)},{sin(10)}){\tiny $p_2$};
			\node[left] at ({5+cos(20)},{sin(20)}){\tiny $p_3$};
			\node[left] at ({5+cos(30)},{sin(30)}){\tiny $p_4$};
			\node[left] at ({5+cos(70)},{sin(70)}){\tiny $p_5$};
			
			\node at (-3,-0.7){$\Lambda_{V_4}$};
			\node at (7,-0.7){$\Lambda_{V'_4}$};
			
			\draw[->] (0,1) -- node[above]{noncharacteristic} node[below]{deformation} (4,1);
		\end{tikzpicture}
		
		\caption{Front projection of the noncharacteristic deformation of $\Lambda_{V_4}$ in $\R^2$}
		\label{fig:vertex-arboreal}
	\end{figure}

	To find $\mSh\dd(V'_n)$, observe that $V'_n$ has the same topology as the front projection of $\Lambda_{V'_n}$ in $\R^2$. The singularities are the ends of the arcs $p_0,p_{n+1}$, and the points $p_1,\ldots,p_n$ where the arc and the edges meet, see Figure \ref{fig:vertex-arboreal}. They are arboreal singularities, see \cite{arboreal}. Since $\mSh\dd$ is a sheaf of dg categories, we can calculate it locally around the singularities $p_i$ and then glue the sections to get the global section. Choose an open covering $\{U_i\}_{i=0}^{n+1}$ of $V'_n$ such that $U_i$ is a sufficiently small neighbourhood of $p_i$. Then we have
	\[\mSh\dd(V'_n)\simeq\holim\left(\begin{tikzcd}[column sep=-25pt]
		\mSh\dd(U_0)\ar[rd] &  & \mSh\dd(U_1)\ar[ld]\ar[rd] & & \mSh\dd(U_2)\ar[ld]\ar[rd] &[25pt]\cdots&[25pt] \mSh\dd(U_{n+1})\ar[ld]\\
		& \mSh\dd(U_0\cap U_1) & &\mSh\dd(U_1\cap U_2) & &\cdots&
	\end{tikzcd}\right)\ .\]
	The subsets $U_i\cap U_{i+1}$ are smooth Lagrangians, hence we have $\mSh\dd(U_i\cap U_{i+1})\simeq\Modk$ for any $i=0,1,\ldots n$. The dg categories $\mSh\dd(U_i)$ associated to arboreal singularities $p_i$ are described in \cite{arboreal}. However, it is already easy to see that $\mSh\dd(U_0)\simeq\mSh\dd(U_{n+1})\simeq 0$, and Proposition \ref{prp:msh-vertex} shows that $\mSh\dd(U_i)\simeq\Modk(A_2)$ for $i=1,\ldots,n$ since $U_i$ is a trivalent vertex. We can label the edges of $U_i$ in such a way that the first edge is $U_{i-1}\cap U_i$ and the third edge is $U_i\cap U_{i+1}$. Then we have
	\[\mSh\dd(V'_n)\simeq\holim\left(\begin{tikzcd}[column sep=-10pt]
		0 \ar[rd] &[15pt] & \Modk(A_2)\ar[ld,"j_1"']\ar[rd,"j_3"] & & \Modk(A_2)\ar[ld,"j_1"']\ar[rd,"j_3"] &[10pt]\cdots&[25pt] 0\ar[ld]\\
		& \Modk & &\Modk & &\cdots&
	\end{tikzcd}\right)\]
	where $j_i$ is the restriction map for $i^{\text{th}}$ edge which is given by Proposition \ref{prp:msh-vertex}, and they are clearly fibrations. This shows homotopy limit becomes ordinary limit. Hence an object $A\in\mSh\dd(V'_n)$ is given by
	\[A=\left(\begin{tikzcd}[column sep=-20pt]
		0 \ar[rd] &[15pt] &[10pt] (0\to A_1)\ar[ld,"j_1"']\ar[rd,"j_3"] & & (A_1[1]\xrightarrow{a_1[1]} A_2[1])\ar[ld,"j_1"']\ar[rd,"j_3"]& &(A_2[2]\xrightarrow{a_2[2]} A_3[2])\ar[ld,"j_1"']\ar[rd,"j_3"] &[15pt]\cdots&[15pt] (A_{n-1}[n-1]\to 0)\ar[ld,"j_1"']\ar[rd,"j_3"] &[10pt] &[15pt] 0\ar[ld]\\
		& 0 & &A_1[1] & & A_2[2]& &\cdots& & 0 &
	\end{tikzcd}\right)\]
	where $(A_i\xrightarrow{a_i} A_{i+1})\in\Modk(A_2)$, and a morphism $f\in\Hom^k_{\mSh\dd(V'_n)}(A,B)$ is given by
	\[f=\left(\begin{tikzcd}[column sep=-15pt]
		0 \ar[rd] &[10pt] &[10pt] (0,0,f_1)\ar[ld,"j_1"']\ar[rd,"j_3"] & & (f_1[1],h_1[1],f_2[1])\ar[ld,"j_1"']\ar[rd,"j_3"]& &(f_2[2],h_2[2],f_3[2])\ar[ld,"j_1"']\ar[rd,"j_3"] &[10pt]\cdots&[10pt] (f_{n-1}[n-1],0,0)\ar[ld,"j_1"']\ar[rd,"j_3"] &[10pt] &[10pt] 0\ar[ld]\\
		& 0 & &f_1[1] & & f_2[2]& &\cdots& & 0 &
	\end{tikzcd}\right)\]
	where $(f_i,h_i,f_{i+1})\in\Hom^k_{\Modk(A_2)}((A_i,a_i,A_{i+1}),(B_i,b_i,B_{i+1}))$. In short, we can write
	\[A=A_1\xrightarrow{a_1}A_2\xrightarrow{a_2}\ldots\xrightarrow{a_{n-2}}A_{n-1}\]
	and
	\[f=(f_1,h_1,f_2,\ldots,h_{n-2},f_{n-1})\ .\]
	The differential, composition rule, identity, and homotopy equivalences can be determined easily by comparing with $\Modk(A_2)$. Since $\mSh\dd(V'_n)\simeq\mSh\dd(V_n)\simeq\Modk(A_{n-1})$, this proves the proposition.
\end{proof}

From now on, by $\Modk(A_{n-1})$ we will mean the dg category $\cC_{n-1}$ in Proposition \ref{prp:quiver-algebra}. We will sometimes consider its objects $A,B\in\Modk(A_{n-1})$ and the morphism $f\in\Hom^k_{\Modk(A_{n-1})}(A,B)$ with zero parts as
\[\begin{tikzcd}
A\dar["f"']\\
B
\end{tikzcd}
=
\begin{tikzcd}
0=A_0\rar["0=a_0"]\dar["0=f_0"']\drar["0=h_0"] & A_1\rar["a_1"]\dar["f_1"]\drar["h_1"] & A_2\rar["a_2"]\dar["f_2"] & \ldots\rar["a_{n-2}"] & A_{n-1}\rar["a_{n-1}=0"]\dar["f_{n-1}"]\drar["h_{n-1}=0"] & A_n=0\dar["f_n=0"]\\
0=B_0\rar["0=b_0"] & B_1\rar["b_1"] & B_2\rar["b_2"] & \ldots\rar["b_{n-2}"] & B_{n-1}\rar["b_{n-1}=0"] & B_n=0
\end{tikzcd}\]
to define functors on $\Modk(A_{n-1})$ compactly. With this presentation, the restriction functor $j_i\colon \Modk(A_{n-1})\to\Modk$ can be given as
\begin{align*}
	j_i(A)&=C(a_{i-1})\\
	j_i(f)&=\fb(f_{i-1},h_{i-1},f_i)
\end{align*}
for $i=1,\ldots,n$. Note that $j_1(A)=A_1$, $j_1(f)=f_1$, $j_n(A)=A_{n-1}[1]$, and $j_n(f)=f_{n-1}$.

Also, we have restriction functors from $\Modk(A_{n-1})$ to the $\Modk(A_2)$:

\begin{dfn}\label{dfn:components}
	The dg functor $u_i\colon\Modk(A_{n-1})\to\Modk(A_2)$ for $1\leq i\leq n-2$ is defined as
	\[u_i(A_1\xrightarrow{a_1}A_2\xrightarrow{a_2}\ldots\xrightarrow{a_{n-2}}A_{n-1})=(A_i\xrightarrow{a_i}A_{i+1})\]
	on objects, and
	\[u_i(f_1,h_1,f_2,\ldots,h_{n-2},f_{n-1})=(f_i,h_i,f_{i+1})\]
	on morphisms. We also define the dg functor $u_{i,j}\colon \Modk(A_{n-1})\to\Modk(A_{j-i+1})$ as
	\[u_{i,j}(A_1\xrightarrow{a_1}A_2\xrightarrow{a_2}\ldots\xrightarrow{a_{n-2}}A_{n-1})=(A_i\xrightarrow{a_i}A_{i+1}\xrightarrow{a_{i+1}}\ldots\xrightarrow{a_{j-1}}A_j)\]
	on objects, and
	\[u_{i,j}(f_1,h_1,f_2,\ldots,h_{n-2},f_{n-1})=(f_i,h_i,f_{i+1},\ldots,h_{j-1},f_j)\]
	on morphisms for $1\leq i< j\leq n-1$.
\end{dfn}

\begin{rmk}
	The functor $u_i$ can be considered as the geometric restriction
	\[\mSh\dd(V_n')\to\mSh\dd(U_{i+1})\]
	in Proposition \ref{prp:quiver-algebra}. Also, since $\mSh\dd(V'_n)$ is glued by $\mSh\dd(U_{i+1})$ for $i=1,\ldots,n-2$, it is enough to describe a given dg functor
	\[F\colon\cD\to\Modk(A_{n-1})\]
	by its components
	\[u_i\circ F\colon\cD\to\Modk(A_2)\]
	for $i=1,\ldots,n-2$, where $\cD$ is a dg category. Note that we can use the functor $u_{i,j}$ instead of using $u_i,u_{i+1},\ldots,u_{j-1}$ altogether. We call $u_i\circ F$ \textit{the $i\th$ component of $F$}, and $u_{ij}$ the $(i,j)\th$ component of $F$. These facts and notations will be also useful for typographical reasons.
\end{rmk}

Now we are ready to study automorphisms of $\Modk(A_n)$, and their geometrical meaning:

\begin{dfn}\label{dfn:coxeter}
	The \textit{(derived) Coxeter functor} $\cox_n\colon \Modk(A_{n-1})\to \Modk(A_{n-1})$ defined as the composition of the derived (source) reflection functors (see \cite{gelfand-manin}) at each object of the $A_{n-1}$-quiver. Explicitly, it can be defined as the dg functor given by
	\begin{align*}
		\begin{tikzcd}[ampersand replacement=\&]
			\cox_n(A)\dar["\cox_n(f)"',"\hspace{2em} ="] \& [-10pt]
			C(a_1)\rar["{\fd(\id,a_2)}"]\dar["{\fb(f_1,h_1,f_2)}"]\drar["{\fd(0,h_2)}"] \& [50pt]
			C(a_2\circ a_1)\rar["{\fd(\id,a_3)}"]\dar["{\fb(f_1,\chi(h)_1^2,f_3)}"]\drar["{\fd(0,h_3)}"] \& [50pt]
			C(a_3\circ a_2\circ a_1)\rar["{\fd(\id,a_4)}"]\dar["{\fb(f_1,\chi(h)_1^3,f_4)}"] \&
			\cdots
			\\
			\cox_n(B) \&
			C(b_1)\rar["{\fd(\id,b_2)}"] \&
			C(b_2\circ b_1)\rar["{\fd(\id,b_3)}"] \&
			C(b_3\circ b_2\circ b_1)\rar["{\fd(\id,b_4)}"] \& 
			\cdots
		\end{tikzcd}\hspace{-34em}&
		\\ 
		&\begin{tikzcd}[ampersand replacement=\&]
			\cdots \rar["{\fd(\id,a_{n-3})}"] \&
			C(a_{n-3}\circ\ldots\circ a_1)\rar["{\fd(\id,a_{n-2})}"]\dar["{\fb(f_1,\chi(h)_1^{n-3},f_{n-2})}"']\drar["{\fd(0,h_{n-2})}"] \&
			C(a_{n-2}\circ\ldots\circ a_1)\rar["{\fr(\id,0)}"]\dar["{\fb(f_1,\chi(h)_1^{n-2},f_{n-1})}"]\drar[gray,dashed,"0"] \& [60pt]
			A_1[1]\dar["f_1"]
			\\
			\cdots \rar["{\fd(\id,b_{n-3})}"] \&
			C(b_{n-3}\circ\ldots\circ a_1)\rar["{\fd(\id,b_{n-2})}"] \&
			C(b_{n-2}\circ\ldots\circ b_1)\rar["{\fr(\id,0)}"] \&
			B_1[1]
		\end{tikzcd}
	\end{align*}
    where $\chi(h)_{i_1}^{i_2}:=\sum_{j=i_1}^{i_2} b_{i_2}\circ\ldots b_{j+1}\circ h_j\circ a_{j-1}\circ \dots\circ a_{i_1}$. Note that $\chi(h)_1^1=h_1$.
\end{dfn}

Before understanding the powers of the Coxeter functor, we present the following useful lemma for $\Modk(A_{n-1})$:

\begin{lem}\label{lem:adjusting-quiver}
	Let $A=(A_1\xrightarrow{a_1}A_2\xrightarrow{a_2}\ldots\xrightarrow{a_{n-2}}A_{n-1})\in\Modk(A_{n-1})$, and assume $m_i\colon A_i\to B_i$ is a homotopy equivalence with the inverse $m_i'$ in homotopy for $i=1,\dots,n-1$. Then $A$ is homotopy equivalent to
	\[B=(B_1\xrightarrow{m_2\circ a_1\circ m_1'}B_2\xrightarrow{m_3\circ a_2\circ m_2'}\ldots\xrightarrow{m_{n-1}\circ a_{n-2}\circ m_{n-2}'}B_{n-1})\]
	via the morphism $m\colon A\to B$ whose component $u_i(m)$ is given by
	\[\begin{tikzcd}[column sep=60pt]
		A_i\rar["a_i"]\dar["m_i"]\drar["m_{i+1}\circ a_i\circ \xi_i" xshift=-0.2cm] & A_{i+1}\dar["m_{i+1}"]\\
		B_i\rar["m_{i+1}\circ a_i\circ m_i'" yshift=-0.05cm] & B_{i+1}
	\end{tikzcd}\]
	where $d\xi_i=\id-m_i'\circ m_i$ for $i=1,\ldots,n-2$.
\end{lem}

\begin{proof}
	Clearly $dm=0$, and since $m_i$ is a homotopy equivalence for every $i$, by Proposition \ref{prp:quiver-algebra} $m$ is a homotopy equivalence.
\end{proof}

\begin{dfn}
	For $1\leq k\leq n-1$, we define the dg functor
	\[\cox_{n,k}\colon \Modk(A_{n-1})\to\Modk(A_{n-1})\]
	such that its components $u_i\circ \cox_{n,k}$ are given by
	\[\begin{tikzcd}
		u_i\circ \cox_{n,k}(A)\dar["u_i\circ \cox_{n,k}(f)"']\\
		u_i\circ \cox_{n,k}(B)
	\end{tikzcd}
	=
	\begin{tikzcd}
		C(a_{k+i-1}\circ\ldots\circ a_k)\rar["{\fd(\id,a_{k+i})}"]\dar["{\fb(f_k,\chi(h)_k^{k+i-1},f_{k+i})}"']\drar["{\fd(0,h_{k+i})}"] & C(a_{k+i}\circ\ldots\circ a_k)\dar["{\fb(f_k,\chi(h)_k^{k+i},f_{k+i+1})}"]\\
		C(b_{k+i-1}\circ\ldots\circ b_k)\rar["{\fd(\id,b_{k+i})}"] & C(b_{i+1}\circ\ldots\circ b_k)
	\end{tikzcd}\]
	for $1\leq i\leq n-k-1$, and
	\[=
	\begin{tikzcd}
		C(a_{k-1}\circ\ldots\circ a_{k+i-n})[1]\rar["{\fd(a_{k+i-n},\id)}"]\dar["{\fb(f_{k+i-n},\chi(h)_{k+i-n}^{k-1},f_k)}"']\drar["{\fd(h_{k+i-n},0)}"] &[20pt] C(a_{k-1}\circ\ldots\circ a_{k+i-n+1})[1]\dar["{\fb(f_{k+i-n+1},\chi(h)_{k+i-n+1}^{k-1},f_k)}"]\\
		C(b_{k-1}\circ\ldots\circ b_{k+i-n})[1]\rar["{\fd(b_{k+i-n},\id)}"] & C(b_{k-1}\circ\ldots\circ b_{k+i-n+1})[1]
	\end{tikzcd}\]
	for $n-k\leq i\leq n-2$. Note that $\cox_{n,1}$ is the Coxeter functor $\cox_n$.
\end{dfn}

\begin{prp}\label{prp:coxeter-power}
	For $1\leq k\leq n-1$, we have the natural equivalence $\cox_n^k\simeq \cox_{n,k}$, i.e. $\cox_n^k$ and $\cox_{n,k}$ are homotopy equivalent in $\Fun(\Modk(A_{n-1}),\Modk(A_{n-1}))$. Moreover, $\cox_n^n\simeq [2]$. In particular, $\cox_{n,k}$ is an auto-quasi-equivalence of $\Modk(A_{n-1})$.
\end{prp}

\begin{proof}
	For any $A=(A_1\xrightarrow{a_1}A_2\xrightarrow{a_2}\ldots\xrightarrow{a_{n-2}}A_{n-1})\in\Modk(A_{n-1})$, it is not hard to see that $\cox_n\circ \cox_{n,k}(A)$ and $\cox_{n,{k+1}}(A)$ are homotopy equivalent since $\cox_n\circ \cox_{n,k}(A_i)\simeq \cox_{n,{k+1}}(A_i)$ for all $i$ by Lemma \ref{lem:cone-exact} and \ref{lem:cone-triple}, and $\cox_{n,k+1}(A)$ is constructed as in Lemma \ref{lem:adjusting-quiver}. Then $\cox_n\circ \cox_{n,k}\simeq \cox_{n,{k+1}}$ since $\cox_{n,{k+1}}$ is constructed as in Lemma \ref{lem:adjusting}. Using this we conclude $\cox_n^k\simeq \cox_{n,k}$. Similarly, we can show $\cox_n^n\simeq [2]$.
\end{proof}

\begin{rmk}\label{rmk:coxeter-general}
	There is a ``counterclockwise'' version of the Coxeter functor, which is the composition of derived (sink) reflection functors at each object of the $A_{n-1}$-quiver. It is the inverse of $\cox_n$. If we define for any integer $n\geq 3$ and $k\in\Z$
	\[\cox_{n,k}:=\begin{cases}
		\cox_{n,k'}[2\lfloor k/n\rfloor] & \text{if }k\not\equiv 0\text{ mod }n\\
		\id[2k/n] & \text{if }k\equiv 0\text{ mod }n
	\end{cases}\]
	where $k'\equiv k\text{ mod }n$ with $1\leq k'\leq n-1$, by Proposition \ref{prp:coxeter-power} we get
	\[\cox_n^{k}\simeq \cox_{n,k}\]
	for any $k\in\Z$. In particular, $\cox_n^{-1}\simeq \cox_{n,n-1}[-2]$.
\end{rmk}

From now on, we will assume all dg categories are $\Z/2$-graded. All the previous results in this section are equally valid for $\Z/2$-graded categories. Note that for $\Modk(A_{n-1})$, in additions to morphism, the objects also become $\Z/2$-graded, i.e. for
\[A=(A_1\xrightarrow{a_1}A_2\xrightarrow{a_2}\ldots\xrightarrow{a_{n-2}}A_{n-1})\in\Modk(A_{n-1})\]
we have $A_i[2]=A_i$ and $a_i[2]=a_i$ for all $i$. Then Remark \ref{rmk:coxeter-general} tells us that
\[\cox_{n,k}:=\begin{cases}
	\cox_{n,k'} & \text{if }k\not\equiv 0\text{ mod }n\\
	\id & \text{if }k\equiv 0\text{ mod }n
\end{cases}\]
where $k'\equiv k\text{ mod }n$ with $1\leq k'\leq n-1$, and $\cox_n^{k}\simeq \cox_{n,k}$. In particular $\cox_n^n\simeq\id$ and we get the following proposition:

\begin{prp}\label{prp:rotation-coxeter}
	For $n\geq 3$ and $k\in\Z$, let $r_{n,k}\colon V_n\to V_n$ be the automorphism of $V_n$ by rotating it by $2\pi k/n$ clockwise, which is explicitly given by $r_{n,k}(z)=ze^{-2\pi i k/n}$. Then $\mSh\dd(r_{n,k})\simeq \cox_{n,k}$.
\end{prp}

\begin{proof}
	Let $r_n:=r_{n,1}$. It is enough to show that $\mSh\dd(r_n)\simeq \cox_n$, since $r_{n,k}=r_n^k$ and $\cox_{n,k}\simeq \cox_n^k$. Observe that we have
	\[i_l\circ r_n=\begin{cases}
		i_{l-1} &\text{if }1<l\leq n\\
		i_n &\text{if }l=1
	\end{cases}\]
	and applying the large microlocal stack $\mSh\dd$ we get
	\[j_l\circ \mSh\dd(r_n)=\begin{cases}
		j_{l-1} &\text{if }1<l\leq n\\
		j_n &\text{if }l=1
	\end{cases}\ .\]
	To understand the $A_{\infty}$-functors satisfying the above condition, we will use the fact that $\Modk(A_{n-1})$ is generated by the injective objects $I_i$ for $i=1,\ldots,n-1$ defined by
	\[I_i:=\k\xrightarrow{\id}\k\xrightarrow{\id}\ldots\xrightarrow{\id}\k\to 0 \to\ldots\to 0\]
	where the first $i$ entries are $\k$ and rest are zero, see \cite{dyckerhoff-kapranov} and \cite{wrapped}. The above condition requires that $\mSh\dd(r_n)(I_i)=P_i[1]$ for all $i=1,\ldots,n-1$, where $P_i$ are the projective objects defined by
	\[P_i:=0\to \ldots \to 0 \to \k\xrightarrow{\id}\k\xrightarrow{\id}\k\]
	where the first $i-1$ entries are zero and rest are $\k$. It also implies that the morphisms between $I_i$'s are required to be mapped to certain morphisms under $\mSh\dd(r_n)$. This shows $\mSh\dd(r_n)$ is uniquely determined up to natural equivalence by imposing this condition. The Coxeter functor $\cox_n$ indeed satisfies this condition, hence we must have $\mSh\dd(r_n)\simeq \cox_n$.
\end{proof}

\begin{rmk}\label{rmk:coxeter-z2}
	Above proof does not apply in $\Z$-graded setting, because $\cox_n^n\not\simeq\id$ whereas $r_n^n=\id$. Hence we cannot define $\mSh\dd(r_{n,k})$ if $\mSh\dd$ is $\Z$-graded.
\end{rmk}
    
\section{Rational Homology Balls and Pinwheels}

In this last chapter, we will apply the theory reviewed and developed in previous chapters to the rational homology balls $B_{p,q}$, defined in \cite{rational-casson}. Being Weinstein manifolds, their symplectic topology is studied in \cite{lekili}. Their skeleta are called pinwheels $L_{p,q}$, following \cite{evans}. In Section \ref{sec:rational-homology-ball}, we will review these objects. In Section \ref{sec:wrapped-rational}, we will calculate the wrapped Fukaya category of $B_{p,1}$ using the theorems reviewed in Section \ref{sec:fukaya-weinstein}. Finally, in Section \ref{sec:msh-pinwheel}, we will calculate the microlocal sheaves on pinwheels $L_{p,1}$, and show that the wrapped microlocal sheaves on $L_{p,1}$ match with the wrapped Fukaya category of $B_{p,1}$, confirming Conjecture \ref{con:microlocal-wrapped}.

\subsection{$A_{p-1}$-Milnor Fibre and Rational Homology Ball $B_{p,q}$}\label{sec:rational-homology-ball}

In this section, we will introduce the Weinstein manifold we will study in this thesis, the rational homology ball $B_{p,q}$, where $p$ and $q$ are coprime integers such that $p\geq 2$ and $0<q<p$. We will describe it as a quotient of $A_{p-1}$-Milnor fibre and summarise some of its properties, following mostly \cite{lekili} and \cite{evans}. Most importantly, we will describe the Legendrian surgery diagram for $B_{p,q}$ which will be used to calculate its wrapped Fukaya category in Section \ref{sec:wrapped-rational}, and the pinwheel $L_{p,q}$ which is the skeleton of $B_{p,q}$. The description of $L_{p,q}$ will enable us to calculate microlocal sheaves on it in Section \ref{sec:msh-pinwheel}.

\begin{dfn}
	Let $p\geq 2$ be an integer. \textit{$A_{p-1}$-Milnor fibre} is defined as
	\[A_{p-1}:=\{(x,y,z)\vb z^p+2xy=1\}\subset\C^3\ .\]
	To see that it is a complex submanifold of $\C^3$, define the holomorphic map
	\begin{align*}
		F\colon \C^3&\rightarrow\C\\
		(x,y,z)&\mapsto z^p+2xy \ . 
	\end{align*}
	Its derivative is $dF_{(x,y,z)}=\mx{2y & 2x & pz^{p-1}}$ which has only one critical point $(0,0,0)$, hence $1$ is a regular value of $F$ and $A_{p-1}=\{F=1\}$ is a complex submanifold of $\C^3$ with the real dimension $4$. This implies that $A_{p-1}$ is a Stein manifold, and consequently has a Weinstein structure coming from the standard Weinstein structure on $\C^3$ given by the Liouville form
	\[\theta=\frac{i}{4}(xd\bar x-\bar x dx+yd\bar y-\bar y dy+zd\bar z-\bar z dz)\]
	with the Lyapunov function
	\[\phi(x,y,z)=\frac{1}{4}(|x|^2+|y|^2+|z|^2)\]
	where $\bar x$ is the conjugate of $x$ (compare with Example \ref{exm:cotangent-weinstein-nonstandard}). See \cite[Section 11.5]{weinstein} for the details.
\end{dfn}

To have better understanding of the geometry of $A_{p-1}$, we consider the exact Lefschetz fibration
\begin{align*}
	 \pi\colon A_{p-1}&\rightarrow\C\\
	 (x,y,z)&\mapsto z\ .
\end{align*}
Its derivative is $d\pi_{(x,y,z)}=\mx{0 & 0 & 1}$. The critical points of this fibration are the points $a\in A_{p-1}$ such that $d\pi_a|_{T_a^{\text{hol}}A_{p-1}}=0$. Since $A_{p-1}=\{F=1\}$, we have $T_a^{\text{hol}}A_{p-1}=\ker dF_a$. Then, $a\in A_{p-1}$ is a critical point if and only if $\ker dF_a\subset\ker d\pi_a$. Hence $\pi$ has $p$ many critical points given by $\{(0,0,z) \vb z^p=1\}$, and the critical values are $z\in\C$ with $z^p=1$. Figure \ref{fig:lef-fib} shows how the fibres look like, where the crosses stand for the critical values of $\pi$.
\begin{figure}[h]
	\centering
	
	\begin{tikzpicture}
	\draw (-3,-1) -- (2,-1) node[right]{$\C_z$};;
	\draw (-2,1) -- (-1.6,1);
	\draw (-1.4,1) -- (-0.1,1);
	\draw (0.1,1) -- (3,1);
	\draw (-3,-1) -- (-2,1);
	\draw (2,-1) -- (3,1);
	
	\newcommand{\cross}[3][thick]{
		\draw[#1] (#2-0.1,#3-0.1) -- (#2+0.1,#3+0.1);
		\draw[#1] (#2-0.1,#3+0.1) -- (#2+0.1,#3-0.1)
	}
	
	\draw[fill] (0,0) circle (2pt) node[right]{$0$};
	\cross{1.5}{0.5};
	\cross{1.5}{-0.5};
	\cross{0}{-0.7};
	\cross{-1.5}{-0.5};
	\cross{-1.5}{0.5};
	
	\draw[domain=-1.3:-0.2,variable=\x] plot ({\x},{(0.5/-1.5)*\x});
	\draw (0,0.2) -- (0,1.3);
	\draw (-1.5,0.7) -- (-1.5,1.3);
	
	\draw (0.4,1.5) arc (0:-180:0.4 and 0.1);
	\draw[dashed] (0.4,1.5) arc (0:180:0.4 and 0.1);
	\draw (0.4,1.9) arc (0:-180:0.4 and 0.1);
	\draw[dashed] (0.4,1.9) arc (0:180:0.4 and 0.1);
	\draw (0,2.3) ellipse (0.4 and 0.1);
	\draw (-0.4,1.5) -- (-0.4,2.3);
	\draw (0.4,1.5) -- (0.4,2.3);
	
	\draw (-1.1,1.5) arc (0:-180:0.4 and 0.1);
	\draw[dashed] (-1.1,1.5) arc (0:180:0.4 and 0.1);
	\draw (-1.9,1.5) -- (-1.1,2.3);
	\draw (-1.1,1.5) -- (-1.9,2.3);
	\draw (-1.5,2.3) ellipse (0.4 and 0.1);
	
	\draw[blue] (-1.5,1.9) .. controls (-0.7,1.8) .. (0,1.8);
	
	\begin{scope}
	\clip (-0.4,1.5) rectangle (-1.4,2.5);
	\draw[blue] (-1.5,1.9) .. controls (-0.7,2) .. (0,2);
	\end{scope}
	
	\begin{scope}
	\clip (-0.4,1.5) rectangle (0.4,2.5);
	\draw[blue,dashed] (-1.5,1.9) .. controls (-0.7,2) .. (0,2);
	\end{scope}
	
	\end{tikzpicture}
	
	\caption{Lefschetz fibration $\pi\colon A_{p-1}\to\C_z$}
	\label{fig:lef-fib}
\end{figure}
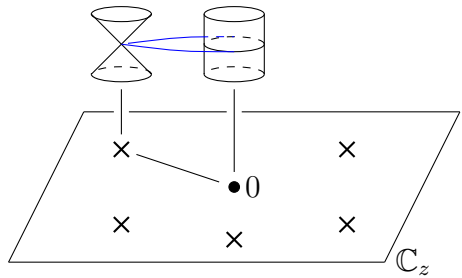

\begin{dfn}
	Let $p,q$ be coprime integers such that $p\geq 2$ and $0<q<p$. Define the action $\Gamma_{p,q}$ of $\Z/p$ on $\C^3$ as
	\begin{align*}
		\Gamma_{p,q}\colon\Z/p\times \C^3 &\rightarrow\C^3\\
		(\xi,(x,y,z))&\mapsto(\xi x,\xi^{-1}y,\xi^q z)
	\end{align*}
	where $\Z/p$ is presented as $\{e^{2\pi i k/p}\vb k\in\Z\}$. Note that $\Gamma_{p,q}$ is a free action except at the origin, and it keeps $A_{p-1}\subset\C^3$ invariant, hence we define the $4$-dimensional manifold $B_{p,q}$ as
	\[B_{p,q}:=A_{p-1}/\Gamma_{p,q}\ .\]
	The Weinstein structure on $A_{p-1}$ induces a Weinstein structure on $B_{p,q}$. Also, $B_{p,q}$ is a rational homology ball, see \cite{lekili}.
\end{dfn}

\begin{rmk}\label{rmk:pq-action}
	The action of $\Gamma_{p,q}$ can be understood in terms of Lefschetz fibration as follows: It rotates the base $\C_z$ around the origin by the angle $2\pi q/p$, and the fibres by $2\pi/p$. See Figure \ref{fig:pq-action}.
	
	\begin{figure}[h]
		\centering
		
		\begin{tikzpicture}
		\draw (-6,-1) -- (-2,-1) node[right]{$\C_z$};
		\draw (-5,1) -- (-1,1);
		\draw (-6,-1) -- (-5,1);
		\draw (-2,-1) -- (-1,1);
		
		\draw (6,1) -- (2,1);
		\draw (5,-1) -- (1,-1);
		\draw (6,1) -- (5,-1) node[right]{$\C_z$};
		\draw (2,1) -- (1,-1);
		
		\newcommand{\cross}[3][thick]{
			\draw[#1] ({#2-0.1},{#3-0.1}) -- ({#2+0.1},{#3+0.1});
			\draw[#1] ({#2-0.1},{#3+0.1}) -- ({#2+0.1},{#3-0.1})
		}
		
		\draw[fill] (-3.5,0) circle (2pt);
		\cross{-3.5+cos(0)}{0+0.8*sin(0)};
		\cross{-3.5+cos(72)}{0+0.8*sin(72)};
		\cross{-3.5+cos(144)}{0+0.8*sin(144)};
		\cross{-3.5+cos(216)}{0+0.8*sin(216)};
		\cross{-3.5+cos(288)}{0+0.8*sin(288)};
		
		\draw ({-3.5+0.2*cos(0)},{0+0.2*sin(0)}) -- ({-3.5+0.8*cos(0)},{0+0.6*sin(0)});
		\draw ({-3.5+0.2*cos(72)},{0+0.2*sin(72)}) -- ({-3.5+0.8*cos(72)},{0+0.6*sin(72)});
		\draw ({-3.5+0.2*cos(144)},{0+0.2*sin(144)}) -- ({-3.5+0.8*cos(144)},{0+0.6*sin(144)});
		\draw ({-3.5+0.2*cos(216)},{0+0.2*sin(216)}) -- ({-3.5+0.8*cos(216)},{0+0.6*sin(216)});
		\draw ({-3.5+0.2*cos(288)},{0+0.2*sin(288)}) -- ({-3.5+0.8*cos(288)},{0+0.6*sin(288)});
		
		\draw[fill] (3.5,0) circle (2pt);
		\draw[thick] ({3.5+cos(0)-0.1},{0+0.8*sin(0)-0.1}) -- ({3.5+cos(0)+0.1},{0+0.8*sin(0)+0.1});
		\draw[thick] ({3.5+cos(0)-0.1},{0+0.8*sin(0)+0.1}) -- ({3.5+cos(0)+0.1},{0+0.8*sin(0)-0.1});
		
		\draw ({3.5+0.2*cos(0)},{0+0.2*sin(0)}) -- ({3.5+0.8*cos(0)},{0+0.6*sin(0)});
		\draw[blue] ({3.5+cos(10)},{0+0.8*sin(10)}) arc (10:350:1 and 0.8);
		
		\draw[->] (-1,0) -- (1,0) node[midway,above]{$\boldsymbol{\cdot}/\Gamma_{p,q}$};
		\end{tikzpicture}
		
		\caption{The orbits of $\Gamma_{p,q}$ on the base $\C_z$}
		\label{fig:pq-action}
	\end{figure}
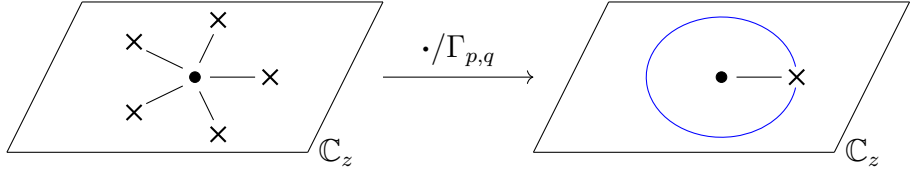
\end{rmk}

\begin{rmk}
	$B_{p,q}$ and $B_{p,p-q}$ are exact symplectomorphic as remarked in \cite{evans}.
\end{rmk}

We have the following fact regarding the Lagrangians in $B_{p,q}$:

\begin{thm}[\cite{lekili}]
	For $p\neq 2$, $B_{p,q}$ has no closed exact Lagrangian submanifolds.
\end{thm}

Note that this does not imply $\cF(B_{p,q})$ is empty since we also consider (unobstructed) exact immersed Lagrangians in the Fukaya category. Moreover, we have an exact immersed Lagrangian sphere in $B_{p,q}$, possibly unobstructed, whose projection to the base is given by the blue curve in Figure \ref{fig:pq-action}: a simple closed curve around $0$ passing through the unique critical value. For $p=2$, $B_{2,1}$ is exact symplectomorphic to $T^*\RP^2$, hence it contains a closed exact Lagrangian submanifold given by the zero section.

Next, we will describe the Weinstein handle decomposition of $B_{p,q}$. For that, we will first describe the Weinstein handle decompoisiton of $A_{p-1}$ (as a Weinstein domain): We will make use of the Lefschetz fibration $\pi$ on $A_{p-1}$. Take a small closed neighbourhood $D_0$ of the origin in $\C_z$ which does not contain the critical values of $\pi$. Note that the fibre $\pi^{-1}(0)$ is a cylinder, which is the union of a 0- and a 1-handle of dimension 2. Since $\pi^{-1}(D_0)\simeq\pi^{-1}(0)\times D^2$, it is a thickened cylinder which can be seen as the union of a 0- and a 1-handle of dimension 4, i.e.
\[\pi^{-1}(D_0)\simeq(D^0\times D^4)\cup (D^1\times D^3)\ .\]
The core of 2-handles attached to $\pi^{-1}(D_0)$ is given by the Lefschetz thimbles associated to the vanishing paths, which are the linear paths connecting the origin to the critical values of $\pi$. See Figure \ref{fig:lef-fib} where blue disk is the Lefschetz thimble whose projection is the linear path below. Explicitly, we have $p$ many 2-handles whose core is given by
\[\pi^{-1}(\gamma_a([0,1])\setminus \inter(D_0))\cap\{(x,y,z)\in\C^3\vb |x|=|y|\}\]
where $\inter(D_0)$ is the interior of $D_0$, $a\in\Z/p\subset\C_z$ is one of the $p$ critical values of $\pi$, and $\gamma_a\colon [0,1]\rightarrow\C_z$ is the linear path connecting the origin and $a$. So, the attaching circle of each $p$ 2-handle is a noncontractible circle on the boundary $S^1\times S^2$ of the thickened cylinder $\pi^{-1}(D_0)$, shown in Figure \ref{fig:leg-milnor} where two spheres are identified respecting the labels on them. This gives the desired Weinstein handle decomposition of $A_{p-1}$. Note that all Legendrian surgery diagrams in this section show the front projection of the Legendrian knots, which are the attaching spheres, and the framing of each attaching sphere is given by $\tb-1$ where $\tb$ is the Thurston–Bennequin number of the Legendrian knot. Then we have the following proposition:

\begin{figure}[h]
	\centering
	
	\begin{tikzpicture}
	\draw (-5,0) circle (1);
	\draw (-4,0) arc (0:-180:1 and 0.15);
	\draw[dashed] (-4,0) arc (0:180:1 and 0.15);
	\draw (5,0) circle (1);
	\draw (6,0) arc (0:-180:1 and 0.15);
	\draw[dashed] (6,0) arc (0:180:1 and 0.15);
	
	\node[left] at ({-5+cos(asin(0.9))},0.9){\tiny 1};
	\node[left] at ({-5+cos(asin(0.7))},0.7){\tiny 2};
	\node[left] at ({-5+cos(asin(0.9))},-0.9){\tiny p};
	
	\node[right] at ({5-cos(asin(0.9))},0.9){\tiny 1};
	\node[right] at ({5-cos(asin(0.7))},0.7){\tiny 2};
	\node[right] at ({5-cos(asin(0.9))},-0.9){\tiny p};
	
	\draw ({-5+cos(asin(0.9))},0.9) -- ({5-cos(asin(0.9))},0.9);
	\draw ({-5+cos(asin(0.7))},0.7) -- ({5-cos(asin(0.7))},0.7);
	\draw ({-5+cos(asin(0.9))},-0.9) -- ({5-cos(asin(0.9))},-0.9);
	
	\draw[dashed] (0,0.6) -- (0,-0.8) node[left,midway]{\small $p$};
	\end{tikzpicture}
	
	\caption{Legendrian surgery diagram of $A_{p-1}$}
	\label{fig:leg-milnor}
\end{figure}
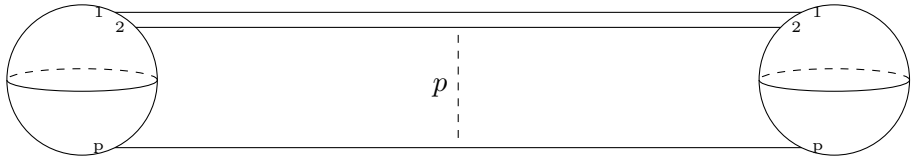

\begin{prp}[\cite{lekili}]
	$B_{p,q}$ has a Weinstein handle decomposition consisting of one 0-handle, one 1-handle, and one 2-handle. The attaching circle of the 2-handle on $S^1\times S^2$ is given by the Legendrian surgery diagram shown in the Figure \ref{fig:leg-rational}.
\end{prp}

\begin{figure}[h]
	\centering
	
	\begin{tikzpicture}
	\draw (-5,0) circle (1);
	\draw (-4,0) arc (0:-180:1 and 0.15);
	\draw[dashed] (-4,0) arc (0:180:1 and 0.15);
	\draw (5,0) circle (1);
	\draw (6,0) arc (0:-180:1 and 0.15);
	\draw[dashed] (6,0) arc (0:180:1 and 0.15);
	
	\node[left] at ({-5+cos(asin(0.9))},0.9){\tiny 1};
	\node[left] at ({-5+cos(asin(0.7))},-0.7){\tiny p-1};
	\node[left] at ({-5+cos(asin(0.9))},-0.9){\tiny p};
	
	\node[right] at ({5-cos(asin(0.9))},0.9){\tiny 1};
	\node[right] at ({5-cos(asin(0.7))},-0.7){\tiny p-1};
	\node[right] at ({5-cos(asin(0.9))},-0.9){\tiny p};
	
	\draw ({-5+cos(asin(0.9))},0.9) -- (2,-0.5) -- (1,1.7) -- ({5-cos(asin(0.9))},0.9);
	\draw ({-5+cos(asin(0.7))},-0.7) -- (-1.8,-1.2) -- (-2.8,1) -- ({5-cos(asin(0.7))},-0.7);
	\draw ({-5+cos(asin(0.9))},-0.9) -- (-2,-1.4) -- (-3,0.8) -- ({5-cos(asin(0.9))},-0.9);
	
	\draw[dashed] (-3.5,0.6) -- (-3.5,-0.8) node[right,midway]{\small $p$};
	\draw[dashed] (3,1.1) -- (3,-0.3) node[right,midway]{\small $p$};
	\draw[dashed] (-2,1) -- (0.8,1.4) node[above,midway]{\small $q$};
	\end{tikzpicture}
	
	\caption{Legendrian surgery diagram of $B_{p,q}$}
	\label{fig:leg-rational}
\end{figure}
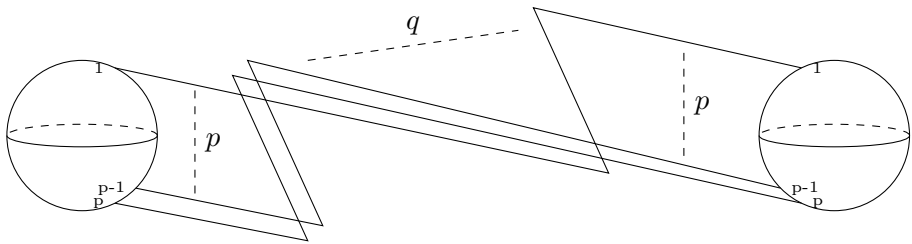

\begin{proof}
	$\Gamma_{p,q}$ action does not change $\pi^{-1}(D_0)$, so $B_{p,q}$ has one 0- and one 1-handle, and hence we only need to understand what happens to 2-handles of $A_{p-1}$ under the action of $\Gamma_{p,q}$. Clearly, $p$ 2-handles become one under the action. To understand its attaching circle, observe that we can apply $q$ full negative twists around the 1-handle in Figure \ref{fig:leg-milnor}, and since $A_{p-1}$ has unique Weinstein structure up to deformation (see \cite{wendl} and \cite{lekili}), it gives an alternative Legendrian surgery diagram of $A_{p-1}$ as shown in Figure \ref{fig:leg-milnor-alternative}.
	
	\begin{figure}[h]
		\centering
		
		\begin{tikzpicture}
			\draw (-5,0) circle (1);
			\draw (-4,0) arc (0:-180:1 and 0.15);
			\draw[dashed] (-4,0) arc (0:180:1 and 0.15);
			\draw (5,0) circle (1);
			\draw (6,0) arc (0:-180:1 and 0.15);
			\draw[dashed] (6,0) arc (0:180:1 and 0.15);
		
			\node[left] at ({-5+cos(asin(0.9))},0.9){\tiny 1};
			\node[left] at ({-5+cos(asin(0.7))},-0.7){\tiny p-1};
			\node[left] at ({-5+cos(asin(0.9))},-0.9){\tiny p};
		
			\node[right] at ({5-cos(asin(0.9))},0.9){\tiny 1};
			\node[right] at ({5-cos(asin(0.7))},-0.7){\tiny p-1};
			\node[right] at ({5-cos(asin(0.9))},-0.9){\tiny p};
		
			\draw ({-5+cos(asin(0.9))},0.9) -- (2,-0.5) -- (1,1.7) -- ({5-cos(asin(0.9))},0.9);
			\draw ({-5+cos(asin(0.7))},-0.7) -- (-1.8,-1.2) -- (-2.8,1) -- ({5-cos(asin(0.7))},-0.7);
			\draw ({-5+cos(asin(0.9))},-0.9) -- (-2,-1.4) -- (-3,0.8) -- ({5-cos(asin(0.9))},-0.9);
		
			\draw[dashed] (-3.5,0.6) -- (-3.5,-0.8) node[right,midway]{\small $p$};
			\draw[dashed] (3,1.1) -- (3,-0.3) node[right,midway]{\small $p$};
			\draw[dashed] (-2,1) -- (0.8,1.4) node[above,midway]{\small $pq$};
		\end{tikzpicture}
		
		\caption{An alternative Legendrian surgery diagram of $A_{p-1}$}
		\label{fig:leg-milnor-alternative}
	\end{figure}
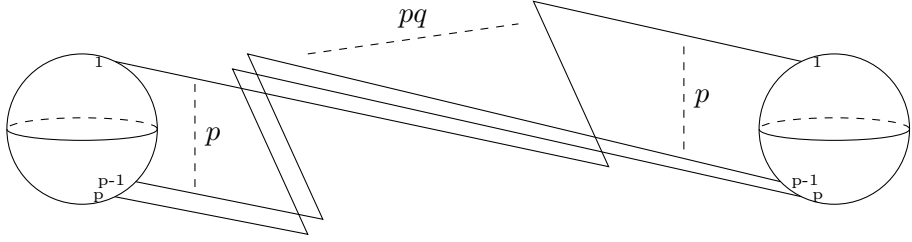

	From this diagram, it is easy to see the effect of the action of $\Gamma_{p,q}$. It gives the Legendrian surgery diagram of $B_{p,q}$ as shown in Figure \ref{fig:leg-rational}.
\end{proof}

\begin{prp}
	The cohomology groups of $B_{p,q}$ are given by
	\[H^1(B_{p,q};\Z)\simeq 0, H^2(B_{p,q};\Z)\simeq \Z/p\]
	and
	\[H^1(B_{p,q};\Z/N)\simeq\Z/\gcd(N,p), H^2(B_{p,q};\Z/N)\simeq\Z/\gcd(N,p)\]
	where $\gcd(N,p)$ is the greatest common divisor of $N$ and $p$. Moreover, the first Chern class of $B_{p,q}$
	\[c_1(B_{p,q})\in H^2(B_{p,q};\Z)\simeq\Z/p\]
	is primitive.
\end{prp}

\begin{proof}
	$B_{p,q}$ consists of one 0-, one 1-, and one 2-handle. The 2-handle is attached as in the Figure \ref{fig:leg-rational}. Hence we have the cochain complex $C^*(B_{p,q};R)$ given by
	\[0\to R\xrightarrow{0}R\xrightarrow{x\mapsto px}R\to 0\]
	for $R=\Z$ or $\Z/N$, which gives the desired cohomology groups. By \cite[Proposition 2.3]{gompf-handlebody}, the class $c_1(B_{p,q})\in H^2(B_{p,q};\Z)$ is
	represented by a cocycle whose value on the oriented 2-handle of $B_{p,q}$ is the rotation number $r(\Lambda_{p,q})$ of the attaching circle $\Lambda_{p,q}$ of that handle, which is shown in Figure \ref{fig:leg-rational}. If we orient $\Lambda_{p,q}$ from left to right, its rotation number can be read from its front projection as
	\[r(\Lambda_{p,q})=\frac{\#(\text{up cusps})-\#(\text{down cusps})}{2}=\frac{2q}{2}=q\ .\]
	This shows that $c_1(B_{p,q})=q\in\Z/p$ is primitive since $\gcd(p,q)=1$.
\end{proof}

Finally, we will describe the skeleton of $B_{p,q}$. To find it we will describe a skeleton of $A_{p-1}$ and then take quotient by $\Gamma_{p,q}$: As remarked in Definition \ref{dfn:weinstein-handle}, a skeleton $\fX_{A_{p-1}}$ can be thought as the union of the cores of the Weinstein handles of $A_{p-1}$ (after extending each core to get whole stable manifold). Hence we get $\fX_{A_{p-1}}$ as the union of the Lefschetz thimbles associated to the linear paths connecting the origin and the critical values of $\pi$:
\[\fX_{A_{p-1}}=\bigcup_{a\in\Z/p\subset\C}\pi^{-1}(\gamma_a([0,1]))\cap\{(x,y,z)\in\C^3\vb |x|=|y|\}\ .\]
$\fX_{A_{p-1}}$ can be also constructed as follows: Start with $V_p\times S^1$, where $V_p$ is the $p$-valent vertex defined in Definition \ref{dfn:vertex} (we take the angles between edges as equal). Note that its boundary consists of $p$ disjoint circles. Then we attach a disk to each of these circles to get $\fX_{A_{p-1}}$. The circle $\{0\}\times S^1$ inside $\fX_{A_{p-1}}$ is called the \textit{core circle}. See Figure \ref{fig:core-circle} to see how it looks like locally.

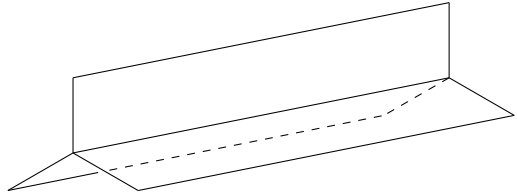
\begin{figure}[h]
	\centering
	
	\begin{tikzpicture}
		\draw (0,0) -- (0,1);
		\draw (0,0) -- ({-sqrt(3)/2},-0.5);
		\draw (0,0) -- ({sqrt(3)/2},-0.5);
		\draw (0,0) -- (5,1);
		\draw (0,1) -- (5,2);
		\draw ({sqrt(3)/2},-0.5) -- ({5+sqrt(3)/2},0.5);
		\draw ({-sqrt(3)/2},-0.5) -- ({5*0.24-sqrt(3)/2},-0.5+0.24);
		\draw[dashed] ({5*0.27-sqrt(3)/2},-0.5+0.27) -- ({5-sqrt(3)/2},0.5);
		\draw (5,1) -- (5,2);
		\draw (5,1) -- ({5+sqrt(3)/2},0.5);
		\draw[dashed] (5,1) -- ({5-sqrt(3)/2},0.5);
		
	\end{tikzpicture}
	
	\caption{A neighbourhood of an arc in the core circle for $p=3$}
	\label{fig:core-circle}
\end{figure}

Now we are ready to describe the skeleton of $B_{p,q}$:

\begin{dfn}[\cite{khodorovskiy}, \cite{evans}]\label{dfn:pinwheel}
	Let $p,q$ be coprime integers such that $p\geq 2$ and $0<q<p$. We define the pinwheel $L_{p,q}$ in a symplectic manifold $(W,\omega)$ as a smooth Lagrangian immersion $f\colon D^2\to W$ satisfying
	\begin{enumerate}
		\item $f|_{D^2\setminus\partial D^2}$ is an embedding,
		\item For $z,z'\in\partial D^2$, $f(z)=f(z')$ if and only if $z/z'\in\Z/p$,
		\item If $z\neq z'$ and $f(z) = f(z')$, then $f_*(T_zD^2)\neq f_*(T_{z'}D^2)$.
	\end{enumerate}
	where we call $f(\partial D^2)$ \textit{core circle}, see Figure \ref{fig:core-circle} again. The characterisation of $q$ is explained in \cite{evans}. We will describe it by the following construction of $L_{p,q}$: Start with $V_p\times[0,1]$, then identify $V_p\times\{0\}$ with $V_p\times\{1\}$ after rotating the former by the angle $2\pi q/p$ clockwise, i.e. via the map $(z,0)\mapsto(ze^{-2\pi q i/p},1)$. This gives the space
	\[S_{p,q} := \colim(V_p\times(0,1)\xleftarrow{(i,i)}(V_p\times(0,1/2))\sqcup (V_p\times(1/2,1))\xrightarrow{(i,r_{p,q})}V_p\times(0,1))\]
	where $i$ is the inclusion and $r_{p,q}(z,t)=(ze^{-2\pi qi/p},t)$ is the inclusion after rotating by the angle $2\pi q/p$ clockwise. Note that the boundary of $S_{p,q}$ is $\partial S_{p,q}=S^1$ whose neighbourhood is $S^1\times B^1$, since $p$ and $q$ are coprime. We get $L_{p,q}$ by attaching a disk to the boundary of $S_{p,q}$, hence we get
	\[L_{p,q}\simeq\colim(S_{p,q} \hookleftarrow S^1\times B^1 \hookrightarrow B^2)\]
	where the maps are standard inclusions into the neighbourhood of boundaries.
\end{dfn}

\begin{prp}
	The skeleton $\fX_{B_{p,q}}$ of $B_{p,q}$ is given by the pinwheel $L_{p,q}$.
\end{prp}

\begin{proof}
	We will get $\fX_{B_{p,q}}$ as the quotient of $\fX_{A_{p-1}}$ by $\Gamma_{p,q}$. Observe that the core circle $\{0\}\times S^1$ of $\fX_{A_{p-1}}$ is a noncontractible circle in the fibre $\pi^{-1}(0)$, and $V_p\times\{s\}$ is just a lift of the linear path connecting the origin and a critical point on the base $\C_z$ for any $s\in S^1$. As observed in Remark \ref{rmk:pq-action}, $\Gamma_{p,q}$ rotates the base $\C_z$ around the origin by the angle $2\pi q/p$, and the fibres by $2\pi/p$, hence $V_p\times S^1$ is sent to $S_{p,q}$ under the quotient by $\Gamma_{p,q}$. Finally, the attached $p$ disks to the boundary of $V_p\times S^1$ becomes a single disk after taking quotient, which is attached to the boundary $\partial S_{p,q}=S^1$. Hence we get $L_{p,q}$ as a result.
\end{proof}

\begin{rmk}
	The topology of $L_{p,q}$ does not depend on $q$, it depends only on $p$. Also, being a skeleton of $B_{p,q}$, $L_{p,q}$ is homotopy equivalent to $B_{p,q}$. However, when calculating microlocal sheaves, we care not only about the topology, but also the order of the edges of the vertices and how they are glued to each other. See Remark \ref{rmk:pants-torus} for an example.
\end{rmk}

\subsection{Wrapped Fukaya Category of Rational Homology Ball $B_{p,1}$}\label{sec:wrapped-rational}

We start the section with some definitions from \cite{subcritical}:

\begin{dfn}
	A \textit{semifree dga} $\cA$ is a dga whose underlying algebra is a unital tensor algebra freely generated by countably many generators. Note that the grading and the differential on $\cA$ is determined by specifying them just on the generators. A \textit{stabilisation} of $\cA$ is a semifree dga obtained by adding a countable collection of generators $\{x_i,y_i\}$ to $\cA$ such that $dx_i=y_i$ for each $i$.
	
	There is a filtration on $\cA$ obtained as follows: Give an ordering on the generators $\{a_i\}_{i=1}^N$ of $\cA$ where $N\in\Z_{\geq 0}$ or $N=\infty$:
	\[a_1<a_2<a_3<\ldots\]
	This induces a filtration on $\cA$ given by
	\[k=\cF^0\cA\subset\cF^1\cA\subset\cF^2\cA\subset\ldots\subset\cA\]
	where $\cF^n\cA$ is a semifree dga generated by $\{a_i\}_{i=1}^n$. An \textit{elementary automorphism} of $\cA$ with an ordered set of generators $\{a_i\}_{i=1}^N$ is a grading preserving algebra map
	\begin{align*}
	F\colon \cA &\to\cA\\
	a_i&\mapsto u_i a_i + v_i
	\end{align*}
	where $u_i$ is a unit in $\k$ and $v_i\in\cF^{i-1}\cA$.  A \textit{tame automorphism} of $\cA$ is a composition of finitely many elementary automorphisms of $\cA$.
	
	A \textit{tame isomorphism} between semifree dgas $\cA$ and $\cB$ is a dg functor $F\colon\cA\to\cB$ such that it can be decomposed as $F=F_2\circ F_1$ where $F_1$ is a tame automorphism of $\cA$ and $F_2\colon \cA \to\cB$ is a grading preserving algebra map sending the generators of $\cA$ to the generators of $\cB$ bijectively. We call $\cA$ and $\cB$ \textit{stable tame isomorphic} if there is a tame isomorphism between a stabilisation of $\cA$ and a stabilisation of $\cB$.
\end{dfn}

\begin{prp}[\cite{subcritical}]\label{prp:stable-tame}
	If semifree dgas $\cA$ and $\cB$ are stable tame isomorphic, then they are quasi-isomorphic.
\end{prp}
    
Now we are ready to calculate the wrapped Fukaya category of $B_{p,1}$ for $p\geq 3$. Note that $\cW(B_{p,1})$ cannot be made $\Z$-graded, since $2c_1(B_{p,1})\neq 0\in H^2(B_{p,1};\Z)\simeq\Z/p$, see Remark \ref{rmk:grading-structure}. Hence $\cW(B_{p,1})$ will be $\Z/2$-graded. If $\k$ is a field of characteristic zero, the combination of Theorem \ref{thm:cdrgg} and \ref{thm:bee} gives
\[\cW(B_{p,1})\simeq\Perf(\CE^*(\Lambda_{p,1}))\]
where $\Lambda_{p,1}$ is the attaching circle of the unique critical 2-handle of $B_{p,1}$, attached on $S^1\times S^2$ which is the boundary of the subcritical part of $B_{p,1}$. Hence we only need to calculate $\Z/2$-graded $\CE^*(\Lambda_{p,1})$ to find $\cW(B_{p,1})$.

\begin{prp}\label{prp:ce-bp1}
	For $p\geq 3$, $\CE^*(\Lambda_{p,1})$ is the semifree dga generated by the elements
	\begin{align*}
		a_i &\text{ for }p\geq i\geq 1\\
		b_{ij} &\text{ for }p\geq i>j\geq 1\\
		c_{ij}&\text{ for }p\geq i>j\geq 1\\
		c'_{ij}&\text{ for }p\geq i,j\geq 1
	\end{align*}
	such that 
	\begin{align*}
		db_{ij} &= \sum_{k=j}^i c_{ik}\circ b_{kj}-b_{ik}\circ c_{(k-1)(j-1)} &\text{with }|b_{ij}|=0 \text{ for }p\geq i>j\geq 1\\
		dc_{ij} &= -\delta_{i-j,p}+\sum_{k=j+1}^{i-1}c_{ik}\circ c_{kj} &\text{with }|c_{ij}|=1\text{ for }p\geq i>j\geq 0\\
		dc'_{ij} &= \delta_{ij} + \sum_{k=1}^{i-1}c_{ik}\circ c'_{kj}+\sum_{k=j+1}^{p}c'_{ik}\circ c_{kj} &\text{with }|c'_{ij}|=1 \text{ for }p\geq i, j\geq 1
	\end{align*}
	where we set $c_{i0}:=a_i$, $b_{ii}:=\id$, $c_{ii}:=\id$, and $\delta_{ij}$ is Kronecker delta.
\end{prp}

\begin{proof}
	For a Legendrian link $\Lambda$ in the boundary of a subcritical Weinstein 4-manifold, $\CE^*(\Lambda)$ is combinatorially described in \cite{subcritical} starting from the front projection of $\Lambda$ on the boundary. We will describe this process for our Legendrian knot $\Lambda_{p,1}$, which is as follows:
	
	By setting $q=1$ in Figure \ref{fig:leg-rational}, we get the front projection of $\Lambda_{p,1}$ on $S^1\times S^2$ as shown in Figure \ref{fig:leg-bp1}.
	
	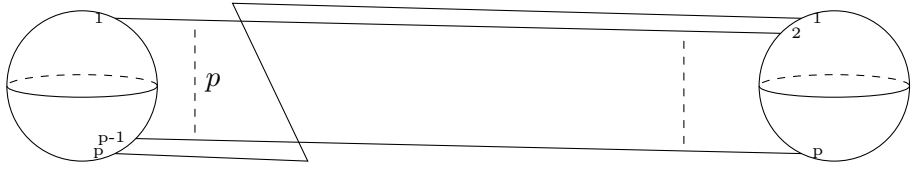
\begin{figure}[h]
		\centering
	        
	    \begin{tikzpicture}    
	    	\draw (-5,0) circle (1);
	        \draw (-4,0) arc (0:-180:1 and 0.15);
	        \draw[dashed] (-4,0) arc (0:180:1 and 0.15);
	        \draw (5,0) circle (1);
	        \draw (6,0) arc (0:-180:1 and 0.15);
	        \draw[dashed] (6,0) arc (0:180:1 and 0.15);
	            
	        \node[left] at ({-5+cos(asin(0.9))},0.9){\tiny 1};
	        \node[left] at ({-5+cos(asin(0.7))},-0.7){\tiny p-1};
	        \node[left] at ({-5+cos(asin(0.9))},-0.9){\tiny p};
	            
	        \node[right] at ({5-cos(asin(0.9))},0.9){\tiny 1};
	        \node[right] at ({5-cos(asin(0.7))},0.7){\tiny 2};
	        \node[right] at ({5-cos(asin(0.9))},-0.9){\tiny p};
	            
	        \draw ({-5+cos(asin(0.9))},0.9) -- ({5-cos(asin(0.7))},0.7);
	        \draw ({-5+cos(asin(0.7))},-0.7) -- ({5-cos(asin(0.9))},-0.9);
	        \draw ({-5+cos(asin(0.9))},-0.9) -- (-2,-1) -- (-3,1.1) -- ({5-cos(asin(0.9))},0.9);
	            
	        \draw[dashed] (-3.5,0.75) -- (-3.5,-0.65) node[right,midway]{\small $p$};
	        \draw[dashed] (3,0.6) -- (3,-0.8);  
		\end{tikzpicture}
	        
	    \caption{Front projection of $\Lambda_{p,1}$ on $S^1\times S^2$}
	    \label{fig:leg-bp1}
	\end{figure}
	
	What we want to work with is actually the Lagrangian projection of $\Lambda_{p,1}$. In \cite{subcritical} it is explained how to pass from a front projection to a Lagrangian projection. It can be summarised as follows: Left cusps are smoothed, right cusps are twisted, the crossings are resolved, and each 1-handle is half-twisted. The Lagrangian projection of $\Lambda_{p,1}$ on $S^1\times S^2$ is shown in Figure \ref{fig:lag-bp1}. Note that the labels on the right sphere are reflected upside down as a result of half-twisting the 1-handle.
	    
	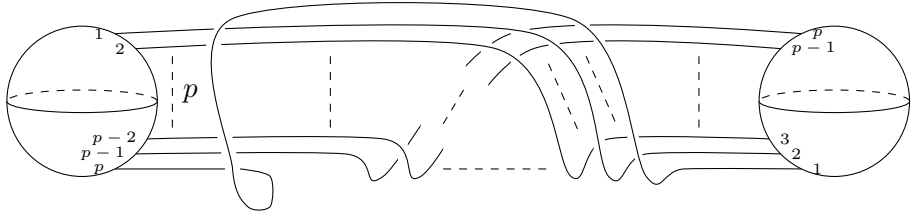
\begin{figure}[h]
		\centering
	        
		\begin{tikzpicture}
	     	\draw (-5,0) circle (1);
	        \draw (-4,0) arc (0:-180:1 and 0.15);
	        \draw[dashed] (-4,0) arc (0:180:1 and 0.15);
	        \draw (5,0) circle (1);
	        \draw (6,0) arc (0:-180:1 and 0.15);
	        \draw[dashed] (6,0) arc (0:180:1 and 0.15);
	            
	        \node[left] at ({-5+cos(asin(0.9))},0.9){\tiny $1$};
	        \node[left] at ({-5+cos(asin(0.7))},0.7){\tiny $2$};
	        \node[left] at ({-5+cos(asin(0.5))},-0.5){\tiny $p-2$};
	        \node[left] at ({-5+cos(asin(0.7))},-0.7){\tiny $p-1$};
	        \node[left] at ({-5+cos(asin(0.9))},-0.9){\tiny $p$};
	            
	        \node[right] at ({5-cos(asin(0.9))},0.9){\tiny $p$};
	        \node[right] at ({5-cos(asin(0.7))},0.7){\tiny $p-1$};
	        \node[right] at ({5-cos(asin(0.5))},-0.5){\tiny $3$};
	        \node[right] at ({5-cos(asin(0.7))},-0.7){\tiny $2$};
	        \node[right] at ({5-cos(asin(0.9))},-0.9){\tiny $1$};
	            
	        \draw ({-5+cos(asin(0.9))},-0.9) -- (-3.1,-0.9);
	        \draw plot[smooth] coordinates{(-2.9,-0.9) (-2.5,-1) (-2.5,-1.4) (-2.8,-1.4) (-3,-0.9) (-3,1.1) (1.5,1.1) (2.5,-1) (3,-0.9) ({5-cos(asin(0.9))},-0.9)};
	            
	        \begin{scope}
	        	\clip (-6,-1.3) rectangle (-3.25,1.3) (-3.1,-1.3) rectangle (2.35,1.3) (2.45,-1.3) rectangle (6,1.3);
	            \draw plot[smooth] coordinates{({-5+cos(asin(0.9))},0.9) (1,0.9) (2,-1) (2.5,-0.7) ({5-cos(asin(0.7))},-0.7)};
	       	\end{scope}
	            
	        \begin{scope}
	            \clip (-6,-1.3) rectangle (-3.35,1.3) (-3.15,-1.3) rectangle (1.8,1.3) (1.95,-1.3) rectangle (2.25,1.3) (2.4,-1.3) rectangle (6,1.3);
	            \draw plot[smooth] coordinates{({-5+cos(asin(0.7))},0.7) (0.5,0.7) (1.5,-1) (2,-0.5) ({5-cos(asin(0.5))},-0.5)};
	        \end{scope}
	            
	        \begin{scope}
	            \clip (-6,-1.3) rectangle (-3.2,1.3) (-3,-1.3) rectangle (-0.2,1.3) (0.4,-1.3) rectangle (0.7,1.3) (0.82,-1.3) rectangle (1.2,1.3) (1.35,-1.3) rectangle (1.8,1.3) (1.95,-1.3) rectangle (6,1.3);
	            \draw plot[smooth] coordinates{({-5+cos(asin(0.5))},-0.5) (-1,-0.5) (-0.5,-1) (1,0.7) ({5-cos(asin(0.7))},0.7)};
	        \end{scope}
	            
	        \begin{scope}
	            \clip (-6,-1.3) rectangle (-3.15,1.3) (-2.95,-1.3) rectangle (-0.8,1.3) (-0.67,-1.3) rectangle (-0.5,1.3) (0.1,-1.3) rectangle (0.3,1.3) (0.45,-1.3) rectangle (0.7,1.3) (0.95,-1.3) rectangle (1.55,1.3) (1.75,-1.3) rectangle (6,1.3);
	            \draw plot[smooth] coordinates{({-5+cos(asin(0.7))},-0.7) (-1.5,-0.7) (-1,-1) (0.7,0.9) ({5-cos(asin(0.9))},0.9)};
	       	\end{scope}
	            
	        \draw[dashed] (-3.8,0.6) -- (-3.8,-0.4) node[right,midway]{\small $p$};
	        \draw[dashed] (-1.7,0.6) -- (-1.7,-0.4);
	        \draw[dashed] (-0.2,-0.3) -- (0.1,0.2);
	        \draw[dashed] (-0.2,-0.9) -- (1.2,-0.9);
	        \draw[dashed] (1.2,0.5) -- (1.6,-0.4);
	        \draw[dashed] (1.7,0.6) -- (2.1,-0.3);
	        \draw[dashed] (3.2,0.6) -- (3.2,-0.4);
		\end{tikzpicture}
	        
	    \caption{Lagrangian projection of $\Lambda_{p,1}$ on $S^1\times S^2$}
	    \label{fig:lag-bp1}
	\end{figure}
	
	Then $\CE^*(\Lambda_{p,1})$ is a semifree dga whose generators can be read off from the Lagrangian projection: We get one generator for each crossing, which we label by
	\begin{align*}
		a_i &\text{ for }p\geq i\geq 1\\
		b_{ij} &\text{ for }p\geq i>j\geq 1
	\end{align*}
	as shown in Figure \ref{fig:gen-bp1}, and we get a collection of additional generators from the 1-handle, labelled by
	\begin{align*}
		c_{ij}&\text{ for }p\geq i>j\geq 1\\
		c'_{ij}&\text{ for }p\geq i,j\geq 1\ .
	\end{align*}
	We visualise the generator $c_{ij}$ in the figure as the arc connecting the points $i$ and $j$ on the sphere, and consider it as an intersection point by collapsing the arc to a point. Figure \ref{fig:gen-bp1} also shows the auxiliary data: an orientation of $\Lambda_{p,1}$ and a marked point $t$ which are given to calculate the differential and grading on the generators.
	
	\begin{figure}[h]
		\centering
		
		\begin{tikzpicture}[scale=1.27]
			\draw (-5,0) circle (1) node[blue]{\small $c_{ij},c'_{ij}$};
			\draw (-4,0) arc (0:-180:1 and 0.15);
			\draw[dashed] (-4,0) arc (0:180:1 and 0.15);
			\draw (5,0) circle (1) node[blue]{\small $c_{ij},c'_{ij}$};
			\draw (6,0) arc (0:-180:1 and 0.15);
			\draw[dashed] (6,0) arc (0:180:1 and 0.15);
		
			\node[left] at ({-5+cos(asin(0.9))},0.9){\tiny $1$};
			\node[left] at ({-5+cos(asin(0.7))},0.7){\tiny $2$};
			\node[left] at ({-5+cos(asin(0.5))},-0.5){\tiny $p-2$};
			\node[left] at ({-5+cos(asin(0.7))},-0.7){\tiny $p-1$};
			\node[left] at ({-5+cos(asin(0.9))},-0.9){\tiny $p$};
		
			\node[right] at ({5-cos(asin(0.9))},0.9){\tiny $p$};
			\node[right] at ({5-cos(asin(0.7))},0.7){\tiny $p-1$};
			\node[right] at ({5-cos(asin(0.5))},-0.5){\tiny $3$};
			\node[right] at ({5-cos(asin(0.7))},-0.7){\tiny $2$};
			\node[right] at ({5-cos(asin(0.9))},-0.9){\tiny $1$};
		
			\draw ({-5+cos(asin(0.9))},-0.9) -- (-3.1,-0.9);
			\draw plot[smooth] coordinates{(-2.9,-0.9) (-2.5,-1) (-2.5,-1.4) (-2.8,-1.4) (-3,-0.9) (-3,1.1) (1.5,1.1) (2.5,-1) (3,-0.9) ({5-cos(asin(0.9))},-0.9)};
			
			\begin{scope}
				\clip (-6,-1.3) rectangle (-3.25,1.3) (-3.1,-1.3) rectangle (2.35,1.3) (2.45,-1.3) rectangle (6,1.3);
				\draw plot[smooth] coordinates{({-5+cos(asin(0.9))},0.9) (1,0.9) (2,-1) (2.5,-0.7) ({5-cos(asin(0.7))},-0.7)};
			\end{scope}
		
			\begin{scope}
				\clip (-6,-1.3) rectangle (-3.35,1.3) (-3.15,-1.3) rectangle (1.8,1.3) 	(1.95,-1.3) rectangle (2.25,1.3) (2.4,-1.3) rectangle (6,1.3);
				\draw plot[smooth] coordinates{({-5+cos(asin(0.7))},0.7) (0.5,0.7) (1.5,-1) (2,-0.5) ({5-cos(asin(0.5))},-0.5)};
			\end{scope}
		
			\begin{scope}
				\clip (-6,-1.3) rectangle (-3.2,1.3) (-3,-1.3) rectangle (-0.2,1.3) (0.4,-1.3) rectangle (0.7,1.3) (0.82,-1.3) rectangle (1.2,1.3) (1.35,-1.3) rectangle (1.8,1.3) (1.95,-1.3) rectangle (6,1.3);
				\draw plot[smooth] coordinates{({-5+cos(asin(0.5))},-0.5) (-1,-0.5) (-0.5,-1) (1,0.7) ({5-cos(asin(0.7))},0.7)};
			\end{scope}
		
			\begin{scope}
				\clip (-6,-1.3) rectangle (-3.15,1.3) (-2.95,-1.3) rectangle (-0.8,1.3) (-0.67,-1.3) rectangle (-0.5,1.3) (0.1,-1.3) rectangle (0.3,1.3) (0.45,-1.3) rectangle (0.7,1.3) (0.95,-1.3) rectangle (1.55,1.3) (1.75,-1.3) rectangle (6,1.3);
				\draw plot[smooth] coordinates{({-5+cos(asin(0.7))},-0.7) (-1.5,-0.7) (-1,-1) (0.7,0.9) ({5-cos(asin(0.9))},0.9)};
			\end{scope}
		
			\draw[dashed] (-3.8,0.6) -- (-3.8,-0.4) node[right,midway]{\small $p$};
			\draw[dashed] (-1.7,0.6) -- (-1.7,-0.4);
			\draw[dashed] (-0.2,-0.3) -- (0.1,0.2);
			\draw[dashed] (-0.2,-0.9) -- (1.2,-0.9);
			\draw[dashed] (1.2,0.5) -- (1.6,-0.4);
			\draw[dashed] (1.7,0.6) -- (2.1,-0.3);
			\draw[dashed] (3.2,0.6) -- (3.2,-0.4);
		
			\fill[blue] (-3.18,0.95) circle (1pt) node[above left]{\tiny $a_1$};
			\fill[blue] (-3.25,0.76) circle (1pt) node[above left,yshift=-0.1cm]{\tiny $a_2$};
			\fill[blue] (-3.1,-0.48) circle (1pt) node[above left,yshift=-0.1cm]{\tiny $a_{p-2}$};
			\fill[blue] (-3.05,-0.68) circle (1pt) node[above left,yshift=-0.1cm]{\tiny $a_{p-1}$};
			\fill[blue] (-3,-0.9) circle (1pt) node[above left,yshift=-0.1cm]{\tiny $a_p$};
		
			\fill[blue] (1.67,1) circle (1pt) node[above]{\tiny $b_{p1}$};
			\fill[blue] (0.82,0.94) circle (1pt) node[above]{\tiny $b_{p2}$};
			\fill[blue] (0.38,0.72) circle (1pt) node[above]{\tiny $b_{p3}$};
			\fill[blue] (-0.73,-0.67) circle (1pt) node[above]{\tiny $b_{p(p-1)}$};
			\fill[blue] (1.88,0.8) circle (1pt) node[above right,yshift=-0.14cm]{\tiny $b_{(p-1)1}$};
			\fill[blue] (1.28,0.77) circle (1pt) node[above,yshift=-0.05cm]{\tiny $b_{(p-1)2}$};
			\fill[blue] (0.76,0.57) circle (1pt) node[above,yshift=-0.05cm]{\tiny $b_{(p-1)3}$};
			\fill[blue] (2.34,-0.48) circle (1pt) node[above right,yshift=-0.08cm]{\tiny $b_{31}$};
			\fill[blue] (1.86,-0.55) circle (1pt) node[above right]{\tiny $b_{32}$};
			\fill[blue] (2.4,-0.74) circle (1pt) node[above right,yshift=-0.06cm]{\tiny $b_{21}$};
			
			\fill[red] (-2.51,-1.41) circle (1pt) node[above right,yshift=-0.1cm]{\tiny $t$};
			
			\draw[purple] (-0.5,1.31) -- (-0.64,1.38);
			\draw[purple] (-0.5,1.31) -- (-0.64,1.24);
			
			\fill[green,opacity=0.2] plot coordinates {(2.4,-0.74) (2.6,-1.1) (3,-0.9) ({5-cos(asin(0.9))},-0.9) ({5-cos(asin(0.7))},-0.7) ({5-cos(asin(0.5))},-0.5) (2.34,-0.48)};
			\node[circle,magenta,draw] at (3.4,-0.6){\tiny 7};
			\fill[purple,opacity=0.2] plot[smooth] coordinates {({5-cos(asin(0.7))},0.7) (1.88,0.8) (1.28,0.77) (1,0.9) (0.82,0.94) (1.67,1) ({5-cos(asin(0.9))},0.9)};
			\node[circle,magenta,draw] at (3.2,0.9){\tiny 6};
			\fill[green,opacity=0.2] plot[smooth] coordinates {(-2.9,-0.9) (-2.5,-1) (-2.5,-1.4) (-2.8,-1.4) (-3,-0.9)};
			\node[circle,magenta,draw] at (-2.65,-1.1){\tiny 1};
			\fill[green,opacity=0.2] plot[smooth] coordinates {({-5+cos(asin(0.7))},-0.7) (-1.5,-0.7) (-1.1,-1.05) (-0.8,-0.79) (-1,-0.5) ({-5+cos(asin(0.5))},-0.5)};
			\node[circle,magenta,draw] at (-1.9,-0.6){\tiny 3};
			\fill[green,opacity=0.2] plot coordinates {({-5+cos(asin(0.9))},0.9) (-0.9,1) (0.82,0.94) (1.28,0.77) (0.76,0.57) (0.38,0.72) (-0.5,0.8) ({-5+cos(asin(0.7))},0.7)};
			\node[circle,magenta,draw] at (-1.7,0.9){\tiny 4};
			\fill[purple,opacity=0.2] plot coordinates {({-5+cos(asin(0.7))},-0.7) (-3.05,-0.68) (-3,-0.9) ({-5+cos(asin(0.9))},-0.9)};
			\node[circle,magenta,draw] at (-3.8,-0.8){\tiny 2};
			\fill[yellow,opacity=0.2] plot coordinates{(-3.18,0.95) (-0.8,1.03) (0.82,0.94) (1.67,1) (1.3,1.18) (-0.5,1.31) (-2.8,1.2)};
			\node[circle,magenta,draw] at (-0.1,1.15){\tiny 5};
		\end{tikzpicture}
		
		\caption{Generators of $\CE^*(\Lambda_{p,1})$}
		\label{fig:gen-bp1}
	\end{figure}
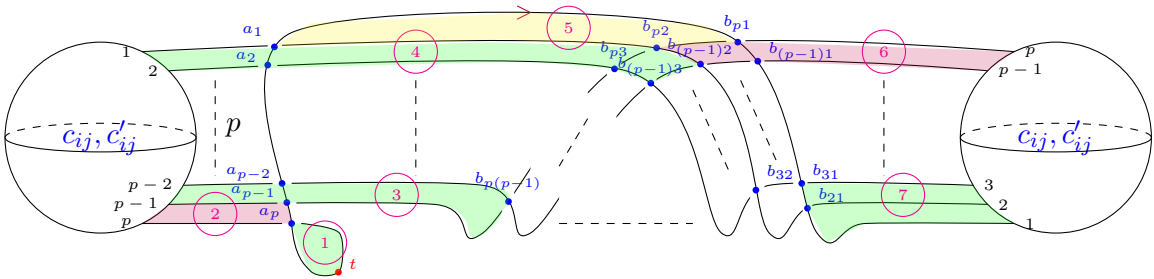
	
	Note that there are infinitely many generators coming from the handle in \cite{subcritical}, however they are simplified in \cite{plumbing} to give finitely many generators, hence we use those. Also, our notation $c_{ij}$ corresponds to $c^0_{ji}$ and $c'_{ij}$ corresponds to $c^1_{ji}$ in \cite{subcritical} and \cite{plumbing}.
	
	Before we describe the differential and grading on the generators, note the following: First, our complexes are cohomological, as opposed to \cite{subcritical} where the complexes are homological. However, since the complexes are $\Z/2$-graded they mean the same thing. Second, in \cite{subcritical} multiplications in dgas are read from left-to-right, whereas we read from right-to-left. This requires us to pass the opposite dga of $\CE^*(\Lambda_{p,1})$. As remarked in \cite{koszul}, this is just changing the differential and multiplication as follows:
	\begin{align*}
		d^{\op}(x) &:= (-1)^{|x|-1}dx\\
		x_2\circ^{\op} x_1 &:= x_1\circ x_2\ .
	\end{align*}
	
	Now, we will use the prescription in \cite[Chapter 2]{subcritical} to get the differential and grading on the generators. For the generators $c_{ij}, c'_{ij}$ coming from the 1-handle, the differential and grading are given by
	\begin{align*}
		dc_{ij} &= \sum_{k=j+1}^{i-1}c_{ik}\circ c_{kj} &\text{with }|c_{ij}|=1\text{ for }p\geq i>j\geq 1\\
		dc'_{ij} &= \delta_{ij} + \sum_{k=1}^{i-1}c_{ik}\circ c'_{kj}+\sum_{k=j+1}^{p}c'_{ik}\circ c_{kj} &\text{with }|c'_{ij}|=1 \text{ for }p\geq i, j\geq 1\ .
	\end{align*}
	For a generator $z_0\in\{a_l,b_{ij}\vb p\geq l\geq 1, p\geq i>j\geq 1\}$, the differential is given by
	\[dz_0=(-1)^{|z_0|-1}\sum_{n\geq 0}\sum_{z_1,\ldots,z_n}\sum_{\Delta\in\Delta(z_0;z_1,\ldots,z_n)}\sgn(\Delta)(-t)^{-n(\Delta)}z_1\circ\ldots\circ z_k\]
	where the second sum is over the generators $a_i,b_{ij},c_{ij}$. $\Delta(z_0;z_1,\ldots,z_n)$ is the moduli space of the immersed disks (up to reparametrisation) with convex corners in Figure \ref{fig:gen-bp1} whose boundary lies inside $\Lambda_{p,1}$ with a positive corner at $z_0$, negative corners at $z_1,\ldots,z_n$ (shown in Figure \ref{fig:reeb-signs}), and they are arranged in clockwise order on the boundary. Any $c_{ij}$ is considered as a negative corner. Also, we don't allow disks to pass through the 1-handle. 
	
	\begin{figure}[h]
		\centering
		
		\begin{tikzpicture}
			\draw (-1,1) -- (1,-1);
			\draw (-1,-1) -- (-0.1,-0.1);
			\draw (0.1,0.1) -- (1,1);
			\node at (0,0.4){--};
			\node at (0,-0.4){--};
			\node at (0.4,0){+};
			\node at (-0.4,0){+};
			
			\fill[green,opacity=0.2] (0,0) -- (-1,1) -- (-1,-1) -- (0,0);
			\fill[purple,opacity=0.2] (0,0) -- (1,1) -- (1,-1) -- (0,0);
			
			\draw (5,1) -- (7,-1);
			\draw (5,-1) -- (5.9,-0.1);
			\draw (6.1,0.1) -- (7,1);
			\node at (6,0.4){--};
			\node at (6,-0.4){--};
			\node at (6.4,0){+};
			\node at (5.6,0){+};
			
			\fill[green,opacity=0.2] (6,0) -- (5,1) -- (7,1) -- (6,0);
			\fill[purple,opacity=0.2] (6,0) -- (5,-1) -- (7,-1) -- (6,0);
		\end{tikzpicture}
		
		\caption{Positive (on the left) and negative (on the right) corners of the the shaded portions of a disk}
		\label{fig:reeb-signs}
	\end{figure}
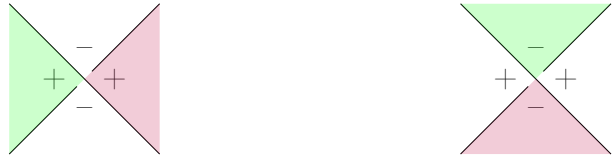

	$n(\Delta)$ is the signed number of times the boundary of $\Delta$ (when transversed counterclockwise) pass through the point $t$. $\sgn(\Delta)$ is defined as
	\[\sgn(\Delta)=\sgn(z_0;\Delta)\times (z_1;\Delta)\times\ldots\times\sgn(z_n;\Delta)\]
	where $\sgn(z_k;\Delta)$ is equal to $1$ if $|z_k|$ is odd, and Figure \ref{fig:orientation-signs} shows the possible values of $\sgn(z_k;\Delta)$ when $|z_k|$ is even.
	
	\begin{figure}[h]
		\centering
		
		\begin{tikzpicture}
			\draw (-1,1) -- (1,-1);
			\draw (-1,-1) -- (-0.1,-0.1);
			\draw (0.1,0.1) -- (1,1);
			\node at (0,0.5){\small +1};
			\node at (0,-0.5){\small --1};
			\node at (0.5,0){\small --1};
			\node at (-0.5,0){\small +1};
			
			\fill[green,opacity=0.2] (0,0) -- (-1,1) -- (1,1) -- (0,0);
			\fill[purple,opacity=0.2] (0,0) -- (-1,1) -- (-1,-1) -- (0,0);
			
			\draw (5,1) -- (7,-1);
			\draw (5,-1) -- (5.9,-0.1);
			\draw (6.1,0.1) -- (7,1);
			\node at (6,0.5){\small +1};
			\node at (6,-0.5){\small --1};
			\node at (6.5,0){\small --1};
			\node at (5.5,0){\small +1};
			
			\fill[green,opacity=0.2] (6,0) -- (5,-1) -- (7,-1) -- (6,0);
			\fill[purple,opacity=0.2] (6,0) -- (7,-1) -- (7,1) -- (6,0);
		\end{tikzpicture}
		
		\caption{The values of $\sgn(z_k;\Delta)$ for the shaded portions of a disk when $|z_k|$ is even}
		\label{fig:orientation-signs}
	\end{figure}
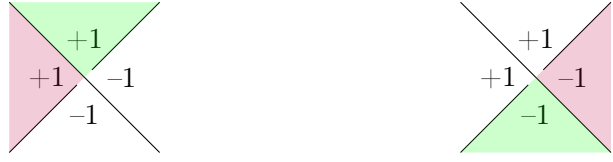
	
	Note that the differential is slightly different compared to \cite{subcritical}. We modified it according to the remark above since we read multiplications from right-to-left, and $t$ is replaced by $-t$ since we use the null-cobordant spin structure on $\Lambda_{p,1}$, as remarked in \cite{subcritical}.
	
	We highlighted some typical disks by green in Figure \ref{fig:gen-bp1}. The contribution of each disk to the differential is as follows:
	\begin{align*}
		\raisebox{.5pt}{\textcircled{\raisebox{-.9pt} {$1$}}}&\to\text{ the term }-t\text{ in }da_p\\
		\raisebox{.5pt}{\textcircled{\raisebox{-.9pt} {$2$}}}&\to\text{ the term }+c_{p(p-1)}\circ a_{p-1}\text{ in }da_p\\
		\raisebox{.5pt}{\textcircled{\raisebox{-.9pt} {$3$}}}&\to\text{ the term }-c_{(p-1)(p-2)}\text{ in }db_{p(p-1)}\\
		\raisebox{.5pt}{\textcircled{\raisebox{-.9pt} {$4$}}}&\to\text{ the term }-b_{(p-1)3}\circ c_{21}\text{ in }db_{(p-1)2}\\
		\raisebox{.5pt}{\textcircled{\raisebox{-.9pt} {$5$}}}&\to\text{ the term }-b_{p2}\circ a_1\text{ in }db_{p1}\\
		\raisebox{.5pt}{\textcircled{\raisebox{-.9pt} {$6$}}}&\to\text{ the term }+c_{p(p-1)}\circ b_{(p-1)2}\text{ in }db_{p2}\\
		\raisebox{.5pt}{\textcircled{\raisebox{-.9pt} {$7$}}}&\to\text{ the term }+c_{31}\text{ in }db_{31}
	\end{align*}
	where we used the fact that $|a_i|=1$ and $|b_{ij}|=0$ when determining the signs, which is justified below. Hence, we get the differentials
	\begin{align*}
		da_i &= \sum_{k=1}^{i-1}c_{ik}\circ a_k &\text{with }|a_i|=1\text{ for }p-1\geq i\geq 1\\
		da_p &= -t+\sum_{k=1}^{p-1}c_{pk}\circ a_k &\text{with }|a_p|=1\\
		db_{i1} &=c_{i1}-a_{i-1}+\sum_{k=2}^{i-1} c_{ik}\circ b_{k1}-b_{ik}\circ a_{(k-1)} &\text{with }|b_{i1}|=0 \text{ for }p\geq i\geq 2\\
		db_{ij} &= c_{ij}-c_{(i-1)(j-1)}+\sum_{k=j+1}^{i-1} c_{ik}\circ b_{kj}-b_{ik}\circ c_{(k-1)(j-1)} &\text{with }|b_{ij}|=0 \text{ for }p\geq i>j\geq 2
	\end{align*}
	where $|t|=0$. All the grading is determined by the grading $|c_{i,j}|=1$ by exploiting the differential relation above. Here, the dga is with the coefficient ring $\k[t,t^{-1}]$. We set $t=1$ to work with coefficient field $\k$. 
\end{proof}

Next, we will simplify $\CE^*(\Lambda_{p,1})$. For that, we define:

\begin{dfn}\label{dfn:dga-rational}
	For $p\geq 3$, $\cA_{p,1}$ is defined as the semifree dga generated by the degree $1$ elements $x_i$ for $p\geq i\geq 1$, and $y_{ij}$ for $p\geq i,j\geq 1$, where
	\begin{align*}
		dx_i&=-\delta_{i,p}+\sum_{j=1}^{i-1}x_{i-j}\circ x_j\\
		dy_{ij} &= \delta_{ij} + \sum_{k=1}^{i-1}x_{i-k}\circ y_{kj}+\sum_{k=j+1}^{p}y_{ik}\circ x_{k-j}\ .
	\end{align*}
\end{dfn}

We are ready to state the main theorem of this section:
    
\begin{thm}\label{thm:wrapped-rational}
	For $p\geq 3$, $\CE^*(\Lambda_{p,1})$ is quasi-isomorphic to $\cA_{p,1}$, hence we have an $A_{\infty}$-quasi-equivalence of pretriangulated $A_{\infty}$-categories over $\k$
    \[\cW(B_{p,1})\simeq\Perf(\cA_{p,1})\]
    where $\k$ is a field of characteristic zero.
\end{thm}
    
\begin{proof}
	Apply the following (recursive)  change of variables on the generators of $\CE^*(\Lambda_{p,1})$:
    \begin{align*}
    	\bar a_i&:=a_i-\sum_{j=1}^{i-1}\beta_{ij}\circ \bar a_j &\text{for }p\geq i\geq 1\\
        \bar b_{ij}&:=b_{ij}&\text{for }p\geq i>j\geq 1\\
        \bar c_{ij}&:=\sum_{k=j}^i c_{ik}\circ b_{kj}-b_{ik}\circ c_{(k-1)(j-1)} &\text{for }p\geq i>j\geq 1\\
        \bar{c'}_{ij} &:= c'_{ij} - \sum_{k=1}^{i-1}\beta_{ik}\circ\bar{c'}_{kj}+\sum_{k=j+1}^{p}c'_{ik}\circ \beta_{kj} &\text{for }p\geq i, j\geq 1
   	\end{align*}
    where degree 0 elements $\beta_{ij}$ for $i\geq j\geq 1$ are determined by the recursive relations
    \[\beta_{ij}=\sum_{k=j}^i b_{ik}\circ \beta_{(k-1)(j-1)}\]
    where we set $\beta_{ii}=\id$, and $\beta_{i0}=0$ if $i>0$.
    
    As a result, we get the following differentials and gradings:
    \begin{align*}
    	d\bar a_i&=-\delta_{i,p}+\sum_{j=1}^{i-1}\bar a_{i-j}\circ\bar a_j&\text{with }|\bar a_i|=1\text{ for }p\geq i\geq 1\\
        d\bar b_{ij} &= \bar c_{ij} &\text{with }|\bar b_{ij}|=0\text{ for }p\geq i>j\geq 1\\
        d\bar c_{ij} &= 0 &\text{with }|\bar c_{ij}|=1\text{ for }p\geq i>j\geq 0\\
        d\bar{c'}_{ij} &= \delta_{ij} + \sum_{k=1}^{i-1}\bar a_{i-k}\circ \bar{c'}_{kj}+\sum_{k=j+1}^{p}\bar{c'}_{ik}\circ \bar a_{k-j} &\text{with }|\bar{c'}_{ij}|=1 \text{ for }p\geq i, j\geq 1\ .
    \end{align*}
    
    We will explain the above equalities in the end of the proof. Let $\cB_{p,1}$ be the semifree dga generated by $\{\bar a_i,\bar b_{ij},\bar c_{ij},\bar{c'}_{ij}\}$ with the differential and grading as above. Then the functor
    \begin{align*}
    	F\colon\CE^*(\Lambda_{p,1})&\to\cB_{p,1}\\
    	a_{ij}&\mapsto\bar a_i\\
    	b_{ij}&\mapsto\bar b_{ij}\\
    	c_{ij}&\mapsto\bar c_{ij}\\
    	c'_{ij}&\mapsto\bar{c'}_{ij}
    \end{align*}
    is obviously a dg functor, moreover it is a tame isomorphism. It can be seen by giving an ordering to the generators $\{a_i,b_{ij},c_{ij},c'_{ij}\}$ of $\CE^*(\Lambda_{p,1})$ satisfying
    \begin{enumerate}
    	\item $b_{i_1j_1}<a_{i_2}<c_{i_3j_3}<c'_{i_4j_4}$ for any $i_l,j_l$,\\
    	
    	\item $a_{i_1}<a_{i_2}$ if $i_1<i_2$,\\
    	
    	\item $c_{i_1j_1}<c_{i_2j_2}$ if $i_1-j_1<i_2-j_2$,\\
    	
    	\item $c_{i_1j_1}<c_{i_2j_2}$ if $i_1-j_1=i_2-j_2$ and $i_1<i_2$,\\
    	
    	\item $c'_{i_1j_1}<c'_{i_2j_2}$ if $i_1-j_1<i_2-j_2$.
    \end{enumerate}
    
    It is easy to see that $\cB_{p,1}$ is actually a stabilisation of $\cA_{p,1}$ after relabelling $x_i$ by $\bar a_i$ and $y_{ij}$ by $\bar{c'}_{ij}$. Hence this shows that $\CE^*(\Lambda_{p,1})$ and $\cA_{p,1}$ are stable tame isomorphic, which implies that they are quasi-isomorphic by Proposition \ref{prp:stable-tame}. This proves the theorem.
        
    Now, back to the differentials we have not explained above: It is obvious that $d\bar b_{ij} = \bar c_{ij}$ and $d\bar c_{ij} = 0$. To see that $d\bar a_i$ is indeed $-\delta_{i,p}+\sum_{j=1}^{i-1}\bar a_{i-j}\circ\bar a_j$, first note that for $i\geq j\geq 0$
    \[d\beta_{ij}=\sum_{k=j+1}^i c_{i(k-1)}\circ \beta_{(k-1)j}-\beta_{ik}\circ\bar a_{k-j}\ .\]
    To show this, assume it is known for $\beta_{i'(j-1)}$ for $i'<i$. Then
    \begin{align*}
    	d\beta_{ij}&=\sum_{k=j+1}^i b_{ik}\circ d\beta_{(k-1)(j-1)}+\sum_{k=j}^{i-1} db_{ik}\circ \beta_{(k-1)(j-1)}\\
        &=\sum_{k=j+1}^i b_{ik}\circ \left(\sum_{l=j}^{k-1} c_{(k-1)(l-1)}\circ \beta_{(l-1)(j-1)}-\beta_{(k-1)l}\circ\bar a_{(l+1)-j}\right)\\
        &\qquad+\sum_{k=j}^{i-1} \left(\sum_{l=k}^{i-1} c_{il}\circ b_{lk}-\sum_{l=k+1}^i b_{il}\circ c_{(l-1)(k-1)}\right)\circ \beta_{(k-1)(j-1)}\\
        &=\sum_{l=j}^{i-1} c_{il}\circ\left(\sum_{k=j}^l b_{lk}\circ \beta_{(k-1)(j-1)}\right)-\sum_{l=j}^{i-1} \left(\sum_{k=l+1}^i b_{ik}\circ \beta_{(k-1)l}\right)\circ\bar a_{(l+1)-j}\\
        &=\sum_{l=j+1}^i c_{i(l-1)}\circ \beta_{(l-1)j}-\beta_{il}\circ\bar a_{l-j}\ .
	\end{align*}
       
    Using this, and assuming $d\bar a_{i'}=\sum_{j=1}^{i'-1}\bar a_{i'-j}\circ\bar a_j$ for $i'<i$, we get
   	\begin{align*}
    	d\bar a_i&=da_i-\sum_{j=1}^{i-1}d\beta_{ij}\circ \bar a_j-\sum_{j=1}^{i-1}\beta_{ij}\circ d\bar a_j\\
        &=-\delta_{i,p}+\sum_{j=1}^{i-1}c_{ij}\circ\left(\sum_{k=1}^j \beta_{jk}\circ\bar a_k\right)-\sum_{j=1}^{i-1}\left(\sum_{k=j+1}^i c_{i(k-1)}\circ \beta_{(k-1)j}-\beta_{ik}\circ\bar a_{k-j}\right)\circ \bar a_j\\
        &\qquad - \sum_{j=1}^{i-1}\beta_{ij}\circ \left(\sum_{k=1}^{j-1}\bar a_{j-k}\circ\bar a_k\right)\\
        &=-\delta_{i,p}+\sum_{j=1}^{i-1} \bar a_{i-j}\circ \bar a_j\ .
	\end{align*}
        
    To see that $d\bar{c'}_{ij}$ is indeed $\delta_{ij} + \sum_{k=1}^{i-1}\bar a_{i-k}\circ \bar{c'}_{kj}+\sum_{k=j+1}^{p}\bar{c'}_{ik}\circ \bar a_{k-j}$, assume $d\bar{c'}_{i'j'}=\delta_{i'j'} + \sum_{k=1}^{i'-1}\bar a_{i'-k}\circ \bar{c'}_{kj'}+\sum_{k=j'+1}^{p}\bar{c'}_{i'k}\circ \bar a_{k-j'}$ for $i'<i$ with $j'=j$, and $i'=i$ with $j'>j$, and get
    \begin{align*}
    	d\bar{c'}_{ij} &= dc'_{ij} - \sum_{k=1}^{i-1}d\beta_{ik}\circ\bar{c'}_{kj}-\sum_{k=1}^{i-1}\beta_{ik}\circ d\bar{c'}_{kj}+\sum_{k=j+1}^{p}dc'_{ik}\circ \beta_{kj}-\sum_{k=j+1}^{p}c'_{ik}\circ d\beta_{kj}\\
       	&= \delta_{ij} + \sum_{k=1}^{i-1}c_{ik}\circ \left(\sum_{l=1}^{k}\beta_{kl}\circ\bar{c'}_{lj}-\sum_{l=j+1}^{p}c'_{kl}\circ \beta_{lj}\right)+\sum_{k=j+1}^{p}c'_{ik}\circ c_{kj}\\
        &\qquad- \sum_{k=1}^{i-1}\left(\sum_{l=k+1}^i c_{i(l-1)}\circ \beta_{(l-1)k}-\beta_{il}\circ\bar a_{l-k}\right)\circ\bar{c'}_{kj}\\
        &\qquad-\sum_{k=1}^{i-1}\beta_{ik}\circ \left(\delta_{kj} + \sum_{l=1}^{k-1}\bar a_{k-l}\circ \bar{c'}_{lj}+\sum_{l=j+1}^{p}\bar{c'}_{kl}\circ \bar a_{l-j}\right)\\
        &\qquad+\sum_{k=j+1}^{p}\left(\delta_{ik} + \sum_{l=1}^{i-1} c_{il}\circ c'_{lk}+\sum_{l=k+1}^{p}c'_{il}\circ c_{lk}\right)\circ \beta_{kj}\\
        &\qquad-\sum_{k=j+1}^{p}c'_{ik}\circ \left(\sum_{l=j+1}^k c_{k(l-1)}\circ \beta_{(l-1)j}-\beta_{kl}\circ\bar a_{l-j}\right)\\
        &=\delta_{ij} + \sum_{k=1}^{i-1}\bar a_{i-k}\circ \bar{c'}_{kj}+\sum_{k=j+1}^{p}\bar{c'}_{ik}\circ \bar a_{k-j}\ .
	\end{align*}
\end{proof}
    
\begin{rmk}
	We can apply Reidemeister moves to the Lagrangian projection of $\Lambda_{p,1}$ shown in Figure \ref{fig:lag-bp1} before calculating the associated dga. Reidemeister moves of Lagrangian projections of Legendrian links are described in \cite{lag-reidemeister}. We expect that one can apply the Reidemeister moves until the associated dga for the resulting Lagrangian projection is directly given by $\cA_{p,1}$. To show this, we need to answer how to deal with the strands of $\Lambda_{p,1}$ in the 1-handle.
\end{rmk}
    
\begin{rmk}\label{rmk:bp1-grading}
	The grading of $x_1$, which is $1$, determines all the grading in $\cA_{p,1}$ by the grading condition induced by the differential. If we choose $|x_1|=0$ instead, and if $p$ is odd, we get inconsistency of gradings in the differential, hence this is not possible. However, if $p$ is even, there is no inconsistency, and we get
	\[|x_i|=\begin{cases}0 &\text{if }i\text{ is odd}\\1 &\text{if }i\text{ is even}\end{cases}, |y_{ij}|=\begin{cases}0 &\text{if }i-j\text{ is odd}\\1 &\text{if }i-j\text{ is even}\end{cases}\ .\]
	Hence we have only one possible grading for $\cA_{p,1}$ if $p$ is odd, and two possible gradings if $p$ is even. Same thing holds consequently for $\cW(B_{p,1})$ by Theorem \ref{thm:wrapped-rational}. Note that this is consistent with Remark \ref{rmk:grading-structure}, since
	\[H^1(B_{p,1};\Z/2)\simeq\begin{cases}0 &\text{if }p\text{ is odd}\\ \Z/2 &\text{if }p\text{ is even}\end{cases}\ .\]
\end{rmk}
    
\subsection{Microlocal Sheaves on Pinwheel $L_{p,1}$}\label{sec:msh-pinwheel}

In this final section, we will calculate the traditional/wrapped microlocal sheaves on the pinwheel $L_{p,1}$ for $p\geq 3$ as defined in Section \ref{sec:microlocal}. For this purpose, we will make use of the algebraic tools developed in Chapter \ref{chp:algebraic-tools}. We will also show that the wrapped microlocal sheaves on $L_{p,1}$ and the wrapped Fukaya category of $B_{p,1}$, given by Theorem \ref{thm:wrapped-rational}, are quasi-equivalent.

For the notation convention about matrices, see the beginning of Section \ref{sec:quiver}. Also, the categories in this section will be $\Z/2$-graded. The reason for this will be explained throughout the section. We already expect this since the wrapped Fukaya category $\cW(B_{p,1})$ also cannot be made $\Z/2$-graded.

Instead of calculating $\mSh(L_{p,1})$ and $\mSh^w(L_{p,1})$ individually, we will calculate the large microlocal sheaves $\mSh\dd(L_{p,1})$, and describe it as $A_{\infty}$-modules over some dga $\cA$, then $\mSh(L_{p,1})$ and $\mSh^w(L_{p,1})$ are automatically obtained thanks to Proposition \ref{prp:mod-perf-microlocal}. Our main theorem is as follows:

\begin{thm}\label{thm:msh-pinwheel}
	Let $\k$ be a commutative ring. For $p\geq 3$, microlocal sheaves on the pinwheel $L_{p,1}$ are given by
	\begin{align*}
		\mSh\dd(L_{p,1})&\simeq\Modk(\cA_{p,1})\\
		\mSh^w(L_{p,1})&\simeq\Perf(\cA_{p,1})\\
		\mSh(L_{p,1})&\simeq\Perfk(\cA_{p,1})
	\end{align*}
	where $\cA_{p,1}$ is the semifree dga introduced in Definition \ref{dfn:dga-rational}.
\end{thm}
    
To prove the theorem, recall that from Definition \ref{dfn:pinwheel}, we have the description of $L_{p,1}$ given by
\[L_{p,1}\simeq\colim(S_{p,1} \overset{i}{\hookleftarrow} S^1\times B^1 \hookrightarrow B^2)\]
where the maps are standard inclusions into the neighbourhood of boundaries. In this case, \hyperlink{disk}{Disk Lemma} gives
\[\mSh\dd(L_{p,1})\simeq\{(\Gamma,\gamma)\vb \Gamma\in\mSh\dd(S_{p,1}),\mon(\Gamma)_0\overset{\gamma}{\sim}\id\}\]
where $\mon=\mSh\dd(i)$, and $\mon(\Gamma)_0$ is the monodromy of $\Gamma$. Recall also that $S_{p,1}$ is defined as
\[S_{p,1} = \colim(V_p\times(0,1)\xleftarrow{(i,i)}(V_p\times(0,1/2))\sqcup (V_p\times(1/2,1))\xrightarrow{(i,r_{p,1})}V_p\times(0,1))\ .\]
After applying the large microlocal stack $\mSh\dd$, and noting $\mSh\dd(V_p\times U)\simeq\mSh\dd(V_p)$ for any open interval $U\subset\R$ by Proposition \ref{prp:stabilisation}, we get
\[\mSh\dd(S_{p,1})\simeq\holim(\mSh\dd(V_p)\xrightarrow{(\id,\id)}\mSh\dd(V_p)\times\mSh\dd(V_p)\xleftarrow{(\id,\mSh\dd(r_{p,1}))}\mSh\dd(V_p))\ .\]
By Proposition \ref{prp:msh-vertex}, we know that $\mSh\dd(V_p)\simeq\Modk(A_{p-1})$, $A_{\infty}$-modules over $A_{p-1}$-quiver, and by Proposition \ref{prp:rotation-coxeter} $\mSh\dd(r_{p,1})\simeq \cox_{p,1}$, the Coxeter functor. Note that we necessarily need to work with $\Z/2$-graded categories at this stage, since otherwise we cannot define $\mSh\dd(r_{p,1})$, see Remark \ref{rmk:coxeter-z2}. Then \hyperlink{circle}{Circle Lemma} gives
\[\mSh\dd(S_{p,1})\simeq\{(A,f)\vb A\in\Modk(A_{p-1}), f\colon A\rightarrow \cox_{p,1}(A)\textup{ is a homotopy equivalence}\}\ .\]
Combining with the previous step, we get
\[\mSh\dd(L_{p,1})\simeq\{(A,f,\gamma)\vb A\in\Modk(A_{p-1}),f\colon A\to \cox_{p,1}(A)\text{ is h.e.},\mon(A,f)_0\overset{\gamma}{\sim}\id\}\]
where ``h.e.'' is an abbreviation of ``homotopy equivalence'', and $\mon(A,f)_0$ is the monodromy of $(A,f)$. The morphisms are explained by \hyperlink{disk}{Disk Lemma} and \hyperlink{circle}{Circle Lemma}.

Next task is to explain $\mSh\dd(L_{p,1})$ explicitly. For that, first we will describe $\mon(A,f)_0$. Note that an object $(A,f)\in\mSh\dd(S_{p,1})$ can be expressed as
\[\begin{tikzcd}
	A\arrow[d,"f"]\\
	\cox_{p,1}(A)
\end{tikzcd}
=
\begin{tikzcd}
	A_1\arrow[r,"a_1"]\arrow[d,"f_1"]\arrow[rd,"h_1"] & A_2\arrow[r,"a_2"]\arrow[d,"f_2"] & \cdots\arrow[r,"a_{p-3}"] & A_{p-2}\arrow[r,"a_{p-2}"]\arrow[d,"f_{p-2}"]\arrow[rd,"h_{p-2}"] & A_{p-1}\arrow[d,"f_{p-1}"]\\
	C(a_1)\arrow[r,"{\fd(\id,a_2)}"] & C(a_2\circ a_1)\arrow[r,"{\fd(\id,a_3)}"] & \cdots\arrow[r,"{\fd(\id,a_{p-2})}"] & C(a_{p-2}\circ\ldots\circ a_1)\arrow[r,"{\fr(\id,0)}" xshift=-0.15cm] & A_1[1]
\end{tikzcd}\]
where $A_i\in\Modk$, $da_i=0$, and $f_i$ is a homotopy equivalence for all $i$, and
\[dh_i=\begin{cases}
	f_{i+1}\circ a_i-\fd(\id,a_{i+1})\circ f_i &\text{if }1\leq i\leq p-3\\
	f_{p-1}\circ a_{p-2}-\fr(\id,0)\circ f_{p-2} &\text{if }i=p-2
\end{cases}\]
by Proposition \ref{prp:quiver-algebra} since $f$ is a homotopy equivalence.

\begin{prp}\label{prp:monodromy}
	Given $(A,f)\in\mSh\dd(S_{p,1})$, we have $\mon(A,f)=(A_1,\mon(A,f)_0)$ where the monodromy $\mon(A,f)_0\colon A_1\to A_1$ is given by
	\[\mon(A,f)_0=f_{p-1}\circ\eta_2\circ\fb(f_{p-2},h_{p-2},f_{p-1})\circ\eta_{3,p-3}\circ\ldots\circ\eta_{3,1}\circ\fb(f_1,h_1,f_2)\circ f_1\]
	where $\eta_2=\fr(0,\id,a_{p-2}\circ\dots\circ a_1)$ and
	\[\eta_{3,i}=\mx{0 & \id & a_i\circ\ldots\circ a_1 & 0 \\ 0 & 0 & 0 & \id}\]
	for $i=1,\ldots,p-3$.
\end{prp}

\begin{proof}
	First recall how the monodromy is defined: We have the inclusion $S^1\times B^1\overset{i}{\hookrightarrow} S_{p,1}$ into a neighbourhood of the boundary $S^1$ of $S_{p,1}$, and after applying the large microlocal stack $\mSh\dd$, we get
	\begin{align*}
		\mSh\dd(S_{p,q})&\xrightarrow{\mon=\mSh\dd(i)}\mSh\dd(S^1\times B^1)\simeq\Loc\dd(S^1)\\
		(A,f)&\mapsto \mon(A,f)=(\mon(A,f)_1,\mon(A,f)_0)\ .
	\end{align*}
	If we fix the vertex $V_p\in S_{p,1}$, and choose its first edge $e_1$, we can revolve it around the core circle of $S_{p,1}$ $p$ times to get $S^1\times B^1$. Then, given $(A,f)\in\mSh\dd(S_{p,1})$, we can read off $\mon(A,f)\in\Loc\dd(S^1)$ from the gluing diagram for $(A,f)$ by focusing to the first edge $e_1$. Turning the vertex $V_p$ around the core circle $p$ times corresponds to composing $f$ with itself (after translation) $p$ times:
	\begin{multline*}
		A\xrightarrow{f}\cox_{p,1}(A)\xrightarrow{\cox_{p,1}(f)}\cox_{p,1}^2(A)\simeq \cox_{p,2}(A)\xrightarrow{\cox_{p,2}(f)}\cox_{p,2}\circ \cox_{p,1}(A)\simeq \cox_{p,3}(A)\xrightarrow{\cox_{p,3}(f)}\ldots\\
		\ldots\xrightarrow{\cox_{p,p-2}(f)}\cox_{p,p-2}\circ \cox_{p,1}(A)\simeq \cox_{p,p-1}(A)\xrightarrow{\cox_{p,p-1}(f)}\cox_{p,p-1}\circ \cox_{p,1}(A)\simeq A
	\end{multline*}
	where the equivalences given by Proposition \ref{prp:coxeter-power}. We restrict this morphism to the first edge $e_1$, i.e. we apply the functor $j_1$, defined in Proposition \ref{prp:msh-vertex}, on this morphism, and get $\mon(A,f)$:
	\begin{multline*}
		A_1\xrightarrow{f_1}C(a_1)\xrightarrow{\fb(f_1,h_1,f_2)}C((a_1)\xrightarrow{\fd(\id,a_2)}C(a_2\circ a_1))\simeq C(a_2)\xrightarrow{\fb(f_2,h_2,f_3)}\ldots\\
		\ldots\xrightarrow{\fb(f_{p-3},h_{p-3},f_{p-2})}C((a_{p-3}\circ\ldots\circ a_1)\xrightarrow{\fd(\id,a_{p-2})}C(a_{p-2}\circ \ldots\circ a_1))\simeq C(a_{p-2})\xrightarrow{\fb(f_{p-2},h_{p-2},f_{p-1})}\\
		\xrightarrow{\fb(f_{p-2},h_{p-2},f_{p-1})}C(C(a_{p-2}\circ\ldots\circ a_1)\xrightarrow{\fr(\id,0)}A_1[1])\simeq A_{p-1}[1]\xrightarrow{f_{p-1}}A_1\simeq A_1
	\end{multline*}
	By inspecting the proof of Proposition \ref{prp:coxeter-power}, the equivalences can be obtained by Lemma \ref{lem:cone-triple} and \ref{lem:cone-exact}, which are $\eta_{3,1},\ldots,\eta_{3,p-3},\eta_2,\id$ in order. This concludes the proof.
\end{proof}

Next, we will simplify the objects of $\mSh\dd(S_{p,1})$. Before that, we define:

\begin{dfn}
	A \textit{simple matrix} $A$ is a lower triangular square matrix where any of its diagonal consists of the same elements, i.e.
	\[A^{ij}=A^{(i+k)(j+k)}\]
	for any $i,j,k$, where $A^{ij}$ is $(ij)^{\text{th}}$ entry of $A$. It is obvious that the simple matrices are determined by their first column, hence if $a$ is the first column of the simple matrix $A$, we say \textit{$A$ is generated by $a$}, and write $A=\fs(a)$. If $a=\fc(a_1,a_2,\ldots,a_n)$, we can write $A=\fs(a_1,a_2,\ldots,a_n)$.
\end{dfn}

\begin{rmk}
	Set of all $n\times n$ simple matrices forms a ring with the usual addition and multiplication. We can use this fact as follows: If $A$ and $B$ are simple matrices and $B=\fs(b)$, then $A\circ B$ is simple and hence $A\circ B=\fs(A\circ b)$.
\end{rmk}

We also introduce the matrices
\[l_n:=\fr(I_n,0)\]
and
\[r_n:=\fc(I_n,0)\]
where $I_n$ is $n\times n$ identity matrix, and dimension of the zero matrix $0$ is determined by the context. We write
\[b_n(A):=l_n\circ A\circ r_n\]
for a matrix $A$ with dimensions bigger than $n$, which gives the $n\times n$ matrix at the top left part of $A$. Note that if $x$ is $1\times 1$, and $a$ is $n\times 1$ matrix, then
\[\fs(x,a)=\fb(x,a,b_n(s(x,a)))\ .\]

\begin{dfn}\label{dfn:simple-pinwheel}
	Let $\cS_{p,1}'$ be the full dg subcategory of $\mSh\dd(S_{p,1})$ with the objects of the form
	\[\begin{tikzcd}
		A_1\arrow[r,"l_{p-2}\circ f_1"]\arrow[d,"f_1"]\arrow[rd,"0"] &
		C(l_{p-3}\circ f_1)\arrow[r,"l_{p-3}"]\arrow[d,"\id"]\arrow[rd,"0"] &
		C(l_{p-4}\circ f_1)\arrow[r,"l_{p-4}"]\arrow[d,"\id"] &[-5pt]
		\cdots\arrow[r,"l_2"] &[-5pt]
		C(l_1\circ f_1)\arrow[r,"l_1"]\arrow[d,"\id"]\arrow[rd,"0"] &
		A_1[1]\arrow[d,"\id"]
		\\
		C(l_{p-2}\circ f_1)\arrow[r,"l_{p-2}"] &
		C(l_{p-3}\circ f_1)\arrow[r,"l_{p-3}"] &
		C(l_{p-4}\circ f_1)\arrow[r,"l_{p-4}"] &
		\cdots\arrow[r,"l_2"] &
		C(l_1\circ f_1)\arrow[r,"l_1"] &
		A_1[1]
	\end{tikzcd}\]
	where $f_1=\fc(f_1^1,f_1^2,\ldots,f_1^{p-1})$ is a homotopy equivalence with degree zero morphisms $f_1^j\colon A_1\to A_1[1]$ for $j=1,\ldots,p-1$. Note that
	\[C(l_k\circ f_1)=A_1[1]\oplus\ldots \oplus A_1[1]\]
	where $A_1[1]$ is repeated $k+1$ times, and
	\[d_{C(l_k\circ f_1)}=\fs(-d,l_k\circ f_1)\]
	for $k=1,\ldots,p-2$.
\end{dfn}

\begin{prp}\label{prp:simplified-pinwheel}                         
	$\mSh\dd(S_{p,1})$ is quasi-equivalent to $\cS_{p,1}'$.
\end{prp}

We can represent an object $(A,f)\in\mSh\dd(S_{p,1})$ by its $(k,l)\th$ components $(u_{k,l}(A),u_{k,l}(f))$ for $1\leq k<l\leq p-1$ where $u_{k,l}$ is defined in Definition \ref{dfn:components}. We shortly write
\[u_{k,l}(A,f):=(u_{k,l}(A),u_{k,l}(f))\ .\]
We also write $\bar a_n:=a_n\circ\ldots\circ a_1$. To prove the proposition, for $i=1,\ldots,p-1$ we define:

\begin{dfn}\label{dfn:simplified-pinwheel}
	$\cS_{p,1}^i$ as the full dg subcategory of $\mSh\dd(S_{p,1})$ with the objects whose $(1,i-2)\nd$ component is of the form
	\[\begin{tikzcd}
		A_1\arrow[r,"a_1"]\arrow[d,"f_1"]\arrow[rd,"h_1"] &[15pt] A_2\arrow[r,"a_2"]\arrow[d,"f_2"] & \cdots\arrow[r,"a_{i-4}"] & A_{i-3}\arrow[r,"a_{i-3}"]\arrow[d,"f_{i-3}"]\arrow[rd,"h_{i-3}"] &[15pt] A_{i-2}\arrow[d,"f_{i-2}"]\\
		C(\bar a_1)\arrow[r,"{\fd(\id,a_2)}"] & C(\bar a_2)\arrow[r,"{\fd(\id,a_3)}"] & \cdots\arrow[r,"{\fd(\id,a_{i-3})}"] & C(\bar a_{i-3})\arrow[r,"{\fd(\id,a_{i-2})}"] & C(\bar a_{i-2})
	\end{tikzcd}\]
	$(i-2,i+1)\st$ component is of the form
	\[\begin{tikzcd}
		A_{i-2}\arrow[r,"a_{i-2}"]\arrow[d,"f_{i-2}"]\arrow[rd,"h_{i-2}"] & [15pt] A_{i-1}\rar["a_{i-1}"]\dar["f_{i-1}"]\drar["h_{i-1}"] & [20pt] A_i\rar["l_{p-(i+1)}\circ f_i"]\dar["f_i"]\drar["0"] & C(l_{p-(i+2)}\circ f_i\circ\bar a_{i-1})\dar["\id"]\\
		C(\bar a_{i-2})\arrow[r,"{\fd(\id,a_{i-1})}"] & C(\bar a_{i-1})\rar["{\fd(\id,l_{p-(i+1)}\circ f_i)}"] & C(l_{p-(i+1)}\circ f_i\circ\bar a_{i-1})\rar["l_{p-(i+1)}"] & C(l_{p-(i+2)}\circ f_i\circ\bar a_{i-1})
	\end{tikzcd}\]
	and $(i+1,p-1)\st$ component is of the form
	\[\begin{tikzcd}
		C(l_{p-(i+2)}\circ f_i\circ\bar a_{i-1})\rar["l_{p-(i+2)}"]\dar["\id"]\drar["0"] & C(l_{p-(i+3)}\circ f_i\circ\bar a_{i-1})\rar["l_{p-(i+3)}"]\dar["\id"]& \cdots\rar["l_1"] & A_1[1]\dar["\id"]\\
		C(l_{p-(i+2)}\circ f_i\circ\bar a_{i-1})\rar["l_{p-(i+2)}"] & C(l_{p-(i+3)}\circ f_i\circ\bar a_{i-1})\rar["l_{p-(i+3)}"] & \cdots\rar["l_1"] & A_1[1]
	\end{tikzcd}\]
	where we consider $C(l_0\circ f_i\circ \bar a_{i-1})$ as $A_1[1]$ and $\fd(\id,l_0\circ f_i)$ as $\fr(\id,0)$. Note that
	\[C(l_k\circ f_i\circ\bar a_{i-1})=A_1[1]\oplus\ldots \oplus A_1[1]\]
	where $A_1[1]$ is repeated $k+1$ times, and
	\[d_{C(l_k\circ f_i\circ\bar a_{i-1})}=\fs(-d,l_k\circ f_i\circ\bar a_{i-1})\]
	for $k=1,\ldots,p-(i+1)$.
\end{dfn}

\begin{dfn}\label{dfn:simplified-pinwheel-homotopy}
	We define $\cT_{p,1}^i$ for $i=1,\ldots,p-2$ as the full dg subcategory of $\mSh\dd(S_{p,1})$ with the objects whose $(1,i-1)\st$ component is of the form
	\[\begin{tikzcd}
		A_1\arrow[r,"a_1"]\arrow[d,"f_1"]\arrow[rd,"h_1"] &[15pt] A_2\arrow[r,"a_2"]\arrow[d,"f_2"] & \cdots\arrow[r,"a_{i-3}"] & A_{i-2}\arrow[r,"a_{i-2}"]\arrow[d,"f_{i-2}"]\arrow[rd,"h_{i-2}"] &[15pt] A_{i-1}\arrow[d,"f_{i-1}"]\\
		C(\bar a_1)\arrow[r,"{\fd(\id,a_2)}"] & C(\bar a_2)\arrow[r,"{\fd(\id,a_3)}"] & \cdots\arrow[r,"{\fd(\id,a_{i-2})}"] & C(\bar a_{i-2})\arrow[r,"{\fd(\id,a_{i-1})}"] & C(\bar a_{i-1})
	\end{tikzcd}\]
	$(i-1,i+1)\st$ component is of the form
	\[\begin{tikzcd}
		A_{i-1}\rar["a_{i-1}"]\dar["f_{i-1}"]\drar["h_{i-1}"] & [20pt] A_i\rar["a_i"]\dar["f_i"]\drar["h_i"] & C(l_{p-(i+2)}\circ \bar a_i)\dar["\id"]\\
		C(\bar a_{i-1})\rar["{\fd(\id,a_i)}"] & C(\bar a_i)\rar["l_{p-(i+1)}"] & C(l_{p-(i+2)}\circ \bar a_i)
	\end{tikzcd}\]
	and $(i+1,p-1)\st$ component is of the form
	\[\begin{tikzcd}
		C(l_{p-(i+2)}\circ \bar a_i)\rar["l_{p-(i+2)}"]\dar["\id"]\drar["0"] & C(l_{p-(i+3)}\circ\bar a_i)\rar["l_{p-(i+3)}"]\dar["\id"]& \cdots\rar["l_1"] & A_1[1]\dar["\id"]\\
		C(l_{p-(i+2)}\circ \bar a_i)\rar["l_{p-(i+2)}"] & C(l_{p-(i+3)}\circ \bar a_i)\rar["l_{p-(i+3)}"] & \cdots\rar["l_1"] & A_1[1]
	\end{tikzcd}\ .\]
	where $C(l_0\circ \bar a_i)$ is considered as $A_1[1]$. Note that
	\[C(l_k\circ\bar a_i)=A_1[1]\oplus\ldots \oplus A_1[1]\]
	where $A_1[1]$ is repeated $k+1$ times, and
	\[d_{C(l_k\circ \bar a_i)}=\fs(-d,l_k\circ \bar a_i)\]
	for $k=1,\ldots,p-(i+1)$, where $l_{p-(i+1)}\circ \bar a_i=\bar a_i$.
\end{dfn}

Note that we have
\[\cS_{p,1}'=\cS_{p,1}^1\subset \cT_{p,1}^1\subset \cS_{p,1}^2\subset\cT_{p,1}^2\subset\ldots\subset\cS_{p,1}^{p-2}\subset \cT_{p,1}^{p-2}\subset \cS_{p,1}^{p-1}=\mSh\dd(S_{p,1})\]
where every category is contained as a full dg subcategory in the next one. Then we have the following two lemmas:

\begin{lem}\label{lem:simplified-pinwheel-1}    
    $\cS_{p,1}^i$ is quasi-equivalent to $\cT_{p,1}^{i-1}$ for $i=2,\ldots,p-1$.
\end{lem}
    
\begin{proof}
    Define the dg functor $F\colon \cT_{p,1}^{i-1}\to \cS_{p,1}^i$ as the inclusion. This implies that $H^*F$ is full and faithful. Next, we want show that $H^0F$ is essentially surjective. Pick $(A',f')\in \cS_{p,1}^i$ which we present as in Definition \ref{dfn:simplified-pinwheel}. Define $(A'',f'')\in\cT_{p,1}^{i-1}$ by
    \[u_{1,i-2}(A'',f''):=u_{1,i-2}(A',f')\ ,\quad u_{i+1,p-1}(A'',f''):=u_{i+1,p-1}(A',f')\]
    and $u_{i-2,i+1}(A'',f'')$ is defined by
    \[\begin{tikzcd}
    	A_{i-2}\rar["a_{i-2}"]\dar["f_{i-2}"]\drar["{\fd(\id,f_i)\circ h_{i-2}}"] & [15pt]
    	A_{i-1}\rar["f_i\circ a_{i-1}"]\dar["{\fd(\id,f_i)\circ f_{i-1}}"]\drar["h_{i-1}"] & [5pt]
    	C(l_{p-(i+1)}\circ f_i\circ\bar a_{i-1})\rar["l_{p-(i+1)}"]\dar["\id"]\drar["0"] &
    	C(l_{p-(i+2)}\circ f_i\circ\bar a_{i-1})\dar["\id"] \\
    	C(\bar a_{i-2})\rar["{\fd(\id,f_i\circ a_{i-1})}"] &
    	C(f_i\circ\bar a_{i-1})\rar["l_{p-i}"] &
    	C(l_{p-(i+1)}\circ f_i\circ\bar a_{i-1})\rar["l_{p-(i+1)}"] &
    	C(l_{p-(i+2)}\circ f_i\circ\bar a_{i-1})
    \end{tikzcd}\ .\]
    
    We should comment on why $(A'',f'')\in\cT_{p,1}^{i-1}$, in particular why $f''$ is a homotopy equivalence. It is clear that showing $u_{i-2,i+1}(f'')$ is a homotopy equivalence is enough. For this to be true, by Proposition \ref{prp:quiver-algebra} we need to show that the vertical maps above should be homotopy equivalences, and the squares should commute up to homotopy inside. Since $f'$ is a homotopy equivalence, $f_j$ is a homotopy equivalence for $j=1,\ldots,i$. To see that $\fd(\id,f_i)\colon C(\bar a_{i-1})\to C(f_i\circ\bar a_{i-1})$ is a homotopy equivalence, observe that
    \[\begin{tikzcd}
    	A_1\rar["\bar a_{i-1}"]\dar["\id"]\drar["0"] & A_i\dar["f_i"] \\
    	A_1\rar["f_i\circ\bar a_{i-1}"] & C(\bar a_i)
    \end{tikzcd}\]
    is a homotopy equivalence, hence $C(\id,0,f_i)=\fd(\id,f_i)$ and consequently $\fd(\id,f_i)\circ f_{i-1}$ are homotopy equivalences. Commutativity of squares easily follows from this and the commutativity of squares for $(A',f')$.
    
    We will show that $(A',f')$ is homotopy equivalent to $F(A'',f'')=(A'',f'')$, and hence prove the essential surjectivity of $H^0F$. Define a morphism $(\alpha,\beta)\colon (A',f')\to (A'',f'')$ in the sense of \hyperlink{circle}{Circle Lemma}, such that $\alpha\colon A'\to A''$ is given by
    \[u_{1,i-2}(\alpha):=\id\ ,\quad u_{i+1,p-1}(\alpha):=\id\]
    and $u_{i-2,i+1}(\alpha)$ is
    \[\begin{tikzcd}
    	A_{i-2}\rar["a_{i-2}"]\dar["\id"]\drar["0"] &
    	A_{i-1}\rar["a_{i-1}"]\dar["\id"]\drar["0"] &
    	A_i\rar["l_{p-(i+1)}\circ f_i"]\dar["f_i"]\drar["0"] &
    	C(l_{p-(i+2)}\circ f_i\circ\bar a_{i-1})\dar["\id"] \\
	    A_{i-2}\rar["a_{i-2}"] &
	    A_{i-1}\rar["f_i\circ a_{i-1}"] &
	    C(l_{p-(i+1)}\circ f_i\circ\bar a_{i-1})\rar["l_{p-(i+1)}"] &
	    C(l_{p-(i+2)}\circ f_i\circ\bar a_{i-1})
    \end{tikzcd}\]
    and $\beta\colon A'\to \cox_{p,1}(A'')$ is given by $\beta=0$. Note that then $\cox_{p,1}(\alpha)\colon \cox_{p,1}(A')\to \cox_{p,1}(A'')$ is given by
    \[u_{1,i-2}(\cox_{p,1}(\alpha))=\id\ ,\quad u_{i+1,p-1}(\cox_{p,1}(\alpha))=\id\]
    and $u_{i-2,i+1}(\cox_{p,1}(\alpha))$ is
    \[\begin{tikzcd}
	    C(\bar a_{i-2})\rar["{\fd(\id,a_{i-1})}"]\dar["\id"]\drar["0"] & [15pt]
	    C(\bar a_{i-1})\rar["{\fd(\id,l_{p-(i+1)}\circ f_i)}"]\dar["{\fd(\id,f_i)}"]\drar["0"] & [10pt]
	    C(l_{p-(i+1)}\circ f_i\circ\bar a_{i-1})\rar["l_{p-(i+1)}"]\dar["\id"]\drar["0"] & [-5pt]
	    C(l_{p-(i+2)}\circ f_i\circ\bar a_{i-1})\dar["\id"] \\
	    C(\bar a_{i-2})\rar["{\fd(\id,f_i\circ a_{i-1})}"] &
	    C(f_i\circ\bar a_{i-1})\rar["l_{p-i}"] &
	    C(l_{p-(i+1)}\circ f_i\circ\bar a_{i-1})\rar["l_{p-(i+1)}"] &
	    C(l_{p-(i+2)}\circ f_i\circ\bar a_{i-1})
    \end{tikzcd}\ .\]
    Clearly $\alpha$ is a homotopy equivalence, and $\cox_{p,1}(\alpha)\circ f' - f''\circ\alpha=d\beta$, in particular the $(i-2,i+1)\st$ component of the equation is given by
    \begin{multline*}
 		(\id,0,\fd(\id,f_i),0,\id,0,\id)\circ(f_{i-2},h_{i-2},f_{i-1},h_{i-1},f_i,0,\id)\\
 		\qquad-(f_{i-2},\fd(\id,f_i)\circ h_{i-2},\fd(\id,f_i)\circ f_{i-1},h_{i-1},\id,0,\id)\circ (\id,0,\id,0,f_i,0,\id)=0\ .
    \end{multline*}
    Hence by \hyperlink{circle}{Circle Lemma}, $(\alpha,\beta)$ is a homotopy equivalence, which shows $H^0F$ is essentially surjective. Consequently, $F$ is a quasi-equivalence between $\cT_{p,1}^{i-1}$ and $\cS_{p,1}^i$.
\end{proof}
    
\begin{lem}\label{lem:simplified-pinwheel-2}
    $\cT_{p,1}^i$ is quasi-equivalent to $\cS_{p,1}^i$ for $i=1,\ldots,p-2$.
\end{lem}
    
\begin{proof}
    Define the dg functor $F\colon\cS_{p,1}^i\to\cT_{p,1}^i$ as the inclusion. This implies that $H^*F$ is full and faithful. Next, we want show that $H^0F$ is essentially surjective. Pick $(A',f')\in \cT_{p,1}^i$ which we present as in Definition \ref{dfn:simplified-pinwheel-homotopy}. Define $(A'',f'')\in\cS_{p,1}^i$ by
    \[u_{1,i-1}(A'',f''):=u_{1,i-1}(A',f')\ ,\]
    $u_{i-1,i+1}(A'',f'')$ is defined by
    \[\begin{tikzcd}
	    A_{i-1}\rar["a_{i-1}"]\dar["f_{i-1}"]\drar["{\bar H_i\circ h_{i-1}+\fd(0,\bar h_i)\circ f_{i-1}}"] & [35pt]
	    A_i\rar["l_{p-(i+1)}\circ\bar H_i\circ f_i"]\dar["\bar H_i\circ f_i"]\drar["0"] & [5pt]
	    C(l_{p-(i+2)}\circ\bar H_i\circ f_i\circ \bar a_{i-1})\dar["\id"] \\
	    C(\bar a_{i-1})\rar["{\fd(\id,l_{p-(i+1)}\circ\bar H_i\circ f_i)}"] &
	    C(l_{p-(i+1)}\circ\bar H_i\circ f_i\circ\bar a_{i-1})\rar["l_{p-(i+1)}"] &
	    C(l_{p-(i+2)}\circ\bar H_i\circ f_i\circ \bar a_{i-1})
    \end{tikzcd}\]
    and $u_{i+1,p-1}(A'',f'')$ is defined by
    \[\begin{tikzcd}
	    C(l_{p-(i+2)}\circ\bar H_i\circ f_i\circ \bar a_{i-1})\rar["l_{p-(i+2)}"]\dar["\id"]\drar["0"] & C(l_{p-(i+3)}\circ\bar H_i\circ f_i\circ \bar a_{i-1})\rar["l_{p-(i+3)}"]\dar["\id"]& \cdots\rar["l_1"] & A_1[1]\dar["\id"]\\
	    C(l_{p-(i+2)}\circ\bar H_i\circ f_i\circ \bar a_{i-1})\rar["l_{p-(i+2)}"] & C(l_{p-(i+3)}\circ\bar H_i\circ f_i\circ \bar a_{i-1})\rar["l_{p-(i+3)}"] & \cdots\rar["l_1"] & A_1[1]
    \end{tikzcd}\]
    where $\bar H_i\colon C(\bar a_i)\to C(l_{p-(i+1)}\circ \bar H_i\circ f_i\circ \bar a_{i-1})$ is given by
    \[\bar H_i:=\fs(\id,\bar h_i\circ \bar a_{i-1})\]
    and $\bar h_i\colon A_i\to C(l_{p-(i+2)}\circ \bar H_i\circ f_i\circ \bar a_{i-1})$ is given by $\bar h_i:=\fc(\bar h_i^1,\bar h_i^2,\ldots,\bar h_i^{p-(i+1)})$ with
    \[\bar h_i^l:=\sum_{k=1}^l\sum_{j_1+\ldots+j_k=l}h_i^{j_1}\circ \bar a_{i-1}\circ h_i^{j_2}\circ \bar a_{i-1}\circ\ldots\circ \bar a_{i-1}\circ h_i^{j_k}\]
   	and $h_i=\fc(h_i^1,h_i^2,\ldots,h_i^{p-(i+1)})$ for $l=1,\ldots,p-(i+1)$.  We can also recursively state that
   	\[\bar h_i^l=h_i^l +\sum_{j=1}^{l-1} h_i^{l-j}\circ \bar a_{i-1}\circ \bar h_i^j=h_i^l +\sum_{j=1}^{l-1} \bar h_i^{l-j}\circ \bar a_{i-1}\circ h_i^j\ .\]
   	Note that
   	\[C(l_k\circ \bar H_i\circ f_i\circ \bar a_{i-1})=A_1[1]\oplus\ldots \oplus A_1[1]\]
   	where $A_1[1]$ is repeated $k+1$ times, and
   	\[d_{C(l_k\circ \bar H_i\circ f_i\circ \bar a_{i-1})}=\fs(-d,l_k\circ \bar H_i\circ f_i\circ \bar a_{i-1})\]
   	for $k=1,\ldots,p-(i+1)$.

	To see that $(A'',f'')\in\cS_{p,1}^i$ indeed, note that the only nontrivial thing to show is why the component $u_{i-1,i}(f'')$ is a homotopy equivalence. Observe that we can write $u_{i-1,i}(f'')=\bar\mu\circ u_{i-1,i}(f')$ where
	\[\bar\mu:=\begin{tikzcd}
		C(\bar a_{i-1})\arrow[r,"{\fd(\id,a_i)}"]\arrow[d,"\id"]\arrow[rd,"{\fd(0,\bar h_i)}"] & [40pt]
		C(\bar a_i)\arrow[d,"\bar H_i"]
		\\
		C(\bar a_{i-1})\arrow[r,"{\fd(\id,l_{p-(i+1)}\circ \bar H_i\circ f_i)}"] &
		C(l_{p-(i+1)}\circ \bar H_i\circ f_i\circ \bar a_{i-1})
	\end{tikzcd}\ .\]
	Since $u_{i-1,i}(f')$ is a homotopy equivalence, we only need to show that $\bar\mu$ is a homotopy equivalence, i.e. we need to show $\bar H_i$ is a homotopy equivalence and the above square commutates up to homotopy inside. For that, we define
	\[H_i\colon C(l_{p-(i+1)}\circ \bar H_i\circ f_i\circ \bar a_{i-1})\to C(\bar a_i)\]
	given by
	\[H_i:=\fs(\id,-h_i\circ \bar a_{i-1})\ .\]
	Stating the recursive relation $\bar h_i^l=h_i^l +\sum_{j=1}^{l-1} h_i^{l-j}\circ \bar a_{i-1}\circ \bar h_i^j$ as $h_i=b_{p-(i+1)}(H_i)\circ\bar h_i$ we get
	\begin{align*}
		H_i\circ \bar H_i&=\fs(\id,-h_i\circ \bar a_{i-1})\circ \fs(\id,\bar h_i\circ \bar a_{i-1})\\
		&=\fs(\fb(\id,-h_i\circ \bar a_{i-1},b_{p-(i+1)}(H_i))\circ \fc(\id,\bar h_i\circ \bar a_{i-1}))\\
		&=\fs(\id,-h_i\circ \bar a_{i-1}+\fb_{p-(i+1)}(H_i)\circ\bar h_i\circ \bar a_{i-1})\\
		&=\fs(\id,0)\\
		&=\id
	\end{align*}
	and stating the recursive relation $\bar h_i^l=h_i^l +\sum_{j=1}^{l-1} \bar h_i^{l-j}\circ \bar a_{i-1}\circ h_i^j$ as $\bar h_i=b_{p-(i+1)}(\bar H_i)\circ h_i$ we get
	\begin{align*}
		\bar H_i\circ H_i&=\fs(\id,\bar h_i\circ \bar a_{i-1})\circ \fs(\id,-h_i\circ \bar a_{i-1})\\
		&=\fs(\fb(\id,\bar h_i\circ \bar a_{i-1},b_{p-(i+1)}(\bar H_i))\circ \fc(\id,-h_i\circ \bar a_{i-1}))\\
		&=\fs(\id,\bar h_i\circ \bar a_{i-1}-b_{p-(i+1)}(\bar H_i)\circ h_i\circ \bar a_{i-1}))\\
		&=\fs(\id,0)\\
		&=\id\ .
	\end{align*}
	Also, by inspecting the object $(A',f')\in\cT_{p,1}^i$ presented in Definition \ref{dfn:simplified-pinwheel-homotopy}, we see
	\[d(h_i\circ \bar a_{i-1})=d_{C(\bar a_{i-1})}\circ h_i\circ \bar a_{i-1}+h_i\circ \bar a_{i-1}\circ d\]
	and the commutativity of the diagram up to homotopy gives
	\begin{align*}d(h_i\circ \bar a_{i-1})&=dh_i\circ \bar a_{i-1}\\
		&=(a_i-l_{p-(i+1)}\circ f_i)\circ\bar a_{i-1}\ .
	\end{align*}
	Using this, we get
	\begin{align*}
		dH_i&=d_{C(\bar a_i)}\circ H_i - H_i\circ d_{C(l_{p-(i+1)}\circ \bar H_i\circ f_i\circ \bar a_{i-1})}\\
		&=\fs(-d,\bar a_i)\circ\fs(\id,-h_i\circ \bar a_{i-1})-\fs(\id,-h_i\circ \bar a_{i-1})\circ \fs(-d,l_{p-(i+1)}\circ \bar H_i\circ f_i\circ \bar a_{i-1})\\
		&=\fs(\fb(-d,\bar a_i,d_{C(\bar a_{i-1})})\circ\fc(\id,-h_i\circ \bar a_{i-1}))\\
		&\hspace{7em}-\fs(\fb(\id,-h_i\circ \bar a_{i-1},b_{p-(i+1)}(H_i))\circ \fc(-d,l_{p-(i+1)}\circ \bar H_i\circ f_i\circ \bar a_{i-1}))\\
		&=\fs(0,\bar a_i-d_{C(\bar a_{i-1})}\circ h_i\circ \bar a_{i-1}-h_i\circ \bar a_{i-1}\circ d\\
		&\hspace{18em}-b_{p-(i+1)}(H_i)\circ l_{p-(i+1)}\circ \bar H_i\circ f_i\circ \bar a_{i-1})\\
		&=\fs(0,\bar a_i-l_{p-(i+1)}\circ f_i\circ \bar a_{i-1}-d(h_i\circ\bar a_{i-1}))\\
		&=0\ .
	\end{align*}	
 	$H_i$ is clearly degree $0$, so by Proposition \ref{prp:closed-inverse} $d\bar H_i=0$ also and $\bar H_i$ is a homotopy equivalence. We also have
 	\begin{align*}
	 	d(\fd(0,\bar h_i))&=d_{C(l_{p-(i+1)}\circ \bar H_i\circ f_i\circ \bar a_{i-1})}\circ \fd(0,\bar h_i)+\fd(0,\bar h_i)\circ d_{C(\bar a_{i-1})}\\
	 	&=\fb(-d,l_{p-(i+1)}\circ \bar H_i\circ f_i\circ \bar a_{i-1},d_{C(l_{p-(i+2)}\circ \bar H_i\circ f_i\circ \bar a_{i-1})})\circ\fd(0,\bar h_i)\\
	 	&\hspace{24em}+\fd(0,\bar h_i)\circ\fb(-d,\bar a_{i-1},d)\\
	 	&=\fb(0,\bar h_i\circ \bar a_{i-1},d_{C(l_{p-(i+2)}\circ \bar H_i\circ f_i\circ \bar a_{i-1})}\circ \bar h_i+\bar h_i\circ d)\\
	 	&=\fb(0,\bar h_i\circ \bar a_{i-1},d\bar h_i)
 	\end{align*}
 	and
 	\begin{align*}
		d\bar h_i&=b_{p-(i+1)}(\bar H_i)\circ dh_i\\
		&=b_{p-(i+1)}(\bar H_i)\circ(a_i - l_{p-(i+1)}\circ f_i)\ .
 	\end{align*}
 	gives
 	\begin{align*}
	 	d(\fd(0,\bar h_i))&=\fb(0,\bar h_i\circ \bar a_{i-1},b_{p-(i+1)}(\bar H_i)\circ a_i - l_{p-(i+1)}\circ\bar H_i\circ f_i)\\
	 	&=\bar H_i\circ \fd(\id,a_i)-\fd(\id,l_{p-(i+1)}\circ\bar H_i\circ f_i)\ .
 	\end{align*}
 	This shows that $d\bar\mu=0$ and $\bar\mu$ is a homotopy equivalence. Therefore, $f''$ is a homotopy equivalence, and we have $(A'',f'')\in \cS_{p,1}^i$ indeed.
    
    Next, we will show that $(A',f')$ is homotopy equivalent to $F(A'',f'')=(A'',f'')$, and hence prove the essential surjectivity of $H^0F$. Define, in the sense of \hyperlink{circle}{Circle Lemma}, a morphism $(\alpha,\beta)\colon (A',f')\to (A'',f'')$ such that $\alpha\colon A'\to A''$ is given by
    \[u_{1,i-1}(\alpha):=\id\ ,\]
    $u_{i-1,i+1}(\alpha)$ is
    \[\begin{tikzcd}
	    A_{i-1}\rar["a_{i-1}"]\dar["\id"]\drar["0"] &
	    A_i\rar["a_i"]\dar["\id"]\drar["\bar h_i"] & [30pt]
	    C(l_{p-(i+2)}\circ \bar a_i)\dar["b_{p-(i+1)}(\bar H_i)"] \\
	    A_{i-1}\rar["a_{i-1}"] &
	    A_i\rar["l_{p-(i+1)}\circ\bar H_i\circ f_i"] &
	    C(l_{p-(i+2)}\circ \bar H_i\circ f_i\circ\bar a_{i-1})
    \end{tikzcd}\]
    and $u_{i+1,p-1}(\alpha)$ is
    \[\begin{tikzcd}
    	C(l_{p-(i+2)}\circ\bar a_i)\rar["l_{p-(i+2)}"]\dar["b_{p-(i+1)}(\bar H_i)"]\drar["0"] & C(l_{p-(i+3)}\circ\bar a_i)\rar["l_{p-(i+3)}"]\dar["b_{p-(i+2)}(\bar H_i)"]& \cdots\rar["l_1"] & A_1[1]\dar["\id"]\\
    	C(l_{p-(i+2)}\circ\bar H_i\circ f_i\circ \bar a_{i-1})\rar["l_{p-(i+2)}"] & C(l_{p-(i+3)}\circ\bar H_i\circ f_i\circ \bar a_{i-1})\rar["l_{p-(i+3)}"] & \cdots\rar["l_1"] & A_1[1]
    \end{tikzcd}\]
    and $\beta\colon A'\to \cox_{p,1}(A'')$ is given by $\beta=0$. Note that then $\cox_{p,1}(\alpha)\colon \cox_{p,1}(A')\to \cox_{p,1}(A'')$ is given by
    \[u_{1,i-1}(\cox_{p,1}(\alpha))=\id\ ,\qquad u_{i+1,p-1}(\cox_{p,1}(\alpha))=u_{i+1,p-1}(\alpha)\]
    and $u_{i-1,i+1}(\cox_{p,1}(\alpha))$ is
    \[\begin{tikzcd}
	    C(\bar a_{i-1})\rar["{\fd(\id,a_i)}"]\dar["\id"]\drar["{\fd(0,\bar h_i)}"] & [35pt]
	    C(\bar a_i)\rar["l_{p-(i+1)}"]\dar["\bar H_i"]\drar["0"] &
	    C(l_{p-(i+2)}\circ \bar a_i)\dar["b_{p-(i+1)}(\bar H_i)"] \\
	    C(\bar a_{i-1})\rar["{\fd(\id,l_{p-(i+1)}\circ\bar H_i\circ f_i)}"] &
	    C(l_{p-(i+1)}\circ\bar H_i\circ f_i\circ\bar a_{i-1})\rar["l_{p-(i+1)}"] &
	    C(l_{p-(i+2)}\circ\bar H_i\circ f_i\circ\bar a_{i-1})
    \end{tikzcd}\ .\]
    We see that $\alpha$ is a homotopy equivalence, since $b_k(\bar H_i)$ is a homotopy equivalence for every $k$ and all squares commute up to homotopy inside.
    
    Finally, $\cox_{p,1}(\alpha)\circ f' - f''\circ\alpha=d\beta$, in particular the $(i-1,i+1)\st$ component of the equation is given by
    \begin{multline*}
    	(\id,\fd(0,\bar h_i),\bar H_i,0,b_{p-(i+1)}(\bar H_i))\circ(f_{i-1},h_{i-1},f_i,h_i,\id)\\
    	\qquad-(f_{i-1},\bar H_i\circ h_{i-1}+\fd(0,\bar h_i)\circ f_{i-1},\bar H_i\circ f_i,0,\id)\circ (\id,0,\id,\bar h_i,b_{p-(i+1)}(\bar H_i))=0\ .
    \end{multline*}
    Hence by \hyperlink{circle}{Circle Lemma}, $(\alpha,\beta)$ is a homotopy equivalence, which shows $H^0F$ is essentially surjective. Consequently, $F$ is a quasi-equivalence between $\cS_{p,1}^i$ and $\cT_{p,1}^i$.
\end{proof}
    
\begin{proof}[Proof of Proposition \ref{prp:simplified-pinwheel}]
    It directly follows from Lemma \ref{lem:simplified-pinwheel-1} and \ref{lem:simplified-pinwheel-2}.
\end{proof}

Next, we will simplify the morphisms in $\cS_{p,1}'$:

\begin{dfn}\label{dfn:simple-pinwheel-morphism}
	Define $\cS_{p,1}$ as the dg subcategory of $\cS_{p,1}'$ with the same objects, and a degree $k$ morphism $(\alpha,\beta)\colon(A,f)\to(B,g)$ is given by a degree $k$ morphism $\alpha\colon A\to B$ of the form
	\[\begin{tikzcd}
		A_1\arrow[r,"l_{p-2}\circ f_1"]\arrow[d,"\alpha_1"]\arrow[rd,"l_{p-2}\circ\beta_1" yshift=-0.15cm] &
		C(l_{p-3}\circ f_1)\arrow[r,"l_{p-3}"]\arrow[d,"{\fs(\alpha_1,\l_{p-3}\circ\beta_1)}"]\arrow[rd,"0"] &[50pt]
		C(l_{p-4}\circ f_1)\arrow[r,"l_{p-4}"]\arrow[d,"{\fs(\alpha_1,\l_{p-4}\circ\beta_1)}"] &
		\cdots\arrow[r,"l_2"] &[-10pt]
		C(l_1\circ f_1)\arrow[r,"l_1"]\arrow[d,"{\fs(\alpha_1,\l_1\circ\beta_1)}"']\arrow[rd,"0"] &[-10pt]
		A_1[1]\arrow[d,"\alpha_1"]
		\\
		B_1\arrow[r,"l_{p-2}\circ g_1"] &
		C(l_{p-3}\circ g_1)\arrow[r,"l_{p-3}"] &
		C(l_{p-4}\circ g_1)\arrow[r,"l_{p-4}"] &
		\cdots\arrow[r,"l_2"] &
		C(l_1\circ g_1)\arrow[r,"l_1"] &
		B_1[1]
	\end{tikzcd}\]
	and a degree $k-1$ morphism $\beta\colon A\to \cox_{p,1}(B)$ of the form
	\[\begin{tikzcd}
		A_1\arrow[r,"l_{p-2}\circ f_1"]\arrow[d,"\beta_1"]\arrow[rd,"0"] &
		C(l_{p-3}\circ f_1)\arrow[r,"l_{p-3}"]\arrow[d,"0"]\arrow[rd,"0"] &
		C(l_{p-4}\circ f_1)\arrow[r,"l_{p-4}"]\arrow[d,"0"] &
		\cdots\arrow[r,"l_2"] & [-10pt]
		C(l_1\circ f_1)\arrow[r,"l_1"]\arrow[d,"0"]\arrow[rd,"0"] & [-5pt]
		A_1[1]\arrow[d,"0"]
		\\
		C(l_{p-2}\circ g_1)\arrow[r,"l_{p-2}"] &
		C(l_{p-3}\circ g_1)\arrow[r,"l_{p-3}"] &
		C(l_{p-4}\circ g_1)\arrow[r,"l_{p-4}"] &
		\cdots\arrow[r,"l_2"] &
		C(l_1\circ g_1)\arrow[r,"l_1"] &
		B_1[1]
	\end{tikzcd}\ .\]
	
	A degree $k$ morphism $(\alpha,\beta)$ can be presented by $(\alpha_1,\beta_1)$ with $\beta_1=\fc(\beta_1^1,\ldots,\beta_1^{p-1})$, where $\alpha_1\colon A_1\to B_1$ and $\beta_1^i\colon A_1\to B_1$ are degree $k$. By \hyperlink{circle}{Circle Lemma}, its derivative is given by
	\[d(\alpha_1,\beta_1)=(d\alpha_1,d\beta_1+(-1)^k(g_1\circ\alpha_1-\fs(\alpha_1,l_{p-2}\circ\beta_1)\circ f_1))\]
	which gives
	\begin{align*}
		d(\alpha_1,(\beta_1^i)_{i=1}^{p-1})=(d\alpha_1,(-d\beta_1^i+(-1)^k(g_1^i\circ\alpha_1-&\alpha_1\circ f_1^i)\\
		&+\sum_{j=1}^{i-1}(g_1^j\circ\beta_1^{i-j} -(-1)^k\beta_1^{i-j}\circ f_1^j))_{i=1}^{p-1})\ .
	\end{align*}
	Composition rule is given by
	\[(\alpha_1',\beta_1')\circ(\alpha_1,\beta_1)=(\alpha_1'\circ \alpha_1,\fs(\alpha_1',l_{p-2}\circ\beta_1')\circ \beta_1+(-1)^k \beta'_1\circ \alpha_1)\]
	which gives
	\[(\alpha'_1,(\beta_1'^i)_{i=1}^{p-1})\circ(\alpha_1,(\beta_1^i)_{i=1}^{p-1})=(\alpha_1'\circ \alpha_1,(\alpha_1'\circ\beta_1^i+(-1)^k\beta_1'^i\circ\alpha_1+\sum_{j=1}^{i-1}\beta_1'^{i-j}\circ\beta_1^j)_{i=1}^{p-1})\]
	and the identity is $(\alpha_1,\beta_1)=(\id,0)$.
\end{dfn}

\begin{prp}\label{prp:simplified-pinwheel-morphism}
	$\cS_{p,1}'$ is quasi-equivalent to $\cS_{p,1}$.
\end{prp}

\begin{proof}
	Define the dg functor $F\colon \cS_{p,1}\to\cS_{p,1}'$ as the inclusion. Then clearly $H^0F$ is (essentially) surjective, and $H^*F$ is injective on morphisms. The surjectivity of $H^*F$ on morphisms can be seen as follows: Pick a closed degree $k$ morphism $(\alpha',\beta')\colon (A,f)\to(B,g)$ in $\cS_{p,1}'$ which is expressed by a degree $k$ morphism $\alpha'\colon A\to B$ given by
	\[\begin{tikzcd}
		A_1\arrow[r,"l_{p-2}\circ f_1"]\arrow[d,"\alpha_1"]\arrow[rd,"\delta_1"] & [10pt]
		C(l_{p-3}\circ f_1)\arrow[r,"l_{p-3}"]\arrow[d,"\alpha_2"]\arrow[rd,"\delta_2"] &
		C(l_{p-4}\circ f_1)\arrow[r,"l_{p-4}"]\arrow[d,"\alpha_3"] &
		\cdots\arrow[r,"l_2"] &
		C(l_1\circ f_1)\arrow[r,"l_1"]\arrow[d,"\alpha_{p-2}"]\arrow[rd,"\delta_{p-2}"] &
		A_1[1]\arrow[d,"\alpha_{p-1}"]
		\\
		B_1\arrow[r,"l_{p-2}\circ g_1"] &
		C(l_{p-3}\circ g_1)\arrow[r,"l_{p-3}"] &
		C(l_{p-4}\circ g_1)\arrow[r,"l_{p-4}"] &
		\cdots\arrow[r,"l_2"] &
		C(l_1\circ g_1)\arrow[r,"l_1"] &
		B_1[1]
	\end{tikzcd}\]
	and a degree $k-1$ morphism $\beta'\colon A\to \cox_{p,1}(B)$ given by
	\[\begin{tikzcd}
		A_1\arrow[r,"l_{p-2}\circ f_1"]\arrow[d,"\beta_1"]\arrow[rd,"\varepsilon_1"] & [-5pt]
		C(l_{p-3}\circ f_1)\arrow[r,"l_{p-3}"]\arrow[d,"\beta_2"]\arrow[rd,"\varepsilon_2"] &
		C(l_{p-4}\circ f_1)\arrow[r,"l_{p-4}"]\arrow[d,"\beta_3"] &
		\cdots\arrow[r,"l_2"] & [-10pt]
		C(l_1\circ f_1)\arrow[r,"l_1"]\arrow[d,"\beta_{p-2}"]\arrow[rd,"\varepsilon_{p-2}"] & [-5pt]
		A_1[1]\arrow[d,"\beta_{p-1}"]
		\\
		C(l_{p-2}\circ g_1)\arrow[r,"l_{p-2}"] &
		C(l_{p-3}\circ g_1)\arrow[r,"l_{p-3}"] &
		C(l_{p-4}\circ g_1)\arrow[r,"l_{p-4}"] &
		\cdots\arrow[r,"l_2"] &
		C(l_1\circ g_1)\arrow[r,"l_1"] &
		B_1[1]
	\end{tikzcd}\ .\]
	Then we claim that $(\alpha',\beta')$ is homotopic to $F(\alpha'',\beta'')=(\alpha'',\beta'')$ where $\alpha''\colon A\to B$ is given by
	\[\begin{tikzcd}
		A_1\arrow[r,"l_{p-2}\circ f_1"]\arrow[d,"\alpha_1"]\arrow[rd,"l_{p-2}\circ\kappa_1" yshift=-0.15cm] &
		C(l_{p-3}\circ f_1)\arrow[r,"l_{p-3}"]\arrow[d,"{\fs(\alpha_1,\l_{p-3}\circ\kappa_1)}"]\arrow[rd,"0"] &[50pt]
		C(l_{p-4}\circ f_1)\arrow[r,"l_{p-4}"]\arrow[d,"{\fs(\alpha_1,\l_{p-4}\circ\kappa_1)}"] &
		\cdots\arrow[r,"l_2"] &[-10pt]
		C(l_1\circ f_1)\arrow[r,"l_1"]\arrow[d,"{\fs(\alpha_1,\l_1\circ\kappa_1)}"']\arrow[rd,"0"] &[-10pt]
		A_1[1]\arrow[d,"\alpha_1"]
		\\
		B_1\arrow[r,"l_{p-2}\circ g_1"] &
		C(l_{p-3}\circ g_1)\arrow[r,"l_{p-3}"] &
		C(l_{p-4}\circ g_1)\arrow[r,"l_{p-4}"] &
		\cdots\arrow[r,"l_2"] &
		C(l_1\circ g_1)\arrow[r,"l_1"] &
		B_1[1]
	\end{tikzcd}\]
	and $\beta''\colon A\to \cox_{p,1}(B)$ is given by
	\[\begin{tikzcd}
		A_1\arrow[r,"l_{p-2}\circ f_1"]\arrow[d,"\kappa_1"]\arrow[rd,"0"] &
		C(l_{p-3}\circ f_1)\arrow[r,"l_{p-3}"]\arrow[d,"0"]\arrow[rd,"0"] &
		C(l_{p-4}\circ f_1)\arrow[r,"l_{p-4}"]\arrow[d,"0"] &
		\cdots\arrow[r,"l_2"] & [-10pt]
		C(l_1\circ f_1)\arrow[r,"l_1"]\arrow[d,"0"]\arrow[rd,"0"] & [-5pt]
		A_1[1]\arrow[d,"0"]
		\\
		C(l_{p-2}\circ g_1)\arrow[r,"l_{p-2}"] &
		C(l_{p-3}\circ g_1)\arrow[r,"l_{p-3}"] &
		C(l_{p-4}\circ g_1)\arrow[r,"l_{p-4}"] &
		\cdots\arrow[r,"l_2"] &
		C(l_1\circ g_1)\arrow[r,"l_1"] &
		B_1[1]
	\end{tikzcd}\]
	where $\kappa_1$ is defined as follows: Let $\lambda_{p-2}:=(-1)^{k-1}\varepsilon_{p-2}$, and for $p-2>i>1$ recursively we define
	\[\lambda_i:=(-1)^{k-1}\varepsilon_i + \fd(0,\lambda_{i+1})\]
	and
	\[\lambda_1:=(-1)^{k-1}\varepsilon_1+\fd(0,\lambda_2)\circ f_1\ .\]
	Also, we define $\kappa_{p-1}:=(-1)^{k-1}\beta_{p-1}$, and for $p-1>i>1$ recursively we define
	\[\kappa_i:=(-1)^{k-1}\beta_i+\fb(0,\chi(\lambda)_1^i,\kappa_{i+1})\]
	and
	\[\kappa_1:=\beta_1+(-1)^{k-1}\fb(0,\lambda_1,\kappa_2)\circ f_1\]
	where $\chi(\lambda)_j^i$ is as defined in Definition \ref{dfn:coxeter}. We define a degree $k-1$ morphism $(\kappa,\tau)\colon (A,f)\to(B,g)$ where $\kappa\colon A\to B$ is given by
	\[\begin{tikzcd}
		A_1\arrow[r,"l_{p-2}\circ f_1"]\arrow[d,"0"]\arrow[rd,"\lambda_1"] & [10pt]
		C(l_{p-3}\circ f_1)\arrow[r,"l_{p-3}"]\arrow[d,"\kappa_2"]\arrow[rd,"\lambda_2"] &
		C(l_{p-4}\circ f_1)\arrow[r,"l_{p-4}"]\arrow[d,"\kappa_3"] &
		\cdots\arrow[r,"l_2"] &
		C(l_1\circ f_1)\arrow[r,"l_1"]\arrow[d,"\kappa_{p-2}"]\arrow[rd,"\lambda_{p-2}"] &
		A_1[1]\arrow[d,"\kappa_{p-1}"]
		\\
		B_1\arrow[r,"l_{p-2}\circ g_1"] &
		C(l_{p-3}\circ g_1)\arrow[r,"l_{p-3}"] &
		C(l_{p-4}\circ g_1)\arrow[r,"l_{p-4}"] &
		\cdots\arrow[r,"l_2"] &
		C(l_1\circ g_1)\arrow[r,"l_1"] &
		B_1[1]
	\end{tikzcd}\]
	and $\tau:=0$. We want to show that
	\[d(\kappa,\tau)=(d\kappa,d\tau+(-1)^{k-1}(g\circ\kappa-\cox_{p,1}(\kappa)\circ f))=(\alpha',\beta')-(\alpha'',\beta'')\ .\]
	Note that $\kappa=(-1)^{k-1}(\beta'-\beta'')+\cox_{p,1}(\kappa)\circ f$. Using this, we get
	\[d\tau+(-1)^{k-1}(g\circ\kappa-\cox_{p,1}(\kappa)\circ f)=\beta'-\beta''\]
	since $g\circ\kappa=\kappa$. Hence only thing left to show is $d\kappa=\alpha'-\alpha''$. To prove this, use the fact that $(\alpha',\beta')$ is closed, in particular this implies
	\[g\circ\alpha'=(-1)^{k-1}d\beta'+\cox_{p,1}(\alpha')\circ f\ .\]
	Also, we have
	\[g\circ d\kappa=(-1)^{k-1}(d\beta'-d\beta'')+\cox_{p,1}(d\kappa)\circ f\]
	which, when combined with above equality, gives
	\[g\circ \xi=(-1)^{k-1}d\beta''+\cox_{p,1}(\xi)\circ f\]
	where we define $\xi:=\alpha'-d\kappa$. Note that $(2,p-1)\st$ component of the above equality gives
	\[u_{2,p-1}(\xi)=u_{2,p-1}\circ \cox_{p,1}(\xi)\ .\]
	Let $u_{1,2}=(x,y,z)$. Then we get
	\begin{align*}
		u_{p-2,p-1}(\xi)&=(\fs(x,l_1\circ y),0,x)\\
		u_{i,i+1}(\xi)&=(\fs(x,l_{p-i-1}\circ y),0,\fs(x,l_{p-i-2}\circ y))&&\text{for }p-2>i>1
	\end{align*}
	which gives $z=\fs(x,l_{p-3}\circ y)$. Also,
	\[x=\alpha_{p-1}-d\kappa_{p-1}=\alpha_{p-1}-(-1)^{k-1}d\beta_{p-1}=\alpha_1\]
	and going back to the equality $g\circ \xi=(-1)^{k-1}d\beta''+\cox_{p,1}(\xi)\circ f$, we get
	\[y=l_{p-2}\circ\kappa_1\]
	which shows that $\xi=\alpha''$. Hence $d(\kappa,\tau)=(\alpha',\beta')-(\alpha'',\beta'')$ and $H^*F$ is surjective on morphisms. This completes the proof.
\end{proof}

Next, we will explicitly present the monodromy functor $\mon\colon\cS_{p,1}\to\Loc\dd(S^1)$ for $\cS_{p,1}$:

\begin{prp}\label{prp:monodromy-morphism}
	Given $(A,f)\in\cS_{p,1}$, we have $\mon(A,f)=(A_1,\mon(A,f)_0)\in\Loc\dd(S^1)$ where the monodromy $\mon(A,f)_0\colon A_1\to A_1$ is given by
	\[\mon(A,f)_0=\sum_{j=1}^{p-1} f_1^{p-j}\circ f_1^j\]
	for $f_1=\fc(f_1^1,\ldots,f_1^{p-1})$. Moreover, if $(\alpha,\beta)\colon(A,f)\to(B,g)$ is a degree $k$ morphism, then
	\begin{align*}
		\mon(\alpha,\beta)_1&=\alpha_1\\
		\mon(\alpha,\beta)_0&=\sum_{j=1}^{p-1} g_1^j\circ\beta_1^{p-j} - (-1)^k\sum_{j=1}^{p-1} \beta_1^{p-j} f_1^j\ .
	\end{align*}
	for $g_1=\fc(g_1^1,\ldots,g_1^{p-1})$ and $\beta_1=\fc(\beta_1^1,\ldots,\beta_1^{p-1})$.
\end{prp}

\begin{proof}
	By Proposition \ref{prp:monodromy} we get
	\[\mon(A,f)_0=\id\circ\eta_2(f)\circ\id\circ\eta_{3,p-3}(f)\circ\id\circ\ldots\circ\id\circ\eta_{3,1}(f)\circ\fb(f_1,0,\id)\circ f_1\]
	where $\eta_2(f)=\fr(0,\id,l_1\circ f_1)$ and
	\[\eta_{3,i}(f)=\mx{0 & I_{p-(i+1)} & l_{p-(i+1)}\circ f_1 & 0 \\ 0 & 0 & 0 & I_{p-(i+2)}}\]
	for $i=1,\ldots,p-3$. We have
	\[\eta_2(f)\circ\eta_{3,p-3}(f)\circ\ldots\circ\eta_{3,1}(f)=\fr(0,\id,f_1^{p-2},\ldots,f_1^1)\]
	and
	\[\fb(f_1,0,\id)\circ f_1=\fc(f_1^1\circ f_1^1,f_1^2\circ f_1^1,\ldots,f_1^{p-1}\circ f_1^1,f_1^2,f_1^3,\ldots,f_1^{p-1})\]
	which gives
	\[\mon(A,f)_0=\sum_{j=1}^{p-1} f_1^{p-j}\circ f_1^j\ .\]
	To find $\mon(\alpha,\beta)$, as in the proof of Proposition \ref{prp:monodromy}, we compose $f$ and $g$ with themselves (after translation) $p$ times, and extend the morphism $(\alpha,\beta)$ between $(A,f)$ and $(B,g)$ to make it a morphism between these compositions. We get the diagram
	\begin{align*}
		\begin{tikzcd}[ampersand replacement=\&]
			A\rar["f"]\dar["\alpha"]\drar["\beta"] \&
			\cox_{p,1}(A)\rar["\cox_{p,1}(f)"]\dar["\cox_{p,1}(\alpha)"]\drar["\cox_{p,1}(\beta)"] \&
			\cox_{p,1}^2(A)\rar["\sim"]\dar["\cox_{p,1}^2(\alpha)"]\drar[gray,dashed] \& 
			\cox_{p,2}(A)\rar["\cox_{p,2}(f)"]\dar["\cox_{p,2}(\alpha)"]\drar["\cox_{p,2}(\beta)"] \&
			\cox_{p,2}\circ \cox_{p,1}(A)\rar["\sim"]\dar["\cox_{p,2}\circ \cox_{p,1}(\alpha)"] \&
			\cdots
			\\
			B\rar["g"] \&
			\cox_{p,1}(B)\rar["\cox_{p,1}(g)"] \&
			\cox_{p,1}^2(B)\rar["\sim"] \& 
			\cox_{p,2}(B)\rar["\cox_{p,2}(g)"] \&
			\cox_{p,2}\circ \cox_{p,1}(A)\rar["\sim"] \&
			\cdots
		\end{tikzcd}\hspace{-27em}&
		\\ 
		&\begin{tikzcd}[ampersand replacement=\&]
			\cdots \rar["\cox_{p,p-2}(f)"] \&
			\cox_{p,p-2}\circ \cox_{p,1}(A)\rar["\sim"]\dar["\cox_{p,p-2}\circ \cox_{p,1}(\alpha)"]\drar[gray,dashed] \&
			\cox_{p,p-1}(A)\rar["\cox_{p,p-1}(f)"]\dar["\cox_{p,p-1}(\alpha)"]\drar["\cox_{p,p-1}(\beta)"] \&
			\cox_{p,p-1}\circ \cox_{p,1}(A)\rar["\sim"]\dar["\cox_{p,p-1}\circ \cox_{p,1}(\alpha)"]\drar[gray,dashed] \&
			A\dar["\alpha"]
			\\
			\cdots \rar["\cox_{p,p-2}(g)"] \&
			\cox_{p,p-2}\circ \cox_{p,1}(B)\rar["\sim"] \&
			\cox_{p,p-1}(B)\rar["\cox_{p,p-1}(g)"] \&
			\cox_{p,p-1}\circ \cox_{p,1}(B)\rar["\sim"] \&
			B
		\end{tikzcd}
	\end{align*}
	where we omitted the labels for the diagonal dashed arrows. Next, we restrict this diagram to the first edge $e_1$, i.e. we apply the functor $j_1$ on this diagram, and get
	\begin{align*}
		\begin{tikzcd}[ampersand replacement=\&]
			A_1\rar["f_1"]\dar["\alpha_1"]\drar["\beta_1"] \&
			C(l_{p-2}\circ f_1)\rar["{\fb(f_1,0,\id)}"]\dar["j_2(\alpha)"]\drar["{\fb(\beta_1,0,0)}"] \&
			C(l_{p-2})\rar["\eta_{3,1}(f)"]\dar["j_2(\cox_{p,1}(\alpha))"]\drar["\xi_{3,1}"] \& [25pt]
			C(l_{p-3})\rar["\id"]\dar["j_3(\alpha)"]\drar["0"] \&
			C(l_{p-3})\rar["\eta_{3,2}(f)"]\dar["j_3(\alpha)"] \& 
			\cdots
			\\
			B_1\rar["g_1"] \&
			C(l_{p-2}\circ g_1)\rar["{\fb(g_1,0,\id)}"] \&
			C(l_{p-2})\rar["\eta_{3,1}(g)"] \& 
			C(l_{p-3})\rar["\id"] \&
			C(l_{p-3})\rar["\eta_{3,2}(g)"] \& 
			\cdots
		\end{tikzcd}\hspace{-25em}&
		\\ 
		&\begin{tikzcd}[ampersand replacement=\&]
			\cdots \rar["\id"] \&
			C(l_2)\rar["\eta_{3,p-3}(f)"]\dar["j_{p-2}(\alpha)"]\drar["\xi_{3,p-3}"] \& [15pt]
			C(l_1)\rar["\id"]\dar["j_{p-1}(\alpha)"]\drar["0"] \& [15pt]
			C(l_1)\rar["\eta_2(f)"]\dar["j_{p-1}(\alpha)"]\drar["\xi_2"] \& [20pt]
			A_1\dar["\alpha_1"]
			\\
			\cdots \rar["\id"] \&
			C(l_2)\rar["\eta_{3,p-3}(g)"] \&
			C(l_1)\rar["\id"] \&
			C(l_1)\rar["\eta_2(g)"] \&
			B_1
		\end{tikzcd}
	\end{align*}
	where $\xi_{3,i}$ and $\xi_2$ are given by Lemma \ref{lem:adjusting} as
	\begin{align*}
		\xi_{3,1}&=(-1)^k\eta_{3,1}(g)\circ j_2(\cox_{p,1}(\alpha))\circ \zeta_{3,1}(f) \\
		\xi_{3,i}&=(-1)^k\eta_{3,i}(g)\circ j_{i+1}(\alpha)\circ\zeta_{3,i}(f) &&\text{for }2\leq i\leq p-3\\
		\xi_2&=(-1)^k\eta_2(g)\circ j_{p-1}(\alpha)\circ\zeta_2(f)
	\end{align*}
	and $\zeta_{3,i}(f)$ and $\zeta_2(f)$ determined by the relations $d\zeta_{3,i}(f)=\eta'_{3,i}(f)\circ\eta_{3,i}(f)-\id$ and $d\zeta_2(f)=\eta'_2(f)\circ\eta_2(f)-\id$ where $\eta_{3,i}'(f)$ and $\eta_2'(f)$ are the inverses of $\eta_{3,i}(f)$ and $\eta_2(f)$, respectively. By inspecting the proofs of Lemma \ref{lem:zero}, \ref{lem:zero-cone}, \ref{lem:cone-exact}, and \ref{lem:cone-triple}, $\zeta_{3,1}(f)$ and $\zeta_2(f)$ can be explicitly given by
	\begin{align*}
		\zeta_{3,i}(f)&=\mx{0_{1,p-i} & -\id & 0_{1,p-2-i} \\ 0_{2p-2-2i,p-i} & 0_{2p-2-2i,1} & 0_{2p-2-2i,p-2-i}} &&\text{for }1\leq i\leq p-3\\
		\zeta_2(f)&=\mx{0_{1,2} & -\id \\ 0_{2,2} & 0_{2,1}}
	\end{align*}
	where $0_{n,m}$ is the $n\times m$ zero matrix. Finally, by ``horizontally composing'' the above diagram, in particular by composing the diagonals as in Definition \ref{dfn:coxeter}, we get the morphism $\mon(\alpha,\beta)\colon (A_1,\mon(A,f)_0)\to (B_1,\mon(B,g)_0)$ as
	\[\begin{tikzcd}
		\mon(A,f)\dar["{\mon(\alpha,\beta)}"',"=" xshift=1cm] & A_1\rar["{\mon(A,f)_0}"]\dar["\alpha_1"]\drar["\chi(\beta_1)"] & [20pt]
		A_1\dar["\alpha_1"]\\
		\mon(B,g) &
		B_1\rar["{\mon(B,g)_0}"] &
		B_1
	\end{tikzcd}\]
	where
	\begin{multline*}
		\chi(\beta_1)=\eta_2(g)\circ\eta_{3,p-3}(g)\circ\ldots\circ\eta_{3,1}(g)\circ(\fb(g_1,0,\id)\circ \beta_1+\fb(\beta_1,0,0)\circ f_1)+\\
		+\sum_{i=1}^{p-3} \eta_2(g)\circ\eta_{3,p-3}(g)\circ\ldots\circ\eta_{3,i+1}(g)\circ\xi_{3,i}\circ\eta_{3,i-1}(f)\circ\ldots\circ\eta_{3,1}(f)\circ\fb(f_1,0,0)\circ f_1 +\\
		+\xi_2\circ\eta_{3,p-3}(f)\circ\ldots\circ\eta_{3,1}(f)\circ \fb(f_1,0,0)\circ f_1
	\end{multline*}
	which gives
	\[\chi(\beta_1)=\sum_{j=1}^{p-1} g_1^j\circ\beta_1^{p-j} - (-1)^k\sum_{j=1}^{p-1} \beta_1^{p-j} f_1^j\ .\]
	This proves the proposition.
\end{proof}

Now we are ready to prove our main theorem:

\begin{proof}[Proof of Theorem \ref{thm:msh-pinwheel}]
	By Proposition \ref{prp:simplified-pinwheel} and \ref{prp:simplified-pinwheel-morphism}, we have
	\begin{align*}
		\mSh\dd(L_{p,1})&\simeq\{(A,f,\gamma)\vb (A,f)\in\cS_{p,1},\mon(A,f)_0\overset{\gamma}{\sim}\id\}\\
		&\simeq\{(A,f,\gamma)\vb A\in\Modk(A_{p-1}),f\colon A\to \cox_{p,1}(A)\text{ is h.e.},\mon(A,f)_0\overset{\gamma}{\sim}\id\}
	\end{align*}
	where $(A,f)\in\cS_{p,1}$ is represented as in Definition \ref{dfn:simple-pinwheel}, and the morphisms are given by \hyperlink{disk}{Disk Lemma}. Note that $f$ is a homotopy equivalence if and only if $f_1\colon A_1\to C(l_{p-2}\circ f_1)$ is a homotopy equivalence. Moreover, $f_1$ is a homotopy equivalence if and only if $|f_1|=0$, $df_1=0$ and $C(f_1)\simeq 0$ by Proposition \ref{prp:cone-zero}. This gives
	\begin{align*}
		\mSh\dd(L_{p,1})\simeq\{(A_1,f_1,\gamma)\vb A_1\in\Modk,f_1\in\Hom^0(A_1,&C(l_{p-2}\circ f_1)),\\
		&df_1=0,C(f_1)\simeq 0,\mon(A,f)_0\overset{\gamma}{\sim}\id\}\ .
	\end{align*}
	Our next task is describe the conditions $df_1=0$, $\mon(A,f)_0\overset{\gamma}{\sim}\id$, and $C(f_1)\simeq 0$. Note that we have $f_1=\fc(f_1^1,f_1^2,\ldots,f_1^{p-1})$. We consider $f_1^i\colon A_1\to A_1$ as degree $1$ morphisms. Then $df_1=0$ reads as
	\[df_1^i=\sum_{j=1}^{i-1} f_1^{i-j}\circ f_1^j\]
	for $i=1,\ldots,p-1$. To understand the degree $1$ morphism $\gamma\colon A_1\to A_1$, recall that by Proposition \ref{prp:monodromy-morphism} we have
	\[\mon(A,f)_0=\sum_{j=1}^{p-1} f_1^{p-j}\circ f_1^j\ .\]
	This implies that
	\[d\gamma=-\id+\sum_{j=1}^{p-1} f_1^{p-j}\circ f_1^j\ .\]
	If we denote $f_1^p:=\gamma$, then we can compactly write
	\[df_1^i=-\delta_{i,p}+\sum_{j=1}^{i-1} f_1^{i-j}\circ f_1^j\ .\]
	for $i=1,\ldots,p$, where $\delta_{ij}$ is Kronecker delta, and get
	\begin{align*}
		\mSh\dd(L_{p,1})\simeq\{(A_1,(f_1^i)_{i=1}^p)\vb A_1\in\Modk,&f_1^i\in\Hom^1(A_1,A_1),\\
		&df_1^i=-\delta_{i,p}+\sum_{j=1}^{i-1} f_1^{i-j}\circ f_1^j,C(f_1)\simeq 0\}\ .
	\end{align*}
	
	Now we can discuss the morphisms in $\mSh\dd(L_{p,1})$. By \hyperlink{disk}{Disk Lemma}, a degree $k$ morphism $(\alpha,\beta,\beta_1^p)\colon (A,f,f_1^p)\to(B,g,g_1^p)$ in $\mSh\dd(L_{p,1})$ is given by a degree $k$ morphism $(\alpha,\beta)\colon (A,f)\to (B,g)$ in $\cS_{p,1}$ as in Definition \ref{dfn:simple-pinwheel-morphism} and a degree $k$ morphism $\beta_1^p\colon A_1\to B_1$ in $\Modk$. Note that by Proposition \ref{prp:monodromy-morphism} we know
	\begin{align*}
		\mon(\alpha,\beta)_1&=\alpha_1\\
		\mon(\alpha,\beta)_0&=\sum_{j=1}^{p-1} g_1^j\circ\beta_1^{p-j} - (-1)^k\sum_{j=1}^{p-1} \beta_1^{p-j} f_1^j\ .
	\end{align*}
	As remarked at the end of Definition \ref{dfn:simple-pinwheel-morphism}, we can equivalently consider
	\[(\alpha_1,(\beta_1^i)_1^p)\colon (A_1,(f_1^i)_{i=1}^p) \to (B_1,(g_1^i)_{i=1}^p)\]
	in place of $(\alpha,\beta,\beta_1^p)$ where $\alpha_1,\beta_1^1,\ldots,\beta_1^p\colon A_1\to B_1$ are degree $k$. Then by \hyperlink{disk}{Disk Lemma}, the differential is given by
	\begin{align*}
		d(\alpha_1,(\beta_1^i)_{i=1}^{p-1},\beta_1^p)&=(d(\alpha_1,(\beta_1^i)_{i=1}^{p-1}),d\beta_1^p+g_1^p\circ \mon(\alpha,\beta)_1+(-1)^k(\mon(\alpha,\beta)_0-\mon(\alpha,\beta)_1\circ f_1^p)\\
		&\hspace{-5em}=(d\alpha_1,(-d\beta_1^i+(-1)^k(g_1^i\circ\alpha_1-\alpha_1\circ f_1^i)+\sum_{j=1}^{i-1}(g_1^j\circ\beta_1^{i-j} -(-1)^k\beta_1^{i-j}\circ f_1^j))_{i=1}^{p-1},\\
		&\hspace{2em}d\beta_1^p+g_1^p\circ \alpha_1-(-1)^k\alpha_1\circ f_1^p+(-1)^k\sum_{j=1}^{p-1} (g_1^j\circ\beta_1^{p-j} - (-1)^k \beta_1^{p-j} f_1^j))\ ,
	\end{align*}
	the composition is given by
	\begin{align*}
		(\alpha'_1,(\beta_1'^i)_{i=1}^{p-1},\beta_1'^p)\circ(\alpha_1,(\beta_1^i)_{i=1}^{p-1},\beta_1^p)&=((\alpha'_1,(\beta_1'^i)_{i=1}^{p-1})\circ(\alpha_1,(\beta_1^i)_{i=1}^{p-1}),\\
		&\hspace{8em}\mon(\alpha',\beta')_1\circ \beta_1^p +(-1)^k \beta_1'^p\circ \mon(\alpha,\beta)_1)\\
		&\hspace{-11em}=(\alpha_1'\circ \alpha_1,(\alpha_1'\circ\beta_1^i+(-1)^k\beta_1'^i\circ\alpha_1+\sum_{j=1}^{i-1}\beta_1'^{i-j}\circ\beta_1^j)_{i=1}^{p-1},\alpha_1'\circ \beta_1^p +(-1)^k \beta_1'^p\circ \alpha_1)\ ,
	\end{align*}
	and the identity is $(\alpha_1,(\beta_1^i)_{i=1}^p)=(\id,0)$. This shows that if we define $\cA'_{p,1}$ as the semifree dga generated by the degree $1$ elements $x_i$ for $p\geq i\geq 1$, where
	\[dx_i=-\delta_{i,p}+\sum_{j=1}^{i-1}x_{i-j}\circ x_j\]
	we get
	\[\mSh\dd(L_{p,1})\simeq\{(A_1,(f_1^i)_{i=1}^p)\in\Modk(\cA'_{p,1})\vb C(f_1)\simeq 0\}\ .\]
	
	Lastly, we will describe $C(f_1)\simeq 0$.  By \cite{drinfeld}, it can be described by adding a morphism $\varepsilon$ of degree $1$ to $\Hom(C(f_1),C(f_1))$ such that $d\varepsilon=\id_{C(f_1)}$ where we do not add new relations between morphisms. This will introduce the morphisms $v\circ\varepsilon\circ u$ in $\Hom(A_1,A_1)$, for any $u\in\Hom(A_1,C(f_1))$ and $v\in\Hom(C(f_1),A_1)$. Note that $u\in\Hom(A_1,C(f_1))$ can be decomposed as
	\[u=\sum_{j=1}^p u^j\circ U^j\]
	where $u^j\colon A_1\to C(f_1)$ such that $u^j=\fc(0,\ldots,0,\id,0,\ldots,0)$ whose $j\th$ term is $\id$, and $U^j\in\Hom(A_1,A_1)$. Similarly, $v\in\Hom(C(f_1),A_1)$ can be decomposed as
	\[v=\sum_{j=1}^p V^j\circ v^j\]
	where $v^j\colon C(f_1)\to A_1$ such that $v^j=\fr(0,\ldots,0,\id,0,\ldots,0)$ whose $j^{\text{th}}$ term is $\id$, and $V^j\in\Hom(A_1,A_1)$. Hence, we just add the morphisms
	\[g^{ij}:=-v^i\circ\varepsilon\circ u^j\]
	for $1\leq i,j\leq p$ in $\Hom(A_1,A_1)$ where we do not add new relations between morphisms. Note that $u^i$, $v^i$, and $g^{ij}$ are degree $1$, and
	\begin{align*}
		du^j&=\sum_{k=j+1}^p u^k\circ f_1^{k-j}\\
		dv^i&=\sum_{k=1}^{i-1} f_1^{i-k}\circ v^k
	\end{align*}
	and consequently
	\[dg^{ij}=\delta_{ij}+\sum_{k=1}^{i-1} f_1^{i-k}\circ g^{kj}+\sum_{k=j+1}^p g^{ik}\circ f_1^{k-j}\ .\]
	Therefore, we get
	\begin{align*}
		\mSh\dd(L_{p,1})\simeq\{&(A,(f_1^i)_{i=1}^p,(g^{ij})_{1\leq i,j\leq p}) \vb A_1\in\Modk, f_1^i,g^{ij}\in\Hom^1(A_1,A_1),\\
		&df_1^i=-\delta_{i,p}+\sum_{j=1}^{i-1} f_1^{i-j}\circ f_1^j, dg^{ij}=\delta_{ij}+\sum_{k=1}^{i-1} f_1^{i-k}\circ g^{kj}+\sum_{k=j+1}^p g^{ik}\circ f_1^{k-j}\}
	\end{align*}
	or alternatively,
	\[\mSh\dd(L_{p,1})\simeq\Modk(\cA_{p,1})\]
	where $\cA_{p,1}$ is defined as in Definition \ref{dfn:dga-rational}. Their morphisms also match as explained before. By Proposition \ref{prp:mod-perf-microlocal}, we get
	\begin{align*}
		\mSh(L_{p,1})&\simeq\Perfk(\cA_{p,1})\\
		\mSh^w(L_{p,1})&\simeq\Perf(\cA_{p,1})\ .
	\end{align*}
\end{proof}

\begin{cor}\label{cor:bp1-lp1}
	Let $\k$ be a field of characteristic zero. We have the $A_{\infty}$-quasi-equivalence
	\[\mSh^w(L_{p,1})\simeq\cW(B_{p,1})\]
	for $p\geq 3$, confirming Conjecture \ref{con:microlocal-wrapped} for $B_{p,1}$.
\end{cor}

\begin{rmk}
	There is a hidden choice when gluing pieces of $L_{p,1}$, namely we can choose shift of the Coxeter functor, see the discussion in Section \ref{sec:microlocal}. If $p$ is even, we expect to get a different category shown in Remark \ref{rmk:bp1-grading}, because there is an alternative grading structure on $B_{p,1}$ if $p$ is even. Hence we expect that for the alternative grading of $B_{p,1}$, Corollary \ref{cor:bp1-lp1} still holds.
\end{rmk}

\bibliographystyle{amsplain}
\bibliography{thesisbibliography}

\end{document}